\documentclass[a4paper,10pt,reqno]{article}

\usepackage[english]{babel}
\usepackage{amsmath}
\usepackage{amsfonts}
\usepackage{amssymb}
\usepackage{amsthm}
\usepackage{float}
\usepackage{graphics}
\usepackage{color}
\usepackage{graphicx}
\usepackage[colorlinks=false]{hyperref}
\usepackage{comment}
\usepackage{vmargin}
\usepackage[T1]{fontenc}
\usepackage{times}
\def\Z{\mathbb{Z}}

\newcommand{\la}{\lambda}
\newcommand{\bt}{\beta}

\newcommand{\op}{{\rm Op}}

\newcommand{\lr}[1]{{\langle{#1}\rangle}}

\def\f{{\varphi}}
\def\im{{\mathtt i}}
\def\g{{\gamma}}

\def\pa{\partial}

\def\cO{\mathcal{O}}
\def\cC{\mathcal{C}}
\def\e{\varepsilon}
\def\N{{\mathbb N}}
\renewcommand{\to}{\rightarrow}
\newcommand{\tk}{{\mathtt k}}

\usepackage{dsfont}

\numberwithin{equation}{section}

\theoremstyle{plain}
\newtheorem{teor}{Theorem}[section]
\newtheorem{teor1}{Theorem}
\newtheorem{ese}[teor]{Example}
\newtheorem{prop}[teor]{Proposition}
\newtheorem{lem}[teor]{Lemma}
\newtheorem{cor}[teor1]{Corollary}
\newtheorem{coro}[teor]{Corollary}

\newcommand{\bdm}{\begin{displaymath}}
\newcommand{\edm}{\end{displaymath}}
\newcommand{\bpb}{\begin{prob}}
\newcommand{\epb}{\end{prob}}

\newcommand{\beq}{\begin{equation}}
\newcommand{\eeq}{\end{equation}}
\newcommand{\bem}{\begin{multline}}
\newcommand{\eem}{\end{multline}}
\newcommand{\bes}{\begin{ese}}
\newcommand{\ees}{\end{ese}}
\newcommand{\bde}{\begin{defi}}
\newcommand{\ede}{\end{defi}}
\newcommand{\bpr}{\begin{prop}}
\newcommand{\epr}{\end{prop}}
\newcommand{\ble}{\begin{lem}}
\newcommand{\ele}{\end{lem}}
\newcommand{\bte}{\begin{teor}}
\newcommand{\ete}{\end{teor}}
\newcommand{\bco}{\begin{cor}}
\newcommand{\eco}{\end{cor}}

\newtheorem{defi}[teor]{Definition}
\newtheorem{remark}[teor]{Remark}

\usepackage{geometry}
\newcommand{\su}{\s_1}
\newcommand{\lD}{{\langle D_x\rangle}}

\newcommand{\R}{\mathbb{R}}

\newcommand{\T}{\mathbb{T}}

\newcommand{\calL}{{\mathcal L}}

\newcommand{\calO}{{\mathcal O}}
\newcommand{\calP}{{\mathcal P}}
\newcommand{\calQ}{{\mathcal Q}}
\newcommand{\RR}{{\mathcal R}}

\newcommand{\gotd}{{\mathfrak d}}

\newcommand{\gotL}{{\mathfrak L}}
\newcommand{\gotM}{{\mathfrak M}}
\newcommand{\gotN}{{\mathfrak N}}

\newcommand{\und}{\underline}

\newcommand{\ol}{\overline}

\newcommand{\al}{\alpha}
\newcommand{\be}{\beta}

\newcommand{\x}{\xi}
\newcommand{\ka}{\kappa}
\newcommand{\s}{\sigma}

\newcommand{\del}{\partial}

\newcommand{\oo}{\omega}

\newcommand{\ta}{\mathtt{a}}
\newcommand{\tb}{\mathtt{b}}

\newcommand{{\resonance}}{relevant self-energy cluster }

\newcommand{\un}{\underline}

\def\cD{{\mathcal D}}
\def\cP{{\mathcal P}}

\def\cA{{\mathcal A}}
\def\cC{{\mathcal C}}
\def\cQ{{\mathcal Q}}
\def\cP{{\mathcal P}}

\def\cD{{\mathcal D}}
\newcommand{\tq}{{\mathtt q}}

\begin{document}

\title{\bf Reducibility for a class of weakly dispersive  linear operators arising from the Degasperis Procesi  equation}
\date{}

\author{\bf 
R. Feola$^{**}$, F. Giuliani$^\dag $, M. Procesi$^\dag$
\\
\small
${}^{**}$ SISSA, Trieste, rfeola@sissa.it; 
\\
\small
${}^\dag$ RomaTre, Roma, procesi@mat.uniroma3.it, fgiuliani@mat.uniroma3.it\footnote{
This research was supported by PRIN 2015 ``Variational methods, with applications to problems in 
mathematical physics and geometry'' and by 
ERC grant ``Hamiltonian PDEs and small divisor problems: a dynamical systems approach n. 306414 under FP7''.
}}

\maketitle

\begin{abstract}
We prove reducibility of a class of quasi-periodically forced linear equations of the form
\[
\partial_tu-\partial_x\circ (1+a(\omega t, x))u+\mathcal{Q}(\omega t)u=0 \qquad  x\in\T:=\mathbb{R}/2\pi\mathbb{Z},
\]
where $u=u(t,x)$, $a$ is a small, $C^{\infty}$ function, $\mathcal{Q}$ is a pseudo differential operator of order $-1$, provided that $\omega\in\mathbb{R}^{\nu}$ satisfies appropriate non-resonance conditions.
Such PDEs arise by linearizing the Degasperis-Procesi (DP) equation at a small amplitude quasi-periodic function. Our work provides a first fundamental step in developing a KAM theory for perturbations of the DP equation on the circle.
Following \cite{Airy}, our approach is based on two main points: first  a reduction in orders  based on an Egorov type theorem  then a KAM diagonalization scheme. In both steps the key difficulites arise from the asymptotically linear dispersion law. In view of the application to the nonlinear context we prove sharp {\it  tame} bounds on the diagonalizing change of variables.
 We remark that the strategy and the techniques proposed are applicable for proving reducibility of more general classes of linear pseudo differential first order operators. 
\end{abstract}

\tableofcontents

\section{Introduction}
The problem of reducibility and stability of Sobolev norms for quasi-periodically forced linear operators on the circle is a classical one, 
and it has received new attention in the past few years. 
Informally speaking, given  a  linear operator, say $\mathcal X_\oo: H^s(\T,\R)\to H^{s-\mu}(\T,\R)$ 
(where $\T:=\R/2\pi \Z$)
depending  on time in a quasi-periodic way, we say that it is reducible if there exists a bounded change of variables  depending quasi-periodically on time (say mapping $H^s\to H^s$ for all times), which conjugates the linear PDE $\del_t u = \mathcal X_\oo u$ 
to the constant coefficient one
\[
\del_{t}v=\mathcal{D}_{\omega}v, \qquad \mathcal{D}_{\omega}:={\rm diag}_{j\in \Z}\{d_{j}\}, \;\;\; d_{j}\in \mathbb{C}.
\]
The notion of reducibility has been first introduced for ODEs (see for instance \cite{Jorba}, \cite{Eli}, \cite{You},  \cite{Avila} and reference therein). In the PDEs context this problem has been studied mostly in a perturbative regime,
 both on compact and non-compact domains. The reducibility of linear operators entails relevant dynamical consequences such as the control on the growth of Sobolev norms for the associated Cauchy problem.
 
 \noindent
 The subject has been studied by many authors:
we mention, among others, \cite{Comb87}, \cite{Bambusi-Graffi}, {\cite{EK2}, \cite{GP}, \cite{BGMR}, \cite{Mon2}}. 
For more details we refer for instance to \cite{Bam} (and reference therein).

\noindent
A strong motivation for the development of reducibility theory comes from KAM theory for nonlinear PDEs.
 Actually, reducibility  is a key ingredient in the construction of quasi-periodic solutions via quadratic schemes, such as Nash-Moser algorithms. 
 Indeed, the main issue is to invert the linearized PDE at a quasi-periodic approximate solution, see \cite{CFP}. 
 This reduces the problem to the study of a quasi-periodically forced linear PDE such as the ones described above. 
 We point out that in this context a sharp quantitative control on the reducing changes of variables is fundamental.
Regarding KAM theory for PDEs, we mention 
\cite{Ku},\cite{W},\cite{KP},\cite{CY} for equations on the circle, \cite{GY},\cite{EK},\cite{GYX},\cite{PP},\cite{EGK} for PDEs on $\T^n$.  
These works  all deal with equations possessing bounded nonlinearities. 
Regarding unbounded cases we mention \cite{Ku2}, \cite{LY}, \cite{BBiP2} for semilinear PDEs and 
\cite{Airy},\cite{BBM16},\cite{FP},\cite{Gi},\cite{BM1},\cite{BBHM} for the quasilinear case.

\noindent
 The main issues in all these problems are related to the geometry/dimension of the domain, 
 the dispersion of the PDE and the number of derivatives appearing in the nonlinearities. In particular the dispersionless case, 
 i.e. the case of (asymptotically) equally spaced spectral gaps, often exhibits 
 unstable behaviours and explosion of Sobolev norms (see \cite{M18}). 
 In this paper we discuss operators of this type, proving  reducibility and  stability for a class of quasi-periodically 
 forced first order linear operators on the circle. In view of possible applications to KAM theory we chose to consider 
 a class of linear operators related to the Degasperis-Procesi equation. However, both the strategy and the techniques are general and, 
 we believe, can be applied to wider classes of first order operators.
 
 \medskip
 
The \textbf{Degasperis-Procesi (DP) equation}
\begin{equation}\label{DP}
u_t-u_{x x t}+u_{xxx}-4 u_x-u u_{xxx}-3 u_x u_{xx}+4 u u_x=0. 
\end{equation}
was singled out  in \cite{DegPro} by  applying a test of asymptotic integrability to a family of third order dispersive PDEs. Later Degasperis-Holm-Hone \cite{Deg} proved its complete integrability by providing a Lax pair and a bi-Hamiltonian structure for this system.

\noindent
Constantin and Lannes showed in \cite{ConstLannes} that the Degasperis-Procesi equation, as well as the Camassa-Holm equation, can be regarded as a model for nonlinear shallow water dynamics and it captures stronger nonlinear effects than the classical Korteweg de Vries equation: 
for example, it exhibits wave-breaking phenomena and it shows peakon-like solutions. 
Unlike the Camassa-Holm equation, the DP system exhibits also shock waves.

\noindent
Since its discovery, lots of works have been written on this equation, mostly on the construction of very special exact solutions such as traveling waves and peaked solitons. We wish to stress that in general the existence of a Lax pair, in the infinite dimensional context, does not directly imply the possibility to construct Birkhoff (or action-angle) variables or even simpler structure, such as finite dimensional invariant tori (the so-called finite gap solutions for KdV and NLS on the circle).  For results on the spectral theory of the DP equation we refer to \cite{Cons,ConsIvLe},\cite{Hou}.
In conclusion the problem of KAM theory for the DP equation is, at the best of our knowledge, still open. This is one of the main motivations for proving this reducibility result.
Before introducing our classes of operators let us briefly describe the structure of the DP equation and in particular its linearized at a quasi-periodic function.

\smallskip

The equation \eqref{DP} can be formulated as a Hamiltonian PDE $u_t=J\,\nabla_{L^2} H(u)$, where $\nabla_{L^2} H$ is the $L^2$-gradient of the Hamiltonian
\begin{equation}\label{DPHamiltonian}
H(u)=\int \frac{u^2}{2}-\frac{u^3}{6}\, dx
\end{equation}
on the real phase space
\begin{equation}\label{H01}
H_0^1(\mathbb{T}):=\Big\{ u\in H^1(\mathbb{T}, \mathbb{R}) : \int_{\mathbb{T}} u\,dx=0\Big\}
\end{equation}
endowed with the non-degenerate symplectic form
\begin{equation}\label{SymplecticForm}
\Omega(u, v):=\int_{\mathbb{T}} (J^{-1} u)\,v\,dx, \quad \forall u, v\in H_0^1(\mathbb{T}), \qquad J:=(1-\partial_{xx})^{-1}(4-\partial_{xx})\partial_x.
\end{equation}
The Poisson bracket induced by $\Omega$ between two functions $F, G\colon H_0^1(\mathbb{T})\to \mathbb{R}$ is
\begin{equation}\label{PoissonBracketDP}
\{ F(u), G(u) \}:=\Omega(X_F, X_G)=\int_{\mathbb{T}} \nabla F(u)\,J \nabla G(u)\,dx,
\end{equation}
where $X_F$ and $X_G$ are the vector fields associated to the Hamiltonians $F$ and $G$, respectively.

\medskip

\noindent
Let $\nu\in \mathbb{N}^{*}:=\N\setminus\{0\}$, $L>0$, $\g\in(0, 1)$.

  Consider $\omega\in \cO_0$ where
\begin{equation}\label{0Meln}
\cO_0:=\left\{ \omega\in [L, 2 L]^{\nu}   : \lvert \omega\cdot \ell \rvert\geq \frac{2\gamma}{\langle \ell \rangle^{\nu}}, \,\,\,\ell\in\mathbb{Z}^{\nu}  \right\} , \quad \langle \ell \rangle:=\max\{\lvert \ell \rvert, 1 \}
\end{equation}
and a quasi-periodic function $u(t,x)$ with zero average in $x$,  small-amplitude and frequency $\omega$,
 \begin{equation}\label{smallqp}
 u(t,x)=\e\mathfrak{I}(\omega t,x), \qquad \e\ll1,
 \end{equation}
 where $\f\mapsto \mathfrak{I}(\f,x)$ belongs to $ C^{\infty}(\T^{\nu+1};\R)$.
 Linearizing equation \eqref{DP} at $u$ one obtains
 \begin{equation}\label{linearDP}
 v_{t}=\mathcal{X}_{\omega}(\oo t) v,\qquad 
 \mathcal{X}_{\omega}(\oo t)=\mathcal{X}_{\omega}(\oo t, \mathfrak{I}):=J\circ(1+a(\omega t,x)), \qquad a(\f,x)=a(\mathfrak{I};\f,x)
 \end{equation}
 with $a(\f,x)\in C^{\infty}(\T^{\nu+1};\R)$  Lipschitz in $\omega$ and $\mathfrak{I}$. In particular 
 one has 
 \[
 \|a\|_{H^{s}(\T^{\nu+1};\R)}\leq \e\|\mathfrak{I}\|_{H^{s}(\T^{\nu+1};\R)}\,,\quad \forall s.
 \]
Note that that $J$ in \eqref{SymplecticForm} can be written as
\begin{equation}\label{Helmotz}
J:=\del_{x}+3\Lambda\del_{x}, \qquad \Lambda:=(1-\del_{xx})^{-1},
\end{equation}
hence the operator $ \mathcal{X}_\oo(\oo t)$ in \eqref{linearDP}
has the form
\begin{equation}\label{sushi}
\begin{aligned}
\mathcal{X}_{\omega}(\oo t)= (1+a(\oo t ,x))\del_{x}+a_x(\oo t,x)+3(1-\del_{xx})^{-1}\del_{x} \circ (1+a(\oo t,x))
\end{aligned}
\end{equation}
and it is a  pseudo-differential operator of order one, moreover $\mathcal{X}_\oo(\oo t) $ 
is a  Hamiltonian vector field w.r.t. the  DP symplectic form \eqref{SymplecticForm}.
\\
In the paper \cite{FGMP}, together with Montalto, we proved that {\sl transport} operators of the form $(1+a(\oo t,x))\del_{x}$, with  $(\omega,1)\in \R^{\nu+1}$  diophantine, are reducible by a change of variables which has very sharp tame estimates in terms of the Sobolev norm of the function $a$.  Here we prove the same result for the more general class \eqref{sushi}. We have to deal with two main issues:
\begin{itemize}
\item the operator \eqref{sushi} is not purely transport;
\item we wish to diagonalize with a change of variables which is  symplectic w.r.t. \eqref{SymplecticForm}.
\end{itemize}
As in \cite{FGMP}, the main difficulties, which turn out to be  particularly delicate in our context, consist in giving sharp estimates of the  
change of variables; in order to do this, we need to  introduce a number of technical tools, for instance a {\sl quantitative} version of Egorov's theorem. 

We prove the following \emph{reducibility} result.

\begin{teor1}\label{MainResult}
Fix $\g\in (0, 1) $, consider  $\mathcal{X}_{\omega}(\oo t)$ 
in \eqref{sushi} with $\omega\in \cO_0$ (see \eqref{0Meln}),  assume that 
$\|\mathfrak{I}\|_{H^{s}(\T^{\nu+1};\R)}\leq 1$
for some $s>1$ large enough and
$|\e|\le \e_0(\gamma)$ (recall \eqref{smallqp}, \eqref{0Meln}). 
Then there  exists a  Cantor set $\calO_{\infty}\subseteq\calO_0$ such that for all $\oo\in \calO_{\infty}$ there exists a  quasi-periodic in time family of bounded symplectic maps $\Psi(\oo t):  H^s(\T;\R)\to H^s(\T;\R)$, which reduces \eqref{linearDP} to a diagonal constant coefficients operator 
with purely imaginary spectrum. 
Moreover  the Lebesgue measure 
of $\calO_0 \setminus \mathcal{O}_{\infty}$ goes to $0$ as $\gamma\to0$.
\end{teor1}

From Theorem \ref{MainResult} we deduce the following dynamical consequence.
\begin{cor}
Consider the Cauchy problem
\begin{equation}\label{probCau}
\left\{\begin{aligned}
&\del_{t}u=\mathcal{X}_{\oo}(\oo t)u,\\
&u(0,x)=u_0(x)\in  H^{s}(\T;\R),
\end{aligned}\right.
\end{equation}
with $s\gg 1$. If the Hypotheses of Theorem \ref{MainResult} are fulfilled then the solution 
 of \eqref{probCau} exists, is unique, and  satisfies
\begin{equation}\label{noncrescita}
\big(1-c(s)\big)\|u_0\|_{H^{s}(\T;\R)}\leq \|u(t,\cdot)\|_{H^{s}(\T;\R)}\leq \big(1+c(s)\big)\|u_0\|_{H^{s}(\T;\R)},
\end{equation}
for some $0<c(s)\ll 1$ for any $t\in \R$.

\end{cor}
We remark that \eqref{noncrescita} means that the Sobolev norms of the solutions of \eqref{probCau} do not increase in time. This is due to the quasi-periodic dependence on time of the perturbation. One could consider also problems with more general time dependence. However one expects to give at best an upper bound on the growth of the norms (see \cite{BGMR2}).

\smallskip

We shall deduce Theorem \ref{MainResult} from 
 Theorem \ref{teoMainRed} below. 
We
first need to introduce some notations.

\paragraph{Functional space.}
Passing to the Fourier representation
\begin{equation}\label{realfunctions}
u(\varphi, x)=\sum_{j\in\mathbb{Z}} u_{j}(\varphi)\,e^{\mathrm{i} j x}=\sum_{\ell\in\mathbb{Z}^{\nu}, j \in \mathbb{Z}} u_{\ell j} \,e^{\mathrm{i}(\ell \cdot \varphi+j x)},\quad 
\overline{u}_j(\varphi)=u_{-j}(\varphi), \quad \overline{u}_{\ell j}=u_{-\ell, -j}\,,
\end{equation}
we define the Sobolev space 
\begin{equation}\label{space} 
H^{s}:=\Big\{ u(\varphi, x)\in L^{2}(\T^{\nu+1}; \mathbb{R}) : \lVert u \rVert_s^2:=\sum_{\ell\in\mathbb{Z}^{\nu}, j\in\mathbb{Z}\setminus\{0\}} \lvert u_{\ell j} \rvert^2 \langle\ell, j \rangle^{2 s}<\infty \Big\}
\end{equation}
where $\langle \ell, j \rangle:=\max\{ 1, \lvert \ell \rvert, \lvert j \rvert\}$, $\lvert \ell \rvert:=\sum_{i=1}^{\nu} \lvert \ell_i \rvert$. 
We denote by $B_{s}(r)$ the ball of radius $r$ centered at the origin of $H^{s}$.
\paragraph{Pseudo  differential operators.}
Following \cite{BM1} and \cite{Meti} we give the following Definitions.
\begin{defi}\label{pseudoR}
	A linear operator $A$ is called pseudo differential 
	of order $\le m$ if its action on any $H^s(\T)$ with $s\ge m$ is given by
	\[
	A\sum_{j\in\Z} u_j e^{\im j x} = \sum_{j\in\Z} a(x,j) u_j e^{\im j x} \,,
	\]
	where   $a(x, j)$, called the {\it symbol} of $A$,  is the restriction to 
	$\mathbb{T}\times \mathbb{Z}$ of a complex valued function 
	$a(x, \xi)$ which is $C^{\infty}$ smooth on $\mathbb{T}\times\mathbb{R}$, 
	$2\pi$-periodic in $x$ and satisfies
	\begin{equation}\label{space3}
	|\del_{x}^{\al}\del_{\x}^{\be}a(x,\x)|\leq C_{\al,\be}\langle\x\rangle^{m-\be},
	\;\;\forall \; \al,\be\in \mathbb{N}.
	\end{equation} 
	We denote by 
	$A[\cdot]=\op(a)[\cdot]$
	the pseudo operator with symbol $a:=a(x, j)$.
	We call $OPS^m$ the class of the pseudo differential operator of order less or equal to $m$ and
	$OPS^{-\infty}:=\bigcap_m OPS^m$.
	We define the class $S^m$ as the set of symbols which satisfy \eqref{space3}. 
\end{defi}

We will consider mainly operator acting on $H^s(\T)$ with a quasi-periodic time dependence.  
In the case of pseudo differential operators this corresponds\footnote{since $\oo$ 
is diophantine we can replace the time variable with angles $\varphi\in\T^{\nu}$. 
The time dependence is recovered by setting $\varphi=\omega t$.} 
to consider symbols $a(\varphi, x, \x)$ with $\varphi\in\T^{\nu}$. 
Clearly these operators can be thought as acting on $H^s(\T^{\nu+1})$.

\begin{defi}
Let $a(\f,x,\x)\in S^{m}$ and set $A=\op(a)\in OPS^{m}$,
\begin{equation}\label{norma}
|A|_{m,s,\al}:=\max_{0\leq \be\leq \al} \sup_{\xi\in\mathbb{R}}\|\del_{\x}^{\be}a(\cdot,\cdot,\x)\|_{s}
\langle\x\rangle^{-m+\be}.
\end{equation}
We will use also the notation 
$\lvert a \rvert_{m, s, \alpha}:=|A|_{m,s,\al}$.
\end{defi}
Note that the norm $|\cdot|_{m,s,\al}$ is non-decreasing in $s$ and $\al$.
Moreover given a symbol $a(\f,x)$ independent of $\x$, the norm of the associated multiplication operator $\op(a)$ is just the $H^{s}$ norm of the function $a$.
If on the contrary the symbol $a(\x)$ depends only on $\x$, then the norm of the corresponding 
Fourier multipliers $\op(a(\x))$ is just controlled by a constant.

\vspace{0.5em}
\noindent
{\bf Linear operators. } Let $A\colon \T^{\nu}\to \mathcal{L}(L^2(\T))$, $\varphi\mapsto A(\varphi)$, be a $\varphi$-dependent family 
of linear operators acting on $L^2(\T)$. We consider $A$ as an operator acting on $H^{s}(\T^{\nu+1})$ by setting
\[
(A u)(\varphi, x)=(A(\varphi)u(\varphi, \cdot))(x).
\]  
This action is represented in Fourier coordinates as
\begin{equation}\label{cervino}
A u(\varphi, x)=\sum_{j, j'\in\mathbb{Z}} A_j^{j'}(\varphi) \,u_{j'}(\varphi)\,e^{\mathrm{i} j x}=\sum_{\ell\in\mathbb{Z}^{\nu}, j\in\mathbb{Z}} \sum_{\ell'\in\mathbb{Z}^{\nu}, j'\in\mathbb{Z}} A_j^{j'}(\ell-\ell')\,u_{\ell' j'}\,e^{\mathrm{i}(\ell\cdot \varphi+j x)}.
\end{equation}
Note that for the pseudo differential operators defined above  the norm \eqref{norma} provides a quantitative control of the action on $H^{s}(\T^{\nu+1})$. 
Conversely, given a T\"opliz in time operator $A$, namely such that its matrix coefficients (with respect to the Fourier basis) satisfy
\begin{equation}\label{topliz}
A_{j, l}^{j', l'}=A_j^{j'}(l-l')\qquad \forall j, j'\in\mathbb{Z},\,\,l, l'\in\mathbb{Z}^{\nu},
\end{equation}
we can associate it a time dependent family of operators acting on $H^s(\T)$ by setting
$$A(\varphi) h=\sum_{j, j'\in\mathbb{Z}, \ell\in\mathbb{Z}^{\nu}} A_j^{j'}(\ell) h_{j'}\,e^{\mathrm{i} j x} e^{\mathrm{i}\ell\cdot \varphi}, \qquad \forall h\in H^s(\T).$$

For $m=1, \dots, \nu$ we define the operators $\partial_{\varphi_m} A(\varphi)$ as
\begin{equation}\label{cervino2}
\begin{aligned}
&(\partial_{\varphi_m} A(\varphi)) u(\varphi, x)=\sum_{\ell\in\mathbb{Z}^{\nu}, j\in\mathbb{Z}} \sum_{\ell'\in\mathbb{Z}^{\nu}, j'\in\mathbb{Z}}\mathrm{i}(\ell-\ell')\, A_j^{j'}(\ell-\ell')\,u_{\ell' j'}\,e^{\mathrm{i}(\ell\cdot \varphi+j x)},
\end{aligned}
\end{equation}
We say that $A$ is a \textit{real} operator if it maps real valued functions in real valued functions. For the matrix coefficients this means that
\[
\overline{A_j^{j'}(\ell)}=A_{-j}^{-j'}(-\ell).
\]

\paragraph{Lipschitz norm.} Fix $\nu\in\mathbb{N}^{*}$ 
and let $\calO$ be a compact subset of $\mathbb{R}^{\nu}$. 
For a function $u\colon \calO\to E$, where $(E, \lVert \cdot \rVert_E)$ is a Banach space, we define the sup-norm and the lip-seminorm of $u$ as
\begin{equation}\label{suplip}
\begin{aligned}
&\lVert u \rVert_E^{\sup}:=\lVert u \rVert_{E}^{\sup, \calO}:=\sup_{\omega\in\calO} \lVert u(\omega) \rVert_E,\qquad \lVert u \rVert_{E}^{lip}:=\lVert u \rVert_{E}^{lip, \calO}:=\sup_{\substack{\omega_1, \omega_2\in \calO,\\ \omega_1\neq \omega_2}} \frac{\lVert u(\omega_1)-u(\omega_2)\rVert_E}{\lvert \omega_1-\omega_2\rvert}.
\end{aligned}
\end{equation}
If $E$ is finite dimensional,  for any $\gamma>0$ we introduce the 
weighted  Lipschitz norm:
\begin{equation}
\label{tazzone}
\lVert u \rVert_E^{\g, \calO}:=\lVert u \rVert_E^{sup, \calO}+\gamma \lVert u \rVert_{E}^{lip, \calO}.
\end{equation}
If $E$ is a scale of Banach spaces, say $E=  H^s$, for $\gamma>0$ we introduce the 
weighted  Lipschitz norms
\begin{equation}\label{tazza10}
\lVert u \rVert_s^{\g, \calO}:=\lVert u \rVert_s^{sup, \calO}+\gamma \lVert u \rVert_{s-1}^{lip, \calO}, \quad \forall s\geq [\nu/2]+3
\end{equation}
where we denoted by $[ r ]$ the integer part of $r\in\R$.
Similarly if $A={\rm Op}(a(\oo,\f,x,\x))\in OPS^{m}$ is a family of 
pseudo differential operators with symbols $a(\oo,\f,x,\x)$ belonging to $S^{m}$ and 
depending in a Lipschitz way on some parameter $\oo\in \calO\subset \mathbb{R}^{\nu}$, we set
\begin{equation}\label{norma2}
|A|_{m,s,\al}^{\g,\calO}:=\sup_{\oo\in \calO}|A|_{m,s,\al}+
\g \sup_{\oo_1,\oo_2\in \calO}\frac{{|\rm Op}\big(a(\oo_1,\f,x,\x)-a(\oo_2,\f,x,\x)\big)|_{m,s-1,\al}}{|\oo_1-\oo_2|}.
\end{equation}

\paragraph{Hamiltonian linear operators.} 
In the paper we shall deal with operators which are Hamiltonian according to the following Definition.
\begin{defi}\label{dynamicaldef}
 We  say that a linear map is symplectic if it preserves  the $2$-form $\Omega$ in \eqref{SymplecticForm};
similarly we say that a linear operator $M$  is Hamiltonian if $Mu $ is a linear hamltonian vector field w.r.t. $\Omega$ in \eqref{SymplecticForm}. This means that each $J^{-1} M$  is real symmetric.
Similarly, we call a family of maps $\f\to A(\f)$  symplectic if  for each fixed $\f$ $A(\f)$ is symplectic, 
same for the Hamiltonians.
We shall say that an operator of the form $\oo\cdot\del_{\f}+M(\f)$ is Hamiltonian if $M(\f)$ is Hamiltonian.
\end{defi}

\paragraph{Notation.} We use the notation $A\le_s B$ to denote $A\le C(s) B$ where $C(s)$ is a constant depending on some real number $s$.

\medskip

For $\omega\in \calO_0$ (see \eqref{0Meln}) we consider 
(in order to keep the parallel with \eqref{sushi})
a quasi-periodic function $\e \mathfrak{I}\in C^\infty(\T^{\nu+1},\R)$ such that,
by possibly rescaling $\e$,
 \begin{equation}\label{ipopiccolezza}
\lVert \mathfrak{I} \rVert_{s_0+\mu}^{\g, \cO_0}\leq 1 \,,\quad s_0 := [\nu/2]+ 3
\end{equation}
for some $\mu>0$ sufficiently large.
We consider  classes of linear Hamiltonian operators of the form
\begin{equation}\label{LomegaDP}
\mathcal{L}_{\omega}=\calL_{\oo}(\mathfrak{I})=\omega\cdot \partial_{\varphi}-J \circ (1+a(\varphi, x))+\mathcal{Q}(\f),
\end{equation}
where $a=a( \varphi, x)=a( \mathfrak{I}; \varphi, x)\in C^{\infty}(\T^{\nu+1}, \mathbb{R})$ 
and 
\begin{equation}\label{opQ}
\mathcal{Q}:=\op(q)[\cdot], \qquad q=q(\mathfrak{I};\f,x,\x)=q(\f,x,\x)\in S^{-1}.
\end{equation}
is Hamiltonian.
We assume that  $a,q$ depend on  
the small quasi-periodic function $\e \mathfrak{I}\in C^\infty(\T^{\nu+1},\R)$
(with $\mathfrak{I}$ as in \eqref{ipopiccolezza}),  as well as  on $\oo\in\calO_0$ in a Lipschitz way
and,   for all $s\ge s_0$  we require that (recall \eqref{norma2})
\begin{equation}\label{opQ2}
\lVert a \rVert_s^{\g, \calO_0}+|q|_{-1,s,\al}^{\g,\calO_0}\le_s \e \lVert \mathfrak{I} \rVert^{\gamma, \calO_0}_{s+\sigma_0},
\end{equation}
for some $\s_0>0$. If $\mathfrak{I}_1, \mathfrak{I}_2\in C^\infty(\T^{\nu+1},\R)$ satisfy \eqref{ipopiccolezza} we assume
\begin{equation}\label{opQ3}
\lVert \Delta_{12}a \rVert_{p}+|\Delta_{12}q|_{-1,s,\al}\le_{p} 
\e \lVert \mathfrak{I}_1-\mathfrak{I}_2 \rVert_{p+\sigma_0},
\end{equation}
for any $p\leq s_0+\mu-\s_0$ ($\mu>\s_0$), 
where we set $\Delta_{12}a:=a(\mathfrak{I}_1;\f,x)-a(\mathfrak{I}_2;\f,x)$ and similarly for $\Delta_{12}q$.

\smallskip

With this formulation our purpose is to diagonalize (in both space and time) the linear operator \eqref{LomegaDP} 
with changes of variables $H^s(\T^{\nu+1})\to H^s(\T^{\nu+1})$. Since $\mathcal{L}_{\omega}$ 
is T\"opliz in time (see \eqref{topliz}), it turns out that  these transformations can be seen 
as a family of quasi-periodically time dependent maps acting on $H^s(\T)$.

Theorem \ref{MainResult} is a consequence of the following result.

\begin{teor}[{\bf Reducibility}]\label{teoMainRed}
Let $\gamma\in (0,1)$ and 
consider $\calL_{\omega}$ in \eqref{LomegaDP} with $\omega\in \calO_0$ satisfying \eqref{opQ}-\eqref{opQ2} with $\e \g^{-5/2}\ll 1$.
Then there exists a sequence
\begin{equation}\label{Djei}
d_j=d_{j}(\mathfrak{I}):=m\, j 
\frac{4+j^2}{1+j^2} 
+r_j\,, \qquad j\in\mathbb{Z}\setminus\{0\}  \,, \quad r_j\in\mathbb{R}\,,\quad r_j= - r_{-j}
\end{equation}
with $m=m(\omega,\mathfrak{I})$, $r_j=r_j(\omega,\mathfrak{I})$ well defined and Lipschitz for $\omega\in \calO_0$ with
$	\lvert m-1 \rvert^{\g,\calO_{0}}, sup_j \langle j \rangle \,\lvert r_j \rvert^{\gamma^{3/2},\calO_0}\le C \e\,,$
such that the following holds:

\noindent
$(i)$	
for $\oo$ in the set  $\calO_{\infty}=\calO_{\infty}(\mathfrak{I}):=\Omega_1 \cap \Omega_2$, where ($\tau\geq 2\nu+6$ )
\begin{align}
& \Omega_1=\Omega_1(\mathfrak{I}):=
\{ \omega\in\cO_0 : \lvert \omega\cdot \ell-m\, j \rvert\geq 2\gamma\langle \ell \rangle^{-\tau}, 
\quad \forall j\in\mathbb{Z}\setminus\{0\},\,\,\ell\in\mathbb{Z}^{\nu}  \}  \label{prime}\\ 
& \Omega_2=\Omega_2(\mathfrak{I}):=\{ \omega\in\cO_0 : \lvert \omega\cdot \ell+ d_j-d_k \rvert\geq 
2\gamma^{3/2}\langle \ell \rangle^{-\tau}, 
\quad \forall j, k\in\mathbb{Z}\setminus\{0\},\,\,\ell\in\mathbb{Z}^{\nu}, \,\,\,\,(j, k, \ell)\neq (j, j, 0) \},\label{seconde}
\end{align}
there exists a linear, symplectic, bounded transformation 
$\Phi\colon \calO_{\infty}\times H^s\to H^s$ 
with bounded inverse $\Phi^{-1}$ such that for all $\oo\in \calO_{\infty}$
\begin{equation}\label{assoOp}
\Phi \mathcal{L}_{\omega} \Phi^{-1}=\omega\cdot \partial_{\varphi}-\mathcal{D}_{\oo}, 
\qquad \mathcal{D}_{\oo}:=\mathrm{diag}_{j\neq 0} (\mathrm{i} d_j)\,;
\end{equation}

\noindent
$(ii)$ the following tame estimates hold
\begin{align}
&\lVert \Phi^{\pm 1} h \rVert_s^{\g^{3/2}, \cO_{\infty}}\le_s 
\lVert h \rVert_s+\varepsilon \g^{-5/2}\lVert \mathfrak{I} \rVert^{\g, \cO_0}_{s+\s} 
\lVert h \rVert_{s_0} \qquad \forall s\geq s_0\label{grano1000}\\
& |\calO_{0}\setminus \cO_{\infty}|\le C\, \gamma \,L^{\nu-1}\label{grano1001}\,,
\end{align}
for some constants $\s, C>0$ depending on $\tau, \nu$;

\smallskip

$(iii)$ the map $\Phi$ is T\"opliz in time and via \eqref{topliz} induces a bounded transformation of the phase space $H^{s}(\T;\R)$ depending quasi-periodically on time.
\end{teor}

\noindent
Let us briefly discuss how to deduce Theorem \ref{MainResult} 
 from Theorem \ref{teoMainRed}.
Consider  the equation
\begin{equation}\label{gnocchi}
\del_{t}u=\mathcal{X}_{\oo}(\oo t)u
\end{equation}
with $\mathcal{X}_{\oo}(\oo t)$ in \eqref{sushi}. 
The operator associated to \eqref{gnocchi} acting on  
quasi-periodic function
is $\calL_{\omega}=\oo\cdot\del_{\f}-\mathcal{X}_{\oo}(\f)$ 
which has the form \eqref{LomegaDP} with $\mathcal{Q}(\f)=0$.

Under the action of the   transformation $v=\Phi(\oo t)u$ 
of the phase space $H^{s}(\T;\R)$ depending quasi-periodically on time 
the equation \eqref{gnocchi} is transformed  into the linear equation
\begin{equation}\label{gnocchi2}
\del_{t}v=\mathcal{D}_{\oo}v, \qquad \mathcal{D}_{\oo}=\Phi(\oo t)\mathcal{X}_{\oo}(\oo t)\Phi^{-1}(\oo t)+\Phi(\oo t)\del_{t}\Phi^{-1}(\oo t).
\end{equation}
The operator associated to \eqref{gnocchi2} is $\Phi\calL_{\omega}\Phi^{-1}$ given in \eqref{assoOp}.

\medskip

Let us makes some comments on the statement of  our main result.
\begin{itemize}
\item If we consider a $C^{\infty}$ Hamiltonian perturbation of the DP equation, say adding to the Hamiltonian \eqref{DPHamiltonian} a term like $\int_{\T} f(u)\,dx$, where the density $f\in C^{\infty}(\mathbb{R}, \mathbb{R})$, then the operator obtained by linearizing at a quasi-periodic function has the same form of the operator $\mathcal{L}_{\omega}$ in \eqref{LomegaDP}.

\item Along the reducibility procedure in order to deal with small divisor problems, we use  that $\omega$  belongs to the intersection of the sets \eqref{prime}, \eqref{seconde}. We point out that the diophantine constants appearing in the first order Melnikov conditions \eqref{prime} and the second order ones \eqref{seconde} consist of different powers of a small constant $\gamma$. This fact is crucial in view of the measure estimates of the sets \eqref{prime} and \eqref{seconde}, in particular for the proof of Lemma \ref{delpiero}. \\
Different scalings in $\g$ for non-resonance conditions are typical in problems with (asymptotically) linear dispersion such as the Klein-Gordon equation, see \cite{P1}, \cite{BBiP1}.

\end{itemize}

As said above, the linear operator $\mathcal{L}_{\omega}$ depends on a smooth function $\mathfrak{I}$ in a Lipschitz way. This dependence is preserved by the reducibility procedure, in the following sense.

\begin{lem}[{\bf Parameter dependence}]\label{paramdependence}
Consider $\mathfrak{I}_{1}, \mathfrak{I}_{2}\in  C^\infty(\T^{\nu+1},\R)$ satisfying \eqref{ipopiccolezza}.
Under the assumptions of Theorem \ref{teoMainRed} the following holds:
 for $\oo\in \calO_{\infty}(\mathfrak{I}_1)\cap\calO_{\infty}(\mathfrak{I}_2)$ there is  $\s>0$ such that

\begin{equation}
\lvert \Delta_{12}m \rvert\le \varepsilon\lVert \mathfrak{I}_1-\mathfrak{I}_2 \rVert_{s_0+\s},\quad 
\sup_j\langle j\rangle\lvert \Delta_{12} r_j \rvert\le \varepsilon \g^{-1}\lVert \mathfrak{I}_1-\mathfrak{I}_2 \rVert_{s_0+\sigma}.
\end{equation}
\end{lem}

The above quantitative lemma is important in view of application to KAM for nonlinear PDEs. Moreover it easily implies an approximate reducibility result, which in turns implies a control of Sobolev norms for long but finite times for all the operators $\mathcal{L}_{\omega}(\mathfrak{I})$ with $\mathfrak{I}$ in a small ball.

\begin{teor}[{\bf Almost reducibility}]\label{almostRED}
 Under the hypotheses of Theorem \ref{teoMainRed}, 
consider $\mathfrak{I}_{1}, \mathfrak{I}_{2}\in  C^\infty(\T^{\nu+1},\R)$ and assume that  
$\calL_{\omega}(\mathfrak{I}_1)$, $\calL_{\omega}(\mathfrak{I}_2)$ as in \eqref{LomegaDP}
satisfy \eqref{opQ2}-\eqref{opQ3}. Assume moreover that \eqref{ipopiccolezza} holds for $\mathfrak{I}_1$, $\mathfrak{I}_2$ and
\begin{equation}\label{piccolezzaI1}
\sup_{\oo\in \calO_0}\|\mathfrak{I}_1-\mathfrak{I}_2\|_{s_0+\mu}\leq  C\rho N^{-(\tau+1)}
\end{equation}
for $N$ sufficiently large, $0\leq \rho<\gamma^{3/2}/2$.
Then
the following holds. For any $\oo\in \calO_{\infty}(\mathfrak{I}_1)$
there exists a linear, symplectic, bounded transformation 
$\Phi_{N}$
with bounded inverse $\Phi_{N}^{-1}$ 
such that 
\begin{equation}\label{assoOpBIS}
\Phi_{N} \mathcal{L}_{\omega}(\mathfrak{I}_2) \Phi_{N}^{-1}=
\omega\cdot \partial_{\varphi}-\mathcal{D}^{(N)}_{\oo}+\RR^{(N)}, 
\qquad \mathcal{D}^{(N)}_{\oo}:=\mathcal{D}^{(N)}_{\oo}(\mathfrak{I}_2):=\mathrm{diag}_{j\neq 0} (\mathrm{i}\, d^{(N)}_j(\mathfrak{I}_2))\,.
\end{equation}
Here $d^{(N)}_j(\mathfrak{I}_2)$ has the form \eqref{Djei} 
for some $m^{(N)}=m^{(N)}(\mathfrak{I}_2)$ and $r_{j}^{(N)}=r_{j}^{(N)}(\mathfrak{I}_2)$
satisfying the bounds \begin{equation}\label{emmeENNE}
|m^{(N)}(\mathfrak{I}_2)-m(\mathfrak{I}_1)|+
\langle j\rangle|r_{j}^{(N)}(\mathfrak{I}_2)-r_{j}(\mathfrak{I}_1)|
\leq \e C \|\mathfrak{I}_1-\mathfrak{I}_2\|_{s_0+\mu}+
C\e N^{-\ka}
\end{equation}
for some $\ka>\tau$ and $C>0$. \\ The remainder 
$\RR^{(N)}=J\circ a^{(N)}+\mathcal{Q}^{(N)}$ with  $a^{(N)}\in C^{\infty}(\T^{\nu+1};\R)$,  $\mathcal{Q}^{(N)} $ T\"opliz in time, bounded on $H^{s}$, $\mathcal{Q}^{(N)}(\f)\colon H^s(\T^{\nu})\to H^{s+1}(\T)$, satisfying
\begin{equation}\label{emmeENNE2}
\|a^{(N)}\|^{\gamma,\calO_{\infty}(\mathfrak{I}_1)}_{s}\leq \e CN^{-\ka}\|\mathfrak{I}_2\|^{\gamma,\calO_0}_{s+\mu}, \qquad
\|\mathcal{Q}^{(N)}v\|_{s}\leq \e CN^{-\ka}(\|v\|_{s}+\|\mathfrak{I}_2\|_{s+\mu}\|v\|_{s_0}), \;\;\; \forall\; v\in H^{s}.
\end{equation}
The maps $\Phi_{N},\Phi_{N}^{-1}$ satisfy bounds like \eqref{grano1000}.
\end{teor}
\begin{remark}\label{Includo}
In order to prove the above theorem the main point is to show the inclusion $\calO_{\infty}(\mathfrak{I}_1)\subset \Omega_1^{(N)} \cap \Omega_2^{(N)}$, where
\begin{align}
& \Omega_1^{(N)}=\Omega_1^{(N)}(\mathfrak{I}_2):=
\{ \omega\in\cO_0 : \lvert \omega\cdot \ell-m^{(N)}\, j \rvert\geq \frac{2(\gamma-\rho)}{\langle \ell \rangle^{\tau}}, 
\quad \forall j\in\mathbb{Z}\setminus\{0\},\,\,|\ell|\leq N \} , \label{prime100}\\ 
& \Omega_2^{(N)}=\Omega_2^{(N)}(\mathfrak{I}_2):=\{ \omega\in\cO_0 : \lvert \omega\cdot \ell+ d^{(N)}_j-d^{(N)}_k \rvert\geq 
\frac{2(\gamma^{3/2}-\rho)}{\langle \ell \rangle^{\tau}}, 
\quad \forall j, k\in\mathbb{Z}\setminus\{0\},\,\,|\ell|\leq N \}.\label{seconde100}
\end{align}
\end{remark}

\noindent
One can deduce the following dynamical consequence.

\begin{coro}
Under the Hypotheses of Theorem \ref{almostRED} consider
 the Cauchy problem
\begin{equation}\label{probCau100}
\left\{\begin{aligned}
&\del_{t}u=J \circ (1+a(\mathfrak{I};\varphi, x)) u-\mathcal{Q}(\mathfrak{I}; \f)u,\\
&u(0,x)=u_0(x)\in  H^{s}(\T;\R),\,s\gg 1.
\end{aligned}\right.
\end{equation}
Consider $\mathfrak{I}_1$ as in Theorem \ref{almostRED} and $\oo\in \calO_{\infty}(\mathfrak{I}_1)$ (which is given in Theorem
\ref{teoMainRed}). Then for any $\mathfrak{I}$ in the ball \eqref{piccolezzaI1},
 \eqref{probCau100} admits a unique solution which satisfies
\begin{equation}\label{garofano}
\sup_{t\in [-T_{N},T_N]}\|u(t,\cdot)\|_{H^{s}(\T;\R)}\leq \big(1+c(s)\big)\|u_0\|_{H^{s}(\T;\R)},
\end{equation}
for some $0<c(s)\ll 1$ and some  $T_{N}\geq  \e^{-1} N^{\ka}$. Finally, if $\mathfrak{I}=\mathfrak{I}_1$ the bound \eqref{garofano} holds for all times.

\end{coro}

\subsection{Strategy of the proof}

In \cite{Airy} Baldi-Berti-Montalto developed a strategy  for the reducibility of a quasi-periodically forced linear operators, as a fundamental step in constructing quasi-periodic solutions for non-linear PDEs, via a Nash-Moser/ KAM scheme.
Indeed, the main point in the Nash-Moser scheme is to obtain tame estimates on \emph{high} Sobolev norms of the inverse of the linearized operator at a sequence of quasi-periodic approximate solutions. Given a diagonal operator, its inverse can be bounded 
 in \emph{any} Sobolev norm  by giving lower bounds on the eigenvalues. Therefore, if an operator is reducible,  the estimates on the inverse follow from corresponding tame bounds on the diagonalizing changes of variables, see for instance \eqref{grano1000}. Note that in order to use \eqref{grano1000} in a Nash-Moser scheme, the crucial point is that the $s$-Sobolev norm of $\Phi$ is controlled by the $(s+\sigma)$-Sobolev norm $\mathfrak{I}$ where $\sigma$ is fixed or at least $\s=\s(s)$ with $\s<s$.
 
 \smallskip
 
The main idea in the reducibility procedure of \cite{Airy} is to perform two steps.\\
 The first step
consists in applying a quasi-periodically depending on time change of coordinates which conjugates $\mathcal{L}_{\omega}$ to an operator $\mathcal{L}_{\omega}^+$ which is the sum of a diagonal unbounded part and a bounded, possibly smoothing, remainder. This is called the \textbf{regularization procedure} and, in fact, reduces the reducibility issue to a semilinear case.\\
The second step consists in performing a \textbf{KAM-like scheme} which completes the diagonalization of $\mathcal{L}_{\omega}^+$.

\smallskip

\paragraph{Step one.}
The operator $\mathcal{L}_{\omega}$ differs from the transport operator considered in \cite{FGMP} by a regularizing pseudo differential operator. 
Then, in order to make  the coefficient of the leading order  constant one can apply a  map 
\begin{equation}\label{diffeotoro}
\mathcal{T}_{\beta} u(\varphi, x)=u(\varphi, x+\beta(\varphi, x)).
\end{equation}
If we choose $\beta$ correctly, this map conjugates $\calL_{\omega}$ to constant coefficients plus a bounded remainder.
Such a map however  is  clearly not symplectic. In order to find the {\it symplectic equivalent} of this transformation we study the flow of the hyperbolic PDE
\begin{equation}\label{sushino}
\left\{
\begin{aligned}
&\del_{\tau}\Psi^{\tau}(u)=(J\circ b) \Psi^{\tau}(u), \qquad b:=\frac{\beta}{1+\tau \beta_x}\\
&\Psi^{0}u=u,
\end{aligned}\right.
\end{equation}
which is generated by the  Hamiltonian
\[
S(\tau,\f, u)= \int b(\tau,\f,x) u^2 dx.
\]
By construction if the flow of \eqref{sushino} is well defined then it is symplectic.
\\
First, 
in Proposition \ref{DPdiffeo} we show that $\Psi^{\tau}$ is the composition of
\[
\mathcal{A}^{\tau} u:=(1+\tau \beta_x)\, u(\varphi, x+\tau \beta(\varphi, x))
\]
with a  pseudo differential  operator $\cO$ plus a remainder.
$\cO$ is one smoothing in the $x$ variable, while the remainder is  $\rho$-smoothing in the $x$ variable for some very large $\rho$.
\begin{remark}
We point out that the strategy used in Proposition \ref{DPdiffeo} for constructing of the symplectic version of the torus diffeomorphism is applicable for more general symplectic structure, provided that $J$ is pseudo differential.
\end{remark}
Next,  we study how the map $\Psi^{\tau}$ conjugates $\mathcal{L}_{\omega}$; 
this is the content of Proposition \ref{ConjugationLemma}.
Egorov's theorem ensures that the main order of the conjugated operator 
$\Psi^{\tau} \mathcal{L}_{\omega} (\Psi^{\tau})^{-1}$ is
\[
a_+(\varphi, x):=-(\omega\cdot\del_{\f} \tilde{\beta})(\varphi, x+\beta(\varphi, x))
+(1+a(\varphi, x+\beta(\varphi, x)))(1+\tilde{\beta}_x(\varphi, x+\beta(\varphi, x)))-1
\]
where $x+\tilde{\beta}(\varphi, x)$ is the inverse of the diffeomorphism of the torus $x\mapsto x+\beta(\varphi, x)$.
The function $\beta$ is chosen as the solution of a quasi-periodic transport equation $a_+(\varphi, x)={\rm const}$. 
This equation has been treated in \cite{FGMP} and the Corollary $3.6$ 
in \cite{FGMP} gives the right $\beta$ with estimates.

\noindent
The map $\Psi^{\tau}$ is the flow of a hyperbolic PDE, 
hence the Egorov theorem guarantees that $\Psi^{\tau} \mathcal{L}_\oo (\Psi^{\tau})^{-1}$ 
is again a pseudo differential operator, whose leading order is constant. The fact that $\Psi^\tau$ is symplectic also ensures that the zero order terms vanish  and
the non-constant coefficients terms are one smoothing in the $x$- variables. 

\noindent
 In order to have sufficiently good bounds on the symbol of the transformed operator, we provide a quantitative version of the Egorov theorem
(see Theorem \ref{EgorovQuantitativo} in Section \ref{reg:procedure}). As before, the idea is to express such operator as a pseudo-differential term 
(whose symbol we can be bounded in a very precise way) plus a {\it remainder}  
which is $\rho$-smoothing in the $x$ variable for some very large $\rho$.

\noindent
The Egorov theorem regards the conjugation of a pseudo differential operator $P_0=\op(p_0)\in OPS^m$ by the flow of a linear pseudo differential vector field $\mathtt{X} u=\op(\chi) u$ of order $d$ with $d\in (0, 1]$. It is well known that the transformed operator $P(\tau)=\op(p(\tau))\in OPS^m$ satisfies the Heisenberg equation 
\begin{equation}\label{Heise}
\partial_{\tau} P=[\mathtt{X}, P]
\end{equation}
(see \eqref{ars}) and that the symbol $p(\tau)$ satisfies $\partial_{\tau} p=\{p, \chi \}_M$, where $\{ \cdot, \cdot\}_M$ are the Moyal brackets. The proof consists in making the ansatz that the new symbol $p$ can be written as sum of decreasing symbols $p=\sum_{i\le m} p_i$ (see \eqref{balconata0}) and solving the Heisenberg equation order by order. This gives a set of triangular ODEs for the symbols $p_i$ (see \eqref{ars2}). The r.h.s of \eqref{Heise} is of order $m+d-1$, hence if $d<1$ the leading order symbol $p_m(\tau)=p_0$. The remaining terms are easily computable by iteration.
A detailed discussion of the case $d=1/2$ can be found in \cite{BM1} and \cite{BBHM}.\\
If $d=1$ then the equation for $p_m$ is a Hamilton equation with Hamiltonian $\chi$, hence $p_m(\tau)$ is given by $p_0$ transported by the flow of the Hamiltonian $\chi$ (see \eqref{ars3}). Consequently the symbols $p_i$, $i<m$, are given by ODE of the same kind but with forcing terms. We need to control the norms $\lvert p_i \rvert_{i, s, \alpha}$ with the norm $\lvert p_0 \rvert_{m, s+\s_1, \alpha+\s_2}$ with $\s_1+\s_2<s$. This requires some careful analysis (see Lemma \ref{Lemmino}).

\medskip


Before stating the main regularization theorem let us briefly describe our class of remainders
 i.e.   operators  which are sufficiently smoothing in the $x$-variable that they can be ignored in the pseudo-differential reduction, and are diagonalized in the KAM scheme. 
 We call such remainders $\gotL_{\rho,p}$ (for some $\rho\geq 3$, $p\geq s_0$).
  Roughly speaking we require that an operator 
$\mathcal{R}$ in $\gotL_{\rho,p}$ is tame as a bounded operator on $H^s$ and $\rho$-regularizing in space; moreover   its derivatives in $\varphi$ of order $\tb\le \rho-2$ are tame and $(\rho-\tb)$-regularizing in space.
This definition is made quantitative by introducing constants $\mathbb{M}_{\mathcal{R}}^{\g}(s, \tb)$, see Definition \ref{ellerho}  in Section \ref{sezione functional setting}.\\
The most important features of this class are that
it is closed for conjugation by maps $\mathcal{T}_\beta$ as in \eqref{diffeotoro} and that
any $\mathcal{R}$ in $\gotL_{\rho,p}$ is modulo-tame and hence can be diagonalized by a KAM procedure.


\begin{teor}[{\bf Regularization}]\label{risultatosez8}
Let $\rho\geq 3$ and consider $\mathcal{L}_{\omega}$ in \eqref{LomegaDP}.
There exist $\mu_1\ge  \mu_2>0$ such that, if condition \eqref{ipopiccolezza}
is satisfied  with $\mu=\mu_1$   then the following holds for all  $p\le s_0+\mu_1-\mu_2$.\\
There exists a constant $m(\oo)$ which depend in a   Lipschitz way  w.r.t.  $\oo\in\cO_0$, satisfying  
\begin{equation}\label{clinica100}
\lvert m-1 \rvert^{\gamma, \cO_0}\le C \varepsilon,
\end{equation}
such that
for all $\oo$ in the set $\Omega_1(\mathfrak{I})$ (see \eqref{prime}) there exists a real bounded linear operator  
$
\Phi_1=\Phi_1(\oo) : \Omega_1\times H^{s}\to H^{s}
$
such that
\begin{equation}\label{operatorefinale}
\calL_{\omega}^+:=\Phi_1 \calL_{\oo}\Phi_1^{-1}=\omega\cdot \partial_{\varphi}-m J +\mathcal{R}.
\end{equation}
The constant $m$ depends on $\mathfrak{I}$ and for $\oo\in \Omega_1(\mathfrak{I}_1)\cap\Omega_1(\mathfrak{I}_2)$ 
one has
\begin{equation}\label{clinica1000}
\lvert \Delta_{12}m \rvert\le \varepsilon\lVert \mathfrak{I}_1-\mathfrak{I}_2 \rVert_{s_0+\mu_1},
\end{equation}
where $\Delta_{12}m:=m(\mathfrak{I}_1)-m(\mathfrak{I}_2)$.
The remainder in \eqref{operatorefinale} has the form 
$\RR={\rm Op}(r)+\widehat{\RR}$ where $r\in S^{-1}$, $\widehat{\RR}$
belongs to $\gotL_{\rho, p}$ 
(see Def. \ref{ellerho}) and 
\begin{equation}
\begin{aligned}\label{pasqua200}
|r|_{-1,s,\al}^{\gamma,\Omega_1}+\mathbb{M}^{\g}_{\widehat{\mathcal{R}}}(s, \tb)  \le_{s,\al} \varepsilon \g^{-1} \lVert \mathfrak{I}\rVert_{s+{\mu_1}}^{\gamma, \cO_0}, 
\quad 0\le \tb\le \rho-2,\\
|\Delta_{12}r|_{-1,p,\al}+
\mathbb{M}_{\Delta_{12}\mathcal{R} }(p, \tb)\le_{p,\al} 
\varepsilon \gamma^{-1}\lVert \mathfrak{I}_1-\mathfrak{I}_2\rVert_{s_0+{\mu_1}}\quad 0\le \tb\le \rho-3.
\end{aligned}
\end{equation}
Moreover if $u=u(\omega)$ depends on  $\omega\in \Omega_1$ in a Lipschitz way then
\begin{equation}\label{grecia}
\lVert \Phi_1^{\pm 1} u \rVert_s^{\g,\Omega_1}\le_s \lVert u \rVert_s^{\g,\Omega_1}+\varepsilon \gamma^{-1}\lVert \mathfrak{I} \rVert^{\g, \calO_0}_{s+\mu_1}\lVert u \rVert^{\gamma, \Omega_1}_{s_0}.
\end{equation}
Finally  $\Phi_1$, $\Phi_1^{-1}$ are  symplectic (according to Def. \ref{dynamicaldef}). 
\end{teor}
\noindent
{\bf Step two}. We apply a KAM algorithm which diagonalizes $\calL^+_\oo$. 
As in the first step, an important point is to implement such algorithm by requiring only a smallness 
condition on a low norm of the remainder of the regularization procedure. 
Hence in order to achieve estimates on high Sobolev norms for the changes of variables 
it is not sufficient that the non- diagonal terms are bounded. 
To this purpose, following \cite{BBHM}, we work in the class of modulo tame operators 
(see  Def. \eqref{menounomodulotame}), more precisely we need that $\RR$ in \eqref{operatorefinale} 
is  modulo tame and one smoothing in the $x$-variable together with its derivatives in times up to some sufficiently 
large order, this follows from our definition of $\mathfrak L_{\rho,p}$ and 
properties of pseudo-differential operators, see Lemma \ref{LemmaAggancio}.
Our strategy is mostly parallel to \cite{BBHM}, hence we give only a sketch of the proof for completeness.

\begin{teor}{\textbf{(Diagonalization)}}\label{ReducibilityDP}
Fix $\mathcal{S}>s_0$.
Assume that $\omega \mapsto \mathfrak{I}(\omega)$ 
is a Lipschitz function defined on 
$\calO_0$, satisfying \eqref{ipopiccolezza} with 
$\mu \geq \mu_1$ where $\mu_1:=\mu_1(\nu)$ is given in 
Theorem \ref{risultatosez8}. 
Then there exists $\delta_0\in (0, 1)$, $N_0>0$, $C_0>0$, such that, if
\begin{equation}\label{PiccolezzaperKamredDP}
N_0^{C_0} \e \g^{-\frac{5}{2}} \le \delta_0, 
\end{equation}
then the following holds.

\vspace{0.5em}
\noindent
$(i)$  \textbf{(Eigenvalues)}. For all $\omega\in \calO_0$ there exists a sequence
\begin{align}\label{FinalEigenvaluesDP}
&d_j(\omega):=d_j(\omega, \mathfrak{I}(\omega)):=m(\omega)\,j 
\frac{4+j^2}{1+j^2}  + r_j(\omega), \quad j\neq 0,
\end{align}
 with $m$ in \eqref{clinica100}. Furthermore, for all $j\neq 0$
\begin{equation}\label{stimeautovalfinaliDP}
 \sup_j \langle j\rangle |r_j|^{\g^{\frac{3}{2}},\calO_0} < C \e, \qquad r_j=-r_{-j}
\end{equation}
for some $C>0$. All the eigenvalues $\mathrm{i} d_j$ are purely imaginary. 

\vspace{0.5em}
\noindent
$(ii)$ \textbf{(Conjugacy)}. For all $\omega$ in the set 
$\cO_{\infty}:=\Omega_1(\mathfrak{I})\cap \Omega_2(\mathfrak{I})$ (see \eqref{prime}, \eqref{seconde})
there is a real, bounded, invertible, linear operator $\Phi_{2}(\omega)\colon H^s\to H^s$, 
with bounded inverse $\Phi_{2}^{-1}(\omega)$, that conjugates 
$\mathcal{L}_{\omega}^{+}$ in \eqref{operatorefinale} to constant coefficients, namely 
\begin{equation}\label{Linfinito}
\begin{aligned}
&\mathcal{L}_{\infty}(\omega):=\Phi_{2} (\omega) \circ \mathcal{L}_{\oo}^{+} \circ \Phi_{2}^{-1}(\omega)=\omega\cdot \partial_{\varphi}+\mathcal{D}(\omega), \qquad \mathcal{D}(\omega):=\mathrm{diag}_{j\neq 0} \{ \mathrm{i} d_j(\omega) \}.
\end{aligned}
\end{equation}
The transformations $\Phi_{2}, \Phi_{2}^{-1}$ are symplectic, tame and they satisfy for $s_0\le s\le \mathcal{S}$
\begin{equation}\label{grano}
\lVert (\Phi^{\pm 1}_{2}-\mathrm{I}) h \rVert^{\gamma^{3/2},\cO_{\infty}}_{s}\le_s  \big(\e \gamma^{-3/2}
+\e \gamma^{-5/2} \lVert \mathfrak{I} \rVert_{s+\mu}^{\gamma, \calO_0}\big) 
\lVert h \rVert^{\gamma^{3/2}, \cO_{\infty}}_{s_0}+\e \gamma^{-3/2} \lVert h \rVert^{\gamma^{3/2}, \cO_{\infty}}_s.
\end{equation}
with $h=h(\omega)$.
Moreover, for $\omega\in \calO_{\infty}(\mathfrak{I}_1)\cap \calO_{\infty}(\mathfrak{I}_2)$ we have the following bound
for some $\s>0$:
\begin{equation}\label{stimaLipR}
\sup_j\langle j\rangle\lvert \Delta_{12} r_j \rvert\le \varepsilon \g^{-1}\lVert \mathfrak{I}_1-\mathfrak{I}_2 \rVert_{s_0+\sigma}.
\end{equation}
\end{teor}
It remains to prove measure estimates for the Cantor set $\cO_{\infty}=\Omega_1\cap\Omega_2$. In Section \ref{measure} we prove the following.

\begin{teor}[{\bf Measure estimates}]\label{stimedimisura}
Let $\calO_{\infty}$ be the set of parameters in \eqref{prime}-\eqref{seconde}.
For some
constant $C>0$ one has that 
\begin{equation}\label{stimedimisuraTeo}
|\calO_0\setminus\cO_{\infty}|\leq C\gamma L^{\nu-1}.
\end{equation}
\end{teor}

\noindent
We discuss the key ideas to prove the above result.
Recalling \eqref{prime}-\eqref{seconde} 
we may write
\begin{equation}\label{UnionDP}
\calO_0 \setminus \mathcal{O}_{\infty}
=\bigcup_{\ell\in\mathbb{Z}^{\nu}, j, k\in \mathbb{Z}\setminus\{0\}} \Big( R_{\ell j k} \cup Q_{\ell j }\Big)
\end{equation}
where 
\begin{equation}\label{BadSetsDP}
\begin{aligned}
&R_{\ell j k}:=\{ \omega\in\calO_0 : \lvert  \omega\cdot \ell+d_j-d_k\rvert< 2\,\gamma^{3/2}\, \langle \ell \rangle^{-\tau} \},\\
&Q_{\ell j }:=\{ \omega\in\calO_0 : \lvert \omega\cdot \ell+m j \rvert< 2\gamma \langle \ell \rangle^{-\tau} \}
\end{aligned}
\end{equation}
where $d_{j}$ are given in \eqref{FinalEigenvaluesDP}.
Since, by \eqref{0Meln} and $\gamma>\gamma^{3/2}$, $R_{\ell j k}=\emptyset$ for $j=k$, in the sequel we assume that $j\neq k$.

\noindent
The strategy of the proof of Theorem \ref{stimedimisura} is the following.

\medskip
\noindent
$(i)$
 Since the union in \eqref{UnionDP} runs over infinite numbers of indices $\ell,j,k$, we first need
some relation between them which is given in Lemma \ref{BrexitDP}. 
Note that, since the dispersion law $j\mapsto j ({1+j^2})^{-1}{(4+j^2)} $ is asymptotically linear, for  fixed $\ell$ there are infinitely many non-empty bad sets 
$R_{\ell j k}$
to be considered. 
It is well known that if the dispersion law grows as $j^{d}$, $d>1$ as $j\to\infty$
then, thanks to  good separation properties of the linear frequencies, there are only a finite number of
sets to be considered for any fixed $\ell\in \Z^{\nu}$. This is the key difficulty to deal with.

\medskip
\noindent
$(ii)$
 We provide the estimates of each  ``bad'' set in \eqref{BadSetsDP} 
when $\ell\in \Z^{\nu},j,k\in \mathbb{Z}\setminus\{0\}$.
This is done in Lemma \ref{singolo}.

\medskip
\noindent
$(iii)$
 We deal with the problem of the summability in $j,k$. We show (in Lemma \ref{delpiero}) 
that, if $|k|,|j|\gg |\ell|$, then the sets
$R_{\ell,j,k}$  are included in sets of  type $Q_{\ell,j-k}$, which depends only on the difference $j-k$
and so are finite for fixed $\ell $.

\section{Functional Setting}\label{sezione functional setting}

In this Section we introduce some notations, definitions and technical tools which will be used along the paper. In particular we introduce rigorously the spaces and the classes of operators on which we work.

We refer to the Appendix $A$ in \cite{FGMP}
for technical lemmata on the tameness properties of the Lipschitz and Sobolev norms in \eqref{space},\eqref{tazza10}.

\paragraph{Linear Tame operators.} 
\begin{defi}[{\bf $\s$-Tame operators}]\label{TameConstants}
Given  $\s\geq 0$ we say that 
a linear operator $A$ is $\sigma$-\textit{tame} w.r.t. a non-decreasing sequence $\{\mathfrak M_A(\s,s)\}_{s=s_0}^\mathcal{S}$ (with possibly $\mathcal{S}=+\infty$) if:
\begin{equation}\label{SigmaTame}
\lVert A u \rVert_{s}\le \mathfrak{M}_A(\s,s) \lVert u \rVert_{s_0+\sigma}+\mathfrak{M}_A(\s,s_0) \lVert u \rVert_{s+\sigma} \qquad u\in H^s,
\end{equation} 
for any $s_0\le s\le \mathcal{S}$. We call $\mathfrak{M}_A(\s,s)$  a {\sc tame constant} for the operator $A$. When the index $\s$ is not relevant 
we write $\gotM_{A}(\s,s)=\gotM_{A}(s)$. 
\end{defi}

\begin{defi}[{\bf Lip-$\s$-Tame operators}]\label{LipTameConstants}
Let $\s\geq 0$ and $A=A(\oo)$ be a linear operator defined for $\oo\in \calO\subset \mathbb{R}^{\nu}$.
Let us define
\begin{equation}\label{defDELTAomega}
\Delta_{\oo,\oo'}A:=\frac{A(\oo)-A(\oo')}{|\oo-\oo'|}, \quad \oo,\oo'\in \calO.
\end{equation}
Then $A$ is \emph{Lip-$\s$-tame} w.r.t. a non-decreasing sequence $\{\mathfrak M_A(\s,s)\}_{s=s_0}^\mathcal{S}$ if
 the following estimate holds
\begin{equation}\label{lipTAME}
\sup_{\oo\in \calO}\|Au\|_{s},\g\sup_{\oo\neq\oo'}\|(\Delta_{\oo,\oo'}A)\|_{s-1}\leq _{s}\gotM^{\g}_{A}(\s,s)\|u\|_{s_0+\s}+
\gotM^{\g}_{A}(\s,s)\|u\|_{s+\s}, \quad u\in H^{s},
\end{equation}
We call $\mathfrak{M}^{\gamma}_A(\s,s)$ a {\sc Lip-tame constant } of the operator $A$. 
When the index $\s$ is not relevant 
we write $\gotM^{\gamma}_{A}(\s,s)=\gotM^{\gamma}_{A}(s)$. 

\end{defi}

\paragraph{Modulo-tame operators and majorant norms.}
The modulo-tame operators are introduced in Section $2.2$ of \cite{BM1}. Note that we are interested only in the Lipschitz variation of the operators respect to the parameters of the problem whereas in \cite{BM1} the authors need to control also higher order derivatives.

\begin{defi}
Let $u\in H^s$, $s\geq 0$, we define the majorant function
$
\underline{u} (\varphi, x):=\sum_{\ell\in\mathbb{Z}^{\nu}, j\in\mathbb{Z}} \lvert u_{\ell j} \rvert e^{\mathrm{i}(\ell\cdot \varphi+j x)}.
$
Note that $\lVert u \rVert_s=\lVert \underline{u} \rVert_s$.
\end{defi}
\begin{defi}[{\bf Majorant operator}]
Let $A\in\mathcal{L}(H^s)$ and recall its matrix representation \eqref{cervino}.
We define the majorant matrix $\underline{A}$ as the matrix with entries
\[
(\underline{A})_j^{j'}(\ell):=\lvert (A)_{j}^{j'}(\ell) \rvert \qquad j, j'\in\mathbb{Z},\,\,\ell\in\mathbb{Z}^{\nu}.
\]
\end{defi}

\noindent
We consider the majorant operatorial norms 
\begin{equation}\label{majorantnorm}
\|\un M\|_{\mathcal L(H^s)}:= \sup_{\lVert u \rVert_s\le 1} \lVert\un Mu \rVert_{s}.
\end{equation}
We have a partial ordering relation in the set of the infinite dimensional matrices, i.e.
if
\begin{equation}\label{partialorder}
M \preceq N \Leftrightarrow |M_{j}^{j'}(\ell)|\leq |N_{j}^{j'}(\ell)|\;\;\forall j,j',\ell\; \Rightarrow \lVert \un{M} \rVert_{\mathcal L(H^s)}\le \lVert \un{N} \rVert_{\mathcal L(H^s)}\,,\quad \lVert  {M}u \rVert_s\le \lVert  \un{M}\,\un u \rVert_s \le \lVert\un{N}\, \un u \rVert_s.
\end{equation}

Since we are working on a majorant norm we have the continuity of the projections on monomial subspace, in particular we define the following functor acting on the matrices
\[
\Pi_K M:=
\begin{cases}
 M_{j}^{j'}(\ell) \qquad \qquad \text{if} \; |\ell|\le K, \\
0 \qquad \qquad \qquad \mbox{otherwise}
\end{cases}
\qquad \qquad  \Pi_K^\perp:= \mathrm{I}-\Pi_K.
\]
Finally we define for $\mathtt b_0\in \N$
\begin{equation}\label{funtore}
(\langle \pa_\f \rangle^{\mathtt b_0} M )_{j}^{j'}(\ell)  =  \langle \ell  \rangle^{\mathtt b_0} M_j^{j'}(\ell).
\end{equation}
If $A=A(\omega)$ is an operator depending on a parameter $\omega$, we control the Lipschitz variation, see formula \ref{defDELTAomega}. In the sequel let $1>\g>\g^{3/2}>0$ be fixed constants.
\begin{defi}[{\bf Lip-$\sigma$-modulo tame}]\label{def:op-tame} 
Let $\s\geq 0$. A  linear operator $ A := A(\omega) $, $\omega\in \calO\subset\mathbb{R}^{\nu}$,  is  
Lip-$\s$-modulo-tame  w.r.t. an increasing sequence $\{	{\mathfrak M}_{A}^{\sharp, \g^{3/2}} (s) \}_{s=s_0}^{\mathcal{S}}$ if 
the majorant operators $  \un{ A }, \un{\Delta_{\omega,\omega'} A}$ are Lip-$\s$-tame w.r.t. these constants, i.e. they 
satisfy the following weighted tame estimates:  
for $\s\geq 0$,  for all 
 $ s \geq s_0 $ and  for any $u \in H^{s} $,  
\begin{equation}\label{CK0-tame}	
\sup_{\omega\in\calO}\| \un{A} u\|_s,
 \sup_{\omega\neq \omega'\in \cO}{\g^{3/2}}  \|\un{\Delta_{\omega,\omega'} A}  u\|_s 
 \leq  
{\mathfrak M}_{A}^{\sharp, \g^{3/2}} (\s,s_0) \| u \|_{s+\s} +
{\mathfrak M}_{A}^{\sharp, \g^{3/2}} (\s,s) \| u \|_{s_0+\s} 
\end{equation}
where  the functions $ s \mapsto  {\mathfrak M}_{A}^{{\sharp, \g^{3/2}}} (\s, s)  \geq 0  $ 
are non-decreasing in $ s $. 
The constant $ {\mathfrak M}_A^{{\sharp, \g^{3/2}}} (\s,s) $ 
is called the {\sc modulo-tame constant} of the operator $ A $. When the index $\s$ is not relevant 
we write $ {\mathfrak M}_{A}^{{\sharp, \g^{3/2}}} (\s, s) = {\mathfrak M}_{A}^{{\sharp, \g^{3/2}}} (s) $. 
\end{defi}
\begin{defi}\label{menounomodulotame}
We say that $A$ is Lip-$-1$-modulo tame if 	
$\langle D_x\rangle^{1/2}{A}  \langle D_x\rangle^{1/2}$ 
is Lip-$0$-modulo tame. We denote
\begin{equation}\label{anagrafe} 
\begin{aligned}
&\mathfrak M^{{\sharp, \g^{3/2}}}_{A}(s):=  
\mathfrak M^{\sharp, \g^{3/2}}_{ \langle D_x \rangle^{1/2}A \langle D_x \rangle^{1/2}}(0,s), \qquad \mathfrak M^{\sharp, \g^{3/2}}_{A}(s,a):= 
\mathfrak M^{\sharp, \g^{3/2}}_{\langle \pa_\f \rangle^{a}  \langle D_x \rangle^{1/2}A \langle D_x \rangle^{1/2}}(0, s), 
\quad a\ge 0.
\end{aligned}
\end{equation}
\end{defi}

In the following we shall systematically use $-1$ modulo-tame operators. 
We refer the reader to Appendix \ref{lemmitecnicitame} for the properties of Tame and Modulo-tame operators.
 
\paragraph{Pseudo differential operators properties.}
We now collect some classical results about pseudo differential operators 
introduced in Def. \ref{pseudoR}  adapted to our setting.

\smallskip

\emph{Composition of pseudo differential operators.}
One of the fundamental properties of pseudo differential operators
is the following: 
given two pseudo differential operators  $\op(a)\in OPS^{m}$ and $\op(b)\in OPS^{m'}$, for some $m,m'\in\mathbb{R}$,
the composition $\op(a)\circ\op(b)$ is a pseudo differential operator of order $m+m'$. In particular 
\begin{equation}\label{compo}
\op(a)\circ\op(b)=\op(a\# b), 
\end{equation}
where the symbol of the composition is given by 
\begin{equation}\label{comp1}
(a\# b)(x,\x)=\sum_{j\in\mathbb{Z}}a(x,\x+j)\hat{b}_{j}(\x)e^{\mathrm{i} jx}=
\sum_{k,j\in\mathbb{Z}} \hat{a}_{k-j}(\x+j)\hat{b}_{j}(\x)e^{\mathrm{i} kx}.
\end{equation}
Here the $\hat{\cdot}$ denotes the Fourier transform of the symbols $a(x,\x)$ and $b(x, \xi)$ in the variable $x$.
The symbol $a\# b$ has the following asymptotic expansion:
for any $N\geq1$ one can write
\begin{equation}\label{comp2}
\begin{aligned}
(a\# b)(x,\x)&=\sum_{n=0}^{N-1}\frac{1}{n! \mathrm{i}^{n}}\del_{\x}^{n}a(x,\x)\del_{x}^{n}b(x,\x)
+r_{N}(x,\x),\qquad r_{N}\in S^{m+m'-N},\\
r_{N}(x,\x)&=\frac{1}{(N-1)!\mathrm{i}^{N}}\int_{0}^{1}(1-\tau)^{N}\sum_{j\in\mathbb{Z}}
(\del_{\x}^{N}a)(x,\x+\tau j)\widehat{\del_{x}^{N}b}(j,\x)e^{\mathrm{i} jx}d \tau.
\end{aligned}
\end{equation}

\begin{defi}\label{cancelletti}
Let $N\in\mathbb{N}$, $0\le k \le N$, $a\in S^m$ and $b\in S^{m'}$, we define (see \eqref{comp2})
\begin{equation}
a\#_{k} b:=\frac{1}{k! \rm i^{k}} (\partial_{\xi}^k a)(\partial_x^k b), \qquad a\#_{< N} b:=\sum_{k=0}^{N-1} a\#_k b, \qquad a\#_{\geq N} b:=r_N.
\end{equation}
\end{defi}

\emph{Adjoint operator.}
Let $A:=\op(a)\in OPS^{m}$. Then its $L^2$-adjoint $A^*$ is a pseudo differential operator such that
\begin{equation}\label{adj}
A^{*}=\op(a^{*}), \qquad a^{*}(x,\x)=\ol{\sum_{j\in\mathbb{Z}}\hat{a}_{j}(\x-j)e^{\mathrm{i} jx}}.
\end{equation}

\emph{Parameter family of pseudo differential operators.} We shall deal also  with pseudo differential operators depending on parameters
$\f\in\T^{\nu}$:
$$
(Au)(\f,x)=\sum_{j\in\mathbb{Z}}a(\f,x,j) u_{j}e^{\mathrm{i} jx}, \quad a(\f,x,j)\in S^{m}.
$$
The symbol $a(\varphi, x, \xi)$ is $C^{\infty}$ smooth also in the variable $\varphi$. We still denote 
$
A:=A(\varphi)=\op(a(\varphi, \cdot))=\op(a).
$
For the symbols of the composition operator with $\op(b(\varphi, x, \xi))$ and the $L^2$-adjoint we have the following formulas
\begin{equation}
\begin{aligned}
&(a\# b)(\varphi, x,\xi)=\sum_{j\in\mathbb{Z}} a(\varphi, x, \xi+j) \hat{b}(\varphi, j, \xi)\,e^{\mathrm{i} j x}=\sum_{\substack{j, j'\in\mathbb{Z},\\\ell, \ell'\in\mathbb{Z}^{\nu}}} \hat{a}(\ell-\ell', j'-j, \xi+j)\hat{b}(\ell', j, \xi)\,e^{\mathrm{i}(\ell \cdot \varphi+j x)},\\
& a^*(\varphi, x, \xi)=\overline{\sum_{j\in\mathbb{Z}} \hat{a}(\varphi, j, \xi-j)\,e^{\mathrm{i}j x}}=\overline{\sum_{\ell\in \mathbb{Z}^{\nu}, j\in\mathbb{Z}} \hat{a}(\ell, j, \xi-j)\,e^{\mathrm{i}(\ell\cdot \varphi+j x)}}.
\end{aligned}
\end{equation}

\paragraph{Classes of Smoothing Remainders.}
The KAM scheme performed in Section \ref{SezioneDiagonalization} is based on an abstract reducibility algorithm which works in the space of modulo-tame operators. In order to control the majorant norm \eqref{majorantnorm} of the remainder of the regularization procedure it is useful to introduce a class of linear ``tame'' smoothing operators.
\begin{defi}\label{ellerho}
Fix $s_0\geq (\nu+1)/2$ and $p,\,\mathcal{S}\in \mathbb{N}$ with $s_0\le p< \mathcal{S}$ with possibly $\mathcal{S}=+\infty$.
Fix $\rho\in\mathbb{N}$,  with $\rho\geq 3$ 
and consider any subset $\calO$ of $\R^\nu$. 
We denote by $\gotL_{\rho, p}=\gotL_{\rho, p}(\calO)=\gotL_{\rho,p}(\calO)$ 
the set of the linear operators 
$A=A(\omega)\colon H^s(\T^{\nu+1})\to H^{s}(\T^{\nu+1})$, $\omega\in \calO$ with the following properties:

\vspace{0.5em}
\noindent
$\bullet$ the operator $A$ is Lipschitz in $\omega$,

\vspace{0.5em}
\noindent		
$\bullet$  the operators $\partial_{\varphi}^{\vec{\tb}} A$,  $[\partial_{\varphi}^{\vec{\tb}} A, \partial_x]$, for all 
$\vec{\tb}=(\tb_1,\ldots,\tb_\nu)\in \mathbb{N}^{\nu}$ with 
$0\le |\vec{\tb}| \leq \rho-2$  have the following properties, 
for any $s_0\le s\le \mathcal{S}$, with possibly $\mathcal{S}=+\infty$:
\begin{itemize}
		
\item[(i)] 
for any $m_{1},m_{2}\in \mathbb{R}$, $m_{1},m_{2}\geq0$
and $m_{1}+m_{2}=\rho-|\vec{\tb}|$ 
one has that  the operator
$\langle D_{x}\rangle^{m_{1}} \del_{\f}^{\vec{\mathtt{b}}}A \langle D_{x}\rangle^{m_{2}}$
is Lip-$0$-tame according to Def. \ref{LipTameConstants} 
and we set
\begin{equation}\label{megaTame2}
\gotM^{\gamma}_{\del_{\f}^{\vec{\mathtt{b}}}A}(-\rho+|\vec{\tb}|,s):=
\sup_{\substack{m_{1}+m_{2}=\rho-|\vec{\tb}|\\
m_{1},m_{2}\geq0}}\gotM^{\gamma}_{\langle D_{x}\rangle^{m_{1}} \del_{\f}^{\vec{\mathtt{b}}}A 
\langle D_{x}\rangle^{m_{2}}}(0,s);
\end{equation}
		
\item[(ii)]
for any $m_{1},m_{2}\in \mathbb{R}$, $m_{1},m_{2}\geq0$
and $m_{1}+m_{2}=\rho- |\vec{\tb}|-1$ 
one has that  
$\langle D_{x}\rangle^{m_{1}} [\del_{\f}^{\vec{\mathtt{b}}}A,\del_{x}] \langle D_{x}\rangle^{m_{2}}$
is Lip-$0$-tame according to Def. \ref{LipTameConstants} 
and we set
\begin{equation}\label{megaTame4}
\gotM^{\gamma}_{[\del_{\f}^{\vec{\mathtt{b}}}A,\del_{x}]}(-\rho+|\vec{\tb}|+1,s):=
\sup_{\substack{m_{1}+m_{2}=\rho-|\vec{\tb}|-1\\
m_{1},m_{2}\geq0}}\gotM^{\gamma}_{\langle D_{x}\rangle^{m_{1}} 
[\del_{\f}^{\vec{\mathtt{b}}}A,\del_{x}] \langle D_{x}\rangle^{m_{2}}}(0,s).
\end{equation}
\end{itemize}
We define for $0\le \tb\le \rho-2$
\begin{equation}\label{Mdritta}
\begin{aligned}\mathbb{M}^{\gamma}_{A}(s, \mathtt{b}):=&\max_{0\leq |\vec{\tb}|\leq \tb}\max\left(
\gotM_{\del_{\f}^{\vec{\mathtt{b}}}A}^{\gamma}(-\rho+|\vec{\tb}|,s),
\gotM_{\del_{\f}^{\vec{\mathtt{b}}}[A,\del_{x}]}^{\gamma}(-\rho+|\vec{\tb}|+1,s)\right).
\end{aligned}
\end{equation}
Assume now that  the set $\calO$ 
and the operator $A$ depend on 
$i=i(\oo)$, and are well defined for $\oo\in \calO_0\subseteq\Omega_{\e}$
for all $i$ satisfying \eqref{ipopiccolezza}.
We consider  $i_{1}=i_{1}(\oo)$, $i_{2}=i_{2}(\oo)$ and 
for $\oo\in \calO(i_1)\cap\calO(i_2)$
we define
\begin{equation}\label{DELTA12}
\Delta_{12}A:=A(i_1)-A(i_2).
\end{equation}
We require the following:

\vspace{0.5em}
\noindent		
$\bullet$ The operators
$\del_{\f}^{\vec{\mathtt{b}}}\Delta_{12}A $, 
$[\del_{\f}^{\vec{\mathtt{b}}}\Delta_{12}A ,\del_{x}]$, for  $0\le |\vec{\tb}|\leq \rho-3$, 
have the following properties, for any $s_0\le s\le \mathcal{S}$, 
with possibly $\mathcal{S}=+\infty$:

\begin{itemize}
\item[(iii)] 
for any $m_{1},m_{2}\in \mathbb{R}$, $m_{1},m_{2}\geq0$
and $m_{1}+m_{2}=\rho-|\vec{\tb}|-1$ 
one has that  
$\langle D_{x}\rangle^{m_{1}} \del_{\f}^{\vec{\mathtt{b}}} \Delta_{12}A  \langle D_{x}\rangle^{m_{2}}$
is  bounded on $H^{p}$ into itself. More precisely there is a positive constant 
$\gotN_{\del_{\f}^{\vec{\mathtt{b}}}\Delta_{12}A }(-\rho+|\vec{\tb}|+1,p)$ such that, for any $h\in H^{p}$,
we have
\begin{equation}\label{megalipTame2}
\sup_{\substack{m_{1}+m_{2}=\rho-|\vec{\tb}|-1\\
m_{1},m_{2}\geq0}}
\|\langle D_{x}\rangle^{m_{1}}
\del_{\f}^{\vec{\mathtt{b}}}\Delta_{12}A\langle D_{x}\rangle^{m_{2}} h\|_{p}\leq 
\gotN_{\del_{\f}^{\vec{\mathtt{b}}}\Delta_{12}A }(-\rho+|\vec{\tb}|+1,p)\|h\|_{p};
\end{equation}
		
\item[(iv)] 
for any $m_{1},m_{2}\in \mathbb{R}$, $m_{1},m_{2}\geq0$
and $m_{1}+m_{2}=\rho-|\vec{\tb}|-2$ 
one has that  
$\langle D_{x}\rangle^{m_{1}} [\del_{\f}^{\vec{\mathtt{b}}} \Delta_{12}A ,\del_{x}] 
\langle D_{x}\rangle^{m_{2}}$
is  bounded on $H^{p}$ into itself. More precisely there is a positive constant 
$\gotN_{[\del_{\f}^{\vec{\mathtt{b}}}\Delta_{12}A,\del_{x}] }(-\rho+|\vec{\tb}|+2,p)$ such that for any 
$h\in H^{p}$ one has
\begin{equation}\label{megalipTame4}
\sup_{\substack{m_{1}+m_{2}=\rho-|\vec{\tb}|-2\\
m_{1},m_{2}\geq0}}
\|\langle D_{x}\rangle^{m_{1}}[
\del_{\f}^{\vec{\mathtt{b}}}\Delta_{12}A,\del_{x}]\langle D_{x}\rangle^{m_{2}} h\|_{p}\leq 
\gotN_{[\del_{\f}^{\vec{\mathtt{b}}}\Delta_{12}A, \del_{x}] }(-\rho+|\vec{\tb}|+2,p)\|h\|_{p}.
\end{equation}
\end{itemize}
		
We define for $0\le \tb\le \rho-3$
\begin{equation}\label{Mdrittaconlai}
\begin{aligned}\mathbb{M}_{\Delta_{12}A }(p, \mathtt{b}):=&\max_{0\le |\vec{\tb}|\leq \tb}\max \left(
\gotN_{\del_{\f}^{\vec{\mathtt{b}}}\Delta_{12}A }(-\rho+|\vec{\tb}|+1,p),
\gotN_{\del_{\f}^{\vec{\mathtt{b}}}[\Delta_{12}A ,\del_{x}]}(-\rho+|\vec{\tb}|+2,p)\right).
\end{aligned}
\end{equation}

By construction one has that 
$\mathbb{M}^{\gamma}_{A}(s, \mathtt{b}_1)\leq 
\mathbb{M}^{\gamma}_{A}(s, \mathtt{b}_2)$
 if $\tb_1\le \tb_2\le \rho-2$ and 
$\mathbb{M}_{\Delta_{12}A }(p, \mathtt{b}_1)\leq 
\mathbb{M}_{\Delta_{12}A }(p, \mathtt{b}_2)$
if $\mathtt{b}_1\leq \mathtt{b}_2\leq \rho-3$.
\end{defi}

For the properties of the classes of operators we introduced above, we refer to Appendix \ref{restismooth}.

\section{Regularization procedure}\label{reg:procedure}

The aim of this section is to prove Theorem \ref{risultatosez8}.


\subsection{Flow of hyperbolic Pseudo differential PDEs}\label{flowPDE}

First we analyze the structure of the flow map that we use to conjugate the operator \eqref{LomegaDP} to a diagonal operator plus a smoothing term.

\medskip

We study the  flow $\Psi^{\tau}$ of the vector field generated by the Hamiltonian
\begin{equation}\label{pseudo}
S(\tau,\f,u)=\frac{1}{2} \int 	b(\tau,\f,x) u^2 dx
\qquad
b(\tau,\f,x):=\frac{\be(\f,x)}{1+\tau \be_{x}(\f,x)}
\end{equation}	
and $\beta$ is some smooth function.
We first need to show that $\Psi^{\tau}$ is well defined as map on $H^{s}$
(see Proposition \ref{DPdiffeo}).
Then
we study  the structure of $\Psi^{\tau}\calL_\oo (\Psi^{\tau})^{-1}$, see Proposition
\ref{ConjugationLemma}.

\noindent
 The  flow associated to the Hamiltonian \eqref{pseudo} is given by
 \begin{equation}\label{diffeotot}
\del_{\tau}\Psi^{\tau}(u)=(J\circ b) \Psi^{\tau}(u),\qquad \; \Psi^{0}u=u,
\end{equation}
where $b(\tau,\f,x)$ is defined  in \eqref{pseudo} with $\beta\in C^{\infty}(\T^{\nu+1})$ to be determined.\\
In the following proposition we prove that the flow of \eqref{diffeotot} $\Psi^{\tau}=\mathcal{C}^{\tau}\circ \mathcal{A}^{\tau}$, where $\mathcal{A}^{\tau}$ is the operator 
\begin{equation}\label{ignobel}
\begin{aligned}
&\mathcal{A}^{\tau}h(\varphi, x):=(1+\tau \beta_x(\varphi, x)) h(\varphi, x+\tau\beta(\varphi, x)), \qquad \quad \,\, \varphi\in \T^{\nu},\, x\in \T,\\
&(\mathcal{A}^{\tau})^{-1}h(\varphi, y):=(1+ \tilde{\beta}_y(\tau,\varphi, y))\,h(\varphi, y+\tilde{\beta}(\tau, \varphi, y)), \quad \varphi\in \T^{\nu},\, y\in \T,
\end{aligned}
\end{equation}
where $\tilde{\be}(\tau;x,\x)$ is such that
\[
x\mapsto y=x+\tau\be(\f,x) \; \Leftrightarrow \; y\mapsto x=y+\tilde{\be}(\tau,\f,x), \;\; \tau\in[0,1],
\]
and $C^{\tau}$ is the sum of a pseudo differential operator of order $-1$ with a smoothing remainder belonging to the class $\gotL_{\rho, p}$ for any $\rho\in\mathbb{N}$, $\rho\geq 3$, $s_0\le p\le p_0(\rho)$ provided that $\beta$ satisfies an appropriate $\rho$-smallness condition (see \eqref{flow1}).

 First we define
\begin{equation}\label{pasta6}
  \Lambda:=(1-\partial_{xx})^{-1},
 \qquad  \mathtt{X}:=\partial_x\circ b \qquad b:=\frac{\beta}{1+\tau \beta_x}.
\end{equation}
 We remark that the torus diffeomorphism $\mathcal{A}^{\tau}$ satisfies 
 \begin{equation}\label{flussoKDV}
\del_{\tau}\mathcal{A}^{\tau}=\mathtt{X}\mathcal{A}^{\tau},\qquad \mathcal{A}^{0}=\mathrm{I}.
\end{equation}
We refer to the Appendix \ref{someprop} for some properties of the operator $\mathcal{A}^{\tau}$ in \eqref{ignobel}. 
\begin{prop}\label{DPdiffeo}
Let $\calO \subseteq \R^\nu$ be a compact set. Fix $\rho\geq 3$, $\mathcal{S}> s_0$  large enough  and consider a function 
$\beta:=\beta(\omega, \mathfrak{I}(\omega))\in C^{\infty}(\T^{\nu+1})$, Lipschitz in 
$\omega\in \calO$ and  in the variable 
$\mathfrak{I}$.
There exist  $\su=\su(\rho)>0$ $\su\ge \tilde{\s}=\tilde{\s}(\rho)>0$  and $1>\delta=\delta(\rho,\mathcal{S})>0$ such that if
\begin{equation}\label{flow1}
\| \beta\|^{\gamma, \calO}_{s_0+\su}\leq \delta, 
\end{equation}
then, for any $\f\in \T^{\nu}$, the equation \eqref{diffeotot}
has a unique solution $\Psi^{\tau}(\f)$ 
in the space 
\[
C^{0}([0,1];H^{s}_{x})\cap C^{1}([0,1];H_x^{s-1}), \quad \forall s_0\le s\le \mathcal{S}.
\]
Moreover, for any $s_0\le p \le s_0+\s_1-\tilde\s$, one  has $\Psi^{\tau}=\mathcal{A}^{\tau}\circ \mathcal{C}^{\tau}$ where $\mathcal{A}^{\tau}$ is defined in \eqref{ignobel} and
\begin{equation}\label{ignobel2}
\mathcal{C}^{\tau}=\Theta^{\tau}+R^{\tau}(\varphi), \qquad \Theta^{\tau}:=\op(1+\vartheta(\tau, \varphi, x, \xi))
\end{equation}
with (recall \eqref{norma}), for any $s\geq s_0$,
\begin{equation}\label{varthetaStima}
\lvert \vartheta \rvert^{\gamma, \calO}_{-1, s, \alpha} \le_{s,\alpha,\rho} \lVert \beta \rVert^{\g, \calO}_{s+\su},
\qquad
\rvert \Delta_{12} \vartheta \lvert_{-1, p, \alpha} \le_{p, \alpha, \rho}  \lVert \Delta_{12} \beta \rVert_{p+\su}.
\end{equation}
and  $R^{\tau}(\varphi)\in \gotL_{\rho,p}(\calO)$ (see Def. \ref{ellerho}) with, for $s_0\le s\le \mathcal{S}$, 
\begin{equation}\label{ignobel3}
\mathbb{M}^{\gamma}_{R^{\tau}}(s,\mathtt{b})\le_{s,\alpha,\rho} \lVert\beta\rVert^{\gamma, \calO}_{s+\su}, \quad 
0\le \tb \le \rho-2, 
\qquad
\mathbb{M}_{\Delta_{12} R^{\tau} }(p,\mathtt{b})\le_{p,\rho} \lVert \Delta_{12} \beta  \rVert_{p+\su},
\quad  0\le \tb  \le \rho-3.
\end{equation}
\end{prop}

\begin{proof}
Let us reformulate the problem \eqref{diffeotot} as
$\Psi^{\tau}=\mathcal{A}^{\tau}\circ C^{\tau}$ where $C^{\tau}:=(\mathcal{A}^{\tau})^{-1}\circ \Psi^{\tau}$
satisfies the following system
\begin{equation}\label{sis}
\partial_{\tau} C^{\tau} u=L^{\tau} C^{\tau} u,\qquad 
C^0 u=u,
\end{equation}
where $L^{\tau}=\op(l(\tau, \varphi, x, \xi))$ is a pseudo differential operator  of order $-1$ of the form
\begin{equation}\label{Ltau}
L^{\tau}:=\mathcal{A}^{\tau} \Big(3 \Lambda \partial_x \circ b(\tau) \Big)(\mathcal{A}^{\tau})^{-1}= -\Big( \mathrm{I}-\Lambda\mathfrak{R}\Big)^{-1}\circ \Lambda\circ g(\tau, \varphi, x)\circ \partial_x\circ \tilde{\beta}(\varphi, x)
\end{equation}
where (recall \eqref{pasta6})
\begin{equation}\label{sissagames}
\begin{aligned}
 g(\tau, \varphi, x)&:=3(1+\tilde{\beta}_x^2(\varphi, x)), 
\qquad \mathfrak{R}:=\op(f_0(\varphi, x)+f_1(\varphi) \mathrm{i} \xi),\\
f_0(\varphi, x)&:=\tilde{\beta}_x^2+2 \tilde{\beta}_x-\frac{(1+\tilde{\beta}_x^2)}{2}\partial_{xx}
\left(\frac{1}{(1+\tilde{\beta}_x)^2} \right), \;\;\;
f_1(\varphi, x):=-\frac{3}{2}(1+\tilde{\beta}_x)^2\,\partial_x \left(\frac{1}{(1+\tilde{\beta}_x)^2} \right).
\end{aligned}
\end{equation}
\textbf{Analysis of $L^{\tau}$}.
The following estimates hold
\begin{equation}\label{ironman}
\begin{aligned}
&\lVert g \rVert^{\g, \calO}_{s}\le_s (1+\lVert \beta \rVert^{\g, \calO}_{s+1}\lVert \beta \rVert^{\g, \calO}_{s_0+1}), 
\qquad 
\lVert f_0 \rVert^{\g, \calO}_s+\lVert f_1 \rVert^{\g, \calO}_s
+\lvert f_0 + f_1\,\mathrm{i}\xi\rvert^{\g, \calO}_{1, s, \alpha}
\le_s \lVert \beta \rVert^{\g, \calO}_{s+3}, \\
& \lVert  \Delta_{12} g  \rVert_{p}+\lVert \Delta_{12} f_0  \rVert_{p}
+\lVert \Delta_{12} f_1  \rVert_{p}
\le_{p}\lVert  \Delta_{12} \beta  \rVert_{p+3}\\
\end{aligned}
\end{equation}
By the fact that $L^{\tau}$ in \eqref{Ltau} is one smoothing in space, the problem \eqref{sis} is locally well-posed in $H^s(\T_x)$.
By the composition Lemma \ref{James} we have that $\mathrm{I}-\Lambda \mathfrak{R}=\mathrm{I}-(\op(r)+R)$
with (see \eqref{ironman})
\begin{equation}\label{ironman2}
\lvert r \rvert^{\g, \calO}_{-1,s, \alpha}\le_{s, \alpha, \rho} \lVert \beta \rVert^{\g, \calO}_{s+\s_0}, \quad \mathbb{M}^{\gamma}_R(s,\tb)\le 
_{s,  \rho} \lVert \beta \rVert^{\g, \calO}_{s+\s_0}\,,\quad 0\le \tb \le \rho-2, 
\end{equation}
\begin{equation}\label{limonoff}
\lvert \Delta_{12} r \rvert _{-1,p, \alpha}\le_{p, \alpha, \rho} \lVert \beta \rVert_{p+\s_0} \quad \mathbb{M}_{\Delta_{12} R}(p,\tb)\le 
_{p,  \rho} \lVert \beta \rVert_{p+\s_0}\,,\quad 0\le \tb \le \rho-3
\end{equation}
for some $\s_0>0$.
By Lemma \ref{InvertibilityUtile}, Lemma \ref{James} and \eqref{ironman2} we have that 
$( \mathrm{I}-\Lambda\mathfrak{R})^{-1}=\mathrm{I}+\op(\tilde{r})+\tilde{R}$, $\Lambda\circ g \circ \partial_x\circ \tilde{\beta}=\op(d)+Q_{\rho}$ with bounds on the symbols and the tame constants similar to \eqref{ironman2}, \eqref{limonoff} with possibly larger $\s_0$.
Then 
\begin{equation*}
\begin{aligned}
L^{\tau} = &(\mathrm{I}+\op(\tilde{r})+\tilde{R})\circ (\op(d)+Q_{\rho})
\stackrel{{\rm Lemma}\; \ref{James}}{=}\op(l)+R_{\rho}
\end{aligned}
\end{equation*}
where
\begin{equation}\label{ellepiccolo} 
\lvert l \rvert^{\g, \calO}_{-1, s, \alpha}\le_{s,\alpha,\rho} \lVert \beta \rVert^{\g, \calO}_{s+\tilde{\s}_1}, \qquad
 \lvert \Delta_{12} l \rvert_{-1, p, \alpha} \le_{p,\rho} \lVert \Delta_{12} \beta  \rVert_{p+\tilde{\s}_1}.
\end{equation}
\begin{equation}\label{tir2}
 \mathbb{M}^{\gamma}_{R_{\rho}}(s,\mathtt{b})\le_{s,\rho} \lVert \beta \rVert^{\g, \calO}_{s+\tilde{\s}_1},\,\,0\le \tb \le \rho-2,
\qquad \mathbb{M}_{\Delta_{12}R_{\rho} }(p,\mathtt{b})\leq _{p,\rho}
\lVert \Delta_{12} \beta  \rVert_{p+\tilde{\s}_1},\,\,0\le \tb \le \rho-3,
\end{equation}
for some constant $\tilde{\s}_1=\tilde{\s}_1(\rho)$. Note that in principle we get a slightly different constant in each inequality, we
are just taking the biggest of them for simplicity.

\smallskip

\noindent
\textbf{Approximate solution of \eqref{sis}}.  
Now we look for an approximate solution 
$\Theta^{\tau}=\op(1+\vartheta(\tau, \varphi, x, \xi))$
 for the system \eqref{sis}.
In order to do that we  look for a symbol $\vartheta=\sum_{k=1}^{\rho-1} \vartheta_{-k}(\tau, \varphi, x, \xi)$ such that
\begin{equation*}
\partial_{\tau} \vartheta=l+l\# \vartheta+S^{-\rho},\qquad 
\vartheta(0, \varphi, x, \xi)=0.
\end{equation*}
We solve it recursively as follows:
\begin{equation}\label{sis2}
\begin{cases}
\partial_{\tau} \vartheta_{-1}=l,\\
\vartheta_{-1}(0, \varphi, x, \xi)=0,
\end{cases}
\qquad 
\begin{cases}
\partial_{\tau} \vartheta_{-k}=\mathtt{r}_{-k},\qquad   1<k\le \rho-1  \\
\vartheta_{-k}(0, \varphi, x, \xi)=0,
\end{cases}
\end{equation}
where 
\begin{equation}
\mathtt{r}_{-k}:=\sum_{j=1}^{k-1} l\#_{k-1-j} \vartheta_{-j}\in S^{-k}.
\end{equation}
Hence we have
\begin{equation}
\vartheta_{-1}(\tau)=\int_0^{\tau} l(s)\,ds, \qquad  \vartheta_{-k}(\tau)=\int_0^{\tau} \mathtt{r}_{-k}(s)\,ds.
\end{equation}
By recursion we have that
\begin{equation}
\lvert\vartheta_{-k}\rvert^{\g, \calO}_{-k, s, \alpha}\le_{s, \alpha, k}  \lVert \beta \rVert^{\g, \calO}_{s+k +\tilde{\s}_1}(\lVert \beta \rVert_{s_0+k+\tilde{\s}_1}^{\g, \calO})^{k-1}, \qquad 1\le k \le \rho-1,
\end{equation}
\begin{equation}
\begin{aligned}
\rvert \Delta_{12} \vartheta_{-k} \lvert_{-k, p, \alpha} 
&\le_{p, \alpha, k} 
\lVert \beta \rVert^{k-1}_{p+k+\tilde{\s}_1}\lVert \Delta_{12} \beta  \rVert_{p+k+\tilde{\s}_1},
\end{aligned}
\end{equation}
and so we get \eqref{varthetaStima}. 
We write $C^{\tau}=\Theta^{\tau}+R^{\tau}$, where $R^{\tau}$ is an  operator which satisfies the equation
\begin{equation}\label{locke}
\partial_{\tau} R^{\tau}=L^{\tau}R^{\tau}+Q^{\tau}, \qquad {\rm with} \qquad
R^{0}=0,
\end{equation}
where 
\begin{equation}
Q^{\tau}:=\op(\mathtt{q}(\tau))+ R_\rho \Theta^{\tau}, \quad \mathtt{q}(\tau):=\sum_{k=1}^{\rho-1}l\#_{\geq \rho-1-k} \theta_{-k}\in S^{-\rho}
\end{equation}
and by Lemma \ref{INCLUSIONEpseudoInclasseL}
\begin{equation}\label{vudu}
\mathbb{M}^{\gamma}_{\op(\mathtt{q})}(s, \tb)\le_{s, \rho}  \lVert \beta \rVert^{\g, \calO}_{s+\tilde{\s}_2}\lVert \beta \rVert^{\g, \calO}_{s_0+\tilde{\s}_2}
\end{equation}
with $\tilde{\s}_2:=\tilde{\s}_2(\rho)>\tilde{\s}_1$. 
By Lemma \ref{idealeds},
the operator $Q^{\tau}$ belongs to $\gotL_{\rho,p}(\calO)$ and 
we have the following bounds
\begin{equation}\label{tir}
\mathbb{M}^{\gamma}_{Q^{\tau}}(s,\mathtt{b}) 
\le \lVert \beta \rVert_{s+\tilde{\s}_2}^{\g, \calO}, \qquad
\mathbb{M}_{\Delta_{12} Q^{\tau} }(p,\mathtt{b})
\le \lVert \Delta_{12} \beta \rVert_{p+\tilde{\s}_2} .
\end{equation}
Note that these bounds hold uniformly for $\tau\in [0, 1]$.
Now we have to prove that $R^{\tau}$ is belongs to the class $\gotL_{\rho, p}$
(see Def. \ref{ellerho}). 
By this fact we will deduce that $C^{\tau}$ and its derivatives are tame on $H^s(\T^{\nu+1})$.

\vspace{0.8em}
\noindent \textbf{Estimates for the remainder $R^{\tau}$}. We prove the bounds \eqref{ignobel3}, i.e. 
 we show that $R^{\tau}$ belongs to $\gotL_{\rho,p}(\calO)$ in Def. \ref{ellerho} for $\tau\in[0,1]$. 
We use the integral formulation for the problem \eqref{locke}, namely
\begin{equation}\label{santana10}
R^{\tau}=\int_0^{\tau} (L^{t} R^{t}+Q^t)\,dt.
\end{equation}
We start by showing that $R^{\tau}$ satisfies item $(i)$
of Definition \ref{ellerho} with $\vec\tb=0$.
Let $m_{1},m_{2}\in \mathbb{R}$, $m_{1},m_{2}\geq0$
and $m_{1}+m_{2}=\rho$. We check that the operator 
$\langle D_{x}\rangle^{m_1}R^{\tau}\langle D_{x}\rangle^{m_{2}}$ is Lip-$0$-tame according to Definition \ref{LipTameConstants}.
 We have
 \begin{equation}\label{santana}
\begin{aligned}
\langle D_{x}\rangle^{m_1}R^{\tau}\langle D_{x}\rangle^{m_{2}}&=
\int_{0}^{\tau}
\langle D_{x}\rangle^{m_1}L^{t}  \langle D_{x}\rangle^{-m_1}  
\langle D_{x}\rangle^{m_1}R^{t}\langle D_{x}\rangle^{m_2}dt
+\int_{0}^{\tau}\langle D_{x}\rangle^{m_1}Q^{t}  \langle D_{x}\rangle^{m_2}dt.
\end{aligned}
\end{equation}
By \eqref{tir} we have, for $s_0\le s\le \mathcal{S}$, that
\begin{equation}\label{santana2}
\|\int_{0}^{\tau}
\langle D_{x}\rangle^{m_1}Q^{t}  \langle D_{x}\rangle^{m_2} u\,dt \|_{s}^{\gamma,\calO}\leq_s
\lVert \beta \rVert^{\g, \calO}_{s+\tilde{\s}_2}\|u\|_{s_0}+
\lVert \beta \rVert^{\g, \calO}_{s_0+\tilde{\s}_2}\|u\|_{s},
\end{equation}
for $\tau\in [0,1]$, $u\in H^{s} $. Moreover, by recalling the definition of $L^{t}$ in \eqref{ellepiccolo}, by using the fact that
$R_{\rho}$ in \eqref{tir2} is in the class $\gotL_{\rho, p}$
and using the estimates \eqref{ellepiccolo} on the symbol $l$
we claim that
\begin{equation}\label{santana3}
\|\int_{0}^{\tau}
\langle D_{x}\rangle^{m_1}L^{t}  \langle D_{x}\rangle^{-m_1}\,u\,dt\|_{s}^{\gamma,\calO}\leq_{s,\rho}
\|\be\|_{s+\tilde{\s}_1}^{\gamma,\calO}
\|u\|_{s_0}+
\|\be\|_{s_0+\tilde{\s}_1}^{\gamma,\calO}\|u\|_{s}.
\end{equation}
Indeed the bound for $\op(l)$ are trivial. In order to treat the remainder $R_{\rho}$ we note that
\[
\langle D_x \rangle^{m_1} R_{\rho} \langle D_x \rangle^{-m_1}=\langle D_x \rangle^{m_1} R_{\rho} \langle D_x \rangle^{\rho-m_1}\langle D_x \rangle^{-\rho}
\]
and $\langle D_x \rangle^{m_1} R_{\rho} \langle D_x \rangle^{\rho-m_1}$ is Lip-$0$-tame, since $R_{\rho}\in\gotL_{\rho, p}$, moreover $\langle D_x \rangle^{-\rho}\in \gotL_{\rho, p}$. Then by Lemma \ref{lem: 2.3.6} our claim follows.
By using bounds \eqref{santana2} and \eqref{santana3} with $s=s_0$ one obtains
\begin{equation}\label{santana4}
\begin{aligned}
\sup_{\tau\in [0,1]}\|\langle D_{x}\rangle^{m_1}R^{\tau}\langle D_{x}\rangle^{m_{2}}
u\|_{s_0}^{\gamma,\calO}&\leq_{\rho}
\|\be\|^{\gamma,\calO}_{s_0+\tilde{\s}_1}
\sup_{\tau\in [0,1]}\|\langle D_{x}\rangle^{m_1}R^{\tau}\langle D_{x}\rangle^{m_{2}}
u\|_{s_0}^{\gamma,\calO}+\lVert \beta \rVert_{s_0+\tilde{\s}_2}\|u\|_{s_0},
\end{aligned}
\end{equation}
hence, by \eqref{tir} and for $\delta$ in \eqref{flow1} small enough, one gets
\begin{equation}\label{santana5}
\sup_{\tau\in [0,1]}\|\langle D_{x}\rangle^{m_1}R^{\tau}\langle D_{x}\rangle^{m_{2}}
u\|_{s_0}^{\gamma,\calO}\leq_{s,\rho}
\|\be\|^{\gamma,\calO}_{s_0+\tilde{\s}_2}\|u\|_{s_0},
\end{equation}
for any $u\in H^{s}$.
Now for any $s_0\le s \le \mathcal{S}$,  by \eqref{santana2}, \eqref{santana3}, the smallness of $\be$ in \eqref{flow1} and estimate \eqref{santana5}, we have
\begin{equation*}
\sup_{\tau\in [0,1]}\|\langle D_{x}\rangle^{m_1}R^{\tau}\langle D_{x}\rangle^{m_{2}}
u\|_{s}^{\gamma,\calO}\leq_{s,\rho}
\|\be\|^{\gamma,\calO}_{s_0+\tilde{\s}_2}\|u\|_{s}+
\|\be\|^{\gamma,\calO}_{s+\tilde{\s}_2}\|u\|_{s_0}.
\end{equation*}
This means that 
\begin{equation}\label{santana12}
\sup_{\tau\in [0,1]}\gotM^{\gamma}_{R^{\tau}}(-\rho,s)
\leq_{s,\rho}\|\be\|^{\gamma,\calO}_{s+\tilde{\s}_2}.
\end{equation}
For $\vec{\tb}\in \mathbb{N}^{\nu}$
with 
$|\vec{\tb}|=\mathtt{b}\leq \rho-2$, 
we consider the operator 
$\del_{\f}^{\vec{\mathtt{b}}}R^{\tau}$
and we show that  $\langle D_{x}\rangle^{m_1} \del_{\f_{m}}^{\mathtt{b}}R^{\tau}
\langle D_{x}\rangle^{m_2}$ is Lip-$0$-tame for any
$m_{1},m_{2}\in \mathbb{R}$, $m_{1},m_{2}\geq0$
and $m_{1}+m_{2}=\rho-\mathtt{b}$.
We prove that
\begin{equation}\label{santana11}
\gotM_{\langle D_{x}\rangle^{m_1}\del_{\f}^{\vec{\mathtt{b}}}R^{\tau}\langle D_{x}\rangle^{m_2}}^{\gamma}(0,s)\leq_{s,\rho}\|\be\|^{\gamma,\calO}_{s+\tilde{\s}_3}
, \quad m_{1}+m_2
=\rho-\mathtt{b},
\end{equation}
for some $\tilde{\s}_3:=\tilde{\s}_3(\rho)\geq \tilde{\s}_2>0$,
by induction on $0\leq \mathtt{b}\leq \rho-1$.
For $\mathtt{b}=0$ the bound follows by \eqref{santana12}.
Assume now that \eqref{santana11} holds for any $\tilde{\mathtt{b}}$ such that $0\leq \tilde{\mathtt{b}}<\mathtt{b}\leq \rho-2$. We show \eqref{santana11} for 
$\mathtt{b}=\tilde{\mathtt{b}}+1$.
By \eqref{santana10} we have
\begin{equation}\label{santana8}
\begin{aligned}
\langle D_{x}\rangle^{m_1}\del_{\f}^{\vec{\mathtt{b}}}R^{\tau}\langle D_{x}\rangle^{m_2}&=
\sum_{\vec{\mathtt{b}_1}+\vec{\mathtt{b}_2}=\vec{\mathtt{b}}}C(|\vec{\mathtt{b}_1}|,|\vec{\mathtt{b}_2}|)\int_{0}^{\tau}
\langle D_{x}\rangle^{m_1}(\del_{\f}^{\vec{\mathtt{b}_1}}L^{t})          
\del_{\f}^{\vec{\mathtt{b_2}}}(R^{t})
\langle D_{x}\rangle^{m_2}dt\\
&+\int_{0}^{\tau}\langle D_{x}\rangle^{m_1}
(\del_{\f}^{\vec{\mathtt{b}}}Q^{t})
\langle D_{x}\rangle^{m_2}dt.
\end{aligned}
\end{equation}
By \eqref{tir} we know that, for any $t\in [0,1]$, the operator
$\langle D_{x}\rangle^{m_1}
(\del_{\f}^{\vec{\mathtt{b}}}Q^{t})
\langle D_{x}\rangle^{m_2}$ is Lip-$0$-tame.
We write
\begin{equation}\label{santana13}
\langle D_{x}\rangle^{m_1}(\del_{\f}^{\vec{\mathtt{b}_1}}L^{t})          
\del_{\f}^{\vec{\mathtt{b_2}}}(R^{t})
\langle D_{x}\rangle^{m_2}=
\langle D_{x}\rangle^{m_1}(\del_{\f}^{\vec{\mathtt{b}_1}}L^{t})    
\langle D_{x}\rangle^{-m_1-|\vec{\mathtt{b}_1}|}
\langle D_{x}\rangle^{m_1+|\vec{\mathtt{b}_1}|}
\del_{\f}^{\vec{\mathtt{b_2}}}(R^{t})
\langle D_{x}\rangle^{m_2}.
\end{equation}
We study the case
 $|\vec{\mathtt{b}_{2}}|\leq \mathtt{b}-1$.
By the inductive hypothesis we have that 
$\langle D_{x}\rangle^{m_1+|\vec{\mathtt{b}_1}|}
\del_{\f}^{\vec{\mathtt{b_2}}}(R^{t})
\langle D_{x}\rangle^{m_2}$
is Lip-$0$-tame since $m_1+|\vec{\mathtt{b}_1}|+m_2=\rho-|\vec{\mathtt{b}_{2}}|$, hence the bound \eqref{santana11} holds for $\mathtt{b}=|\vec{\mathtt{b}_{2}}|$.
By reasoning as for the proof of the bound \eqref{santana3} we have
\begin{equation}\label{santana14}
\|
\langle D_{x}\rangle^{m_1}(\del_{\f}^{\vec{\mathtt{b}_1}}L^{t})    
\langle D_{x}\rangle^{-m_1-|\vec{\mathtt{b}_1}|}u
\|_{s}^{\gamma,\calO}\leq_{s,\rho}
\|\be\|_{s+\tilde{\s}_3}^{\gamma,\calO}\|u\|_{s_0}+
\|\be\|_{s_0+\tilde{\s}_3}^{\gamma,\calO}\|u\|_{s},
\end{equation}
for $u\in H^{s}$, $s_0\le s\le \mathcal{S}$.
By \eqref{santana14}, the inductive hypothesis
on $\del_{\f}^{\vec{\mathtt{b}_2}}R^{\tau}$
 and \eqref{tir}
we get
\begin{equation}\label{santana15}
\gotM^{\gamma}_{
\langle D_{x}\rangle^{m_1}(\del_{\f}^{\vec{\mathtt{b}_1}}L^{t})          
\del_{\f}^{\vec{\mathtt{b_2}}}(R^{t})
\langle D_{x}\rangle^{m_2}
}(0,s)\leq_{s,\rho}\|\be\|^{\gamma,\calO}_{s+\tilde{\s}_3}.
\end{equation}
Note also that
By Lemma \ref{PROP},
bounds \eqref{tir2} and \eqref{ellepiccolo} we have that \eqref{santana14}
holds for $\mathtt{b}_1=0$. Hence
\begin{equation}\label{santana16}
 \begin{aligned}
\sup_{\tau\in [0,1]}\|
\langle D_{x}\rangle^{m_1}\del_{\f}^{\vec{\mathtt{b}}}R^{\tau}\langle D_{x}\rangle^{m_2}
u\|_{s}^{\gamma,\calO}
&\stackrel{(\ref{santana14})}{\leq_{s,\rho}}
\|\be\|^{\gamma,\calO}_{s+\tilde{\s}_3}
\sup_{\tau\in [0,1]}\|
\langle D_{x}\rangle^{m_1}\del_{\f}^{\vec{\mathtt{b}}}R^{\tau}\langle D_{x}\rangle^{m_2}
u\|_{s_0}^{\gamma,\calO}\\
&+
\|\be\|^{\gamma,\calO}_{s_0+\tilde{\s}_3}
\sup_{\tau\in [0,1]}\|
\langle D_{x}\rangle^{m_1}\del_{\f}^{\vec{\mathtt{b}}}R^{\tau}\langle D_{x}\rangle^{m_2}
u\|_{s}^{\gamma,\calO}\\
&+\|\be\|^{\gamma,\calO}_{s+\tilde{\s}_3}\|u\|_{s_0}+
\|\be\|^{\gamma,\calO}_{s_0+\tilde{\s}_3}\|u\|_{s}.
\end{aligned} 
\end{equation}
Hence using  \eqref{santana16} for $s=s_0$ and the smallness of $\be$ in 
\eqref{flow1}  
we get
\begin{equation}\label{santana18}
 \sup_{\tau\in [0,1]}\|
\langle D_{x}\rangle^{m_1}\del_{\f}^{\vec{\mathtt{b}}}R^{\tau}\langle D_{x}\rangle^{m_2}
u\|_{s_0}^{\gamma,\calO}\leq_{s,\rho}
\|\be\|^{\gamma,\calO}_{s_0+\tilde{\s}_3}\|u\|_{s_0}.
 \end{equation}
 Then using again \eqref{santana18} one obtains the bound for any $s_0\le s \le \mathcal{S}$
 \begin{equation}\label{santana19}
\sup_{\tau\in [0,1]}\gotM^{\gamma}_{R^{\tau}}(-\rho+\mathtt{b},s):=
\sup_{\tau\in [0,1]}\sup_{\substack{m_1+m_2=\rho-\mathtt{b}\\
m_1,m_2\geq0\\
|\vec{\tb}|\leq \tb}}\gotM^{\gamma}_{\langle D_{x}\rangle^{m_1} \del^{\vec{\tb}}_{\f}R^{\tau}
\langle D_{x}\rangle^{m_2}}(0,s)
\leq_{s,\rho}\|\be\|^{\gamma,\calO}_{s+\tilde{\s}_3}.
 \end{equation}
The estimates for $\mathfrak{M}_{[R^{\tau}, \partial_x]}(s)$ 
and $\mathfrak{M}_{[\partial_{\varphi}^{\vec{\tb}}R^{\tau}, \partial_x]}(s)$ follow by the same arguments. We have obtained the estimate for $\mathbb{M}^{\g}_{R^{\tau}}(s, \tb)$ in \eqref{ignobel3}.
\noindent
The estimate on the Lipschitz variation with respect to the variable $i$ \eqref{ignobel3}
follows by
 by Leibnitz rule and by \eqref{ignobel3} for $R^{\tau}$,
\eqref{ellepiccolo}, \eqref{tir} as in the previous cases.
We proved \eqref{ignobel3} with $\su=\tilde{\s}_3$.
\end{proof}

\begin{coro}\label{CoroDPdiffeo}
Fix $n \in \mathbb{N}$.
There exists $\s=\s(\rho)$ such that, if $\| \beta\|^{\gamma, \calO}_{s_0+\s}\leq 1$, then
 the flow $\Psi^{\tau}(\f)$ of \eqref{diffeotot} 
satisfies 
for $s\in [s_0, \mathcal{S}]$,
\begin{equation}\label{flow2}
\begin{aligned}
&\sup_{\tau\in[0,1]} \lVert \Psi^{\tau} u \rVert_s^{\gamma, \calO}
+\sup_{\tau\in[0,1]} \lVert ({\Psi}^{\tau})^{*} u \rVert_s^{\gamma, \calO}
\le_s \left(
\|u\|_{s}+\| b\|_{s+\s}^{\gamma, \calO}\|u\|_{s_0}
\right),
\end{aligned}
\end{equation}
\begin{equation*}
\sup_{\tau\in[0,1]} \lVert (\Psi^{\tau}-\mathrm{I})u \rVert_s^{\gamma, \calO}+
\sup_{\tau\in[0,1]} \lVert (({\Psi}^{\tau})^{*}-\mathrm{I})u \rVert_s^{\gamma, \calO}
\le_s \left(
\lVert \be \rVert_{s_0+\s}^{\gamma, \calO} \|u\|_{s+1}
+\| \be\|_{s+\s}^{\gamma, \calO}\|u\|_{s_0+1}
\right).
\end{equation*}
%
For any $|\al |\leq n$, $m_1,m_{2}\in \mathbb{R}$ such that $m_1+m_{2}=|\al|$, for any 
$s\geq s_0$ there exist $\mu_*=\mu_*(|\al|,m_1,m_2)$, $\s_*=\s_*(|\al|,m_1,m_2)$ and $\delta=\delta(m_1,s)$
such that if $
\|\be\|^{\gamma,\calO}_{s_0+\mu_*}\leq\delta,
$ and $\|\be\|^{\gamma,\calO}_{p+\s_*}\leq1$ for $p+\s_{*}\leq s_0+\mu_{*} $,
then one has
\begin{equation*}
\sup_{\tau\in[0,1]}\|\langle D_{x}\rangle^{-m_1}\del_{\f}^{\al}\Psi^{\tau}(\f)
\langle D_{x}\rangle^{-m_2}u\|^{\gamma,\calO}_{s}
\leq_{s,\mathtt{b},m_1,m_2}
\| u \|^{\gamma,\calO}_{s}+\|\be\|^{\gamma,\calO}_{s+\mu_*}\| u\|^{\gamma,\calO}_{s_0}
\end{equation*}
\begin{equation*}
\sup_{\tau\in[0,1]}  \| \langle D_{x}\rangle^{-m_1}\del_{\f}^{\al}\Delta_{12}\Psi^{\tau}(\f) 
\langle D_{x}\rangle^{-m_2}u\|_{p} \leq_{p,\mathtt{b},m_1,m_2}
\lVert u \rVert_{p}\lVert \Delta_{12} \be   \rVert_{p+\mu_*}, \quad m_{1}+m_2=|\al|+1.
\end{equation*}
\end{coro}

\begin{proof} The estimates on $\Psi^{\tau}$ follow by using  Lemmata \ref{bastalapasta}, \ref{buttalapasta2} and the result of Proposition \ref{DPdiffeo}.
In order to prove the bounds \eqref{flow2} for the adjoint $(\Psi^{\tau})^{*}$ it is sufficient to reformulate 
Proposition \ref{DPdiffeo} in terms of $(\Psi^{\tau})^{*}$. 
\end{proof}

\subsection{Quantitative Egorov analysis}

The system \eqref{diffeotot} 
 is an Hyperbolic PDE, thus we shall use a version of Egorov Theorem 
to study how pseudo differential operators change under the flow $\Psi^{\tau}$.
This is the content of Theorem \ref{EgorovQuantitativo} 
which provides precise estimates for the transformed operators. 

\paragraph{Notation.} Consider an integer $n\in \N$.
To simplify the notation for now on we shall write, 
$\Sigma^{*}_{n}$ the sum over indexes $k_1,k_2,k_3\in\N$ such that 
$k_1<n$, $k_1+k_2+k_3=n$ and $k_1+k_2\geq 1$.


\smallskip

We need the following lemma. 

\begin{lem}\label{Lemmino}
Let $\calO$ be a subset of $\mathbb{R}^{\nu}$. Let $A$ be the operator defined for $w\in S^m$ as
\begin{equation}
A w=w(f(x), g(x) \xi), \quad f(x):=x+\beta(x), \quad g(x):=(1+\beta_{x}(x))^{-1}
\end{equation}
for some smooth function $\beta$ such that
$
\lVert \beta \rVert^{\g, \calO}_{2s_0+2}< 1$. Then $A$ is bounded, namely $A w\in S^m$ and 
\begin{equation}\label{stima}
\lvert A w \rvert^{\g, \calO}_{m, s, \alpha}\le \lvert w \rvert^{\g, \calO}_{m, s, \alpha}
+\sum_{s}^{*}
\lvert w \rvert^{\g, \calO}_{m, k_1, \alpha+k_2} \lVert \beta \rVert^{\g, \calO}_{k_3+s_0+2}.
\end{equation}
for $s\ge 0$. For $s=s_0$  it is convenient to consider the rougher estimate
$\lvert A w \rvert^{\g, \calO}_{m, s_0, \alpha}\le \lvert w \rvert^{\g, \calO}_{m, s_0, \alpha+s_0}$.
\end{lem}
\begin{proof}
It follows directly by Lemma \ref{Lemminobis} in Appendix \ref{lemmitecnici}.
\end{proof}

\begin{teor}[{\bf Egorov}]\label{EgorovQuantitativo}
Fix $\rho \geq 3$, $p\geq s_0$, $m\in\mathbb{R}$ with $\rho+m> 0$. Let $w(x, \xi)\in S^m$ with 
$w=w(\omega, \mathfrak{I}(\omega))$, Lipschitz in $\omega\in \calO\subseteq\R^\nu$ and in the variable $\mathfrak{I}$. 
Let $\mathcal{A}^{\tau}$ be the flow of the system \eqref{flussoKDV}.
There exist $\su:=\su(m, \rho)$ and $\delta:=\delta(m, \rho)$ such that, if
\begin{equation}
\lVert \beta \rVert^{\gamma, \calO}_{s_0+\su}<\delta,
\end{equation}
then $\mathcal{A}^{\tau} \op(w) (\mathcal{A}^{\tau})^{-1}=\op(q(x, \xi))+R$ where $q\in S^m$ and 
$R\in\gotL_{\rho,p}(\calO)$. Moreover,  one has that 
 the following estimates hold:
\begin{align}
\lvert q \rvert^{\gamma, \calO}_{m, s, \alpha}&\le_{m, s, \alpha, \rho}  
\lvert w \rvert^{\gamma, \calO}_{m, s, \alpha+\su}+
\sum_{s}^{*} 
\lvert w \rvert^{\gamma, \calO}_{m, k_1, \alpha+k_2+\su} 
\lVert \beta \rVert^{\gamma, \calO}_{k_3+\su},\label{zeppelin}\\
\lvert \Delta_{12} q   \rvert_{m, p, \alpha} 
&\le_{m, p, \alpha, \rho} 
\lvert w \rvert_{m, p+1, \alpha+\su}\lVert \Delta_{12} \beta   \rVert_{p+1}
+\lvert \Delta_{12} w   \rvert_{m, p, \alpha+\su}\nonumber\\
&+\sum_{p+1}^{*}  
\lvert w \rvert_{m, k_1, \alpha+k_2+\su}\lVert \beta \rVert_{k_3+\su} \lVert \Delta_{12} \beta   \rVert_{s_0+1}
+\sum_{p}^{*}  
\lvert \Delta_{12} w   \rvert_{m, k_1, k_2+\alpha+\su}\lVert \beta \rVert_{k_3+\su}.\label{deeppurple}
\end{align}
Furthermore for any $\tb\le \rho-2$ and $s_0\le s\le \mathcal{S}$
\begin{equation}\label{CostanteTameR}
\begin{aligned}
\mathbb{M}^{\g}_{R}(s, \tb) &\le_{s, m, \rho}   \lvert w \rvert^{\gamma, \calO}_{m, s+\rho, \su}+\sum_{s+\rho}^{*} 
\lvert w \rvert^{\gamma, \calO}_{m, k_1, k_2+\su}\lVert \beta \rVert^{\gamma, \calO}_{k_3+\su},
\end{aligned}
\end{equation}
and for any $\tb\le \rho-3$,  
\begin{equation}\label{CostanteTameRdei}
\begin{aligned}
\mathbb{M}_{\Delta_{12} R  }(p, \tb) &\le_{m, p, \rho}  
\lvert w \rvert_{m, p+\rho, \su}\lVert \Delta_{12} \beta   \rVert_{p+\su}
+\lvert \Delta_{12} w   \rvert_{m, s+\rho, \su}\\
&+\sum_{p+\rho}^{*}  
\lvert w \rvert_{m, k_1, k_2+\su} \lVert \beta \rVert_{k_3+\su}\lVert \Delta_{12} \beta   \rVert_{s_0+\su}+\sum_{p+\rho}^{*} 
\lvert \Delta_{12} w   \rvert_{m, k_1, k_2+\su}\lVert \beta \rVert_{k_3+\su}.
\end{aligned}
\end{equation}
\end{teor}
\begin{proof}
The operator $P(\tau):=\mathcal{A}^{\tau} \op(w) (\mathcal{A}^{\tau})^{-1}$ satisfies the Heisenberg equation
\begin{equation}\label{ars}
\begin{cases}
\partial_{\tau} P(\tau)=[\mathtt{X}, P(\tau)], \qquad \mathtt{X}=\partial_x\circ b=:\op(\chi),\\
P(0)=\op(w).
\end{cases}
\end{equation}
We construct an approximate solution of \eqref{ars} by considering a pseudo differential operator $\op(q)$ with 
\begin{equation}\label{balconata0}
q=\sum_{k=0}^{m+\rho-1} q_{m-k}(x, \xi)
\end{equation}
such that  (see \eqref{ars} and note that $\chi:=b\,\mathrm{i}\xi+b_x$)
\begin{equation}\label{ars2}
\begin{cases}
\partial_{\tau} q_m=\{b \xi, q_m \},\\
q_m(0)=w
\end{cases}
\qquad
\begin{cases}
\partial_{\tau} q_{m-k}=\{b \xi, q_{m-k} \}+r_{m-k}\,\\
q_{m-k} (0)=0
\end{cases} \, \qquad k\ge 1
\end{equation}
where for $k\ge 1$ (recall \eqref{tazza6}), denoting by $\mathtt{w}=\mathtt{w}(h, k):=k-h+1$,
\begin{align*}
r_{m-k}:&= \frac{1}{\mathrm{i}}\{ b_x, q_{m-k+1}\}-\sum_{h=0}^{k-1}q_{m-h}\#_{\mathtt{w}} \chi\\
& =-\frac{1}{\mathrm{i}} \partial_{\xi} q_{m-k+1}\,b_{xx}-\sum_{h=0}^{k-1} \frac{1}{\mathrm{i}^{\mathtt{w}}(\mathtt{w})!}(\partial_{\xi}^{\mathtt{w}} q_{m-h})(\partial_{x}^{\mathtt{w}} \chi)\,\in S^{m-k}\,.
\end{align*}
By Lemma \ref{James}, or directly by interpolation, one has 
\begin{equation}\label{stimaResti}
\lvert r_{m-k} \rvert^{\gamma, \calO}_{m-k, s, \alpha}\le \sum_{h=0}^{k-1}  \lvert q_{m-h} \rvert^{\gamma, \calO}_{m-h, s, \alpha+\mathtt{w}} +\sum_{h=0}^{k-1}  \lvert q_{m-h} \rvert^{\gamma, \calO}_{m-h, s_0, \alpha+\mathtt{w}} \lVert \beta \rVert^{\gamma, \calO}_{s+\mathtt{w}+2},
\end{equation}
\begin{equation}\label{anagrafe00}
\begin{aligned}
\lvert \Delta_{12} r_{m-k}  \rvert_{m-k, p, \alpha} &\le 
\sum_{h=0}^{k-1}  \lvert \Delta_{12} q_{m-h}  \rvert_{m-h, p, \alpha+\mathtt{w}} 
+\sum_{h=0}^{k-1}  \lvert \Delta_{12} q_{m-h}   \rvert_{m-h, s_0, \alpha+\mathtt{w}} \lVert \beta \rVert_{p+\mathtt{w}+2}\\
&+\sum_{h=0}^{k-1}  \lvert  q_{m-h} \rvert_{m-h, p, \alpha+\mathtt{w}} 
\lVert \Delta_{12} \beta   \rVert_{p+\mathtt{w}+2}.
\end{aligned}
\end{equation}
Hence we can solve \eqref{ars2} iteratively. 
Let us denote by $\gamma^{\tau_0,\tau}(x,\x)$ the solution of the characteristic system
\begin{equation}\label{charsys}
\left\{\begin{aligned}
&\frac{d}{ds}x(s)=-b(s,x(s))\\
&\frac{d}{ds}\x(s)=b_{x}(s,x(s))\x(s)
\end{aligned}\right.
\end{equation}
with initial condition $\gamma^{\tau_0,\tau_0}=(x,\x)$.
Then the first equation in \eqref{ars2} has the solution
\begin{equation}\label{ars3}
q_m(\tau, x, \xi)=w(\gamma^{\tau, 0}(x, \xi))
\end{equation}
where 
\begin{equation}
\gamma^{\tau, 0}(x, \xi)=\big(f(\tau, x), \xi g(\tau, x)\big), \qquad f(\tau, x):=x+\tau\beta(x), \quad g(\tau, x):=\frac{1}{1+\tau\beta_x(x)}.
\end{equation}
Hence by Lemma \ref{Lemmino} we have
\begin{equation}\label{Qm}
\lvert q_m \rvert^{\gamma, \calO}_{m, s, \alpha}\le_{s, \alpha} \lvert w \rvert^{\gamma, \calO}_{m, s, \alpha}
+\sum_{s}^{*} 
\lvert w \rvert^{\gamma, \calO}_{m, k_1, \alpha+k_2} \lVert \beta \rVert^{\gamma, \calO}_{k_3+s_0+2}.
\end{equation}
For any $k \geq 1$, the solution of \eqref{ars2} is
\begin{equation}\label{gnomo}
q_{m-k}(\tau, x, \xi)=\int_0^{\tau} r_{m-k}(\gamma^{0, t} \gamma^{\tau, 0}(x, \xi)) \,dt.
\end{equation}
We observe that 
\begin{equation}
\gamma^{0, t} \gamma^{\tau, 0}(x, \xi)=(\tilde{f}, \tilde{g}\,\xi)
\end{equation}
with
\begin{equation}\label{freccina}
\tilde{f}(t, \tau, x):=x+\tau \beta(x)+\tilde{\beta}(t, x+\tau \beta(x)), \qquad \tilde{g}(t, \tau, x):= \frac{1+t\beta_x(\tilde{f}(t,\tau, x))}{1+\tau \beta_x(x)}.
\end{equation}
Thus if $\tilde{A} r:=r(\tilde{f}, \tilde{g}\,\xi)$ we have (recall that $\tau\in [0, 1]$)
\begin{equation}
\lvert q_{m-k} \rvert^{\gamma, \calO}_{m-k, s, \alpha}\le_{s, \alpha} 
\lvert \tilde{A} r_{m-k} \rvert^{\gamma, \calO}_{m-k, s, \alpha}, 
\quad \lvert q_{m-k} \rvert^{\gamma, \calO}_{m-k, s_0, \alpha}\le_{\alpha} 
\lvert \tilde{A} r_{m-k} \rvert^{\gamma, \calO}_{m-k, s_0, \alpha}\le 
\lvert r_{m-k} \rvert^{\gamma, \calO}_{m-k, s_0, \alpha+s_0}
\end{equation}
and by Lemma \ref{Lemmino} with $A\rightsquigarrow \tilde{A}$
\begin{equation}\label{interpol}
\lvert q_{m-k} \rvert^{\gamma, \calO}_{m-k, s, \alpha}\le_{s, \alpha} 
\lvert r_{m-k} \rvert^{\gamma, \calO}_{m-k, s, \alpha}
+\sum_{s}^{*} 
\lvert r_{m-k} \rvert^{\gamma, \calO}_{m-k, k_1, \alpha+k_2} \lVert \beta \rVert^{\gamma, \calO}_{k_3+s_0+2}.
\end{equation}
We want to prove inductively,  for $k=0,\dots, m+\rho$,
\begin{equation}\label{stimaQ}
\begin{aligned}
\lvert q_{m-k} \rvert^{\gamma, \calO}_{m-k, s, \alpha} 
\le_{s, \alpha, \rho} &\lvert w \rvert^{\gamma, \calO}_{m, s, \alpha+ 2 k}
+\sum_{s}^{*} \lvert w \rvert^{\gamma, \calO}_{m, k_1, \alpha+k_2+k(s_0+2)} 
\lVert \beta \rVert^{\gamma, \calO}_{k_3+s_0+2+k}, \\
 \lvert q_{m-k} \rvert^{\gamma, \calO}_{m-k, s_0, \alpha} \le_{\alpha, \rho}& \lvert w \rvert^{\gamma, \calO}_{m, s_0, \alpha+s_0+ k(s_0+2)}.
\end{aligned}
\end{equation}
For $k=0$ this is proved in \eqref{Qm}. Now assume that \eqref{stimaQ} holds, up to some $k-1\ge 0$.
We use \eqref{stimaResti} to bound $q_{m-k}$. First we give a bound for $r_{m-k}$ in terms of the norm of the symbol $w$. To shorten the formulas let us denote $\mathtt{t}:=s_0+2$.

\noindent
By \eqref{stimaResti} and the inductive hypothesis \eqref{stimaQ} we get
\begin{equation}\label{interpol2}
\begin{aligned}
\lvert r_{m-k} \rvert^{\gamma, \calO}_{m-k, s, \alpha} 
&
\le_{s,\al,\rho}  \lvert w \rvert^{\gamma, \calO}_{m, s, \alpha+ 2k}+\sum_{s}^{*} 
\lvert w \rvert^{\gamma, \calO}_{m, k_1, \alpha+k_2+k\mathtt{t}} \lVert \beta \rVert^{\gamma, \calO}_{k_3+\mathtt{t}+k}. 
\end{aligned}
\end{equation}
Then by \eqref{interpol} and \eqref{interpol2}
\begin{equation*}
\begin{aligned}
\lvert q_{m-k} \rvert^{\gamma, \calO}_{m-k, s, \alpha} &
\le_{ s, \alpha, k} 
\sum_{s}^{*} 
\Big( \sum_{n_1+n_2+n_3=k_1+k} \lvert w \rvert^{\gamma, \calO}_{m, n_1, \alpha+n_2+k \mathtt{t}+k_2} \lVert \beta \rVert^{\gamma, \calO}_{n_3+\mathtt{t}+k} \Big) \lVert \beta \rVert^{\gamma, \calO}_{k_3+\mathtt{t}}\\
&+\lvert w \rvert^{\gamma, \calO}_{m, s, \alpha+ 2k}+\sum_{s}^{*} \lvert w \rvert^{\gamma, \calO}_{m, k_1, \alpha+k_2+k \mathtt{t}} 
\lVert \beta \rVert^{\gamma, \calO}_{k_3+\mathtt{t}+k} +\sum_{s}^{*} 
\lvert w \rvert^{\gamma, \calO}_{m, k_1, \alpha+k_2+2 k} 
\lVert \beta \rVert^{\gamma, \calO}_{k_3+\mathtt{t}}\\
&
\le_{s, \alpha, k}  \lvert w \rvert^{\gamma, \calO}_{m, s, \alpha+2 k}+\sum_{s}^{*} 
\lvert w \rvert^{\gamma, \calO}_{m, k_1, \alpha+k_2+k \mathtt{t}} 
\lVert \beta \rVert^{\gamma, \calO}_{k_3+\mathtt{t}+k}
\end{aligned}
\end{equation*}
that is the estimate \eqref{stimaQ}.
By \eqref{gnomo} we have
\begin{equation}\label{gnomo2}
\Delta_{12} q_{m-k}(\tau, x, \xi)  =\int_0^{\tau} \Delta_{12} \big(r_{m-k}(\gamma^{0, s} \gamma^{\tau, 0}(x, \xi)) \big) \,ds
\end{equation}
and recalling \eqref{freccina}
\begin{equation}\label{balconata}
\begin{aligned}
\lvert \Delta_{12} q_{m-k} \rvert_{m-k, s, \alpha} \le_{s, \alpha}  &\lvert \tilde{A}( \partial_x  r_{m-k})\,(\Delta_{12} \tilde{f}  ) \rvert_{m-k, s, \alpha}+\lvert \tilde{A} ( \partial_{\xi}  r_{m-k})\,(\Delta_{12} \tilde{g} \,\xi  ) \rvert_{m-k, s, \alpha}\\
&+\lvert \tilde{A} (\Delta_{12} r_{m-k}  ) \rvert_{m-k, s, \alpha}.
\end{aligned}
\end{equation}
The first two terms of the right hand side in \eqref{balconata} are bounded by \eqref{interpol2} and 
Lemma $A.1$ in Appendix $A$ of \cite{FGMP}.
For the last summand we proceed by induction as above using \eqref{anagrafe00}. We obtain
\begin{equation}
\begin{aligned}
\lvert \Delta_{12} q_{m-k}   \rvert_{m-k, p, \alpha} &\le  
\lvert w \rvert_{m, p+1, \alpha+2 k+1}\lVert \Delta_{12} \beta   \rVert_{p+1}\\
&+\sum_{p+1}^{*} \lvert w \rvert_{m, k_1, \alpha+k_2+s_0+1+k \mathtt{t}}\lVert \beta \rVert_{k_3+s_0+\mathtt{t}+k} \lVert \Delta_{12} \beta   \rVert_{s_0+1}\\
&+\lvert w \rvert_{m, s_0+1, \alpha+s_0+1+k \mathtt{t}}
\lVert \Delta_{12} \beta   \rVert_{s_0+1}+\lvert \Delta_{12} w   \rvert_{m, p, \alpha+2 k}\\
&+\sum_{p}^{*} \lvert \Delta_{12} w  \rvert_{m, k_1, k_2+\alpha+k \mathtt{t}}\lVert \beta \rVert_{k_3+ s_0+\mathtt{t}+k}.\\
\end{aligned}
\end{equation}

\noindent
Then we have \eqref{zeppelin} and \eqref{deeppurple}.
Now we have (recall \eqref{balconata0})
\begin{equation}
P(\tau)=Q+R, \qquad Q=\op(q)\in OPS^m
\end{equation}
and by the construction of $Q$ we get that
\begin{equation}
\begin{cases}
\partial_{\tau} R(\tau)=[\mathtt{X}, R]+\mathcal{M},\\
R(0)=0
\end{cases}
\end{equation}
where 
\begin{equation}
\mathcal{M}=-\op\Big(\mathrm{i}\{b_x, q_{-\rho+1}\}+\sum_{k=0}^{m+\rho-1} q_{m-k}\#_{\geq m-k+1+\rho} \chi\Big)\in OPS^{-\rho}.
\end{equation}
By Lemma \ref{INCLUSIONEpseudoInclasseL} we deduce 
that $\mathcal{M}\in \gotL_{\rho, p}$ and using \eqref{tazza1} (recall also the Definition \eqref{cancelletti}) 
we have for all $s_0\le s\le \mathcal{S}$ 
\begin{align}
\mathbb{M}^{\g}_{\mathcal{M}}(s, \tb)
&\le_{s, \rho, m}  \lvert w \rvert^{\gamma, \calO}_{m, s+\rho, \su}+\sum_{s+\rho}^{*}
\lvert w \rvert^{\gamma, \calO}_{m, k_1, k_2+\su}\lVert \beta \rVert^{\gamma, \calO}_{k_3+\su}, \quad \tb\le \rho-2,\label{anagrafe44}\\
\mathbb{M}_{\Delta_{12} \mathcal{M}  }(p, \tb) &\le_{p} 
\lvert w \rvert_{m, p+\su, \su}\lVert \Delta_{12} \beta   \rVert_{p+\su}
+\lVert \Delta_{12} \beta   \rVert_{s_0+\su}\sum_{p+\rho}^{*}
\lvert w \rvert_{m, k_1, k_2+\su}\lVert \beta \rVert_{k_3+\su}\nonumber\\
&+\lvert \Delta_{12} w   \rvert_{m, p+\su, \su}+\sum_{p+\rho}^{*}
\lvert \Delta_{12} w   \rvert_{m, k_1, k_2+\su}\lVert \beta \rVert_{p+\su}, \quad \tb\le \rho-3\label{anagrafe4}
\end{align}
for some $\su>0$.
If $V(\tau):=R(\tau) \mathcal{A}^{\tau}$ then
it solves
$\partial_{\tau} V = \mathtt{X} V+\mathcal{M} \mathcal{A}^{\tau}$
and so 
\begin{equation}\label{blacksabbath}
V^{\tau}=\int_0^{\tau} \mathcal{A}^{\tau} (\mathcal{A}^s)^{-1} \mathcal{M} \mathcal{A}^{s}\,ds \quad \Rightarrow \quad R(\tau)=\int_0^{\tau} \mathcal{A}^{\tau} (\mathcal{A}^s)^{-1} \mathcal{M} \mathcal{A}^{s} (\mathcal{A}^{\tau})^{-1}\,ds.
\end{equation}
By Lemma \ref{preparailsugo} $R^{\tau}\in \gotL_{\rho, p}$ for any $\tau\in [0, 1]$.
By \eqref{casalotti} we have that, for any $\tau\in [0, 1]$, taking $\su$ possibly larger than before in order 
to fit the assumptions of Lemma \ref{preparailsugo},
\begin{equation}
\mathbb{M}^{\gamma}_{R^{\tau}}(s, \tb)\le_s \mathbb{M}^{\gamma}_{\mathcal{M}}(s)+\lVert \beta \rVert^{\gamma,\calO}_{s+\su}\mathbb{M}^{\gamma}_{\mathcal{M}}(s_0).
\end{equation}
Then by Leibniz rule and Lemma \ref{buttalapasta60} we have by \eqref{anagrafe4}
\begin{equation*}
\begin{aligned}
\mathbb{M}_{\Delta_{12} R  }(s, \tb) &\le_s 
\mathbb{M}^{\g}_{\mathcal{M}}(p, \tb)\lVert \Delta_{12} \beta  \rVert_{p}+
\mathbb{M}^{\g}_{\mathcal{M}}(p, \tb)\lVert \Delta_{12} \beta  \rVert_{p}\lVert \beta \rVert_{p+\su}\\
&+\mathbb{M}_{\Delta_{12} \mathcal{M}  }(p, \tb)+
\mathbb{M}_{\Delta_{12} \mathcal{M}  }(p, \tb) \lVert \beta  \rVert_{p+\su}.
\end{aligned}
\end{equation*}
We obtain \eqref{CostanteTameR} and \eqref{CostanteTameRdei} by using respectively \eqref{anagrafe44} and \eqref{anagrafe4}.
\end{proof}

\subsection{Conjugation of a class of first order operators}\label{conjFirst}

In this Section we prove an important abstract conjugation Lemma which is needed 
to prove Theorem \ref{risultatosez8}. 
We shall also recall a Moser-like theorem for first order linear operators
(see Proposition \ref{moser})
which has been proved in \cite{FGMP}.

\paragraph{A conjugation Lemma for a class of pseudo differential operators.}

\noindent
The following proposition describes the structure of an operator like $\mathcal{L}_{\omega}$ conjugated by the flow of a system like 
\eqref{diffeotot}.

\begin{prop}[\textbf{Conjugation}]\label{ConjugationLemma}
Let $\mathcal{O}$ be a subset of $\R^\nu$. Fix $\rho\geq 3$, $\alpha\in \mathbb{N}$, $p\geq s_0$ 
and consider a linear operator
\begin{equation}\label{onizuka6}
\mathcal{L}:= \oo\cdot\del_{\f}-J \circ (m+a(\varphi, x)) +\mathcal{Q}
\end{equation}
where $m=m(\omega)$ is a real constant,
$a=a(\omega,\mathfrak{I}(\omega))\in C^\infty(\T^{\nu+1})$ is real valued, 
both are Lipschitz in $\omega\in \calO$ and $a $ is Lipschitz in the variable $\mathfrak{I}$.
Moreover $\calQ=\op(\mathtt{q}(\varphi, x, \xi))+\widehat{\calQ}$ with  $\widehat{\mathcal{Q}}\in\gotL_{\rho,p}(\calO)$  
and $\tq=\tq(\omega,\mathfrak{I}(\omega))\in S^{-1}$ satisfying
\begin{equation}\label{docq}
\lvert \mathtt{q} \rvert_{-1, s, \alpha}^{\gamma, \calO}\le_{s, \alpha} \tk_1+ \tk_2 \lVert \mathtt{p} \rVert^{\g,\calO}_{s+\sigma_2},
\end{equation}
\begin{equation}\label{docqi}
\begin{aligned}
\lvert \Delta_{12} \mathtt{q}  \rvert_{-1, p, \alpha} &\le_{p, \alpha} \tk_3\,\lVert \Delta_{12} \mathtt{p}  \rVert_{p+\sigma_2}
(1+\lVert \mathtt{p} \rVert_{p+\sigma_2}).
\end{aligned}
\end{equation}
Here $\tk_1, \tk_2, \tk_3,\sigma_2>0$ are constants depending on $\mathtt{q}$ while 
$\mathtt{p}=\mathtt{p}(\omega,\mathfrak{I}(\omega))\in C^\infty(\T^{\nu+1})$, 
is Lipschitz in $\omega$ and in the variable $\mathfrak{I}$ .

\noindent
There are $\s_{3}=\s_{3}(\rho)\geq \tilde{\s_2}=\tilde{\s_2}(\rho)>0$ and $\delta_*:=\delta_*(\rho)\in (0, 1)$ such that, if
\begin{equation}\label{filini}
\lVert \bt \rVert^{\gamma, \calO}_{s_0+\s_3}+\lVert a \rVert^{\gamma, \calO}_{s_0+\s_3}+\tk_2 \lVert \mathtt{p}\rVert^{\gamma, \calO}_{s_0+\s_3}+\tk_1+\mathbb{M}^{\g}_{\widehat{\mathcal{Q}}}(s_0, \tb) \le \delta_*\,,
\end{equation}
the following holds for $p\leq s_0+\s_{3}-\tilde{\s_2}$.
Consider $\Psi:=\Psi^{1}$ 
the flow at time one of the system \eqref{diffeotot}, where $b$ is defined in \eqref{pasta6}. 
Then we have
\begin{equation}\label{ellepiu}
\mathcal{L}_+:=\Psi \mathcal{L}\Psi^{-1}=\oo\cdot\del_{\f}-J \circ (m+ a_+(\varphi, x))+\mathcal{Q}_+
\end{equation}
\begin{equation}\label{Round}
m+a_+(\varphi, x):=-(\omega\cdot\del_{\f} \tilde{\beta})(\varphi, x+\beta(\varphi, x))+(m+a(\varphi, x+\beta(\varphi, x)))(1+\tilde{\beta}_x(\varphi, x+\beta(\varphi, x)))
\end{equation}
with $\tilde{\beta}$ the function such that $x+\tilde{\beta}(\varphi, x)$ is the inverse of the diffeomorphism of the torus $x\mapsto x+\beta(\varphi, x)$.
The operator $\mathcal{Q}_+:=\op(\mathtt{q}_+(\varphi, x, \xi))+\widehat{\mathcal{Q}}_+$, with
\begin{equation}\label{lapiuimportante}
\begin{aligned}
\lvert \mathtt{q}_+ \rvert^{\g, \calO}_{-1, s, \alpha} &\le_{s, \alpha, \rho} \tk_1+\tk_2 \lVert \mathtt{p}\rVert^{\g, \calO}_{s+\sigma_3}+\lVert \beta \rVert^{\g, \calO}_{s+\sigma_3}+\lVert a \rVert^{\g, \calO}_{s+\sigma_3},\\
\lvert \Delta_{12} \mathtt{q}_+ \rvert_{-1, p, \alpha} &\le_{p, \alpha, \rho} 
\mathtt{k}_3(\lVert \Delta_{12} \mathtt{p} \rVert_{p+\sigma_3}+\lVert \Delta_{12} \mathtt{p} \rVert_{p+\sigma_3}\lVert \mathtt{p} \rVert_{p+\sigma_3})+\lVert \Delta_{12}\beta \rVert_{p+\sigma_3}+\lVert \Delta_{12} a \rVert_{p+\sigma_3}\\
\end{aligned}
\end{equation}
and
$\widehat{\mathcal{Q}}_+\in\gotL_{\rho,p}(\calO)$
with, for $s_0\le s \le \mathcal{S}$,
\begin{equation}\label{JeTame}
\mathbb{M}^{\g}_{\widehat{\mathcal{Q}}_+}(s, \tb)\le_{s, \rho} \mathbb{M}^{\g}_{\widehat{\calQ}}(s, \tb)
+ \|\be\|^{\g,\calO}_{s+\sigma_3} +\tk_1+\tk_2 \|\mathtt{p}\|^{\g,\calO}_{s+\sigma_3}+\|a\|^{\g,\calO}_{s+\sigma_3},
\quad \tb\le \rho-2,
\end{equation}
\begin{equation}\label{ServelloniMazzantiVienDalMare}
\begin{aligned}
\mathbb{M}_{\Delta_{12} \widehat{\calQ}_+  }(p, \tb) &\le_{p,\rho} \mathbb{M}_{\Delta_{12} \widehat{\calQ}}(p, \tb)
+\mathtt{k}_3\lVert \Delta_{12} \mathtt{p} \rVert_{p+\sigma_3}(1+\lVert \mathtt{p} \rVert_{p+\sigma_3})
+\lVert \Delta_{12}\beta \rVert_{p+\sigma_3}+\lVert \Delta_{12} a \rVert_{p+\sigma_3}
\end{aligned}
\end{equation}
for any $\tb\le \rho-3$.
\end{prop}

\begin{proof}
Let $\Psi^{\tau}$ be the flow in \eqref{diffeotot}. We can write 
$\Psi^{\tau}:=\mathcal{A}^{\tau}\circ(\Theta^{\tau}+R^{\tau})$,
where $\mathcal{A}^{\tau}$ is defined in \eqref{ignobel}, and $\Theta^{\tau},R^{\tau}$
given by Prop. \ref{DPdiffeo} in \eqref{ignobel2}.
We define the map
$W^{\tau}:=\mathcal{A}^{\tau}\circ \Theta^{\tau}$.
 We claim that setting $\widehat R^\tau = (\Theta^\tau)^{-1} R^{\tau}$ we have
 \begin{equation*}
 \begin{aligned}
 S^{\tau}&:=W^{\tau} \mathcal{L}^0 (W^{\tau})^{-1}-\Psi^{\tau} \mathcal{L}^0 (\Psi^{\tau})^{-1}=
 \mathcal A^{\tau}\Theta^{\tau}  [\calL^{0},\widehat R^\tau](I+\widehat R^\tau)^{-1}(\Theta^{\tau})^{-1}(\mathcal A^{\tau})^{-1}\in \gotL_{\rho, p}\,,
 \end{aligned}
 \end{equation*}
 and
 $\sup_{\tau\in[0,1]}\mathbb{M}^{\gamma}_{S^{\tau}}(s,\tb),\sup_{\tau\in[0,1]} \mathbb{M}_{\Delta_{12} S^{\tau}  }(s,\tb)$ satisfy bounds \eqref{JeTame} and \eqref{ServelloniMazzantiVienDalMare}. We first study the conjugation of $\mathcal{L}^0$ by $W^{\tau}$.  In order to prove our claim we just have to note that $\widehat R^\tau\in \gotL_{\rho+1, p}$ by Lemma \ref{idealeds},
 moreover, by formula  \eqref{miserialadra} , $ [\omega\cdot \partial_{\varphi},\widehat R^\tau]= \omega\cdot\del_{\f}\widehat R^\tau$ and $ [\del_x,\widehat R^\tau]\in \gotL_{\rho, p}$. This means that $[\calL^{0},\widehat R^\tau]\in \gotL_{\rho, p}$, so that our claim follows by Lemmata \ref{chiusuracompoclasseL}, \ref{idealeds}, \ref{InvertibilityUtile} and \ref{preparailsugo} .

\smallskip

\noindent \textbf{Conjugation by $\Theta^{\tau}$.}
By Lemma \ref{InvertibilityUtile} we have $(\Theta^{\tau})^{-1}:=\mathrm{I}-\op(\tilde{\vartheta})+\mathtt{R}_{\rho}$, 
with
\begin{equation}\label{inter}
\begin{aligned}
&\lvert \tilde\vartheta \rvert^{\g, \calO}_{-1, s, \alpha}\le_{s, \alpha, \rho}\lVert \beta \rVert^{\g, \calO}_{s+\gotd_0},\,\qquad \mathbb{M}^{\gamma}_{\mathtt{R}_{\rho}}(s, \tb)\le_{s,  \rho} \lVert \beta \rVert^{\g, \calO}_{s+\gotd_0}, \quad \tb\le 0\le \rho-2,\\
&\lvert \Delta_{12}  \tilde\vartheta \rvert_{-1, p, \alpha}\le_{p, \alpha} \lVert \Delta_{12} \beta  \rVert_{p+\gotd_0},  \qquad \mathbb{M}_{\Delta_{12} \mathtt{R}_{\rho}  }(p, \tb)\le_{p,  \rho} \lVert \Delta_{12} \beta  \rVert_{p+\gotd_0}
\quad 0\le \tb\le \rho-3,
\end{aligned}
\end{equation}
 for $s_0\le s \le \mathcal{S}$ and
for some $\gotd_0=\gotd_0(\rho)$. Throughout the proof we shall denote by $\gotd_i$ an increasing sequence of constants, depending on $\rho$, which keeps track of the loss of derivatives in our procedure. Moreover we shall omit writing the constraints $s_0\le s\le \mathcal{S}$, $0\le \tb\le \rho-2$, $0\le \tb\le \rho-3$ when we write the bounds for the operators belonging to $\gotL_{\rho, p}$.

We wish to compute
\begin{equation*}
\Theta^\tau B (\Theta^\tau)^{-1}= 
B+ [\op(\vartheta), B] \op(1-\tilde\vartheta)  +[\op(\vartheta), B] \mathtt R_\rho
\end{equation*}
for $B= \omega\cdot\partial_{\varphi}, J\circ (m+a), \op(\tq),\widehat \calQ$. 
\\
Let us start by studying the commutator  $ [\op(\vartheta), B]$, our purpose is to write it as a pseudo differential term plus a remainder in $\gotL_{\rho, p}$.
We have (recalling the Definition \ref{cancelletti} and formula \eqref{starcontro})
\begin{align}
[\op(\vartheta), \oo\cdot\del_{\f}]  & =  - \op(\oo\cdot\del_{\f} \vartheta)\\
  [\op(\vartheta), J\circ (m+a)] & = \op\big(\vartheta\star_{<\rho+1}(\omega(\xi)\#_{<\rho+1}(m+a))\big)\label{tetaJ}\\  & + \op\big(\vartheta\star_{\ge \rho+1}(\omega(\xi)\#(m+a)) + \vartheta\star_{< \rho+1}(\omega(\xi)\#_{\ge \rho+1}(m+a)) \big)\notag\\
  [\op(\vartheta), \op(\tq)]& = \op\big(\vartheta\star_{<\rho-1}\tq\big)+\op\big(\vartheta\star_{\ge \rho-1}\tq) \big)\,.\label{tetaq}
\end{align}
Here $\omega(\xi)$ is the symbol of the Fourier multiplier $J=\partial_x+3 \Lambda\partial_x$ , i.e.
$\omega(\xi):=\mathrm{i}\xi+3 \frac{\mathrm{i} \xi}{1+\xi^2}$.
One  can directly verify that all the symbols above are in $S^{-1}$, indeed the commutator of two pseudo differential operators has as order the sum of the orders minus one.
By Lemma \ref{idealeds} we verify that $[\op(\vartheta), \widehat{\cQ}], [\op(\vartheta), B] \mathtt R_\rho\in \gotL_{\rho, p}$ for all choices of $B$. By Lemma \ref{INCLUSIONEpseudoInclasseL} and \eqref{comp2} we have that the second summands in 
\eqref{tetaJ} and \eqref{tetaq} belong to $\gotL_{\rho, p}$.
We have proved that
\[
[\op(\vartheta), B]= \op(r_B) + R_B \,,\quad r_B\in S^{-1}\,,\quad R_B\in \gotL_{\rho, p}\,.
\]
Using \eqref{varthetaStima}, \eqref{docq} and \eqref{filini}, we have by \eqref{crawford} 
\begin{equation}\label{stimaB}
|r_B|^{\g,\calO}_{-1,s,\al} \le_{s,\al,\rho} \|\be\|^{\g,\calO}_{s+\gotd_1} + \|\be\|^{\g,\calO}_{s_0+\gotd_1}(\tk_1+\tk_2 \|p\|^{\g,\calO}_{s+\gotd_1}+\|a\|^{\g,\calO}_{s+\gotd_1}).
\end{equation}
Similarly, by \eqref{jamal} we have
\begin{equation}\label{stimaBresto}
\mathbb M_{R_B}^{\g}(s,\tb) \le_{s,\rho} \|\be\|^{\g,\calO}_{s+\gotd_1} + \|\be\|^{\g,\calO}_{s_0+\gotd_1}(\tk_1+\tk_2 \|p\|^{\g,\calO}_{s+\gotd_1}+\|a\|^{\g,\calO}_{s+\gotd_1}+ \mathbb M_{\widehat{\mathcal Q}}^{\g}(s,\tb)).
\end{equation}
Analogously by \eqref{crawford2} and \eqref{jamal2} we have
\begin{equation*}
|\Delta_{12}r_B|_{-1,p,\al} \le_{p,\al,\rho} 
\|\Delta_{12}\be\|_{p+\gotd_1} +  \|\be\|_{p+\gotd_1}(\mathtt{k}_3(\|\Delta_{12}p\|_{p+\gotd_1}+\lVert \Delta_{12} p \rVert_{s_0+\gotd_1}
\lVert p \rVert_{p+\gotd_1})+\|\Delta_{12}a\|_{p+\gotd_1}.
\end{equation*}
Similarly, by \eqref{jamal} we have
\begin{align}
\mathbb M_{\Delta_{12}R_B}(p,\tb) &\le_{p,\rho} \|\Delta_{12}\be\|_{p+\gotd_1} + 
\notag\\
&+\|\be\|_{p+\gotd_1}(\mathtt{k}_3(\|\Delta_{12}p\|_{p+\gotd_1}+\lVert \Delta_{12} p \rVert_{p+\gotd_1}
\lVert p \rVert_{p+\gotd_1})+\|\Delta_{12}a\|_{p+\gotd_1}+ \mathbb M_{\Delta_{12}\widehat{\mathcal Q}}(p,\tb)).\label{stimaBrestoi}
\end{align}
By Lemmata \ref{James}, \ref{idealeds} and  \ref{chiusuracompoclasseL} we have that 
\[
[\op(\vartheta), B]\op(1-\tilde\theta)= \op(\tilde r_B) + \tilde R_B \,,\quad \tilde r_B\in S^{-1}\,,\quad \tilde R_B\in \gotL_{\rho, p}\,,
\]
and $\tilde r_B,\tilde R_B$ satisfy bounds like \eqref{stimaB}-\eqref{stimaBrestoi}, with possibly a larger $\gotd_1$.
Analogously, by Lemmata \ref{idealeds} and  \ref{chiusuracompoclasseL}, we have that $[\op(\theta),B]{\mathtt R}_\rho\in \gotL_{\rho, p}$ satisfies estimates like \eqref{stimaBresto}, \eqref{stimaBrestoi}.
We conclude that
\[
\Theta^{\tau} \mathcal{L}^{0} (\Theta^{\tau})^{-1}=\mathcal{L}^0+\op(r_0)+\mathcal{R}_0
\]
where $r_0\in S^{-1}$, $\mathcal{R}_0\in\gotL_{\rho, p}$ and satisfy the bounds \eqref{stimaB}-\eqref{stimaBrestoi} with possibly larger $\gotd_1$.

\smallskip

\noindent \textbf{Conjugation by $\mathcal{A}^{\tau}$.} We proved that
\begin{equation}\label{inbianco}
W^{\tau} \mathcal{L}^{0} (W^{\tau})^{-1}=\mathcal{A}^{\tau}\mathcal{L}^0 (\mathcal{A}^{\tau})^{-1}+\mathcal{A}^{\tau}\op(r_0)  (\mathcal{A}^{\tau})^{-1}+\mathcal{A}^{\tau}\mathcal{R}_0  (\mathcal{A}^{\tau})^{-1}.
\end{equation}
%
By an explicit computation one has that
\begin{equation*}
\begin{aligned}
\mathcal{A}^{\tau} \mathcal{D}_{\omega} (\mathcal{A}^{\tau})^{-1} =
\mathcal{D}_{\omega}+J\circ (\mathcal{T}_{\tau\bt} \mathcal{D}_{\omega}\tilde{\beta})+\op(r_1)+\mathcal{R}_1
\end{aligned}
\end{equation*}
where $r_1\in S^{-1}$, $\mathcal{R}_1\in \gotL_{\rho, p}$ are  defined by
\begin{equation}
r_1:=-3 (\mathrm{i}\xi/(1+\xi^2)) \#_{< \rho-1} \mathcal{T}_{\tau \beta} (\mathcal{D}_{\omega}\tilde{\beta}),\quad \mathcal{R}_1:=-3\op( (\mathrm{i}\xi/(1+\xi^2)) \#_{\ge \rho-1} \mathcal{T}_{\tau \beta} (\mathcal{D}_{\omega}\tilde{\beta})),
\end{equation}
and, by \eqref{crawford},\eqref{crawford2}, \eqref{jamal}, \eqref{jamal2},  satisfy the following bounds
\begin{equation*}
\lvert r_1 \rvert^{\g, \calO}_{-1,s, \alpha}+\mathbb{M}^{\gamma}_{\mathcal{R}_1}(s,\tb)
\le_{s, \alpha,\rho} \lVert \beta \rVert^{\g, \calO}_{s+\gotd_2},
\qquad 
\lvert \Delta_{12}  r_1   \rvert_{-1,p, \alpha}+
\mathbb{M}_{\Delta_{12} \mathcal{R}_1 }(p,\tb)
\le_{p, \alpha,\rho} \lVert \Delta_{12} \beta   \rVert_{p+ \gotd_2}.
\end{equation*}
Moreover
\begin{equation}
\mathcal{A}^{\tau} (J\circ (m+a)) (\mathcal{A}^{\tau})^{-1}=J\circ\mathcal{T}_{\tau\beta}\Big((1+\tilde{\beta}_x)(m+a)  \Big)+\mathsf{R}^{(2)}
\end{equation}
where 
\begin{equation}
\begin{aligned}
\mathsf{R}^{(2)}&:=  \Big( (1-\Lambda \mathfrak{R})^{-1} -1\Big)\circ \Lambda \circ g \circ \partial_x \circ \mathcal{T}_{\tau \beta}\Big((1+\tilde{\beta}_x)(m+a)  \Big)\\
&+ \Big( (1-\Lambda \mathfrak{R})^{-1} -1\Big)\circ \Lambda \circ \left(g-3\right)\circ \partial_x \circ \mathcal{T}_{\tau \beta}\Big((1+\tilde{\beta}_x)(m+a)  \Big)\\
&+ \Big( (1-\Lambda \mathfrak{R})^{-1} \Big)\circ \Lambda \circ \left(g-3\right)\circ \partial_x \circ \mathcal{T}_{\tau \beta}\Big((1+\tilde{\beta}_x)(m+a)  \Big)\\
\end{aligned}
\end{equation}
with $g$ and $\mathfrak{R}$ defined in \eqref{sissagames}. In particular $\mathsf{R}^{(2)}=\op(r_2)+\mathcal{R}_2$, $r_2\in S^{-1}$, $\mathcal{R}_2\in \gotL_{\rho, p}$ and satisfy the following bounds
\begin{equation*}
\begin{aligned}
\lvert r_2 \rvert^{\g, \calO}_{-1, s, \alpha}+\mathbb{M}^{\gamma}_{\mathcal{R}_2}(s, \tb)
 &\le_{s, \alpha, \rho} \lVert \beta \rVert^{\g, \calO}_{s+\gotd_3}+\lVert \beta \rVert^{\g, \calO}_{s_0+\gotd_3}\lVert a \rVert^{\g, \calO}_{s+\gotd_3},\\
\lvert \Delta_{12} r_2   \rvert_{-1, p, \alpha}+
\mathbb{M}_{\Delta_{12} \mathcal{R}_2  }(p, \tb)
 &\le_{p, \alpha, \rho} \lVert \Delta_{12} \beta   \rVert_{p+\gotd_3}.
\end{aligned}
\end{equation*}
Then,  by \eqref{inbianco}, we conclude
\begin{align}
&W^{\tau} \mathcal{L}^0 (W^{\tau})^{-1}=\mathcal{D}_{\omega}-J\circ (m+a_+)+\calQ_*,\label{ipa}\\
\calQ_* &:=\mathcal{A}^{\tau} \op(\mathtt{q}+r_0) (\mathcal{A}^{\tau})^{-1}+\mathcal{A}^{\tau}(\widehat{\mathcal{Q}}+\mathcal{R}_0) (\mathcal{A}^{\tau})^{-1}+\op(r_1+r_2)+\mathcal{R}_1+\mathcal{R}_2.\label{ipa2}
\end{align}
By Theorem \ref{EgorovQuantitativo} and Lemma \ref{preparailsugo} we have
\begin{equation}\label{ipa3}
\mathcal{A}^{\tau} \op(\mathtt{q}+r_0) (\mathcal{A}^{\tau})^{-1}=\op(r_3)+\mathcal{R}_3, \quad \mathcal{A}^{\tau} (\widehat{\mathcal{Q}}+\mathcal{R}_0)(\mathcal{A}^{\tau})^{-1}=\mathcal{R}_4
\end{equation}
where $r_3\in S^{-1}$ and $\mathcal{R}_3$, $\mathcal{R}_4\in \gotL_{\rho, p}$. In order to bound $r_3$ we use \eqref{zeppelin} with $w=\mathtt{q}+r_0$ so that
\begin{equation}
\lvert w \rvert^{\g, \calO}_{-1, s, \alpha}\le_{s, \alpha, \rho} \tk_1+\tk_2 \lVert p\rVert^{\g, \calO}_{s+\gotd_4}+\lVert \beta \rVert^{\g, \calO}_{s+\gotd_4}+\lVert a \rVert^{\g, \calO}_{s+\gotd_4}.
\end{equation}
Note that in the formula \eqref{zeppelin} (recall the notations used in formula \eqref{zeppelin} and the fact that 
$k_1$, $k_2$, $k_3\geq 0$ and $k_1+k_2+k_3=s$) 
we have by interpolation
\begin{equation*}
\begin{aligned}
\lvert w \rvert^{\gamma, \calO}_{-1, k_1, \alpha+k_2+\su} \lVert \beta \rVert^{\gamma, \calO}_{k_3+\su}
&\le_s  (\tk_2 \lVert p\rVert^{\g, \calO}_{s+\gotd_5}+\lVert \beta \rVert^{\g, \calO}_{s+\gotd_5}+\lVert a \rVert^{\g, \calO}_{s+\gotd_5})\lVert \beta \rVert^{\gamma, \calO}_{s_0+\gotd_5}\\
&+\lVert \beta \rVert_{s+\gotd_5}(\tk_1+\tk_2 \lVert p\rVert^{\g, \calO}_{s_0+\gotd_5}+\lVert \beta \rVert^{\g, \calO}_{s_0+\gotd_5}+\lVert a \rVert^{\g, \calO}_{s_0+\gotd_5}).
\end{aligned}
\end{equation*}
Thus we get by \eqref{filini}
\begin{equation*}
\begin{aligned}
\lvert r_3 \rvert^{\g, \calO}_{-1, s, \alpha} +\mathbb{M}^{\gamma}_{\mathcal{R}_3}(s, \tb)
&\le_{s, \alpha, \rho} \tk_1+\tk_2 \lVert p\rVert^{\g, \calO}_{s+\gotd_5}
+\lVert \beta \rVert^{\g, \calO}_{s+\gotd_5}+\lVert a \rVert^{\g, \calO}_{s+\gotd_5},\\
\lvert \Delta_{12} r_3   \rvert_{-1, p, \alpha} +
 \mathbb{M}_{\Delta_{12} \mathcal{R}_3}(p, \tb) &\le_{p, \alpha, \rho} 
\mathtt{k}_3(\lVert \Delta_{12} p \rVert_{p+\gotd_5}+\lVert \Delta_{12} p \rVert_{s_0+\gotd_5}\lVert p \rVert_{p+\gotd_5})+\lVert \Delta_{12}\beta \rVert_{p+\gotd_5}+\lVert \Delta_{12} a \rVert_{p+\gotd_5}.
\end{aligned}
\end{equation*}
Moreover by \eqref{filini}
\begin{equation*}
\begin{aligned}
\mathbb{M}^{\g}_{\mathcal{R}_4}(s, \tb)&\le_{s, \rho} \mathbb{M}^{\g}_{\widehat{\calQ}}(s, \tb)+ \|\be\|^{\g,\calO}_{s+\gotd_6} + \|\be\|^{\g,\calO}_{s_0+\gotd_6}(\tk_1+\tk_2 \|p\|^{\g,\calO}_{s+\gotd_6}+\|a\|^{\g,\calO}_{s+\gotd_6}),\\
\mathbb{M}_{\Delta_{12} \mathcal{R}_4}(p, \tb) &\le_{p, \rho}   \mathbb{M}_{\Delta_{12}\widehat{\calQ}}(s, \tb)+ \| \Delta_{12}\be\|_{s+\gotd_6}  \\ 
 &+ \|\be\|_{p+\gotd_6}(\mathtt{k}_3( \|\Delta_{12}p\|_{p+\gotd_6}+ \|\Delta_{12}p\|_{p+\gotd_6}\|p\|_{p+\gotd_6})+ \|\Delta_{12}a\|_{p+\gotd_6}.\\
\end{aligned}
\end{equation*}
By \eqref{ipa2} and \eqref{ipa3} $\calQ_*$ in \eqref{ipa} is
\begin{equation*}
\calQ_*=\op(\mathtt{q}_+)+\widehat{\mathcal{Q}}_*, \quad \mathtt{q}_+:=r_1+r_2+r_3, \quad \widehat{\mathcal{Q}}_*:=\mathcal{R}_1+\mathcal{R}_2+\mathcal{R}_3+\mathcal{R}_4.
\end{equation*}
In particular, by the discussion above we have that 
the bounds
\eqref{lapiuimportante} hold with $\s_{3}\geq \gotd_5$ while
bounds \eqref{JeTame} and \eqref{ServelloniMazzantiVienDalMare} hold with $\s_{3}\geq\gotd_6$.
This concludes the proof.
\end{proof}

\paragraph{Straightening theorem.}\label{tiraddrizzotutto}

By Proposition \ref{ConjugationLemma} the coefficient $a_+$ of the transformed operator $\mathcal{L}_+=\Psi \mathcal{L} \Psi^{-1}$ (see \eqref{ellepiu}) is given by \eqref{Round}.
The aim of this section is to find a function $\beta$ (see \eqref{pseudo}), or equivalently a flow $\Psi$ of \eqref{diffeotot}, such that $a_+$ is a constant, namely such that the following equation is solved (recall \eqref{ignobel})
\begin{equation}\label{cubalibera}
\omega\cdot\partial_{\varphi} \tilde{\beta}-(m+a)(1+\tilde{\beta}_x)=\mbox{constant}.
\end{equation}
This issue is tantamount to finding a change of coordinates that straightens the $1$-order vector field
\[
\omega\cdot \frac{\partial}{\partial \varphi}-(m+a(\varphi, x)) \frac{\partial}{\partial x}.
\]
This is the content of the following proposition.
Actually this is a classical result on vector fields on a torus (\cite{M2}), but for our purposes we need a version which provides quantitative tame estimates on the Sobolev norms.

\begin{prop}\label{moser}
Let $\cO_0\subseteq \mathbb{R}^{\nu}$ be a compact set. Consider for $\omega\in \cO_0$ a Lipschitz family of vector fields  on $\T^{\nu+1}$
\begin{equation}\label{mao1}
\begin{aligned}
&X_0:=	\omega\cdot\frac{\partial}{\partial \varphi} - (m_0 + a_0(x,\varphi;\omega))\frac{\partial}{\partial x} \,,\quad \frac23<m_0<\frac32\,, |m_0|^{lip}\le M_0<1/2\\
& a_0\in H^s(\T^{\nu+1},\R)\quad  \forall s\geq s_0.
\end{aligned}
\end{equation}
Moreover $a_0(x,\f;\oo)=a_0(x,\varphi,\mathfrak{I}(\oo);\omega)$ and it is  Lipschitz in the variable $\mathfrak{I}$.
There exists $\delta_\star=\delta_\star(s_1)>0$ and $s_1 \geq s_0+2\tau+4$  such that, for any $\g>0$  if
\begin{equation}\label{picci}
C(s_1)\g^{-1} \|a_0\|^{\g,\mathcal O_0}_{s_1} :=\delta \le \delta_\star 
\end{equation}
then there exists a Lipschitz  function $m_{\infty}(\omega)=m_{\infty}(\oo,\mathfrak{I}(\oo))$ with  
$1/2<m_{\infty}< 2\,$ and  $ \lvert m_{\infty}-m_0\rvert^{\gamma}\leq \gamma \delta $ with 
$\forall\omega\in \Omega_{\varepsilon}\,$
such that in the set
\begin{equation}\label{buoneacque}
\calP_{\infty}^{2 \gamma}=\calP_{\infty}^{2\gamma}(\mathfrak{I}):= \left\{\omega\in \cO_0: \; |\omega\cdot \ell - m_{\infty}(\omega) j |>\frac{2\g }{\langle \ell\rangle^{\tau}}\,,\;\forall \ell \in \Z^\nu, \,\,\forall j\in \mathbb{Z}\setminus\{0\}\right\}
\end{equation}
the following holds. For all $\omega\in \calP_{\infty}^{2 \gamma}$ one has $ |\Delta_{12}m_{\infty} |\leq 2 \lvert \Delta_{12} \langle a_0 \rangle  \rvert$ 
and there exists a smooth map 
\begin{equation}\label{perunpugnodidollari}
\beta^{(\infty)}: \calP_0\times \T^{\nu+1}\to \R \,,\quad  \|\beta^{(\infty)} \|^{\g,\mathcal O_0}_{s} \le_s  \g^{-1}\|a_0\|^{\g,\mathcal O_0}_{s+2\tau+4},
\quad \forall s\geq s_0
\end{equation}
so that $\Psi^{(\infty)}: (\f,x) \mapsto (\f,x+\beta^{(\infty)}(\f,x))$ is a diffeomorphism of $\T^{\nu+1}$ and for all 
$\omega\in \calP_{\infty}^{2 \gamma}$
\begin{equation}\label{tordo6}
\Psi^{(\infty)}_* X_0 := \omega\cdot\frac{\partial}{\partial \varphi} + (\Psi^{(\infty)})^{-1}\big(\omega\cdot \partial_{\varphi}\beta^{(\infty)} - (m_0+a_0)(1+\beta^{(\infty)}_x)\big)\frac{\partial}{\partial x}= \omega\cdot\frac{\partial}{\partial \varphi} - m_\infty (\omega)\frac{\partial}{\partial x}.
\end{equation}
\end{prop}

\begin{proof}
We refer to Corollary $3.6$ of \cite{FGMP} which is a generalization of Proposition \ref{moser} in the case $x\in \T^d$ with $d\geq 1$.
\end{proof}
\begin{lem}\label{Anagrafe6}
Under the assumption of Proposition \ref{moser}, the function $\beta^{(\infty)}$ defined in the Proposition \ref{moser} satisfies the following estimate on the variation of the variable $i(\oo)$: 
$\lVert \Delta_{12} \beta^{(\infty)}  \rVert_{p}\le C \gamma^{-1} \lVert \Delta_{12} a_0   \rVert_{p+\s}$
for some $\s>0$ such that $p+\s<s_1$.
\end{lem}
\begin{proof}
We refer to Corollary $3.3$ of \cite{FGMP}.
\end{proof}

\subsection{Proof of Theorem \ref{risultatosez8}}

Consider the vector field
\begin{equation}\label{vectorA}
\omega\cdot\frac{\partial}{\partial \varphi} - (1 + a(x,\varphi;\omega))\frac{\partial}{\partial x}
\end{equation}
for $\omega\in\calO_0$ given in \eqref{0Meln}.
By taking $\mu$ in \eqref{ipopiccolezza} large enough and
$\e$ in \eqref{opQ2} small enough
we have that the condition \eqref{picci} is satisfied. Thus 
we apply Proposition \ref{moser} with $a_0\rightsquigarrow a$ in \eqref{vectorA} and $m_0\rightsquigarrow 1$.
Then there exist a constant $m(\omega)=m_{\infty}(\omega)$ 
and a function $\tilde{\beta}(\omega)$ defined on the set 
$\calP_{\infty}^{2\gamma}=\Omega_1$ (see \eqref{buoneacque} and  \eqref{prime}) 
such that (recall \eqref{tordo6})
\begin{equation}\label{liverpool}
\mathcal{T}_{\tilde{\beta}}^{-1}\Big(  \omega\cdot \partial_{\varphi} \tilde{\beta} - (1+a)(1+\tilde{\beta}_x)  \Big) =-m.
\end{equation}
Let $\beta$ be the function such that $(\varphi,x)\mapsto (\varphi,x+\beta(\varphi, x))$ is the inverse diffeomorphism  of $(\varphi,x)\mapsto (\varphi,x+\tilde{\beta}(\varphi, x))$ and let $\Psi^{\tau}$ be the flow of the Hamiltonian PDE
\begin{equation*}
u_{\tau}=\big(J \circ b(\tau)\big)\,u, \qquad b(\tau):=b(\tau, \varphi, x)=\frac{\beta}{1+\tau \beta_x}.
\end{equation*}
Let us call $\Phi_1:=\Psi^{1}$ and recall that $\Phi_1=\Phi_1(\omega)$ is defined for $\omega\in \Omega_1$. We apply Proposition \ref{ConjugationLemma} to $\mathcal{L}_{\omega}$ in \eqref{LomegaDP} and we get
\begin{equation}\label{BrigitteBardot}
\Phi_1\,\mathcal{L}_{\omega}\,\Phi^{-1}_1=\mathcal{D}_{\omega}-J\circ(1+a_+)+\mathcal{R},
\end{equation}
where, by \eqref{Round} and \eqref{liverpool},
\[
a_+(\varphi, x)=m-1
\]
and $\mathcal{R}=\op(\mathtt{r})+\widehat{\mathcal{R}}$, $\mathtt{r}=\mathtt{r}(\omega)\in S^{-1}$, $\widehat{\mathcal{R}}\in\gotL_{\rho, p}(\Omega_1)$.
Hence we have
\begin{equation}\label{BrigitteBardot1000}
\Phi_1\,\mathcal{L}_{\omega}\,\Phi^{-1}_1=\mathcal{D}_{\omega}-m J+\mathcal{R}.
\end{equation}
By \eqref{opQ2} one has that Proposition \ref{moser} implies
\eqref{clinica100}, \eqref{clinica1000}.
By \eqref{opQ2}, \eqref{perunpugnodidollari} the bound \eqref{lapiuimportante} reads as
\[
\lvert \mathtt{r} \rvert^{\g, \Omega_1}_{-1, s, \alpha}\le \varepsilon \g^{-1} \lVert \mathfrak{I}\rVert^{\g, \cO_0}_{s+\hat{\s}}, \quad \lvert \Delta_{12} \mathtt{r}  \rvert_{-1, p, \alpha}\le_{p} 
\varepsilon \gamma^{-1} (1+\lVert \mathfrak{I} \rVert_{p+\hat{\s}})
\lVert \mathfrak{I}_1-\mathfrak{I}_2 \rVert_{p+\hat{\s}},
\]
for some $\hat{\s}>0$, since $\mathtt{k}_1=0$, $\mathtt{k_2}=\varepsilon$,$\mathtt{p}=\mathfrak{I}$.
Moreover by \eqref{JeTame}, since $\mathtt{k}_3=\varepsilon$, for $0\le \tb\le \rho-2$ and $s_0\le s\le \mathcal{S}$
\begin{equation}\label{capoinb}
\mathbb{M}^{\gamma}_{\widehat{\mathcal{R}}}(s, \tb) \le_s \varepsilon \g^{-1}\lVert \mathfrak{I}\rVert_{s+\hat{\s}}^{\gamma, \calO_0}
\end{equation}
and by \eqref{ServelloniMazzantiVienDalMare}, and Lemma \ref{Anagrafe6}, 
for $0\le \tb\le \rho-3$ and $s_0\le s\le \mathcal{S}$, we get
\begin{equation}\label{capoinb2}
\begin{aligned}
\mathbb{M}_{ \Delta_{12} \widehat{\mathcal{R}}  }(p, \tb)&\le_{p} 
\varepsilon \gamma^{-1} (1+\lVert \mathfrak{I} \rVert_{p+\hat{\s}})
\lVert \mathfrak{I}_1-\mathfrak{I}_2 \rVert_{p+\hat{\s}}.
\end{aligned}
\end{equation}
The bound \eqref{grecia} follows by Corollary \ref{CoroDPdiffeo}, in particular by \eqref{flow2}, and \eqref{perunpugnodidollari}.

\section{Diagonalization}\label{SezioneDiagonalization}

The aim of this section is to prove Theorem \ref{ReducibilityDP}. We first provide an abstract result for $-1$-modulo tame operators.

\subsection{A KAM reducibility  result for modulo-tame vector fields}

\noindent
We say that a bounded linear operator $\mathbf B=\mathbf B(\f)$ 
is Hamiltonian  if $\mathbf B(\f) u$ is a linear  Hamiltonian vector field w.r.t. the symplectic form $J$. 
This means that
the corresponding Hamiltonian  
$\frac12(u, J^{-1}\mathbf B(\f) u )$ is a real quadratic function 
provided that $u_j=\bar u_{-j}$ and $\f\in \T^\nu$.
In matrix elements this means that 
\[
(J^{-1}\mathbf B(\f))_j^{j'}=  (J^{-1}\mathbf B(\f))_{j'}^{j}\,,
\quad \overline{(J^{-1}\mathbf B)_j^{j'}(\ell)}=  (J^{-1}\mathbf B)_{-j}^{-j'}(-\ell)
\]
or more explicitely:
\begin{equation}\label{lemettoanchio}
\mathbf B_j^{j'}(\f)=  -\frac{\omega(j)}{\omega( j')} \mathbf B_{-j'}^{-j}(\f)\,,
\quad \overline{\mathbf B_j^{j'}(\ell)}=  \mathbf B_{-j}^{-j'}(-\ell).
\end{equation}
This representation is convenient in the present setting because it keeps track of the Hamiltonian structure and
\[
\mathcal B= \frac12(u,J^{-1}\mathbf B(\f) u )\,,\quad  
\mathcal G= \frac12(u,J^{-1}\mathbf G(\f) u )\;\Rightarrow 
\{\mathcal B,\mathcal G\}= \frac12(u, J^{-1}[\mathbf B,\mathbf G] u ).
\]
We introduce the following parameters 
\begin{equation}\label{parametriKAM}
\tau= 2\nu+6,\quad \tb_0:=6\tau+6. 
\end{equation}
In order to prove the Theorem \ref{ReducibilityDP} we need to work in the class of Lip-$-1$-majorant tame operators (see Definition \ref{LipTameConstants})
and the proof is based on an abstract reducibility scheme for a class of tame operators.

We investigate the reducibility of a Hamiltonian operator of the form 
\begin{equation}\label{natali}
\mathbf{M}_0= \cD_0+\cP_0\,,\quad \cD_0= {\rm diag}(\mathrm{i} \, d_j^{(0)})\,,\quad d_j^{(0)}= m \left(\frac{j(4+j^2)}{1+j^2}\right).	
\end{equation}
Here the functions $d_j^{(0)}$ are well defined and Lipschitz in the set $\calO_0$, 
$ |m-1|^{\g,\cO_0} \le C\e$, 
while  $\cP_0$ is defined and Lipschitz in $\omega$ belonging to the set $\Omega_1$.
  We fix 
\begin{equation}\label{anagrafe2}
\ta:=6\tau+4, \qquad {\tau_1:=2\tau+2}, 
\end{equation}
we require that  $\cP_0,\langle \partial_{\f}\rangle^{\tb_0} \cP_0$ are Lip- $-1$- modulo tame, with modulo-tame constants denoted by 
$\mathfrak M^{{{\sharp, \g^{3/2}}}}_{\cP_0}(s)$ and 
$\mathfrak M^{{{\sharp, \g^{3/2}}}}_{\cP_0}(s,{\mathtt b_0})$ 
respectively (recall Definitions \ref{def:op-tame}, \ref{menounomodulotame}), in the set $\Omega_1$.  
Moreover $m$ and $\cP_0$ and the set $\Omega_1$  
depend on $\mathfrak{I}=\mathfrak{I}(\omega)$ and satisfy the bounds 
\begin{align}
&   |\Delta_{12} m |\le K_1  \|\mathfrak{I}_1-\mathfrak{I}_2\|_{s_0+\s} \label{nomidimerda2} \\
&\|\lD^{1/2}\und{\Delta_{12}\cP_0 }\lD^{1/2}\|_{\mathcal L(H^{s_0})}, \,\, \|\lD^{1/2}\und{ \Delta_{12}\langle \partial_{\f}\rangle^{{\mathtt b_0}} \cP_0  }\lD^{1/2}\|_{\mathcal L(H^{s_0})} \le   K_2  \|\mathfrak{I}_1-\mathfrak{I}_2\|_{s_0+\s}\,,\notag
\end{align}
for some $\s,K_1,K_2>0$, 
for all $\omega\in \Omega_1(\mathfrak{I}_1)\cap \Omega_1(\mathfrak{I}_2)$ with 
\begin{equation}\label{constaintCiao}
K_1,\mathfrak M^{{{\sharp, \g^{3/2}}}}_{\cP_0}(s_0),\mathfrak M^{{{\sharp, \g^{3/2}}}}_{\cP_0}(s_0,{\mathtt b_0})\le K_2.
\end{equation}
We recall that
$\|\cdot \|_{\mathcal L(H^{s_0})}$ is the operatorial norm.
We associate to the operator \eqref{natali} the Hamiltonian
\[ 
\mathcal H_0(\eta, u):= \omega\cdot \eta + \frac12 (u, J^{-1} \mathbf{M}_0 u)_{L^2(\T_x)}.
\]
\begin{prop}[{\bf Iterative reduction}]\label{iterazione riducibilita}
Let $\s>0$ be the loss of derivatives in \eqref{nomidimerda2} and
consider an operator of the form \eqref{natali}.
For all $s \in [s_0, \mathcal{S}]$,
there is $ N_0 := N_0 (\mathcal{S}, {{\mathtt b_0}}) >0$ such that, if 
\begin{equation}\label{KAM smallness condition1}
N_0^{\tau_1}  {\mathfrak M}^{{{\sharp, \g^{3/2}}}}_{\cP_0}(s_0, {{\mathtt b_0}}) \gamma^{- 3/2} \leq 1\,, 
\end{equation}
 (recall \eqref{anagrafe2}) then, for all $k\geq 0$: 

\smallskip
\noindent
${\bf(S1)}_{k}$ there exists a sequence  of Hamiltonian operators
\begin{equation}\label{cal L nu}
\mathbf{M}_k= {\cD}_k + {\cP}_k\,, \quad 
\quad {\cal D}_k := {\rm diag}_{j \in \mathbb{Z}\setminus\{0\}} (\mathrm{i} \, d_j^{(k)})\,,
\end{equation}
with $d_j^{(k)}$ defined for $\omega\in\calO_0$ and 
\begin{equation}\label{mu j nu}
d_j^{(k)}(\omega) := d_j^{(0)} + r_j^{(k)}(\omega)\,,\quad \,r_j^{(0)} :=0 \;,\,\, r_j^{(k)}\in\R \,,\,\,\; r^{(k)}_j= - r^{(k)}_{-j}.
\end{equation}
The operators $\cP_k$ are defined for $k\geq 1$ in a set $\Omega_{k}^{\g^{3/2}}:=\Omega_{k}^{\g^{3/2}}(\mathfrak{I})$ defined as
\begin{align}\label{Omega nu + 1 gamma}
\Omega_k^{\g^{3/2}} & := 
\Big\{\omega \in \Omega_{k - 1}^{\g^{3/2}}  \, :  \, |\omega \cdot \ell  + d_j^{(k - 1)} -  d_{j'}^{(k - 1)}| \geq  
\frac{{\g^{3/2}}}{\langle \ell\rangle^\tau},\,\forall \lvert \ell \rvert\le N_{k-1},\, \forall j, j'\in \mathbb{Z}\setminus\{0\},\,(j, j', \ell)\neq (j, j, 0)\Big\}
\end{align}
where $\Omega_{0}^{\g^{3/2}}:=\Omega_{1}$ and $N_k:=N_0^{(3/2)^k}$.
Moreover  $\cP_k $ and  
	$ \langle\partial_\f \rangle^{{\mathtt b_0} }  {\cal P}_k$
	are $-1$-modulo-tame  with modulo-tame constants respectively
\begin{equation}\label{def:msharp}
{\mathfrak M}_k^{{{\sharp, \g^{3/2}}}} (s) :=  {\mathfrak M}_{{\cP}_k}^{{{\sharp, \g^{3/2}}}} (s) \, , \quad
{\mathfrak M}_k^{{{\sharp, \g^{3/2}}}} (s, {{\mathtt b_0}}) := 
{\mathfrak M}_{  {\cP}_k}^{{{\sharp, \g^{3/2}}}} (s,\tb_0),  \quad k\geq 0
\end{equation}
for all $s \in [s_0, \mathcal{S}] $. Setting $N_{-1}=1$, we have
\begin{equation}\label{stimaM}
{\mathfrak M}_k^{{{\sharp, \g^{3/2}}}} (s) \leq 
{\mathfrak M}^{{{\sharp, \g^{3/2}}}} _0 (s, {{\mathtt b_0}}) N_{k - 1}^{- {\mathtt a}}\,,\quad 
{\mathfrak M}_k^{{{\sharp, \g^{3/2}}}} ( s, {\mathtt b_0}) \leq  {\mathfrak M}^{{{\sharp, \g^{3/2}}}} _0 (s, {{\mathtt b_0}}) N_{k - 1}\,,
\end{equation}
while for all  $k\ge 1$ 
\begin{equation}\label{stima cal R nu}
\begin{aligned}
&\langle j\rangle |d_j^{(k)}- d_j^{(k-1)}|\le 
\mathfrak M_{0}^{{\sharp, \g^{3/2}}} (s_0, {{\mathtt b_0}})  N_{k - 2}^{- {\mathtt a}}.
\end{aligned}
\end{equation}

\noindent	
${\bf(S2)}_{k}$  For $ k \geq 1 $, 
there exists a linear symplectic change of variables $\cQ_{k - 1}$, defined in $\Omega_{k}^{\g^{3/2}}$ and  such that 
\begin{equation}\label{coniugionu+1}
{\mathbf M}_k := \cQ_{k-1} \omega\cdot \partial_\f \cQ_{k-1}^{-1} + \cQ_{k-1}\mathbf{M}_{k-1}\cQ_{k-1}^{-1} . 
\end{equation}
The operators $\Psi_{k - 1} :=\cQ_{k-1}- \mathrm{I}$
and $ \langle \partial_\f \rangle^{{\mathtt b_0}}  \Psi_{k-1} $,
are $ -1$-modulo-tame with modulo-tame constants satisfying, for all $s \in [s_0, \mathcal{S}] $,
\begin{equation}\label{tame Psi nu - 1}
{\mathfrak M}_{\Psi_{k - 1}}^{{{\sharp, \g^{3/2}}}} \! (s) \leq \g^{-3/2} 
N_{k - 1}^{{\tau_1}} N_{k - 2}^{- \mathtt a} {\mathfrak M}_0^{{\sharp, \g^{3/2}}} (s, {{\mathtt b_0}}) \, , \quad
{\mathfrak M}_{  \Psi_{k - 1}}^{{{\sharp, \g^{3/2}}}} \! (s,\tb_0) \leq  
\g^{-3/2} N_{k - 1}^{{\tau_1}} N_{k - 2}  {\mathfrak M}_0^{{\sharp, \g^{3/2}}} (s, {{\mathtt b_0}})\, . 
\end{equation}

\noindent		
${\bf(S3)}_{k}$ Let $ \mathfrak{I}_1(\omega )$, $ \mathfrak{I}_2(\omega) $ such that  
$ {\cP}_0(\mathfrak{I}_1)$,  ${\cP}_0(\mathfrak{I}_2 )$ satisfy \eqref{nomidimerda2}.
Then for all $\omega \in \Omega_k^{\gamma_1}(\mathfrak{I}_1) \cap \Omega_k^{\gamma_2}(\mathfrak{I}_2)$
with $\gamma_1, \gamma_2 \in [\g^{3/2}/2, 2 \g^{3/2}]$ we have
\begin{align}\label{stima R nu i1 i2}
& \! \! \| \lD^{1/2}\underline{\Delta_{12}{\cP}_k }\lD^{1/2} \|_{\calL(H^{s_0})}  \leq  
 K_2N_{k - 1}^{- \mathtt a}\|\mathfrak{I}_1-\mathfrak{I}_2\|_{s_0 +  \s }, \\ 
& \label{stima R nu i1 i2 norma alta}
\! \! \|  \lD^{1/2}\underline{\langle \partial_\f \rangle^{{\mathtt b_0}}\Delta_{12}{\cP}_k }\lD^{1/2} \|_{\calL(H^{s_0})} 
\leq K_2
 N_{k - 1} \|\mathfrak{I}_1-\mathfrak{I}_2\|_{ s_0 +  \s }\,.
\end{align}
Moreover for all $k = 1, \ldots , n$, for all $j \in { S}^c$, 
\begin{align}\label{r nu - 1 r nu i1 i2}
& \lr{j}\big|\Delta_{12}r_j^{(k)}  - \Delta_{12} r_j^{(k - 1)}   \big| \leq 
\| \lD^{1/2}\underline{\Delta_{12}{\cP}_k }\lD^{1/2} \|_{\calL(H^{s_0})} \,, \\
& \lr{j}\ |\Delta_{12} r_j^{(k)}  | \leq  
K_2 \|\mathfrak{I}_1-\mathfrak{I}_2   \|_{ s_0  + \s }\,. \label{r nu i1 - r nu i2}
\end{align}
		
\noindent		
${\bf(S4)}_{k}$ Let $\mathfrak{I}_1$, $\mathfrak{I}_2$ be like in ${\bf(S3)}_{k}$ and $0 < \rho < \g^{3/2}/2$. Then 
\begin{equation}\label{tab}
K_2  N_{k - 1}^{\tau +1} \|\mathfrak{I}_1 - \mathfrak{I}_2 \|_{ s_0 + \s } 
\leq \rho \quad \Longrightarrow \quad 
\Omega_k^{\g^{3/2}}(\mathfrak{I}_1) \subseteq \Omega_k^{\g^{3/2} - \rho}(\mathfrak{I}_2) \, . 
\end{equation}	
\end{prop}


\vspace{0.9em}
\noindent
The Proposition \ref{iterazione riducibilita} is proved by applying repeatedly the following {\bf KAM reduction procedure} :

\smallskip

\noindent
Fix any $N\gg 1$	and consider any operator of the form 
\[
\mathbf M= \mathcal D(\omega) +\mathcal P(\f,\omega )\,,\qquad
\mathcal D(\omega)= {\rm diag} (\mathrm{i} \,d_j(\omega))_{j\in \Z}\,,\quad d_j= d^{(0)}_j+ r_j, \quad d_j^{(0)}:=m(\omega)\,\frac{j (4+j^2)}{(1+j^2)}.
\]
Here the $m, r_j\in \R$  are well defined and Lipschitz for $\omega\in \calO_0$ 
with 
\begin{equation}\label{miseria}
|1-m|^{\g,\calO_0} \le C\e\,,\quad r_j= - r_{-j}\,,\quad  \sup_j \langle j\rangle |r_j|^{{\g^{3/2},\calO_0}} <  
2\, \mathfrak{M}^{\sharp, \g^{3/2}}_{\cP_0}(s_0, \tb_0) .
\end{equation}
Assume that (recall \eqref{0Meln}, \eqref{prime})  
in a set $\cO\equiv\cO(\mathfrak{I})\subseteq \Omega_1(\mathfrak{I})\subseteq \calO_0$ 
the operators $\cP,\langle \pa_\f \rangle^{{\mathtt b_0}} \cP$ are  Hamiltonian, real 
and  $-1$-modulo tame with 
\begin{equation}\label{piccolo0}
\g^{-3/2} N^{2\tau+2}\mathfrak M^{\sharp, \g^{3/2}}_{\cP}(s_0, \tb_0)<1 \,. 
\end{equation}
Assume finally that $d_j=d_j(\mathfrak{I})$, 
$\cP(\mathfrak{I}),\langle \pa_\f \rangle^{{\mathtt b_0}} \cP(\mathfrak{I})$ 
are Lipschitz  w.r.t. $\mathfrak{I}$ namely for all $\omega\in\cO(\mathfrak{I}_1)\cap \cO(\mathfrak{I}_2)$
\begin{equation}\label{miserialip}
\begin{aligned}
&\lvert \Delta_{12} m \rvert\le K_1\, \|\mathfrak{I}_1-\mathfrak{I}_2\|_{s_0+\s},  \quad \sup_j \langle j\rangle |\Delta_{12} r_j| 
< 2\,K_0  \|\mathfrak{I}_1-\mathfrak{I}_2\|_{s_0+\s}\,\\
& \|\lD^{1/2}\und{ \Delta_{12}\langle \partial_{\f}\rangle^{a} \cP  }\lD^{1/2}\|_{\mathcal L(H^{s_0})} 
\le K_2 \|\mathfrak{I}_1-\mathfrak{I}_2\|_{s_0+\s} \,,\quad a=0,\tb_0
\end{aligned}
\end{equation}
for some constants $K_1\le K_0$ (recall  $K_2$  in \eqref{nomidimerda2}).
Let us define $\cC\equiv \cC_{\cD}^{({\g^{3/2}} ,\tau,N,\cO)}$  as
\begin{equation}\label{Cantonà}
\cC:=\{\omega\in\cO\,:\,\, |\omega\cdot \ell +d _{j}-d _{j'}|> \frac{{\g^{3/2}}}{ \langle\ell\rangle^{\tau}} \,,\quad \forall (\ell,j,j')\neq (0,j,j),\; |\ell|\le N,\,j,j'\in \Z\setminus\{0\}\}.
\end{equation}
For  $\omega\in\cC$ let  $\cA(\f)$ be defined as follows
\begin{equation}\label{soluzioneomologica}
\mathcal A_j^{j'}(\ell)= 
\dfrac{\mathcal P_j^{j'}(\ell) }{\mathrm{i} (\omega\cdot \ell + d _j-d _{j'})}, \quad {\rm for} \quad |\ell| \le N,\;\;\;{\rm and}
\qquad  \mathcal A_j^{j'}(\ell)= 0\quad \text{otherwise}.
\end{equation}
\begin{lem}[{\bf KAM step}]\label{KAMstep}
The following holds:

\smallskip
\noindent
$(i)$  The operator $\cA$ in \eqref{soluzioneomologica} is a Hamiltonian, $-1$-modulo tame  
matrix with the bounds
\begin{align}
& \mathfrak M^{{{\sharp, \g^{3/2}}}}_\cA(s,a) \le \g^{-3/2}  N^{2\tau+1}
\mathfrak M^{{{\sharp, \g^{3/2}}}}_\cP(s,a)\,,  \label{costantetameA}\\ \label{costantetamedelta12A}
& \|\lD^{1/2}\und{ \Delta_{12}\langle \partial_{\f}\rangle^{a} \cA  }\lD^{1/2}\|_{\mathcal L(H^{s_0})} \le  
C\,\g^{-3/2}N^{2\tau+1} \big(K_2 +{ K_0\,\g^{-3/2} }
\mathfrak M^{{{\sharp, \g^{3/2}}}}_\cP(s_0,a) \big) \|\mathfrak{I}_1-\mathfrak{I}_2\|_{s_0+\s})\,,
\end{align}
for $ a=0,{\mathtt b_0}$, for all $\omega\in \cC(\mathfrak{I}_1)\cap\cC(\mathfrak{I}_2)$ and for some $\s>0$.

\smallskip
\noindent
$(ii)$ The operator  $\cQ=e^\cA:=\sum_{k\geq 0} \frac{\mathcal{A}^k}{k!}$ 
is well defined and invertible, moreover $\Psi= \cQ-\mathrm{I}$  
 is a $-1$-modulo tame operator with the bounds
\begin{equation}
\begin{aligned}
&\mathfrak M^{{{\sharp, \g^{3/2}}}}_{\cQ-\mathrm{I}}(s,a) \le 
2 \mathfrak M^{{{\sharp, \g^{3/2}}}}_\cA(s,a)\le 2\g^{-3/2} N^{2\tau+1}\mathfrak M^{{{\sharp, \g^{3/2}}}}_\cP(s,a)\,,\\
&\|\lD^{1/2}\und{ \Delta_{12}\langle \partial_{\f}\rangle^{a} \cQ  }
\lD^{1/2}\|_{\mathcal L(H^{s_0})} \le  2\g^{-3/2}N^{2\tau+1} 
\big(K_2 +K_0\g^{-3/2}\mathfrak M^{{{\sharp, \g^{3/2}}}}_\cP(s_0,a) \big) 
\|\mathfrak{I}_1-\mathfrak{I}_2\|_{s_0+\s})\,,\notag
\end{aligned}
\end{equation}
for $ a=0,{\mathtt b_0}$ and for some $\s>0$.
Finally $z\to \cQ z$   is a  symplectic change of variables 
generated by the time one flow of the Hamiltonian $\mathcal S_0 = \frac12(z, J^{-1} \cA z)$.

\smallskip
\noindent
$(iii)$   Set, for  $\omega\in\cC$ (see \eqref{Cantonà}), 
\begin{equation}\label{coniuga+}
\cQ (\omega\cdot \partial_\f \cQ^{-1}) + \cQ\left(\mathcal D(\omega) 
+\mathcal P(\f,\omega )\right)\cQ^{-1}:=  \mathbf{M}_+=
\mathcal D^+(\omega) +\mathcal P^+(\f,\omega )
\end{equation}
where $\mathcal D^+(\omega)= \mathrm{diag}(\mathrm{i} \, d^+_j)$ 
is Hamiltonian, diagonal, independent of $\f$ and defined for all $\omega\in \calO_0$ with

\begin{equation}\label{d+}
\begin{aligned}
&d^+_j= d^{(0)}_j+ r^+_j\,,\quad r^+_j= -r^+_{-j}\,,
\quad \sup_j \langle j\rangle  |r _j-r ^+_j|^{\g^{3/2} ,\cO_0} \le  
\mathfrak M^{{{\sharp, \g^{3/2}}}}_\cP(s_0)\,, \\
&\sup_j \langle j\rangle  |\Delta_{12}(r _j-r ^+_j) | \le  
K_2 \| \mathfrak{I}_1-\mathfrak{I}_2\|_{s_0+\s}\,,\quad 
\forall \omega\in \cC(\mathfrak{I}_1)\cap \cC(\mathfrak{I}_2).
\end{aligned}	
\end{equation}		 
For $\omega\in \cC$  we have the bounds
\begin{align}	
{\mathfrak M}_{\cP^+}^{{{\sharp, \g^{3/2}}}} (s ) &\leq  
N^{- {{\mathtt b_0}}} {\mathfrak M}_\cP^{{{\sharp, \g^{3/2}}}} (s, {{\mathtt b_0}}) 
+C(s) N^{2\tau+1} \g^{-3/2} {\mathfrak M}_\cP^{{{\sharp, \g^{3/2}}}} (s) 
{\mathfrak M}_\cP^{{{\sharp, \g^{3/2}}}} (s_0)\,.\label{P+}\\
{\mathfrak M}_{\cP^+}^{{{\sharp, \g^{3/2}}}} (s, {{\mathtt b_0}}) &\leq 
{\mathfrak M}_\cP^{{{\sharp, \g^{3/2}}}} (s, {\mathtt b_0}) \label{P+b}\\
&+ N^{2\tau+1} \g^{-3/2} C(s,\tb_0) 
\left({\mathfrak M}_\cP^{{{\sharp, \g^{3/2}}}} (s, {\mathtt b_0}) 
{\mathfrak M}^{{{\sharp, \g^{3/2}}}}_\cP(s_0) 
+  {\mathfrak M}_\cP^{{{\sharp, \g^{3/2}}}} (s_0, {\mathtt b_0}) {\mathfrak M}_\cP^{{{\sharp, \g^{3/2}}}} (s)\right)  \,.\notag
\end{align}
Moreover for all $\omega\in \cC(\mathfrak{I}_1)\cap \cC(\mathfrak{I}_2)$
\begin{align}\label{P+i}
\|\und{\Delta_{12}\cP^+ }\|_{\mathcal L(H^{s_0})} 
&\leq  N^{- {{\mathtt b_0}}} K_2 \|\mathfrak{I}_1-\mathfrak{I}_2\|_{s+\s} \\
&+  C(s_0)N^{2\tau+1} \g^{-3/2} {\mathfrak M}_\cP^{{{\sharp, \g^{3/2}}}} (s_0) 
\left(K_2+\g^{-3/2} {\mathfrak M}^{{{\sharp, \g^{3/2}}}}_\cP(s_0) K_0 \right)  
\|\mathfrak{I}_1-\mathfrak{I}_2\|_{s+\s}\,\notag\\
\|\und{\Delta_{12}\langle \partial_{\f}\rangle^{{\mathtt b_0}}\cP^+ }\|_{\mathcal L(H^{s_0})} 
&\leq  K_2  \|\mathfrak{I}_1-\mathfrak{I}_2\|_{s+\s} 
+  N^{2\tau+1} \g^{-3/2} C(s_0,\tb_0) 
\left({\mathfrak M}_\cP^{{{\sharp, \g^{3/2}}}} (s_0,{\mathtt b_0}) K_2 \right. \label{P+ib}\\
&\!\!\!\!\!\!\!\!+ {\mathfrak M}_\cP^{{{\sharp, \g^{3/2}}}} (s_0) \left(K_2
+ \g^{-3/2} {\mathfrak M}^{{{\sharp, \g^{3/2}}}}_\cP(s_0,\tb_0) K_0 \right) \notag\\
&\!\!\!\!\!\!\!\!+\left.\g^{-3/2}N^{2\tau+1} {\mathfrak M}_\cP^{{{\sharp, \g^{3/2}}}} (s_0)
{\mathfrak M}_\cP^{{{\sharp, \g^{3/2}}}} (s_0,\tb_0) \left( K_2 
+ \g^{-3/2} {\mathfrak M}^{{{\sharp, \g^{3/2}}}}_\cP(s_0) K_0 \right) \right) 
\|\mathfrak{I}_1-\mathfrak{I}_2\|_{s+\s}\,\notag
\end{align}
for some $\s>0$.
 The  action of $\cQ$ on the Hamiltonian $\mathcal H$ is given by (see \eqref{coniuga+})
\[
\mathcal H_+:= e^{\{\mathcal S_0,\cdot\} }\mathcal H = \omega\cdot \eta + \frac12(w, J^{-1} \mathbf M^+ w).
\]
\end{lem}

\begin{proof}
\emph{Proof of $(i)$}: First we prove that $\mathcal{A}$ is a $-1$-modulo tame operator. By \eqref{Cantonà}, \eqref{soluzioneomologica} (recall \eqref{partialorder}, \eqref{funtore})
\[
\langle \partial_\f\rangle^a\cA \preceq \g^{-3/2} N^\tau	
\langle \partial_\f\rangle^a  \cP \,,\quad {\rm for}\; a=0,{\mathtt b_0},
\]
while
\[
\langle \partial_\f\rangle^a \Delta_{\omega,\omega'}\cA \preceq 
\g^{-3/2} N^\tau \langle \partial_\f\rangle^a \Delta_{\omega,\omega'} \cP 
+ \g^{-3} N^{2\tau+1} 	\langle \partial_\f\rangle^a  \cP \,,\quad {\rm for}\; a=0,{\mathtt b_0}
\]
since 
\[
\Delta_{\omega, \omega'} \mathcal{A}_j^{j'}(\ell)=
\frac{\Delta_{\omega, \omega'} \cP_j^{j'}(\ell)}{\mathrm{i}\big(\omega\cdot \ell +d_j-d_{j'} \big)}
-\mathrm{i}\frac{\cP_j^{j'}(\ell)\,
\big([(\omega-\omega')\cdot \ell/(|\omega-\omega'|)]
+\Delta_{\omega, \omega'}(d_j-d_{j'}) \big)}{\big(\omega\cdot \ell +d_j-d_{j'} \big)^2}.
\]
By Lemma \ref{proprietatame}-{\it (i)} and \eqref{miseria}, \eqref{piccolo0} we deduce \eqref{costantetameA}.
 The bounds \eqref{costantetamedelta12A} come from applying the Leibniz rule and by \eqref{miserialip}
\begin{equation}\label{pargolo}
|\Delta_{12}\cA_j^{j'}(\ell) |\leq \frac{|\Delta_{12}\cP_j^{j'}(\ell) |}{|\omega\cdot \ell + d_j -d_{j'}|} 
+\frac{ |\cP_j^{j'}(\ell)| |\Delta_{12} d_j  - \Delta_{12} d_{j'}  | }{(\omega\cdot \ell + d_j -d_{j'})^2} .
\end{equation}
We remark that in the second summand (recall that $K_1\le K_0$) 
\begin{align*}
\frac{ \lvert \Delta_{12} d_j  - \Delta_{12} d_{j'} \rvert}{|\omega\cdot \ell + d_j -d_{j'}|} 
&\le |\Delta_{12} m  |\frac{|\omega(j)-\omega(j')|}{|\omega\cdot \ell + d_j -d_{j'}|}
+ \frac{|\Delta_{12} r_j | +|\Delta_{12} r_{j'} |}{|\omega\cdot \ell + d_j -d_{j'}|} \\	
&\stackrel{(\ref{miserialip}),(\ref{nomidimerda2})}{\le} 
C\, \g^{-3/2}(K_1 N^{\tau+1} 
+  N^\tau  K_0)\| \mathfrak{I}_1-\mathfrak{I}_2\|_{s_0+\s}
 \le C  \g^{-3/2}N^{\tau+1} K_0 \|\mathfrak{I}_1-\mathfrak{I}_2\|_{s_0+\s}.
\end{align*}
The  estimate on the first summand follows from  the estimates on $\Delta_{12} m  $ and the fact that if $|\omega(j)-\omega(j')|> C |\ell|$ with $C>1$ then $|\omega\cdot \ell + d_j -d_{j'}|> \tilde{C}|\omega(j)-\omega(j')|$ with $\tilde{C}>0$; the estimate on the second summand comes from \eqref{miseria}, \eqref{piccolo0}.
In conclusion we get (recall \eqref{miserialip} for the definition of $K_2$)
\[
\|\lD^{1/2}\und{ \Delta_{12}\langle \partial_{\f}\rangle^{a} \cA  }\lD^{1/2}\|_{\mathcal L(H^{s_0})}
\le  C \big(\g^{-3/2}N^\tau K_2 +  {\g^{-3}}N^{2\tau+1}  
K_0\, \mathfrak M^{{{\sharp, \g^{3/2}}}}_\cP(s_0,a) \big)\|\mathfrak{I}_1-\mathfrak{I}_2\|_{s_0+\s}
\]	
for all $\omega\in\cC(\mathfrak{I}_1)\cap\cC(\mathfrak{I}_2)$.
The fact that $\cA$ is Hamiltonian  
follows from \eqref{lemettoanchio} and 
from the fact that $d_j$ is odd in $j$ (recall  \eqref{natali}) and $\mathcal{P}$ is Hamiltonian.
	
\smallskip	

\emph{Proof of $(ii)$}:	
By the boundness of $\mathcal{A}$, the bound on its modulo-tame constant and the smallness condition \eqref{piccolo0} we have that $\mathcal{Q}$ is well defined and invertible. The bounds are a consequence of 
Lemma \ref{proprietatame} ({\it iv})-({\it v}), the smallness condition \eqref{piccolo0} and the estimates proved in statement $(i)$. 

\smallskip

\emph{Proof of $(iii)$}:	
We start by observing that 
\begin{equation}\label{marvel}
\mathcal D^+ +\mathcal P^+ =  
\mathcal D +\mathcal P -\omega\cdot\partial_\f \cA +[\cA, \mathcal D +\mathcal P]+
\sum_{k\ge 2} \frac{{\rm ad}(\cA)^k}{k!} ( \mathcal D +\mathcal P )-
\sum_{k\ge 2} \frac{{\rm ad}(\cA)^{k-1}}{k!} ( \omega\cdot\partial_\f \cA  ).
\end{equation}
Again by definition, $\cA$ solves the equation 
\[
\omega\cdot\partial_\f \cA +[  \mathcal{D} , \mathcal{A}] =  
\Pi_N \mathcal P -[\cP]
\]
where $[\mathcal{P}]$ is the diagonal matrix with $j$-th eigenvalue $\mathcal{P}_j^j(0)$.
Substituting in \eqref{marvel} we get
\begin{equation}\label{marvel2}
\mathcal D^+ +\mathcal P^+ =  
\mathcal D +[\cP] + \Pi_N^\perp \cP+
\sum_{k\ge 1} \frac{{\rm ad}(\cA)^k}{k!} ( \mathcal P )-
\sum_{k\ge 2} \frac{{\rm ad}(\cA)^{k-1}}{k!} ( \Pi_N \mathcal P -[\cP]\  ).
\end{equation}
By the reality condition \eqref{lemettoanchio} we get 
$\overline{\cP_j^j(0)}= \cP_{-j}^{-j}(0)= - \cP_{j}^{j}(0)$, 
which shows that $\cP_{j}^{j}(0)$ is real and odd in $j$.
By Kirtzbraun Theorem we extend $\cP_j^j(0)$ to the whole 
$\calO_0$ 
preserving the $|\cdot|^{\g^{3/2}}$ norm.
We set 
\[
d^+_j= d _j+ (\cP_j^j(0))^{\rm Ext}=d_j^{(0)}+r_j+(\cP_j^j(0))^{\rm Ext}\,,\quad r_j^+:=r_j+(\cP_j^j(0))^{\rm Ext}
\]
where $(\cdot)^{\rm Ext}$ denotes the extension of the eigenvalue 
at $\cO_0$, so that the bound \eqref{d+} follows, 
by Lemma \ref{proprietatame} - {\it (i)} and the bounds  \eqref{miserialip} on $\cP$ and $\Delta_{12}\cP $. Now for $\omega\in \cC$
\begin{equation}\label{pipiu}
\mathcal P^+=\Pi_N^\perp \mathcal P +
\sum_{k\ge 1} \frac{{\rm ad}(\cA)^k}{k!} ( \mathcal P )-
\sum_{k\ge 2} \frac{{\rm ad}(\cA)^{k-1}}{k!} ( \Pi_N \mathcal P -[\cP]\  ).
\end{equation}
By Lemma \ref{proprietatame}-{\it (iv)} we have
\begin{equation}
{\mathfrak M}_{({\rm ad}\cA)^k\cP}^{{{\sharp, \g^{3/2}}}} (s) \leq 
C(s)^k   \Big( ({\mathfrak M}_{\cA}^{{{\sharp, \g^{3/2}}}} (s_0))^{k}{\mathfrak M}_{ \cP }^{{{\sharp, \g^{3/2}}}} (s)+
k({\mathfrak M}_{A}^{{{\sharp, \g^{3/2}}}} (s_0))^{k-1}{\mathfrak M}_{ \cA}^{{{\sharp, \g^{3/2}}}} (s) 
{\mathfrak M}_{\cP}^{{{\sharp, \g^{3/2}}}} (s_0) \Big)
\end{equation}
which implies \eqref{P+}, by using also \ref{proprietatame}{\it (iii)}.
Finally 
\begin{equation}\label{adAkB2}
\begin{aligned}
{\mathfrak M}_{({\rm ad}\cA)^k\cP}^{{{\sharp, \g^{3/2}}}} (s,{\mathtt b_0}) &\leq 
C(s,{{\mathtt b_0}})^k   \Big( ({\mathfrak M}_{\cA}^{{{\sharp, \g^{3/2}}}} (s_0))^{k}
{\mathfrak M}_{ \cP }^{{{\sharp, \g^{3/2}}}} (s,{\mathtt b_0})\\
&+k({\mathfrak M}_{\cA}^{{{\sharp, \g^{3/2}}}} (s_0))^{k-1}
\left({\mathfrak M}_{ \cA}^{{{\sharp, \g^{3/2}}}} (s,{\mathtt b_0}) 
{\mathfrak M}_{\cP}^{{{\sharp, \g^{3/2}}}} (s_0) + 
{\mathfrak M}_{\cA }^{{{\sharp, \g^{3/2}}}} (s_0,{\mathtt b_0}) 
{\mathfrak M}_{\cP}^{{{\sharp, \g^{3/2}}}} (s)\right)  \\  
&+  k(k-1)({\mathfrak M}_{\cA}^{{{\sharp, \g^{3/2}}}} (s_0))^{k-2}
{\mathfrak M}_{\cA}^{{{\sharp, \g^{3/2}}}} (s){\mathfrak M}_{ \cA}^{{{\sharp, \g^{3/2}}}} (s_0,{\mathtt b_0})  {\mathfrak M}_{ \cP}^{{{\sharp, \g^{3/2}}}} (s_0) \Big)
\end{aligned}
\end{equation}
which implies \eqref{P+b}.
In order to obtain the bounds \eqref{P+i} and \eqref{P+ib} on  $\Delta_{12}$,  
we just apply Leibniz rule repeatedly in \eqref{pipiu} 
and then procede as before. More precisely we have for all 
$\omega\in \cC(\mathfrak{I}_1)\cap \cC(\mathfrak{I}_2)$
\[
\Delta_{12} ({{\rm ad}(\cA)^{k}}  \mathcal P)   =  
{{\rm ad}(\cA)^{k}}  \Delta_{12}\mathcal P    
+ \sum_{k_1+k_2= k-1} {{\rm ad}(\cA)^{k_1}{\rm ad}(\Delta_{12}\cA    ){\rm ad}(\cA)^{k_2}} 
 \mathcal P. \quad \footnote{Recall the usual convention that
$a (\Delta_{12} b) c \equiv a(\mathfrak{I}_1) (\Delta_{12} b) c(\mathfrak{I}_2)$.}
\]
Now we note that 
$\|\lD^{1/2}\und\cA\lD^{1/2}\|_{\mathcal L(H^{s_0})}\le  {\mathfrak M}_{\cA}^{{{\sharp, \g^{3/2}}}} (s_0) $ and that for any matrices $A,B$ we have
\[
\|\lD^{1/2}\und{{\rm ad}(A)B}\lD^{1/2}\|_{\mathcal L(H^{s_0})}\le C(s_0) 
\|\lD^{1/2}\und A\lD^{1/2}\|_{\mathcal L(H^{s_0})} 
\|\lD^{1/2}\und B\lD^{1/2}\|_{\mathcal L(H^{s_0})}.
\]	
This implies that for all $\omega\in \cC(\mathfrak{I}_1)\cap \cC(\mathfrak{I}_2)$ 
(recall \eqref{miserialip} for the definition of $K_2$)
\begin{align}
&\|\lD^{1/2}\und{\Delta_{12} ({{\rm ad}(\cA)^{k}}  \mathcal P)  }
\lD^{1/2}\|_{\mathcal L(H^{s_0})}\le  (C(s_0) {\mathfrak M}_{\cA}^{{{\sharp, \g^{3/2}}}} (s_0))^k K_2 \label{prima} \\ 
&+ k C(s_0)^k( {\mathfrak M}_{\cA}^{{{\sharp, \g^{3/2}}}} (s_0))^{k-1}
\g^{-3/2}{\mathfrak M}_{\cP}^{{{\sharp, \g^{3/2}}}} (s_0)(N^\tau K_2 +  \g^{-3/2}N^{2\tau+1}  
K_0 \mathfrak M^{{{\sharp, \g^{3/2}}}}_\cP(s_0))
\|\mathfrak{I}_1-\mathfrak{I}_2\|_{s_0+\s}. \notag
\end{align}
Now by definition
\begin{align}\label{pipiu2}
\Delta_{12}	\mathcal P^+ =\Pi_N^\perp \Delta_{12}\mathcal P   +
\sum_{k \ge 1 }  \Delta_{12}(\frac{{\rm ad}(\cA)^{k}}{k!} \mathcal P)    -
\sum_{k\ge 2}  \Delta_{12} \big(\frac{{\rm ad}(\cA)^{k-1}}{k!} (\Pi_N \mathcal P -[\cP])\big) \,, 
\end{align}
so we use Lemma \ref{proprietatame}- {\it (iii)} 
in oder to bound the first summand and \eqref{prima} 
in order to bound the remaining ones. In the same way
\begin{align*}
\Delta_{12} \lr{\partial_\f}^{\tb_0}({{\rm ad}(\cA)^{k}}  \mathcal P)    
&=  {{\rm ad}(\cA)^{k}}  \Delta_{12}\lr{\partial_\f}^{\tb_0}\mathcal P   
+  \sum_{k_1+k_2= k-1} {{\rm ad}(\cA)^{k_1}{\rm ad}(\Delta_{12}\cA    ){\rm ad}(\cA)^{k_2}} 
 \lr{\partial_\f}^{\tb_0}\mathcal P  \\ &  +  \sum_{k_1+k_2= k-1} {{\rm ad}(\cA)^{k_1}
 {\rm ad}(\lr{\partial_\f}^{\tb_0}\cA  ){\rm ad}(\cA)^{k_2}} \Delta_{12} \mathcal P    \\ 
 & +\sum_{k_1+k_2= k-1} {{\rm ad}(\cA)^{k_1}{\rm ad}(\Delta_{12}
 \lr{\partial_\f}^{\tb_0}\cA    ){\rm ad}(\cA)^{k_2}}  \mathcal P\\ 
 &+\sum_{k_1+k_2+k_3= k-2} {\rm ad}(\cA)^{k_1}{\rm ad}
 (\lr{\partial_\f}^{\tb_0}\cA) {\rm ad}(\cA)^{k_2}{\rm ad}(\Delta_{12}\cA    ){\rm ad}(\cA)^{k_3}
  \mathcal P \\ 
  &+  \sum_{k_1+k_2+k_3= k-2} {\rm ad}(\cA)^{k_1}{\rm ad}(\Delta_{12}\cA  ) 
  {\rm ad}(\cA)^{k_2}{\rm ad}(\lr{\partial_\f}^{\tb_0}\cA  ){\rm ad}(\cA)^{k_3}\,,
\end{align*}
where the last two terms appear only if $k\ge 2$. We proceed as for \eqref{prima} and obtain the bound
\begin{align}
&\|\lD^{1/2}\und{\Delta_{12} \lr{\partial_\f}^{\tb_0}({{\rm ad}(\cA)^{k}}  
\mathcal P)  }\lD^{1/2}\|_{\mathcal L(H^{s_0})}\le  (C(s_0) {\mathfrak M}_{\cA}^{{{\sharp, \g^{3/2}}}} (s_0))^k
K_2 \label{vera} \\ 
&\quad+ k C(s_0)^k( {\mathfrak M}_{\cA}^{{{\sharp, \g^{3/2}}}} (s_0))^{k-1}\g^{-3/2}
{\mathfrak M}_{\cP}^{{{\sharp, \g^{3/2}}}} (s_0,\tb_0)(N^\tau K_2 
+  \g^{-3/2}N^{2\tau+1}  K_0 \mathfrak M^{{{\sharp, \g^{3/2}}}}_\cP(s_0))\notag \\ 
&\quad+ k C(s_0)^k( {\mathfrak M}_{\cA}^{{{\sharp, \g^{3/2}}}} (s_0))^{k-1}
{\mathfrak M}_{\cA}^{{{\sharp, \g^{3/2}}}} (s_0,\tb_0)K_2 \notag \\ 
&\quad+ k C(s_0)^k( {\mathfrak M}_{\cA}^{{{\sharp, \g^{3/2}}}} (s_0))^{k-1}\g^{-3/2}
{\mathfrak M}_{\cP}^{{{\sharp, \g^{3/2}}}} (s_0)(N^\tau K_2 
+  \g^{-3/2}N^{2\tau+1}  K_0 \mathfrak M^{{{\sharp, \g^{3/2}}}}_\cP(s_0,\tb_0))\notag \\
 &\quad+ 2 k(k-1) C(s_0)^k( {\mathfrak M}_{\cA}^{{{\sharp, \g^{3/2}}}} (s_0))^{k-2}{\mathfrak M}_{\cA}^{{{\sharp, \g^{3/2}}}} (s_0,\tb_0)\g^{-3/2}{\mathfrak M}_{\cP}^{{{\sharp, \g^{3/2}}}} (s_0)\\
&\quad(N^\tau K_2 +  \g^{-3/2}N^{2\tau+1}  K_0 
\mathfrak M^{{{\sharp, \g^{3/2}}}}_\cP(s_0)) \|\mathfrak{I}_1-\mathfrak{I}_2\|_{s_0+\s} \notag
\end{align}
from which one deduces the \eqref{P+ib}.
\end{proof}

\subsection{Proof of Theorem \ref{ReducibilityDP}}

In this section we conclude the proof of Theorem \ref{ReducibilityDP}. We first provide a preliminary result.

\begin{lem}\label{natalino10}
Consider $\rho:=s_0+\tb_0+3$, $p=s_0$ and the operator $\calL_{\omega}^{+}$ (see \eqref{operatorefinale}) in Theorem \ref{risultatosez8} .
We have that $\mathcal{P}_0:=\mathcal{R}$ (with $\mathcal{R}$ in \eqref{operatorefinale})
is $-1$-modulo-tame with modulo-tame constants satisfying 
 the \eqref{constaintCiao}
  with 
  \begin{equation}\label{fixparam}
  \s:=\mu_1,\quad  K_1:=\e, \quad K_{2}:=\e\gamma^{-1},
  \end{equation}
where $\mu_1$ is given by Theorem \ref{risultatosez8}.

\noindent
 Moreover the constant $m$ and the operator  $\cP_0$  satisfy, 
 for all $\omega\in \Omega_1(\mathfrak{I}_1)\cap \Omega_1(\mathfrak{I}_2)$, the bounds \eqref{nomidimerda2}.
  \end{lem}

\begin{proof} 
Recalling the form of $\mathcal{R}$ in Theorem \ref{risultatosez8} we have that
 Lemma \ref{LemmaAggancio} implies that $\cP_0$  is $-1$-modulo tame with modulo tame constants satisfying 
 (recalling the Definition \ref{Mdritta2} and the fact that $\g^{3/2}<\g$) 
 \begin{equation}\label{MaledettaCondizioneKAM}
\mathfrak{M}^{{\sharp, \g^{3/2}}}_{\mathcal{P}_0}(-1, s), \mathfrak{M}^{{\sharp, \g^{3/2}}}_{\mathcal{P}_0}(-1, s, \tb_0)\le_s 
\mathbb{B}^{\g}_{\mathcal{R}}(s, s_0+\tb_0)\stackrel{(\ref{pasta10})}{\leq}\mathbb{M}^{\gamma}_{\mathcal{R}}(s,\rho-2)
\end{equation}
which implies 
 \begin{equation}\label{PiccolezzaperKamredDP1000}
  \mathfrak{M}^{\sharp, \g^{3/2}}_{\cP_0}(s,\tb_0)\le \mathbb{M}^{\g}_{\mathcal{R}}(s, \tb), \qquad
 \mathfrak{M}^{\sharp, \g^{3/2}}_{\cP_0}(s_0,\tb_0)\le \e \g^{-1}.
 \end{equation}
Using \eqref{ipopiccolezza},  \eqref{pasqua200}   one gets the \eqref{constaintCiao} with the parameters fixed in \eqref{fixparam}.
In the same way, by Lemma \ref{LemmaAggancio}, \eqref{clinica1000}, \eqref{pasqua200}, \eqref{ipopiccolezza} we get the \eqref{nomidimerda2}.
\end{proof}

\begin{proof}[{\bf Proof of Theorem \ref{ReducibilityDP}}]
We want to apply Proposition \ref{iterazione riducibilita} to the operator $\mathcal{L}_{\omega}^+$ in \eqref{operatorefinale} 
(see also Theorem \ref{risultatosez8}).
It is convenient to remark that $\mathcal{L}^{+}_{\omega}$ 
gives the dynamics of a quadratic time-dependent Hamiltonian. 
Passing to the extended phase space, $\mathcal{L}^{+}_{\omega}$ 
corresponds to the Hamiltonian
\[
\mathcal H:=\mathcal{H}( \eta, u)= \omega\cdot \eta + \frac12 (u,J^{-1}\mathbf{M}_0 u)_{L^2(\T_x)}\,,\qquad \mathbf{M}_0= \cD_0+\cP_0
\]
where  
\begin{equation}\label{natalino}
 \cD_0= {\rm diag}(\mathrm{i} \, d_j^{(0)})_{j\in \mathbb{Z}\setminus\{0\}}\,,\quad d_j^{(0)}= m \left(\frac{j(4+j^2)}{1+j^2}\right), 
\quad \mathcal{P}_0:=\mathcal{R}.
\end{equation}
By Lemma \ref{natalino10} we have that $m$ and $\mathcal{P}_0$
 satisfy \eqref{nomidimerda2}, \eqref{constaintCiao} with the choice of parameters in 
 \eqref{fixparam}. 
Then the smallness assumption  \eqref{KAM smallness condition1} follows
by the smallness condition on $\e$ in \eqref{PiccolezzaperKamredDP}
provided that $N_0$ in formula \eqref{PiccolezzaperKamredDP} is chosen as in Proposition \ref{iterazione riducibilita}.
We can conclude that Proposition \ref{iterazione riducibilita} applies to $\mathcal{L}^{+}_{\omega}$ 
in \eqref{operatorefinale}.

\noindent
By \eqref{stima cal R nu} we have that the sequence $(d_j^k)_{k\in\mathbb{N}}$ in \eqref{mu j nu} is Cauchy, hence the limit $d_j^{\infty}=d_j^{(0)}+r_j^{\infty}$ exists and, also by \eqref{mu j nu}, $r_j^{\infty}$ satisfies \eqref{stimeautovalfinaliDP}.\\
Now we claim that (recall \eqref{prime}-\eqref{seconde} and \eqref{Omega nu + 1 gamma})
\begin{equation}\label{4.48}
\calO_{\infty}\subseteq \bigcap_{k\geq 0} \Omega_k^{\g^{3/2}}.
\end{equation}
Indeed we have for $\lvert \ell \rvert\le N_{k}$
\begin{align*}
\lvert \omega\cdot \ell+d_j^k-d_{j'}^{k}\rvert &\geq \lvert \omega\cdot \ell+d_j^{\infty}-d_{j'}^{\infty}\rvert-\lvert r_j^k-r_j^{\infty} \rvert-\lvert r_{j'}^k-r_{j'}^{\infty} \rvert\\
&\stackrel{ (\ref{stima cal R nu})}{\geq} \frac{2 \g^{3/2}}{\langle \ell\rangle^{\tau}}-\frac{ \mathfrak{M}^{{\sharp, \g^{3/2}}}_{0}(s_0,\tb_0)\,}{N_{k-2}^{\ta}}\geq \frac{\g^{3/2}}{\langle \ell \rangle^{\tau}}
\end{align*}
since $\mathfrak{M}^{{\sharp, \g^{3/2}}}_{0}(s_0,\tb_0)\leq \g^{3/2} N_0^{-\tau_1}$ and $ \langle \ell \rangle^{\tau}\le N_k^\tau\le N_{k-2}^{\ta}$
due to \eqref{anagrafe2}.
We conclude  that $\calO_{\infty}\subseteq  \Omega_{k+1}^{\g^{3/2}}$. Thus the sequence $(\Psi_k)_{k\in\mathbb{N}}$ (recall 
item $({\bf S2})_{k}$ in Prop. \ref{iterazione riducibilita}
) is well defined on $\calO_{\infty}$.

\noindent
We define 
\[
\Phi_k=\calQ_0\circ \dots \circ \calQ_k.
\]
We claim that there exists $\Phi_{\infty}:=\lim_{k\to\infty} \Phi_k$ in the topology induced by the operatorial norm. 
First we note that, by using \eqref{tame Psi nu - 1} and \eqref{KAM smallness condition1}, for any $k$ we have
\begin{equation}
\mathfrak{M}^{\sharp, \g^{3/2}}_{\Phi_k}(s)\le \sum_{j=0}^{k} \left(\mathfrak{M}^{\sharp, \g^{3/2}}_{\calQ_j}(s) \prod_{i\neq j} \mathfrak{M}^{\sharp, \g^{3/2}}_{\calQ_i}(s_0)\right)\le 2\sum_{j=0}^{k} \mathfrak{M}^{\sharp, \g^{3/2}}_{\calQ_j}(s)\le C \big(1+\max_{j=0, \dots, k}\mathfrak{M}^{\sharp, \g^{3/2}}_{\Psi_j}(s)\big).
\end{equation}
By Lemmata \ref{proprietatame} and \ref{LemmaAggancio} we have
\begin{align*}
\mathfrak{M}^{\g^{3/2}}_{\Phi_k - \Phi_{k-1}}(s, \tb_0) &\le_s 
\mathfrak{M}^{\sharp, \g^{3/2}}_{\Phi_k - \Phi_{k-1}}(s, \tb_0)  \le_s
\mathfrak{M}^{\sharp, \g^{3/2}}_{\Psi_k}(s, \tb_0)+\mathfrak{M}^{\sharp, \g^{3/2}}_{\Psi_k}(s_0, \tb_0) 
\max_{j=0,\dots, k} \mathfrak{M}^{\sharp, \g^{3/2}}_{\Psi_j}(s, \tb_0)\\
&+\mathfrak{M}^{\sharp, \g^{3/2}}_{\Psi_k}(s, \tb_0) \max_{j=0,\dots, k} 
\mathfrak{M}^{\g^{3/2}, \sharp}_{\Psi_j}(s_0, \tb_0)\stackrel{(\ref{tame Psi nu - 1})}{\le_s} 
 N_k^{\tau_1} N_{k-1}^{-\mathtt{a}} \mathfrak{M}^{{\sharp, \g^{3/2}}}_{0}(s, \tb_0) \g^{-3/2}.
\end{align*}
Thus by
\[
\lVert (\Phi_{k+m}-\Phi_k)h\rVert^{\g^{3/2}, \calO_{\infty}}_s \le 
\sum_{j=k}^{k+m} \lVert (\Phi_{j}-\Phi_{j-1})h\rVert^{\g^{3/2}, \calO_{\infty}}_s
\]
and by \eqref{MaledettaCondizioneKAM} we have that (recall \eqref{KAM smallness condition1} and \eqref{anagrafe2})
\[
\mathfrak{M}^{\g^{3/2}}_{\Phi_{k+m}-\Phi_k}(s, \tb_0) \le_{s_0, \tb_0} \mathbb{M}^{\gamma}_{\mathcal{R}}(s,\rho-2)   N_k^{\tau_1} N_{k-1}^{-\mathtt{a}}  \g^{-3/2} \stackrel{(\ref{pasqua200}), (\ref{PiccolezzaperKamredDP1000})}{\le_{s_0, \tb_0}} 
\varepsilon \gamma^{-1}\lVert \mathfrak{I}\rVert_{s+\mu_1}^{\gamma, \cO_0}  N_k^{-2( \tau+(1/3))},
\]
hence $(\Phi_k)_{k\in\mathbb{N}}$ is a Cauchy sequence in $\mathcal{L}(H^s)$ and for $\Phi_{\infty}$ the estimate \eqref{grano} holds. The operators $\Phi_k$ are close to the identity, hence the same is true for $\Phi_{\infty}$ and by Neumann series it is invertible. One can  prove that for $\Phi_{\infty}^{-1}$ the estimate \eqref{grano} holds.
 
  Let us prove the \eqref{stimaLipR}.
 We first show that, for any $n\in \N$ one has
 \begin{equation}\label{stimaENNE}
 \langle j\rangle| r_j(\mathfrak{I}_1)-r_j (\mathfrak{I}_2)|\leq \e\gamma^{-1}\|\mathfrak{I}_1-\mathfrak{I}_2\|_{s_0+\s}+
 \e\gamma^{-1}CN_{n-1}^{-\mathtt{a}},
 \end{equation}
 with $N_n$ defined in Prop. \ref{iterazione riducibilita}. This would implies the thesis.
For $k=n+1$ one can estimates
 \begin{equation*}
 \begin{aligned}
| r_j(\mathfrak{I}_1)-r_j (\mathfrak{I}_2)|&\leq   | r_j(\mathfrak{I}_1)-r_j^{(k)} (\mathfrak{I}_1)|+
| r^{(k)}_j(\mathfrak{I}_1)-r^{(k)}_j (\mathfrak{I}_2)|+  | r_j^{(k)}(\mathfrak{I}_2)-r_j (\mathfrak{I}_2)|\\
 \end{aligned}
 \end{equation*}
by using  \eqref{stima cal R nu}, \eqref{r nu i1 - r nu i2}, with $K_2\sim \e\gamma^{-1}$ to get the \eqref{stimaENNE}.

\end{proof}

\section{Measure estimates and conclusions}\label{measure}


Here we conclude the proof of Theorem \ref{teoMainRed} by showing that Theorem  \ref{stimedimisura}
 holds. 
 We first need some preliminary results. Let us  define
 \begin{equation}\label{omeghinoj}
 \omega(j):=\frac{4+j^{2}}{1+j^{2}} j\,,
 \end{equation}
 and remark that if $j\neq k$ (both non-zero)
\begin{equation}
\label{prat}
|\omega(j)-\omega(k)| = |j-k| \Big|1+ 3\frac{1-jk}{(1+j^2)(1+k^2)}\Big|\geq \frac1{2} |j-k|\,.
\end{equation}
Recall that $\tau> 2\nu +1$ is fixed  in \eqref{0Meln}.
\begin{lem}\label{BrexitDP}
	If $R_{\ell j k}\neq \emptyset$,  
	then $\lvert \ell \rvert\geq C_1 \lvert \omega(j)-\omega(k)|$
	for some constant $C_1>0$. 
	
	\noindent
	If $Q_{\ell j}\neq \emptyset$ then $\lvert \ell \rvert\geq C_2 \lvert j \rvert$ for some constant $C_2>0$. 
\end{lem}

\begin{proof}
	Since $|\omega||\ell|\geq |\omega\cdot\ell|$  our first claim follows,  setting $C_1:=(8\lvert {\omega} \rvert)^{-1}$, provided that we prove 
	\[
	8\lvert {\omega}\cdot \ell\rvert\geq \lvert \omega(j)-\omega(k)\rvert
	\]
If $R_{\ell j k}\neq \emptyset$, then there exist $\omega$ such that 
\begin{equation}
\lvert d_j(\omega)-d_k(\omega) \rvert< 
2 \gamma^{3/2}\langle \ell \rangle^{-\tau}+2 \lvert {\omega}\cdot \ell \rvert.
\end{equation}
Moreover, recall \eqref{clinica100} and \eqref{stimeautovalfinaliDP}, we get
\begin{equation}
\begin{aligned}
\lvert d_j(\omega)-d_k(\omega) \rvert &\geq \lvert m \rvert \lvert \omega(j)-\omega(k) \rvert
-\lvert r_j(\omega) \rvert-\lvert r_k(\omega) \rvert 
\geq \frac{1}{3} \lvert \omega(j)-\omega(k) \rvert.
\end{aligned}
\end{equation}
Thus, for $\varepsilon$ small enough
\[
2 \lvert {\omega} \rvert \lvert \ell \rvert\geq 2 \lvert {\omega}\cdot \ell \rvert\geq 
\left( \frac{1}{3}-\frac{2 \gamma^{3/2}}{\langle \ell \rangle^{\tau}\lvert \omega(j)-\omega(k) \rvert} \right)
\lvert \omega(j)-\omega(k) \rvert\geq \frac{1}{4}\lvert \omega(j)-\omega(k) \rvert
\]
and this proves the first claim on $R_{\ell j k}$. 
If $\lvert m j \rvert\geq 2 \lvert \omega\cdot \ell \rvert$ then by \eqref{0Meln}
\[
\lvert \omega\cdot \ell +m j \rvert\geq \lvert m \rvert \lvert j \rvert-\lvert \omega\cdot \ell \rvert\geq 2 \lvert \omega\cdot \ell \rvert- \lvert \omega\cdot \ell \rvert=\lvert \omega\cdot \ell \rvert\geq \gamma \langle \ell \rangle^{-\tau}.
\]
Hence if $Q_{\ell j}\neq \emptyset$ we have
$\lvert j \rvert\le 2\lvert \omega\cdot \ell \rvert \lvert m \rvert^{-1}\le C_{2}^{-1}\lvert \ell \rvert$, 
where $C_2:=\lvert m \rvert (4 \lvert {\omega} \rvert)^{-1}$.
This concludes the proof.
\end{proof}

By \eqref{BadSetsDP}, we have to bound the measure of the sublevels of the function $\omega\mapsto\phi(\omega)$ defined by
\begin{equation}\label{phi(omega)DP}
\begin{aligned}
&\phi_R(\omega):= \omega\cdot \ell+d_j(\omega)-d_k(\omega)=\omega\cdot \ell+\mathrm{i} m(\omega) (\omega(j)-\omega(k))
+(r_j-r_k)(\omega),\\
&\phi_Q(\omega):= \omega\cdot \ell+m(\omega) j.
\end{aligned}
\end{equation}
Note that $\phi$ also depends on $\ell, j, k, \mathfrak{I}$.

By Lemma \ref{BrexitDP} it is sufficient to study the measure of the resonant sets $R_{\ell j k}$ 
defined in \eqref{BadSetsDP} for $(\ell, j, k)\neq (0, j, j)$. In particular we will prove the following Lemma.

\begin{lem}\label{singolo}
Let us define for $\eta\in (0, 1)$ and $\sigma\in\mathbb{N}>0$ 
\begin{equation*}
R_{\ell j k}(\eta, \sigma):=\big\{ \omega\in \calO_0\; : \lvert \omega\cdot \ell+d_j-d_k\rvert\le {2\eta}\langle \ell \rangle^{-\sigma}\big\},\quad
Q_{\ell j }(\eta, \sigma):=\big\{ \omega\in \calO_0 : \lvert \omega\cdot \ell+m j \rvert\le {2\eta}{\langle \ell \rangle^{-\sigma}}\big\}.
\end{equation*}
Recalling that $\calO_0\in[-L,L]$,   we have that $\lvert R_{\ell j k}(\eta, \sigma) \rvert\le C L^{(\nu-1)} \eta \langle \ell \rangle^{-\sigma}$. The same holds for $Q_{\ell j}(\eta,\s)$.
\end{lem}

\begin{proof}
We give the proof of Lemma \ref{singolo} for the set $R_{\ell j k}$ (with $\ell\neq 0$)
which is the most difficult case.

\noindent
Split $\omega=s \hat{\ell}+v$ where $\hat{\ell}:=\ell /\lvert \ell \rvert$ and $v\cdot \ell=0$. 
Let $\Psi_R(s):=\phi_R(s \hat{\ell}+v)$, defined in \eqref{phi(omega)DP}.
By using \eqref{clinica100},\eqref{stimeautovalfinaliDP} and    Lemma \ref{BrexitDP} 
we have 
\begin{equation}\label{maggio}
\begin{aligned}
\lvert \Psi_R(s_1)-\Psi_R(s_2) \rvert&\geq \lvert s_1-s_2 \rvert \big( 
|\ell|-|j-k||m|^{lip,\calO_0}-(|r_{j}|^{lip,\calO_0}+|r_{k}|^{lip,\calO_0}) \big)
{\geq}\frac{\lvert\ell\rvert}{2} \lvert s_1-s_2 \rvert 
\end{aligned}
\end{equation}
for $\e$ small enough (see \eqref{PiccolezzaperKamredDP}).
As a consequence, the set $\Delta_{\ell j k}:=\{ s: s \hat{\ell}+v\in R_{l j k}\}$ has Lebesgue measure
\[
\lvert \Delta_{\ell j k} \rvert\le {2}{\,\lvert \ell\rvert^{-1}} \,{4\,\eta\,}{\langle \ell \rangle^{-\s }}= {8\,\eta}{\langle \ell \rangle^{-\s -1}}.
\]
 The Lemma follows by Fubini's theorem.
\end{proof}

\begin{lem}\label{delpiero}
There exists $\mathtt C>0$ such that setting $\tau_1= \nu+2$ then, for all $j,k$ such that $\lvert j \rvert, \lvert k \rvert\geq \mathtt{C} \langle \ell \rangle^{\tau_1}\gamma^{-(1/2)}$, 
one has 
$R_{\ell j k}(\gamma^{3/2}, \tau)\subseteq Q_{\ell, j-k}(\gamma, \tau_1)$.
\end{lem}
\begin{proof}
By \eqref{stimeautovalfinaliDP}, \eqref{clinica100}
we have (recall also \eqref{prat}) that
\begin{equation}
\begin{aligned}
\lvert \omega\cdot \ell +d_j-d_k\rvert 
&\geq \frac{2\gamma}{\langle \ell \rangle^{\tau_1}}-2 \lvert j-k \rvert \frac{C}{\lvert j \rvert \lvert k \rvert}-\frac{\tilde{C}\varepsilon}{\min\{ \lvert j \rvert, \lvert k \rvert \}}\geq \frac{2\gamma}{\langle \ell \rangle^{\tau_1}}
- \frac{C\gamma}{\mathtt{C}\langle \ell \rangle^{2\tau_1-1}}
-\frac{\tilde{C}\varepsilon \sqrt{\gamma}}{\mathtt{C} \langle \ell \rangle^{\tau_1}}
\geq \frac{\gamma^{3/2}}{\langle \ell \rangle^{\tau}}
\end{aligned}
\end{equation}
for $\mathtt{C}$ big enough and since $\varepsilon (\sqrt{\gamma})^{-1}\ll 1$.
\end{proof}

\begin{proof}[{\bf Proof of Theorem \ref{stimedimisura}}]
Let  $\tau>2\nu+4$.
We have 
\begin{equation*}
\left\lvert \bigcup_{\ell\in\mathbb{Z}^{\nu}, j, k\in \Z\setminus\{0\}} R_{\ell j k}  \right\rvert\le 
\sum_{\ell\in \Z^{\nu}, \lvert j \rvert, \lvert k \rvert\geq \mathtt{C}\langle \ell \rangle^{\tau_1}\gamma^{-(1/2)}}|R_{\ell j k}|
+\sum_{\ell\in \Z^{\nu}, \lvert j \rvert, \lvert k \rvert\leq \mathtt{C} \langle \ell \rangle^{\tau_1}\gamma^{-(1/2)}} |R_{\ell j k}|.
\end{equation*}
On one hand we have that, using Lemmata \ref{delpiero} and \ref{singolo}, 
\begin{equation*}
\begin{aligned}
\sum_{\ell\in \Z^{\nu}, \lvert j \rvert, \lvert k \rvert\geq \mathtt{C} \langle \ell \rangle^{\tau_1}\gamma^{-(1/2)}}|R_{\ell j k}|&\le 
C\sum_{j-k=h, \lvert h \rvert\le C \lvert \ell \rvert} L^{\nu-1}\gamma \langle \ell \rangle^{-\tau_1}\le 
C L^{\nu-1}\gamma\sum_{\ell\in\Z^{\nu}}  \langle \ell \rangle^{-(\tau_1-1)}\leq
\tilde{C} L^{\nu-1}\g ,
\end{aligned}
\end{equation*}
for some $\tilde{C}\geq C>0$.
On the other hand
\begin{equation*}
\begin{aligned}
\sum_{\substack{\ell\in \Z^{\nu}, \lvert j \rvert, \lvert k \rvert\leq \mathtt{C} \langle \ell \rangle^{\tau_1}\gamma^{-(1/2)},\\
 \lvert j-k \rvert 
\le C\lvert \ell \rvert} }|R_{\ell j k} | 
&
\le  C\gamma^{(3/2)} L^{\nu-1}  
\sum_{\ell\in\Z^{\nu}} \frac{\lvert \ell \rvert \langle \ell \rangle^{\tau_1}}{\sqrt{\gamma}\langle \ell \rangle^{\tau}} 
\le C \gamma L^{\nu-1} 
\sum_{\ell\in\Z^{\nu}}  \langle \ell \rangle^{-(\tau-\tau_1-1)}
\leq C \gamma L^{\nu-1}.
\end{aligned}
\end{equation*}
The discussion above  implies estimates \eqref{stimedimisuraTeo}.
\end{proof}

\begin{proof}[{\bf Proof of Theorem \ref{teoMainRed}} ({\bf Reducibility})]
It is sufficient to set $\Phi:=\Phi_2\circ\Phi_{1}$ where $\Phi_1(\omega)$ is the map given in 
 Theorem \ref{risultatosez8} while $\Phi_2$ in Theorem \ref{ReducibilityDP}.
 The bound \eqref{grano1000}
 follows by \eqref{grecia} and \eqref{grano}. Theorem \ref{stimedimisura} provides the measure estimate
 on the set $\mathcal{O}_{\infty}$  in \eqref{grano1001}.
\end{proof}

\begin{proof}[{\bf Proof of Theorem \ref{almostRED}} ({\bf Almost Reducibility})]
Consider $\calL_{\omega}(\mathfrak{I}_1)$, $\calL_{\omega}(\mathfrak{I}_2)$ under the hypotheses of Theorem \ref{almostRED}.
 Theorems \ref{risultatosez8} and \ref{ReducibilityDP} applies to the operator $\calL_{\omega}(\mathfrak{I}_1)$ hence
the results of Theorem \ref{teoMainRed} holds for $\omega$ in the set $\Omega_{1}(\mathfrak{I}_1)$ (see \eqref{prime}). 
Recalling Remark \ref{Includo} let us assume that 
\begin{equation}\label{includo2}
\calO_{\infty}(\mathfrak{I}_1)\subset\Lambda_{N}(\mathfrak{I}_2)\stackrel{(\ref{prime100}), (\ref{seconde100})}{=}\Omega_1^{(N)} \cap \Omega_2^{(N)}.
\end{equation}
We show that the thesis will follows. 
Indeed we can apply the iterative Lemma $5.2$ in Section $5$ of \cite{FGMP}
for $n=1,2,\ldots,k<\infty$ where the  larger is $N$  the larger is $k$. Actually $k$ has to be chosen in such a way $N_{k}\equiv N$
where $N_{n}=N_{0}^{(\frac{3}{2})^{n}}$.
 Hence $\calL_{\omega}(\mathfrak{I}_2)$ can be conjugated to an operator of the form
\[
\widetilde{\calL}_{\omega}(\mathfrak{I}_2):=\oo\cdot\del_{\f}-m^{(N)}J-J\circ a^{(N)}(\mathfrak{I}_2;\f,x)
+\widetilde{\RR}^{(N)}(\mathfrak{I}_2)
\]
where the constant $m^{(N)}$ and the real function $a^{(N)}$ satisfy the bounds \eqref{emmeENNE}, \eqref{emmeENNE2} respectively.
The linear operator $\widetilde{\RR}^{(N)}={\rm Op}(\widetilde{r})+\widehat{\RR}^{+}$ where $\widetilde{r}\in S^{-1}$, 
$\widehat{\RR}^{+}\in\gotL_{\rho, p}$ and satisfy the hypotheses of Proposition \ref{iterazione riducibilita}.
For $\oo\in \Omega_{2}^{(N)}(\mathfrak{I}_2)$ one can iterate the procedure of Prop. \ref{iterazione riducibilita}
with $1\leq n\leq k<\infty$. It is important to note that the maps $\calQ_{n-1}$ given in $({\bf S2})_{n}$ are the identity plus
$\Psi_{n-1}$ a $-1$-modulo-tame operator. By \eqref{tame Psi nu - 1} and \eqref{emmeENNE2} on $a^{(N)}$ one has that
\[
\calQ_{n-1} \circ J\circ a^{(N)}(\f,x)\circ\calQ_{n-1}^{-1}=J\circ a^{(N)}(\f,x)+\calP_{n}
\]
with $\calP_{n}$ satisfying the second bound in \eqref{emmeENNE2} for any $n\leq k$. In other words these terms 
are already ``small'' and they are not to be taken into account in
 the
reducibility procedure. 
By the reasoning above one can prove \eqref{assoOpBIS} and \eqref{emmeENNE2}.
It remains to show that \eqref{includo2} and the \eqref{emmeENNE}. 
First we have $\Omega_1(\mathfrak{I}_1)\subset\Omega_{1}^{(N)}(\mathfrak{I}_2)$
 Remark $5.3$ in \cite{FGMP}. To show the inclusion $\Omega_2(\mathfrak{I}_1)\subset\Omega_{2}^{(N)}(\mathfrak{I}_2)$
 we reason as follows.

 \noindent
 We first note that, by Lemma \ref{BrexitDP}, if $|\omega(j)-\omega(k)|> C_1^{-1}|\ell|$ then 
 $R_{\ell j k}(\mathfrak{I}_1)= R_{\ell j k}(\mathfrak{I}_2)=\emptyset$ (recall \eqref{BadSetsDP}), so that our claim is trivial.  Otherwise,  if 
$|\omega(j)-\omega(k)|\le  C_1^{-1}|\ell|\le C_1^{-1} N\,$
we claim that for all $j, k\in\mathbb{Z}$ we have (recall \eqref{anagrafe2})
\begin{equation}\label{marathon1}
\lvert (d_j^{(N)}-d^{(N)}_k)(\mathfrak{I}_2)-(d_j-d_k)(\mathfrak{I}_1)\rvert
\le \varepsilon\gamma^{-1}N\big(\sup_{\oo\in \calO_0}\|\mathfrak{I}_1-\mathfrak{I}_2\|_{s_0+\mu}+ N^{-\frac{3}{2}\mathtt{a}}\big)
\qquad \forall\omega\in \calO_{\infty}(\mathfrak{I}_1).
\end{equation}
The \eqref{marathon1} imply the \eqref{emmeENNE}.
We now prove that \eqref{marathon1} implies that $\Omega_{2}(\mathfrak{I}_1)\subset\Omega_{2}^{(N)}(\mathfrak{I}_2)$. 
For all $j\neq k$, $\lvert \ell \rvert\le N$, $\omega\in\calO_{\infty}(\mathfrak{I}_1)$ by \eqref{marathon1} 
\begin{equation}
\begin{aligned}
&\lvert \omega\cdot \ell +d^{(N)}_j(\mathfrak{I}_2)-d^{(N)}_k(\mathfrak{I}_2)\rvert
\geq \lvert \omega\cdot \ell +d_j(\mathfrak{I}_1)-d_k(\mathfrak{I}_1)\rvert-
\lvert (d^{(N)}_j-d^{(N)}_k)(\mathfrak{I}_2)-(d_j-d_k)(\mathfrak{I}_1)\rvert\\
&\geq 2\gamma^{3/2}\langle \ell \rangle^{-\tau}-\varepsilon\gamma^{-1}N^{-\frac{3}{2}\mathtt{a}}
\geq 2(\gamma^{3/2}-\rho)\langle \ell \rangle^{-\tau}
\end{aligned}
\end{equation}
where we used \eqref{piccolezzaI1}.

\noindent
\textit{Proof of \eqref{marathon1}}. By \eqref{FinalEigenvaluesDP} (recalling \eqref{omeghinoj})
\begin{equation}
\begin{aligned}
(d_j^{(N)}-d_k^{(N)})(\mathfrak{I}_2)-(d_j-d_k)(\mathfrak{I}_1)&=(m^{(N)}(\mathfrak{I}_2)-m(\mathfrak{I}_1))
(\omega(j)-\omega(k))\\
&+(r^{(N)}_j(\mathfrak{I}_2)-r_j (\mathfrak{I}_1))+(r^{(N)}_k(\mathfrak{I}_2)-r_k (\mathfrak{I}_1)).
\end{aligned}
\end{equation}
Choose $k\in \N$ such that $N_{k-1}\equiv N$. In this way we have that $r_{j}^{(N)}(\mathfrak{I}_2)$ coincides with $r_{j}^{(k)}$
given in Proposition \ref{iterazione riducibilita}.
We apply Proposition \ref{iterazione riducibilita}-$({\bf S4})_k$ 
in order to conclude that
\begin{equation}\label{marathon2}
\Omega^{\gamma^{3/2}}_{k}(\mathfrak{I}_1)\subseteq \Omega_{k}^{\gamma^{3/2}-\rho}(\mathfrak{I}_2),
\end{equation}
since
the smallness condition in \eqref{tab} is satisfied
by \eqref{piccolezzaI1}.
Then by  \eqref{4.48} 
\begin{equation}
 \calO_{\infty}(\mathfrak{I}_1)\subseteq \bigcap_{j\geq 0}\Omega^{\gamma^{3/2}}_j(\mathfrak{I}_1)\subseteq 
 \Omega_{k}^{\gamma^{3/2}}(\mathfrak{I}_1)\stackrel{(\ref{marathon2})}{\subseteq} \Omega_{k}^{\gamma^{3/2}-\rho}(\mathfrak{I}_2).
\end{equation}
For all $\omega\in \calO_{\infty}(\mathfrak{I}_1)\subseteq \Omega_{k}^{\gamma^{3/2}}(\mathfrak{I}_1) \cap 
\Omega_{k}^{\gamma^{3/2}-\rho}(\mathfrak{I}_2)$, 
 we deduce by Proposition \ref{iterazione riducibilita}-$({\bf S3})_{k}$  
\begin{equation}\label{marathon3}
\begin{aligned}
\langle j \rangle\lvert r_j^{(k)}(\mathfrak{I}_2)-r_j^{(k)}(\mathfrak{I}_1) \rvert 
&\stackrel{(\ref{r nu i1 - r nu i2})}{\le}  \varepsilon\g^{-1} \lVert \mathfrak{I}_2-\mathfrak{I}_1 \rVert_{s_0+\sigma}.
\end{aligned}
\end{equation}
We have, by \eqref{stima cal R nu}, for any $k\in\mathbb{N}$
\begin{equation}\label{marathon4}
\langle j \rangle\lvert r_j(\mathfrak{I}_1)-r_j^{(n+1)}(\mathfrak{I}_1) \rvert 
\le \langle j \rangle \sum_{j\geq k} \lvert r_j^{(j+1)}(\mathfrak{I}_1)-r_j^{(j)}(\mathfrak{I}_1)\rvert
\le  {\mathfrak M}^{\sharp, \g_*}_{0} (s_0, {\mathtt b}) \sum_{j\geq n}  N_{j}^{- {\mathtt a}}\stackrel{(\ref{PiccolezzaperKamredDP1000})}{\le} \e\gamma^{-1}  N_k^{-\mathtt{a}}.
\end{equation}
Therefore $\forall \omega\in\calO_{\infty}(\mathfrak{I}_1)$, 
$\forall j\in\mathbb{Z}$ we have (recall the choice of $k$ above)
\begin{equation*}
\begin{aligned}
&\langle j \rangle\lvert r_j^{(N)}(\mathfrak{I}_2)-r_j(\mathfrak{I}_1) \rvert\le\langle j \rangle
\Big( \lvert r_j^{(k)}(\mathfrak{I}_2)-r^{(k)}_j(\mathfrak{I}_1) \rvert+\lvert r_j(\mathfrak{I}_1)-r^{(k)}_j(\mathfrak{I}_1) \rvert\Big)\\&\stackrel{(\ref{marathon3}), (\ref{marathon4})}{\le} \varepsilon\g^{-1} \lVert \mathfrak{I}_1-\mathfrak{I}_1 \rVert_{s_0+\sigma}
+C\mathfrak{M}_{0}^{\sharp, \g_*} (s_0, {\mathtt b})  N_k^{-\mathtt{a}}.
\end{aligned}
\end{equation*}
Using similar reasonings, the iterative Lemma $5.2$ in Section $5$ of \cite{FGMP} and recalling $|j-k|\lesssim |\ell|$ 
one can prove
that
\begin{equation}
\begin{aligned}
\lvert m^{(N)}(\mathfrak{I}_2)-m(\mathfrak{I}_1) \rvert\lvert j \rvert{\le} & C \varepsilon
\lVert \mathfrak{I}_2-\mathfrak{I}_1\rVert_{s_0+2}\lvert \ell \rvert .
\end{aligned}
\end{equation}
This concludes the proof of  \eqref{marathon1}.
\end{proof}

\appendix

\section{Technical Lemmata}\label{lemmitecnici}

\subsection{Tame and Modulo-tame operators}\label{lemmitecnicitame}
In the following we collects some properties of operators 
which are ``Lip-tame'' or  ``Lip-modulo-tame'' according to Definitions 
\ref{LipTameConstants} and \ref{def:op-tame}.

\begin{lem}[{\bf Composition of Lip-Tame operators}]\label{lem: 2.3.6}
Let $A$ and $B$ be respectively Lip-$\s_A$-tame and Lip-$\s_B$-tame operators with tame constants respectively $\mathfrak{M}^{\g}_A(\s_A, s)$ and $\mathfrak{M}_B^{\g}(\s_B, s)$. Then the composition $A\circ B$ is a Lip-$(\s_A+\s_B)$-operator with
\begin{equation}
\mathfrak{M}^{\g}_{A\circ B} (\s_{A+B}, s)\le \mathfrak{M}^{\g}_A(\s_A, s)\mathfrak{M}_B^{\g}(\s_B, s_0+\s_A)+\mathfrak{M}^{\g}_A(\s_A, s_0)\mathfrak{M}_B^{\g}(\s_B, s+\s_A).
\end{equation}
The same holds for $\s$-tame operators. 
\end{lem}
\begin{proof}
The proof follows by the definitions and by using triangle inequalities.
\end{proof}
\begin{lem}
Let $A$ be a Lip-$\s$-tame operator. Let $u(\omega)$, $\omega\in \calO\subset \mathbb{R}^{\nu}$ be a $\omega$-parameter family of Sobolev functions $H^s$, for $s\geq s_0$. Then
\begin{equation}
\lVert A u \rVert_s^{\g, \calO}\le_s \mathfrak{M}_A^{\g}(\s, s) \lVert u \rVert_{s_0}^{\g, \calO}+\mathfrak{M}_A^{\g}(\s, s_0) \lVert u \rVert_{s}^{\g, \calO}.
\end{equation}
\begin{proof}
By definition \eqref{lipTAME} we have $\mathfrak{M}_A(\s, s)\le \mathfrak{M}^{\g}_A(\s, s)$ and $\lVert u \rVert_s\le \lVert u \rVert_s^{\g, \calO}$. Then the thesis follows by the triangle inequalities
\[
\lvert \omega-\omega' \rvert^{-1}\lVert A(\omega)u(\omega)-A(\omega')u(\omega') \rVert_s\le \lVert (\Delta_{\omega, \omega'} A) u (\omega) \rVert_s+\lVert A(\omega') \Delta_{\omega, \omega'} u \rVert_s.
\]
\end{proof}
\end{lem}
\begin{lem}\label{actiononsobolev}
Let $A=\op(a(\varphi, x, D))\in OPS^{0}$ be a family of pseudo differential operators which are Lipschitz in a parameter $\omega\in \calO\subset\mathbb{R}^{\nu}$. If $\lvert A \rvert^{\g, \calO}_{0, s, 0}<+\infty$ (recall \eqref{norma2}) then $A$ is a $0$-tame operator with
\begin{equation}
\mathfrak{M}^{\g}_{A}(\s, s)\le C(s) \lvert A \rvert^{\g, \calO}_{0, s, 0}.
\end{equation}
\end{lem}

\begin{proof}
We refer to the proof of Lemma $2. 21$ of \cite{BM1}.
\end{proof}

\noindent
Given an operator $A\in \mathcal{L}_{\rho,p}$ we define
\begin{equation*}
\begin{aligned}
\mathfrak{M}^{\g}_{\partial_{\varphi_m}^{\tb_1}A}(-1, s):=\mathfrak{M}^{\g}_{\langle D_x \rangle^{1/2}\partial_{\varphi_m}^{\tb_1}A \langle D_x \rangle^{1/2}}(0, s),\qquad \mathfrak{M}^{\g}_{\partial_{\varphi_m}^{\tb_1}[A, \partial_x]}(-1, s) :=\mathfrak{M}^{\g}_{\langle D_x \rangle^{1/2}\partial_{\varphi_m}^{\tb_1}[A, \partial_x]\langle D_x \rangle^{1/2}}(0, s) .
\end{aligned}
\end{equation*}
the  Lip-$0$-tame constant of 	
$\langle D_x \rangle^{1/2} A \langle D_x \rangle^{1/2}$, $\langle D_x \rangle^{1/2}\partial_{\varphi_m}^{\tb} A \langle D_x \rangle^{1/2}$, 
 $\langle D_x \rangle^{1/2}[\partial_{\varphi_m}^{\tb} A, \partial_x]\langle D_x \rangle^{1/2}$,
  for any $m=1, \dots, \nu$, $0\le \tb_1\le \tb$ and we set
\begin{equation}\label{Mdritta2}
\begin{aligned}
\mathbb{B}^{\g}_A(s, \tb):=\max_{\substack{0\le \tb_1\le \tb\\ m=1,\dots, \nu}} 
\max\Big(\mathfrak{M}^{\g}_{\partial_{\varphi_m}^{\tb_1}A}(-1, s), \mathfrak{M}^{\g}_{\partial_{\varphi_m}^{\tb_1}[A, \partial_x]}(-1, s)    \Big).
\end{aligned}
\end{equation}

\noindent
We have the following result.

\begin{lem}\label{LemmaAggancio}
Let $s_0\geq [\nu/2]+3, \,s_0 \in\mathbb{N}$, $\mathtt{b}_0\in\mathbb{N}$ and recall \eqref{CK0-tame}, Def. \ref{ellerho} and \eqref{anagrafe}.

\noindent
$(i)$
 Let $A\in \mathcal{L}_{\rho, p}$ with $\rho:=s_0+\tb_0+3$, $p=s_0$,  
then $A$ is a $-1$-modulo tame operator. Moreover
\begin{equation}\label{chavez1}
\mathfrak{M}^{\sharp, \g^{3/2}}_A( s)\le \max_{m=1, \dots, \nu} \mathfrak{M}^{\g^{3/2}}_{\partial_{\varphi_m}^{s_0}  [A, \partial_x]}(-1,s),
\end{equation}
\begin{equation}\label{chavez2}
\mathfrak{M}^{\sharp, \g^{3/2}}_A( s, \tb_0)\le \max_{m=1,\dots, \nu}  \mathfrak{M}^{\g^{3/2}}_{\partial_{\varphi_m}^{s_0+\tb_0}  [A, \partial_x]}(-1,s).
\end{equation}
\begin{equation}\label{chavez20}
\|\lD^{1/2}\und{\Delta_{12} A}\lD^{1/2}\|_{\mathcal L(H^{s_0})}, \,\, \|\lD^{1/2}\und{\Delta_{12} \langle \partial_{\f}\rangle^{{\mathtt b_0}} A  }\lD^{1/2}\|_{\mathcal L(H^{s_0})} \le  \mathbb{B}_{\Delta_{12} A}(s_0, \mathtt{b}_0)
\end{equation}
where
\begin{equation}\label{Mdrittaconlai2}
\begin{aligned}
\mathbb{B}_{\Delta_{12} A}(s_0, \tb):=\max_{\substack{0\le \tb_1\le \tb\\ m=1,\dots, \nu}} 
\max\Big(\mathfrak{N}_{\partial_{\varphi_m}^{\tb_1}\Delta_{12}A}(-1, s_0), 
\mathfrak{N}_{\partial_{\varphi_m}^{\tb_1}[\Delta_{12}A, \partial_x]}(-1, s_0)    \Big).
\end{aligned}
\end{equation}

\noindent
$(ii)$ If $A:=\op(a)$ with $a=a(\oo, i(\omega))$ in $S^{m}$ with $m\leq-1$
		depending on $\oo\in \calO_0\subset \mathbb{R}^{\nu}$ in a Lipschitz way and on $i$ in a  Lipschitz way, then
  $A$ is a $-1$-modulo tame operator and bounds \eqref{chavez1}-\eqref{chavez20} hold.
\end{lem}
\begin{proof}
	Consider $\mathtt{b}\in \mathbb{N}$ and $\rho\in \mathbb{N}$ with $\rho\geq\mathtt{b}+3$.
We claim that if $A\in \gotL_{\rho, p}$ (see Def. \ref{ellerho}) then one has
\begin{equation}\label{pasta10}
\mathbb{B}^{\gamma}_{A}(s,\mathtt{b})\leq_{\rho,s}\mathbb{M}^{\gamma}_{A}(s,\rho-2),
\qquad
\mathbb{B}_{\Delta_{12}A }(p,\mathtt{b})\leq_{\rho, p}\mathbb{M}_{\Delta_{12}A }(p,\rho-3).
\end{equation}
The fact that  $\langle D_{x}\rangle^{1/2}A\langle D_x \rangle^{1/2}$  is Lip-$0$-tame  follows by \eqref{megaTame2} since
	$\rho\geq 1$. 
	Indeed $\langle D_{x}\rangle^{-\rho+1}$ is bounded in $x$ and for any $h\in H^{s}$
	\[
	\begin{aligned}
	\|\langle D_{x}\rangle^{\frac{1}{2}}A\langle D_x \rangle^{\frac{1}{2}}h\|_{s}^{\g,\calO_0}
	&\leq 
	\|\langle D_{x}\rangle^{-\rho+1}\big(\langle D_{x}\rangle^{\rho-\frac{1}{2}}A
	\langle D_x \rangle^{\frac{1}{2}}\big)  h\|_{s}^{\g,\calO_0} \\
	&\leq_{s}\gotM^{\gamma}_{A}(-\rho,s)\|h\|^{\gamma,\calO_0}_{s_0}+
	\gotM^{\gamma}_{A}(-\rho,s_0)\|h\|^{\gamma,\calO_0}_{s}. 
	\end{aligned}
	\]
By  studying the tameness constant of 
$
\partial_{\varphi}^{\vec{\tb}} A, [A, \partial_x], [\partial_{\varphi}^{\vec{\tb}} A, \partial_x]
\Delta_{12}A ,\del_{\f}^{\vec{\tb}}\Delta_{12}A , [\Delta_{12}A ,\del_{x}]$ and 
$[\del_{\f}^{\vec{\tb}}\Delta_{12}A ,\del_{x}]$
for $\vec{\tb}\in \mathbb{N}^{\nu}$, $|\vec{\tb}|=\tb$, following the same reasoning above one gets the \eqref{pasta10}. 
	
\medskip
\noindent
We have, by Cauchy-Schwarz,
\[
\begin{aligned}
&\lVert \langle D_x \rangle^{1/2} \underline{A} \langle D_x \rangle^{1/2} u \rVert_s^2\le \sum_{\ell\in\mathbb{Z}^{\nu}, j \in \mathbb{Z}} \langle \ell, j \rangle^{2 s} \Big( \sum_{\ell'\in\mathbb{Z}^{\nu}, j' \in \mathbb{Z}} \langle j' \rangle^{1/2} \lvert (A)_j^{j'}(\ell-\ell') \rvert \langle j \rangle^{1/2} \lvert u_{\ell' j'} \rvert \Big)^2\\
&\qquad\le \sum_{\ell\in\mathbb{Z}^{\nu}, j \in \mathbb{Z}} \langle \ell, j \rangle^{2 s} \Big( \sum_{\ell'\in\mathbb{Z}^{\nu}, j' \in \mathbb{Z}} \frac{\langle \ell-\ell' \rangle^{s_0} \lvert j-j' \rvert}{\langle \ell-\ell' \rangle^{s_0} \lvert j-j' \rvert} \langle j' \rangle^{1/2} \lvert (A)_j^{j'}(\ell-\ell') \rvert \langle j \rangle^{1/2} \lvert u_{\ell' j'} \rvert \Big)^2\\
&\qquad\le  \sum_{\ell\in\mathbb{Z}^{\nu}, j \in \mathbb{Z}} \langle \ell, j \rangle^{2 s} ( \sum_{\ell'\in\mathbb{Z}^{\nu} , j' \in \mathbb{Z}}  C_{\ell j} )\, ( \sum_{\ell'\in\mathbb{Z}^{\nu} , j' \in \mathbb{Z}}  \langle j \rangle \langle j' \rangle \lvert j-j' \rvert^2 \langle \ell-\ell' \rangle^{2 s_0} \lvert (A)_j^{j'}(\ell-\ell')\rvert^2 \lvert u_{\ell' j'} \rvert^2)\\
&\qquad\le C  \sum_{\ell'\in\mathbb{Z}^{\nu}, j' \in \mathbb{Z}} \lvert u_{\ell' j'} \rvert^2 ( \sum_{\ell\in\mathbb{Z}^{\nu} , j \in \mathbb{Z}}  \langle \ell, j \rangle^{2 s}\, \,\langle j \rangle \langle j' \rangle \lvert j-j' \rvert^2 \langle \ell-\ell' \rangle^{2 s_0} \lvert (A)_j^{j'}(\ell-\ell')\rvert^2)
\end{aligned}
\]
since
\[
C:= \sum_{\ell, \ell'\in\mathbb{Z}^{\nu} , j, j' \in \mathbb{Z}}  C_{\ell j}<\infty,\,\,\qquad   C_{\ell j}:=\sum_{\ell'\in\mathbb{Z}^{\nu}, j' \in \mathbb{Z}} \frac{1}{\langle \ell-\ell' \rangle^{2 s_0} \lvert j-j' \rvert^2}.
\]
By the fact that for any $1\le m\le\nu$ (recall \eqref{cervino2})
\begin{equation*}
\begin{aligned}
&\sum_{\ell\in\mathbb{Z}^{\nu}, j \in \mathbb{Z}} \langle \ell, j \rangle^{2 s} \langle j \rangle \langle j' \rangle \lvert j-j' \rvert^2 \langle \ell_m-\ell^{'}_m \rangle^{2 s_0} \lvert (A)_j^{j'}(\ell-\ell')\rvert^2\\
&\qquad\qquad\le 2 (\mathfrak{M}^{\g^{3/2}}_{\partial_{\varphi_m}^{s_0}[A, \partial_x]}(-1, s))^2 \langle \ell', j' \rangle^{2 s_0}+2 (\mathfrak{M}^{\g^{3/2}}_{\partial_{\varphi_m}^{s_0}[A, \partial_x]}(-1, s_0))^2 \langle \ell', j' \rangle^{2 s}
\end{aligned}
\end{equation*}
and $\langle \ell-\ell' \rangle\le \max_{m=1, \dots, \nu} \langle \ell_m-\ell^{'}_m\rangle$
we obtain 
\[
\lVert \langle D_x \rangle^{1/2} \underline{A} \langle D_x \rangle^{1/2} u \rVert_s^2\le 2 \max_{m=1, \dots, \nu} (\mathfrak{M}^{\g^{3/2}}_{\partial_{\varphi_m}^{s_0}[A, \partial_x]}(-1, s_0))^2 \lVert u \rVert_s^2+2 \max_{m=1, \dots, \nu} (\mathfrak{M}^{\g^{3/2}}_{\partial_{\varphi_m}^{s_0}[A, \partial_x]}(-1, s))^2\lVert u \rVert_{s_0}^2.
\]
Following the same reasoning above we conclude the same bound for $\lVert \langle D_x \rangle^{1/2} \underline{\Delta_{\omega, \omega'} A} \langle D_x \rangle^{1/2} u \rVert_s^2$,
it is sufficient to substitute $A_{j}^{j'}(\ell-\ell')$ with $\big(A_{j}^{j'}(\ell-\ell', \omega)-A_{j}^{j'}(\ell-\ell', \omega')\big)/(\omega-\omega')$ in the computations above.
By the fact that $\g^{3/2}<1$ we deduce \eqref{chavez1}. The proofs of \eqref{chavez2}, \eqref{chavez20} are analogous.
The proof of item $(ii)$ follows using the above computations 
by noting that $\del_{\f_m} A$ and the commutator $[A,\del_{x}]$ are still pseudo-differential operators of order $-1$. 
\end{proof}

\begin{lem}\label{proprietatame}
Recall \eqref{partialorder}. The following holds.

\vspace{0.5em}
\noindent
$(i)$ If $A\preceq B$  and $\Delta_{\omega,\omega'} A\preceq \Delta_{\omega,\omega'} B$ for all $\omega\neq \omega'\in \mathcal O$, we may choose the modulo-tame constants of $A$ so that
\[
{\mathfrak M}_{A}^{{\sharp, \g^{3/2}}} (s)\le {\mathfrak M}_{B}^{{\sharp, \g^{3/2}}} (s)\,.
\] 

\noindent
$(ii)$ Let $ A$ be a $-1$ modulo-tame operator with modulo-tame constant 
$ {\mathfrak M}_A^{{\sharp, \g^{3/2}}}(s) $. Then  $\lD^{1/2}A\lD^{1/2}$ is majorant bounded $H^s\to H^{s}$ 
\[
\|\lD^{1/2}\un{A}\lD^{1/2}\|_{\mathcal L(H^s)}\le 2 {\mathfrak M}_A^{{\sharp, \g^{3/2}}}(s)\,, \qquad |A_j^j(0)|^{\g^{3/2}}\le 
{\mathfrak M}_A^{{\sharp, \g^{3/2}}}(s_0) \langle j\rangle^{-1}.
\]
\noindent
$(iii)$ Suppose that $ \langle \pa_\f \rangle^{{\mathtt b_0}} A $, $ {{\mathtt b_0}} \geq  0 $, is $-1$  modulo-tame. Then 
the operator $ \Pi_N^\bot A $ is $-1$ modulo-tame with modulo-tame constant
\begin{equation}
\label{proprieta tame proiettori moduli}
{\mathfrak M}_{\Pi_N^\bot A}^{{\sharp, \g^{3/2}}} (s) \leq \min\{ N^{- {{\mathtt b_0}} }{\mathfrak M}_{ \langle \pa_\f \rangle^{{\mathtt b_0}} A}^{{\sharp, \g^{3/2}}} (s),{\mathfrak M}_{ A}^{{\sharp, \g^{3/2}}} (s)\} \, . 
\end{equation}
\noindent
$(iv)$	Let $A,\, B$ be two $-1$ modulo-tame operators with modulo-tame constants 
$ {\mathfrak M}_A^{{\sharp, \g^{3/2}}}(s),\,\, {\mathfrak M}_B^{{\sharp, \g^{3/2}}}(s) $. Then 
$ A+ B $ is $-1$ modulo-tame with modulo-tame constant
\begin{equation}\label{modulo-tame-A+B}
{\mathfrak M}_{A + B}^{{\sharp, \g^{3/2}}} (s) \leq {\mathfrak M}_A^{{\sharp, \g^{3/2}}} (s)  + {\mathfrak M}_B^{{\sharp, \g^{3/2}}} (s)  \,.
\end{equation}
The composed operator 
$  A  \circ B $ is $-1$ modulo-tame with modulo-tame constant
\begin{equation}\label{modulo tame constant for composition}
{\mathfrak M}_{A B}^{{\sharp, \g^{3/2}}} (s) \leq  C(s) \big( {\mathfrak M}_{A}^{{\sharp, \g^{3/2}}}(s) 
{\mathfrak M}_{B}^{{\sharp, \g^{3/2}}} (s_0) + {\mathfrak M}_{A}^{{\sharp, \g^{3/2}}} (s_0) 
{\mathfrak M}_{B}^{{\sharp, \g^{3/2}}} (s) \big)\,.
\end{equation}
Assume in addition that  $ \langle \partial_\f \rangle^{{\mathtt b_0}} A $, 
$ \langle \partial_\f \rangle^{{\mathtt b_0}}  B $ are $-1$ modulo-tame with  modulo-tame constants 
 $ {\mathfrak M}_{\langle \partial_\f \rangle^{{\mathtt b_0}} A}^{{\sharp, \g^{3/2}}} (s) $ and  
$ {\mathfrak M}_{\langle \partial_\f \rangle^{{\mathtt b_0}} B}^{{\sharp, \g^{3/2}}} (s) $ respectively, then 
$ \langle \partial_\f \rangle^{{\mathtt b_0}} (A  B) $ is $-1$ modulo-tame with 
modulo-tame constant satisfsying
\begin{align}\label{K cal A cal B}
{\mathfrak M}_{\langle \partial_\f \rangle^{{\mathtt b_0}} (A  B)}^{{\sharp, \g^{3/2}}} (s) & \leq
C(s,{{\mathtt b_0}}) \Big( 
{\mathfrak M}_{\langle \partial_\f \rangle^{{\mathtt b_0}} A}^{{\sharp, \g^{3/2}}} (s) 
{\mathfrak M}_{B}^{{\sharp, \g^{3/2}}} (s_0) + 
{\mathfrak M}_{\langle \partial_\f \rangle^{{\mathtt b_0}} A }^{{\sharp, \g^{3/2}}} (s_0) 
{\mathfrak M}_{B}^{{\sharp, \g^{3/2}}} (s) \nonumber \\ 
& \qquad \qquad \qquad \quad + {\mathfrak M}_{A}^{{\sharp, \g^{3/2}}} (s) {\mathfrak M}_{ \langle \pa_\f \rangle^{{\mathtt b_0}} B}^{{\sharp, \g^{3/2}}} (s_0) 
+ {\mathfrak M}_{A}^{{\sharp, \g^{3/2}}} (s_0) {\mathfrak M}_{ \langle \pa_\f \rangle^{{\mathtt b_0}} B}^{{\sharp, \g^{3/2}}} (s)\Big) \,.
\end{align}
Finally, for any $k\ge 1$ we have, setting $ L= \mathrm{ad}^k(A) B$, $\mathrm{ad}(A) B:=A B-B A$:
\begin{equation}\label{adAkB}
\begin{aligned}
&{\mathfrak M}_{\langle \partial_\f \rangle^{{\mathtt b_0}} L}^{{\sharp, \g^{3/2}}} (s) \leq  C(s,{{\mathtt b_0}})^k   \Big[ ({\mathfrak M}_{A}^{{\sharp, \g^{3/2}}} (s_0))^{k}{\mathfrak M}_{\langle \partial_\f \rangle^{{\mathtt b_0}} B }^{{\sharp, \g^{3/2}}} (s)\\
 &+k({\mathfrak M}_{A}^{{\sharp, \g^{3/2}}} (s_0))^{k-1}\left({\mathfrak M}_{\langle \partial_\f \rangle^{{\mathtt b_0}} A}^{{\sharp, \g^{3/2}}} (s) 
{\mathfrak M}_{B}^{{\sharp, \g^{3/2}}} (s_0) + 
{\mathfrak M}_{\langle \partial_\f \rangle^{{\mathtt b_0}} A }^{{\sharp, \g^{3/2}}} (s_0) 
{\mathfrak M}_{B}^{{\sharp, \g^{3/2}}} (s)\right)  \\  
 &+  k(k-1)({\mathfrak M}_{A}^{{\sharp, \g^{3/2}}} (s_0))^{k-2}{\mathfrak M}_{A}^{{\sharp, \g^{3/2}}} (s){\mathfrak M}_{\langle \partial_\f \rangle^{{\mathtt b_0}} A}^{{\sharp, \g^{3/2}}} (s_0)  {\mathfrak M}_{ B}^{{\sharp, \g^{3/2}}} (s_0) \Big]\,.
\end{aligned}
\end{equation}
The same bound holds if we set $L= A^k B$.

\smallskip

\noindent
$(v)$ Let $\Phi := {\rm I} + A $ and assume,  for some ${\mathtt b_0}\ge 0$, that $A$, $\langle\partial_\f\rangle^{{\mathtt b_0}}A$ are Lip--1-modulo tame  and 
the smallness condition 
\begin{equation}\label{piccolezza neumann tame}
8 C(\mathcal{S},{\mathtt b_0}) {\mathfrak M}_{A}^{{\sharp, \g^{3/2}}} (s_0)  < 1\,,\qquad C(\mathcal{S},{\mathtt b_0})=\max_{s_0\le s\le \mathcal{S}}C(s,\tb_0)
\end{equation}
holds.
 Then the operator $ \Phi $ is invertible,  
$\check A :=   \Phi^{- 1} - {\rm I}  $  is 
$-1$  modulo-tame together with $\langle \partial_\f \rangle^{{\mathtt b_0}}  A$ with modulo-tame constants
\begin{align}\label{disuguaglianza constante tame A tilde A}
&{\mathfrak M}_{\check A}^{{\sharp, \g^{3/2}}} (s) \leq  2 {\mathfrak M}_A^{{\sharp, \g^{3/2}}} (s) \, , \\
\label{Psi tilde alta Neumann moduli}
&{\mathfrak M}_{\langle \partial_\f \rangle^{{\mathtt b_0}}  \check A}^{{\sharp, \g^{3/2}}} (s)  \leq 
2 {\mathfrak M}_{ \langle \partial_\f \rangle^{{\mathtt b_0}}A}^{{\sharp, \g^{3/2}}} (s)  + 
8  C(\mathcal{S},{\mathtt b_0})  {\mathfrak M}_{ \langle \partial_\f \rangle^{{\mathtt b_0}}A}^{{\sharp, \g^{3/2}}} (s_0)\, {\mathfrak M}_A^{{\sharp, \g^{3/2}}}(s) \, .  
\end{align}
\end{lem}
\begin{proof}
	In the following we shall sistematically use the fact that if $B$ is an operator with matrix coefficients $\ge 1$, then $A\preceq \un{A\circ B}=\un{A}\circ \un{B} =\un{A}\circ B$. Note that $\langle D_x \rangle^{1/2}$ is a diagonal operator with positive eigenvalues.
\\
{\it (i)} Assume that $A\preceq B$ i.e. $|A_{j}^{j'}(\ell)|\le 	|B_{j}^{j'}(\ell)|$ for all $j,j',\ell$. Then, by \eqref{partialorder},
\[
\| \lD^{1/2}\un{A} \lD^{1/2}u\|_s \le \|\lD^{1/2}\un{A}\lD^{1/2} \un{u}\|_s \le \|\lD^{1/2}\un{B} \lD^{1/2}\un{u}\|_s. 
\]
The same  reasoning holds for $\lD^{1/2}\un{\Delta_{\omega,\omega'} A} \lD^{1/2}$, so that the result follows.
\\
{\it (ii)} The first bound  is just a reformulation of the definition, indeed
\[
\sup_{\|u\|_s\le 1}\| \lD^{1/2}\un{A} \lD^{1/2}u\|_s \le  \sup_{\|u\|_s\le 1}({\mathfrak M}_{A}^{{\sharp, \g^{3/2}}} (s_0) \| u \|_{s} +
{\mathfrak M}_{A}^{{\sharp, \g^{3/2}}} (s) \| u \|_{s_0} )\le 2{\mathfrak M}_{A}^{{\sharp, \g^{3/2}}} (s).
\]
 In order to prove the second bound we notice that setting 
 \[
 B_{j}^{j'}(\ell)= \begin{cases} \langle j\rangle A_j^j(0) & \ell=  0\;\; \mbox{and} \;\; j = j', \\ 0 & \mbox{otherwise}, \\
  \end{cases}
 \]
 we have $B\preceq \lD^{1/2}{A}\lD^{1/2}$, same for $\Delta_{\omega, \omega'}B$. Fix any $j_0$ and  consider the unit vector $u^{(j_0)}$  in $H^{s_0}(\T^{\nu+1})$ defined by
 $u_{j,\ell}= 0$ if $(j,\ell)\neq (j_0,0)$ and $u_{j_0,0}= \langle j_0\rangle^{-s_0}$.  We have by \eqref{partialorder}
\[
 \langle j_0\rangle |A_{j_0}^{j_0}(0)|= \| \un{B}u^{(j_0)}\|_{s_0} \le  \| \lD^{1/2}\un{A}\lD^{1/2} \un{u}^{(j_0)}\|_{s_0} \le  {\mathfrak M}_{A}^{{\sharp, \g^{3/2}}} (s_0).
\]  
The same holds for $ \g^{3/2}\langle j_0\rangle |\Delta_{\omega, \omega'} A_{j_0}^{j_0}(0)|$.
\\
{\it (iii)} We remark that $ |A_{j}^{j'}(\ell) | \le N^{-\tb_0} \langle\ell\rangle^{\tb_0} |A_{j}^{j'}(\ell) |$ if $|\ell|\ge N$ and the same holds for $|\Delta_{\omega, \omega'} A_{j}^{j'}(\ell) | $. Therefore we have
\[
\Pi_N^\perp A\preceq N^{-\tb_0}\langle \partial_\f \rangle^{\tb_0} \Pi_N^\perp A \preceq N^{-\tb_0}\langle \partial_\f \rangle^{\tb_0}  A
\]
and clearly $\Pi_N^\perp A\preceq A$ and the result follows by $(i)$. See also Lemma $2.27$  of \cite{BM1}.
\\
{\it (iv)} The computations involved in this proof are similar to the ones in Lemma $2. 25$ of \cite{BM1}.
For the first bound we just remark that \[\lD^{1/2}(\un{A+B})\lD^{1/2} \preceq \lD^{1/2}\un{A}\lD^{1/2}+ \lD^{1/2}\un{B}\lD^{1/2},\] and the same for the Lipschitz variation, so that \eqref{modulo-tame-A+B} follows.
Regarding the second we note that
\begin{align*}
&\lD^{1/2}\un{A\circ B}\lD^{1/2} \preceq \lD^{1/2}\un{A}\circ \un{B}\lD^{1/2}\preceq  \lD^{1/2}\un{A}\lD^{1/2}\circ \lD^{1/2} \un{B}\lD^{1/2}\,,\\
&\lD^{1/2}\un{\Delta_{\omega, \omega'}A\circ B}\lD^{1/2} \preceq 
  \lD^{1/2}\un{\Delta_{\omega, \omega'}A}\lD^{1/2}\circ \lD^{1/2} \un{B}\lD^{1/2}\\ &+ \lD^{1/2}\un{A}\lD^{1/2}\circ \lD^{1/2} \un{\Delta_{\omega, \omega'}B}\lD^{1/2},
\end{align*}

so that \eqref{modulo tame constant for composition} follows.
For the third bound we 
note that
\begin{equation}\label{allday}
\langle \ell\rangle^{\tb_0} \sum_{j_1, \ell_1+\ell_2=\ell} A_{j}^{j_1}(\ell_1) B_{j_1}^{j'}(\ell_2) \le C(\tb_0)  \sum_{j_1, \ell_1+\ell_2=\ell} (\langle \ell_1\rangle^{\tb_0}+\langle \ell_2\rangle^{\tb_0} ) A_{j}^{j_1}(\ell_1) B_{j_1}^{j'}(\ell_2)
\end{equation}
and the same holds for $\Delta_{\omega, \omega'} A \circ B$ and $ A \circ \Delta_{\omega, \omega'} B$. Hence by \eqref{allday}
\begin{align*}
\lD^{1/2}\un{\langle\partial_\f\rangle^{\tb_0}( A\circ B)}\lD^{1/2}  &\preceq C(\tb_0)\Big(\lD^{1/2}\un{\langle\partial_\f\rangle^{\tb_0} A}\lD^{1/2}\circ \lD^{1/2}\un{B}\lD^{1/2}\\ &+ \lD^{1/2}\un{A}\lD^{1/2}\circ \lD^{1/2}\un{\langle\partial_\f\rangle^{\tb_0} B}\lD^{1/2}\Big) \,,
\end{align*}
same for the Lipshitz variations.  The result follows from the estimate on the composition.

In order to prove \eqref{adAkB} we note that
\[
\lD^{1/2} {\mathrm{ad}}^k(A) B \lD^{1/2} \preceq \un{\mathrm{ad}}^k\left(\lD^{1/2}\un{A}\lD^{1/2})\right)\lD^{1/2}\un B\lD^{1/2}\,,
\]
where $\un{\mathrm{ad}}(A)B:= AB+BA$, since $\mathrm{ad}^k(A) B \preceq \un{\mathrm{ad}}^k(\un{A}) \un{B }$.
Similarly
\begin{align*}
\langle\partial_\f  \rangle^{\tb_0}\lD^{1/2}  \mathrm{ad}^k(A) & B \lD^{1/2} \preceq \un{\mathrm{ad}}^k\left(\lD^{1/2}\un{A}\lD^{1/2})\right)\lD^{1/2}\langle\partial_\f^{\tb_0} \rangle\un B\lD^{1/2}
\\
+\sum_{\substack{k_1+k_2=k-1,\\ k_1, k_2 \geq 0}} & \un{\mathrm{ad}}^{k_1}\left(\lD^{1/2}\un{A}\lD^{1/2})\right)\un{\mathrm{ad}}(\lD^{1/2}\langle\partial_\f^{\tb_0} \rangle\un A\lD^{1/2})\\
 &\un{\mathrm{ad}}^{k_2}\left(\lD^{1/2}\un{A}\lD^{1/2})\right)\lD^{1/2}\un B\lD^{1/2}\,.
\end{align*}
Completely analogous bounds can be proved for the Lipschitz variations, by recalling that
\[
\Delta_{\omega, \omega'}\mathrm{ad} (A)B= \mathrm{ad} (\Delta_{\omega, \omega'}A)B(\omega)+ \mathrm{ad} (A(\omega'))\Delta_{\omega, \omega'}B.
\]
The result follows, by induction, from the estimate on the composition. The estimate \eqref{adAkB} when $C= A^k\circ B $ follows in the same way using
\begin{align*}
\langle\partial_\f \rangle^{\tb_0}\lD^{1/2} (A)^k\circ & B \lD^{1/2} \preceq (\lD^{1/2} \un{A}\lD^{1/2})^k\circ \lD^{1/2}\langle\partial_\f \rangle^{\tb_0}\un{B} \lD^{1/2} \\ +
\sum_{k_1+k_2=k-1} &\left(\lD^{1/2}\un{A}\lD^{1/2})\right)^{k_1} \left(\lD^{1/2}\langle\partial_\f^{\tb_0} \rangle\un{A}\lD^{1/2} \right)\left(\lD^{1/2}\un{A}\lD^{1/2})\right)^{k_2}\\
&\lD^{1/2}\un{B}\lD^{1/2}.
\end{align*}
 {\it (v)} follows by Neumann series, $\check A=\sum_{k\geq 1} (-1)^k\,A^k$, and from \eqref{adAkB} with $L= A^k\circ B $, $B=\rm{I}$.
\end{proof}

\subsection{Pseudo differential operators}

First of all we note that the norm \eqref{norma2} satisfies
\begin{equation}\label{norma3}
\begin{aligned}
&\forall\; s\leq s',\; \al\leq \al'\; \Rightarrow |\cdot|_{m,s,\al}^{\g,\calO}\leq |\cdot|_{m,s',\al}^{\g,\calO},\;\;
 \; |\cdot|_{m,s,\al}^{\g,\calO}\leq |\cdot|_{m,s,\al'}^{\g,\calO}\\
&m\leq m' \quad  \Rightarrow |\cdot|_{m',s,\al}^{\g,\calO}\leq |\cdot|_{m',s,\al}^{\g,\calO}.
\end{aligned}
\end{equation}
In the following lemma we collect properties of pseudo differential operators which will be used in the sequel.
 We remark that along the Nash-Moser iteration we shall control the Lipschitz variation respect to the torus embedding $i:=i(\varphi)$ of the terms of the linearized operator at $i$. Hence we consider pseudo differential operators which depend on this variable.

\begin{lem}\label{PROP}
Fix $m,m',m''\in \mathbb{R}$. Let $i$ be a torus embedding.
Consider symbols
 \[
 a(i, \la,\f,x,\x)\in S^{m},\,\,b(i, \la,\f,x,\x)\in S^{m'},\,\,c(\la,\f,x,\x)\in S^{m''}, \,\, d(\la,\f,x,\x)\in S^{0}
 \]
  which depend on $\lambda\in\calO$ and $i\in H^s$ in a Lipschitz way.
Set 
\[
\begin{aligned}
&A:={\rm Op}(a(\la,\f,x,\x)), \quad B:={\rm Op}(b(\la,\f,x,\x)),\\
&C:={\rm Op}(c(\la,\f,x,\x)), \quad D:={\rm Op}(d(\la,\f,x,\x)).
\end{aligned}
\]
Then one has\\
\noindent
$(i)$ for any $\al\in \mathbb{N}$, $s\geq s_0$,
\begin{equation}\label{tazza}
|A\circ B|^{\g,\calO}_{m+m',s,\al}\leq_{m,\al} C(s)|A|^{\g,\calO}_{m,s,\al}|B|^{\g,\calO}_{m',s_0+\al+|m|,\al}+ 
C(s_0)|A|^{\g,\calO}_{m,s_0,\al}|B|^{\g,\calO}_{m',s+\al+|m|,\al}.
\end{equation}
One has also that, for any $N\geq1$, the operator $R_{N}:={\rm Op}(r_{N})$ with $r_{N}$ defined in \eqref{comp2} satisfies
\begin{equation}\label{tazza1}
\begin{aligned}
|R_{N}|^{\g,\calO}_{m+m'-N,s,\al}&\leq_{m,N,\al} \frac{1}{N!}\left(
C(s)|A|^{\g,\calO}_{m,s,\al+N}|B|^{\g,\calO}_{m',s_0+2N+\al+|m|,\al}+\right.\\
&\left. \qquad \quad
C(s_0)|A|^{\g,\calO}_{m,s_0,\al+N}|B|^{\g,\calO}_{m',s+2N+\al+|m|,\al}
\right);
\end{aligned}
\end{equation}
\begin{equation}\label{rNdei}
\begin{aligned}
| \Delta_{1 2} R_{N} [i_1-i_2]|^{\g,\calO}_{m+m'-N,s,\al}&\leq_{m,N,\al} \frac{1}{N!}\Big(
C(s)| \Delta_{1 2} A [i_1-i_2] |^{\g,\calO}_{m,s,\al+N}|B|^{\g,\calO}_{m',s_0+2N+\al+|m|,\al}+\\
&
C(s_0)|\Delta_{1 2} A [i_1-i_2] |^{\g,\calO}_{m,s_0,\al+N}|B|^{\g,\calO}_{m',s+2N+\al+|m|,\al}\Big)\\
&+ \frac{1}{N!}\Big(
C(s)|  A  |^{\g,\calO}_{m,s,\al+N}| \Delta_{1 2} B [i_1-i_2]|^{\g,\calO}_{m',s_0+2N+\al+|m|,\al}\\
& +C(s_0)| A  |^{\g,\calO}_{m,s_0,\al+N}| \Delta_{1 2}  B [i_1-i_2] |^{\g,\calO}_{m',s+2N+\al+|m|,\al}
\Big);
\end{aligned}
\end{equation}

\noindent
$(ii)$ the adjoint operator $C^{*}:={\rm Op}(c^{*}(\la,\f,x,\x))$ in \eqref{adj} satisfies
\begin{equation}\label{tazza2}
|C^{*}|^{\g,\calO}_{m'',s,0}\leq_{m} |C|^{\g,\calO}_{m'',s+s_0+|m''|,0};
\end{equation}

\noindent
$(iii)$  consider the map $\Phi:=\mathrm{I}+D$, then
there are constants $C(s_0,\al),C(s,\al)\geq1$ such that if
\begin{equation}\label{tazza3}
C(s_0,\al)|D|^{\g,\calO}_{0,s_0+\al,\al}\leq \frac{1}{2},
\end{equation}
then, for all $\la$, the map $\Phi$ is invertible and $\Phi^{-1}\in OPS^{0}$ and for any $s\geq s_0$
one has
\begin{equation}\label{tazza4}
|\Phi^{-1}-\mathrm{I}|^{\g,\calO}_{0,s,\al}\leq C(s,\al)|D|^{\g,\calO}_{0,s+\al,\al}.
\end{equation}
\end{lem}

\begin{proof}
Item $(i)$ and $(iii)$ are proved respectively in Lemmata $2.13$ and $2.17$ of \cite{BM1}. The estimates \eqref{tazza} and \eqref{tazza1} are proved in Lemma $2. 16$ of \cite{BM1}.
The bound \eqref{rNdei} is obtained following the proof of Lemma $2.16$ of \cite{BM1} and exploiting the Leibniz rule. 
\end{proof}
\begin{remark}
When the domain of parameters $\calO$ depends on the variable $i$ then we are interested in estimating the variation $\Delta_{1 2}A:=A(i_1)-A(i_2)$ on $\calO(i_1)\cap \calO(i_2)$ instead of the derivative $\partial_i$. The bound \eqref{rNdei} holds also for $\Delta_{1 2}$ by replacing $i_1-i_2\rightsquigarrow \hat{\imath}$.
\end{remark}
\paragraph{Commutators.}
By formula \eqref{comp2}
the commutator between two pseudo differential operators\\
\noindent
 $A:={\rm Op}(a(\la,\f,x,\x))$, $B:={\rm Op}(b(\la,\f,x,\x))$ with $a\in S^{m}$ and $b\in S^{m'}$, is a pseudo differential operator
such that
\begin{equation}\label{tazza5}
[A,B]:={\rm Op}(a\star b), \quad a\star b(\la,\f,x,\x):=\big(a\# b-b\# a\big)(\la,\f,x,\x).
\end{equation}
The symbols $a\star b$ (called the Moyal parenthesis of $a$ and $b$) admits the expansion
\begin{equation}\label{tazza6}
a\star b=-\mathrm{i}\{a,b\}+\mathtt{r}_{2}(a,b), \quad \{a,b\}=\del_{\x}a\del_{x}b-\del_{x}a\del_{\x}b\in S^{m+m'-1}, 
\end{equation}
where 
\begin{equation}\label{tazza7}
\mathtt{r}_{2}(a,b)=\Big[(a\#b)-\frac{1}{\mathrm{i}}\del_{\x}a\del_{x}b\Big] - 
\Big[(b\#a)-\frac{1}{\mathrm{i}}\del_{\x}b\del_{x}a\Big] \in S^{m+m'-2}.
\end{equation}
Following Definition \ref{cancelletti} we also set
\begin{equation}\label{starcontro}
a\star_{k} b:=a\#_k b- b\#_k a, \qquad a\star_{< N} b:=\sum_{k=0}^{N-1} a\star_k b, \qquad a\star_{\geq N} b:=a\#_{\geq N} b- b\#_{\geq N} a.
\end{equation}
As a consequence, using bounds \eqref{tazza} and \eqref{tazza1}, one has
\begin{equation}\label{tazza8}
\begin{aligned}
|[A,B]|^{\g,\calO}_{m+m'-1,s,\al}&\leq_{m, m'} C(s)
|A|^{\g,\calO}_{m,s+2+|m'|+\al,\al+1}|B|^{\g,\calO}_{m',s_0+2+\al+|m|,\al+1} \\
&\qquad +C(s_0)
|A|^{\g,\calO}_{m,s_0+2+|m'|+\al+1,\al+1}|B|^{\g,\calO}_{m',s+2+\al+|m|,\al+1}.
\end{aligned}
\end{equation}
The last inequality is proved in Lemma $2.15$ of \cite{BM1}.

We now give a lemma on symbols defined on $\T^d$. 
Recalling Definition \ref{pseudoR} and \eqref{norma} we define
\begin{equation}
\lvert A w \rvert_{m, s, \alpha}:=\sup_{\xi\in\mathbb{R}^d} \max_{0\le \lvert \vec{\alpha}_1 \rvert \le \alpha} \lVert \partial_{\xi}^{\vec{\alpha}_1} A w \rVert_s \langle \xi \rangle^{-m+\lvert\vec{\alpha}_1 \rvert},
\end{equation}
we recall the notation
\[
\partial_{y}^{\vec{\alpha}} := \prod_{i=1}^{d} \partial_{y_i}^{\alpha^{(i)}}, \quad \vec{\alpha}:=(\alpha^{(1)}, \dots, \alpha^{(d)}). 
\]

\begin{lem}\label{Lemminobis}
Let $\calO$ be a subset of $\mathbb{R}^{\nu}$. Let $p=p_{\lambda}$ as in the previous lemma, let $A$ be the linear operator defined for all $w=w_{\lambda}(x, \xi)\in S^m(\mathbb{T}^d)$, $\lambda\in\calO$, as
\begin{equation}
A w=w(f(x), g(x) \xi), \quad f(x):=x+p(x), \quad g(x)=(\mathrm{I}+D p)^{-1}, \quad x\in \T^d, \xi\in\mathbb{R}^d
\end{equation}
 such that $\lVert p \rVert^{\g, \calO}_{2s_0+2}< 1$. Then $A$ is bounded, namely $A w\in S^m$, with
\begin{equation}\label{stimabis}
\lvert A w \rvert^{\g, \calO}_{m, s, \alpha}\le_{s, m, \alpha} \lvert w \rvert^{\g, \calO}_{m, s, \alpha}+\sum_{\substack{k_1+k_2+k_3=s,\\ k_1<s, k_1, k_2, k_3\geq 0,\\ k_1+k_2\geq 1}} 
\lvert w \rvert^{\g, \calO}_{m, k_1, \alpha+k_2} \lVert p \rVert^{\g, \calO}_{k_3+s_0+2}.
\end{equation}
\end{lem}	

\begin{proof}
We adopt the notation $\lvert \cdot \rvert_{W^{s, \infty}}$ instead of $\lvert \cdot \rvert_{s, \infty}$ (see 
estimate $(A.1)$ in \cite{FGMP}
) 
in order to avoid confusion with the norm of the symbols.  We also denote with $D^s_{\xi}$ the $s$-th Fr\'echet derivative with respect to $\xi$.\\
We study
\begin{equation}\label{Faa}
D_{\xi}^{\alpha}D^s w(f, g \xi)=\sum_{k=1}^s \sum_{\substack{r=0, \\ \sum (j_i+ n_i)=s}}^k C_{k r j n}\,\, (D_{\xi}^{k-r+\alpha} D^r w)[D^{j_1} f,\dots, D^{j_r} f, D^{n_1} g\,\xi,\dots, D^{n_{k-r}} g\, \xi,\underbrace{ g, \dots, g}_{\alpha\,\,\mbox{times}}]
\end{equation}
where $j:=(j_1, \dots, j_r)$, $n:=(n_1, \dots, n_{k-r})$. In the following formulas  we shall denote $\underbrace{ g, \dots, g}_{\alpha\,\,\mbox{times}}$ by $g^\al$.
For $k=1$ and $r=0$ we get from the expression \eqref{Faa} (and estimating $|g|_{L^\infty} \le 2$)
\begin{equation}\label{stima0}
\lVert  (D^{1+\alpha}_{\xi} w) [D^s g\,\xi, g^\al] \rVert_{L^2(\T^d)}\le_\al \lvert w \rvert_{m, 0, \alpha+1} \lvert D^2p\rvert_{W^{s-1, \infty}}
\end{equation}
and for $r=1$
\begin{equation}\label{stima1}
\lVert (D^{\alpha}_{\xi}D w) [D^s f,g^\al]\rVert_{L^2(\T^d)}\le_\al \lvert w \rvert_{m, 1, \alpha} \lvert D^2p \rvert_{W^{s-2, \infty}}.
\end{equation}
For $k=s$ we have that $j_i=n_i=1$ for all $i$ and we get from \eqref{Faa}
\begin{equation}\label{stima2}
\begin{aligned}
&\lVert \sum_{r=0}^s (D_{\xi}^{s-r+\al} D^r w) [\underbrace{D f, \dots, D f}_{r\,\,\mbox{times}}, \underbrace{D g\,\xi, \dots, D g\,\xi}_{s-r\,\,\mbox{times}}, g^\al] \rVert_{L^2(\T^d)}\le \sum_{r=0}^s \lvert w \rvert_{m, r, \alpha+(s-r)} \lvert f \rvert^r_{W^{1, \infty}} \lvert D^2p \rvert_{L^{\infty}}^{s-r}\\
&\le_s \sum_{\substack{s_1+p=s,\\ s_1, p \geq 0}} \lvert w \rvert_{m, s_1, \alpha+p}  \lvert D^2p \rvert_{L^{\infty}}^{p}\le_s \lvert w \rvert_{m, s, \alpha}+ \sum_{\substack{s_1+p=s,\\ s_1, p \geq 0, s_1<s}} \lvert w \rvert_{m, s_1, \alpha+p}  \lvert D^2p\rvert_{L^{\infty}}.
\end{aligned}
\end{equation}
It remains to estimate
\begin{equation}\label{Faa2}
\sum_{k=2}^{s-1} \sum_{\substack{r=0, \\ \sum (j_i+ n_i)=s}}^k C_{k r j n}\,\, (D_{\xi}^{k-r+\al} D^r w)[D^{j_1} f,\dots, D^{j_r} f, D^{n_1} g\,\xi,\dots, D^{n_{k-r}} g\, \xi,g^\al].
\end{equation}
We call $\ell\geq 1$ the number of indices $j_i$ that are $\geq 2$ and we rename these ones $\sigma_i$.
Then $\sum_i (\sigma_i+n_i)=s-(k-\ell)=s-k+\ell$. The $L^2$-norm of \eqref{Faa2} can be estimated by
\begin{equation}\label{stima3}
\begin{aligned}
&\sum_{k=2}^{s-1} \sum_{r=0}^k \sum_{\ell\geq 1} \lvert w \rvert_{m, r, \alpha+(k-r)} \lvert D f \rvert_{L^{\infty}}^{k-\ell} \lvert D^{\sigma_1} f \rvert_{L^{\infty}}\dots \lvert D^{\sigma_{\ell}} f \rvert_{L^{\infty}}\lvert D^{n_1} g \rvert_{L^{\infty}}\dots \lvert D^{n_{k-r}} g \rvert_{L^{\infty}}\\
& \le_s\sum_{k=2}^{s-1} \sum_{r=0}^k \sum_{\ell\geq 1} \lvert w \rvert_{m, r, \alpha+(k-r)} \lvert D^{\sigma_1-2} D^2 p \rvert_{L^{\infty}}\dots \lvert D^{\sigma_{\ell}-2}  D^2 p \rvert_{L^{\infty}}\lvert D^{n_1-1}  D^2 p \rvert_{L^{\infty}}\dots \lvert D^{n_{k-r}-1}  D^2 p \rvert_{L^{\infty}}\\
&\le_s \sum_{k=2}^{s-1} \sum_{r=0}^k \sum_{\ell\geq 1} \lvert w \rvert_{m, r, \alpha+(k-r)}  \lvert D^2p \rvert^{k+\ell-r-1}_{L^{\infty}} \lvert D^2p \rvert_{W^{s-2k-\ell+r, \infty}}\\
& \le_s \sum_{k=2}^{s-1} \sum_{r=0}^k  \lvert w \rvert_{m, r, \alpha+(k-r)}  \lvert D^2p \rvert_{W^{s-k-1, \infty}}\le \sum_{\substack{s_1+p+s_3=s-1, \\ s>s_1, p, s_3\geq 0}}  \lvert w \rvert_{m, s_1, \alpha+p}  \lvert D^2 p \rvert_{W^{s_3, \infty}}.
\end{aligned}
\end{equation}
Then by \eqref{stima0}, \eqref{stima1}, \eqref{stima2}, \eqref{stima3} we have \eqref{stimabis} for $\lvert A w \rvert_{m, s, \alpha}$. For the Lipschitz variation we observe that
\begin{equation}\label{ba}
\Delta_{\lambda, \lambda'}(w(\lambda, f(\lambda), g(\lambda)\xi))=A (\Delta_{\lambda, \lambda'} w)+A \,D w [\Delta_{\lambda, \lambda'} f]+A\, D_{\xi}w[\Delta_{\lambda, \lambda'} g\,\xi].
\end{equation}
One follows exactly the strategy above but considering $s-1$ derivatives instead of $s$ (recall \eqref{norma2}). This is important since in formula \eqref{ba} we have one extra derivative either in $x$ or $\xi$.
\end{proof}

\section{Pseudo differential calculus and the classes of remainders
}

\subsection{Properties of the smoothing remainders}\label{restismooth}

In the first step of our reduction procedure 
in order to prove Theorem \ref{EgorovQuantitativo} we need to work with operators which are pseudo differential up to a remainder in the class $\mathfrak{L}_{\rho}$. 
In the following we shall study properties of such operators under composition, inversion etc...

\smallskip

The following Lemma guarantees that the class of operators in Def. \ref{ellerho} 
is closed under composition.

\begin{lem}\label{chiusuracompoclasseL}
If $A$ and $B$ belong to $\mathfrak{L}_{\rho}$, for $\rho\geq3$ (see Def. \ref{ellerho}) ,
then $A\circ B\in\gotL_{\rho, p}$ and, for $s_0\le s\le \mathcal{S}$,
\begin{equation}\label{composizioneTame}
\mathbb{M}_{A\circ B}^{\gamma}(s,\mathtt{b})\le_{s, \rho} 
\sum_{\tb_1+\tb_2=\tb}\left(\mathbb{M}^{\gamma}_A(s_0,\mathtt{b}_1)\mathbb{M}^{\gamma}_B(s,\mathtt{b}_2)+\mathbb{M}^{\gamma}_A(s,\mathtt{b}_1)\mathbb{M}^{\gamma}_B(s_0,\mathtt{b}_2)\right),\;\;\;\; \mathtt{b}\leq \rho-2,
\end{equation}
\begin{equation}\label{composizioneTamedei}
\begin{aligned}
\mathbb{M}_{\Delta_{12} (A\circ B)  }(p,\tb)&
\le_{p, \rho} \sum_{\tb_1+\tb_2=\tb}\Big(\mathbb{M}_{\Delta_{12} A}(p,\mathtt{b}_1)\mathbb{M}_B(p,\mathtt{b}_2)+\mathbb{M}_A(p,\mathtt{b}_1)\mathbb{M}_{\Delta_{12} B  }(p,\mathtt{b}_2)\Big)\,,\;\;\; \mathtt{b}\leq \rho-3.
\end{aligned}
\end{equation}
\end{lem}

\begin{proof}
We start by noting that $\gotM^{\gamma}_{A\circ B}(-\rho,s)$ defined in 
\eqref{megaTame2} with $A\rightsquigarrow A\circ{B}$
is controlled by the r.h.s. of \eqref{composizioneTame}.
Let $m_{1},m_{2}\in \mathbb{R}$, $m_{1},m_{2}\geq0$ and $m_{1}+m_{2}=\rho$.
We can write
\[
\langle D_{x}\rangle^{m_1}A\circ B\langle D_{x}\rangle^{m_{2}}=
\langle D_{x}\rangle^{m_1}A  
\langle D_{x}\rangle^{m_2}\langle D_{x}\rangle^{{-\rho}}
 \langle D_{x}\rangle^{m_1}B\langle D_{x}\rangle^{m_2}.
\]
By hypothesis we know that $A$ belongs to the class $\calL_{\rho}$, hence  by  $(i)$ of Definition \ref{ellerho} one has that 
$\langle D_{x}\rangle^{m_1}A  
\langle D_{x}\rangle^{m_2}$ is a $0-$tame operator. 
For the same reason also $\langle D_{x}\rangle^{m_1}B
\langle D_{x}\rangle^{m_2}$ is  a $0-$tame operator.
Note also that, since $\rho\geq0$, then $\langle D_{x}\rangle^{{-\rho}}: H^{s}(\mathbb{T}^{\nu+1})\to H^{s}(\mathbb{T}^{\nu+1})$ is  a $0-$tame operator. 
Hence,  using Lemma \ref{lem: 2.3.6}  for any $u\in H^{s}$ one has
\begin{equation}\label{Extron}
\begin{aligned}
\|\langle D_{x}\rangle^{m_1}A\circ B\langle D_{x}\rangle^{m_2}u\|_{s}&\leq_{s}
(\gotM_{A}(-\rho,s)\gotM_{B}(-\rho,s_0)+\gotM_{A}(-\rho,s_0)\gotM_{B}(-\rho,s))
\|u\|_{s_0}\\
&+\gotM_{A}(-\rho,s_0)\gotM_{B}(-\rho,s_0)\|u\|_{s},
\end{aligned}
\end{equation}
where $\gotM_{A}(-\rho,s)$, $\gotM_{B}(-\rho,s)$   are defined in \eqref{megaTame2}.
Then we may set
\[
\gotM_{A\circ B}(-\rho,s)= C(s)\Big(\gotM_{A}(-\rho,s)\gotM_{B}(-\rho,s_0)+\gotM_{A}(-\rho,s_0)\gotM_{B}(-\rho,s)\Big).
\]
Reasoning as in \eqref{Extron}
one can check that 
\[
\gotM^{\gamma}_{A\circ B}(-\rho,s)\leq C(s)\Big(\gotM^{\gamma}_{A}(-\rho,s)\gotM^{\gamma}_{B}(-\rho,s_0)+\gotM^{\gamma}_{A}(-\rho,s_0)\gotM^{\gamma}_{B}(-\rho,s)\Big).
\]

\noindent
Let us study the operator $\del_{\f}^{\vec{\mathtt{b}}}(A\circ{B})$ for $\vec{\tb}\in\mathbb{N}^{\nu}$ and $|\vec{\tb}|\leq \rho-2$. We have
\begin{equation}\label{Extron0}
\del_{\f}^{\vec{\mathtt{b}}}(A\circ{B})=\sum_{\vec{\mathtt{b}_{1}}+\vec{\mathtt{b}_2}=\vec{\mathtt{b}}}
(\del_{\f}^{\vec{\mathtt{b}_1}}A)(\del^{\vec{\mathtt{b}_2}}_{\f}B).
\end{equation}
We show that any summand in \eqref{Extron0}
satisfies item $(i)$ of Def. \eqref{ellerho}.
Let $m_{1},m_{2}\in \mathbb{R}$, $m_{1},m_{2}\geq0$ and $m_{1}+m_{2}=\rho-|\vec\tb|$.
We write
\[
\begin{aligned}
 \langle D_{x}\rangle^{m_1} (\del_{\f}^{\vec{\mathtt{b}_1}}A)(\del^{\vec{\mathtt{b}_2}}_{\f}B)
  \langle D_{x}\rangle^{m_2}&=
 \langle D_{x}\rangle^{m_1}
(\del_{\f}^{\vec{\mathtt{b}_1}}A) 
  \langle D_{x}\rangle^{y}
 \langle D_{x}\rangle^{-y-w}  
  \langle D_{x}\rangle^{w}(\del^{\vec{\mathtt{b}_2}}_{\f}B)
 \langle D_{x}\rangle^{m_2}
\end{aligned}
\]
with $y:=\rho-|\vec{\mathtt{b}_1}|-m_1$, $w=\rho-|\vec{\mathtt{b}_2}|-m_2$ and note that
$-y-w=-\rho\leq0$. Moreover 
$m_1+y=\rho-|\vec{\mathtt{b}_1}|$, and $w+m_{2}=\rho-|\vec{\mathtt{b}_2}|$. Hence
the operators 
$ \langle D_{x}\rangle^{m_1}(\del_{\f}^{\vec{\mathtt{b}_1}}A) 
  \langle D_{x}\rangle^{y}$ 
  and 
  $ \langle D_{x}\rangle^{w}(\del_{\f}^{\vec{\mathtt{b}_2}}b) 
  \langle D_{x}\rangle^{m_2}$
  are Lip-$0$-tame operator. 
Hence, using Lemma \ref{lem: 2.3.6}
one has
\begin{align}\label{scoccio}
\| \langle D_{x}\rangle^{m_1} (\del_{\f}^{\vec{\mathtt{b}_1}}A)
(\del^{\vec{\mathtt{b}_2}}_{\f}B)
  \langle D_{x}\rangle^{m_2}\|^{\gamma,\calO}_{s}&\leq
 \gotM^{\gamma}_{\del_{\f}^{\vec{\mathtt{b}_1}}A}(-\rho+|\vec{\mathtt{b}_1}|,s)
 \gotM^{\gamma}_{\del_{\f}^{\vec{\mathtt{b}_2}}B}(-\rho+|\vec{\mathtt{b_{2}}}|,s_0)\|u\|_{s_0}\\\notag
 &+
 \gotM^{\gamma}_{A}(-\rho+|\vec{\mathtt{b}_1}|,s_0)
 \gotM^{\gamma}_{B}(-\rho+|\vec{\mathtt{b}_2}|,s)
\|u\|_{s_0}\\ \notag
&+\gotM^{\gamma}_{A}(-\rho+|\vec{\mathtt{b}_1}|,s_0)
\gotM^{\gamma}_{B}(-\rho+|\vec{\mathtt{b}_2}|,s_0)\|u\|_{s}, 
  \end{align}
for $u\in H^{s}$. We can conclude that $\gotM^{\gamma}_{\del_{\f}^{\vec{\mathtt{b}}}(A\circ{B})}(-\rho+|\vec{\tb}|,s)$
is controlled by the r.h.s. of \eqref{composizioneTame}.
Regarding the operator $[A\circ B,\del_{x}]$ we reason as follows.
We prove that
\begin{equation}\label{Extron3}
[A\circ B,\del_{x}]=A[B,\del_{x}]+[A,\del_{x}]B.
\end{equation}
satisfies item $(ii)$ of Definition \eqref{ellerho}.
Let $m_{1},m_{2}\in \mathbb{R}$, $m_{1},m_{2}\geq0$ and $m_{1}+m_{2}=\rho-1$.
Moreover
\[
\langle D_{x}\rangle^{m_1}[A,\del_{x}]B\langle D_{x}\rangle^{m_2}
=\langle D_{x}\rangle^{m_1}[A,\del_{x}]\langle D_{x}\rangle^{y}
\langle D_{x}\rangle^{{-y-z}}
\langle D_{x}\rangle^{z}B
\langle D_{x}\rangle^{m_2},
\]
with $y=\rho-1-m_1$, $z=\rho-m_2$. Hence by definition 
(see Def. \eqref{ellerho}) we have that 
$\langle D_{x}\rangle^{m_1}[A,\del_{x}]\langle D_{x}\rangle^{y}$
and $\langle D_{x}\rangle^{z}B
\langle D_{x}\rangle^{m_2}$
are Lip-$0-$tame.
Thus one can conclude, as done above, that $\gotM_{[A,\del_{x}]B}(-\rho+1,s)$
is controlled by the r.h.s. of \eqref{composizioneTame}.
One can reason in the same way for the first summand in \eqref{Extron3}
and for the operator $[\del_{\f}^{\vec{\mathtt{b}}}(AB),\del_{x}]$.
This proves \eqref{composizioneTame}.

\noindent
Let us study the term
\begin{equation}\label{Extron4}
\Delta_{12}(A\circ{B}) =(\Delta_{12}A)\, B(\mathfrak{I}_2)+
A(\mathfrak{I}_1)\,(\Delta_{12}B) .
\end{equation}
By definition both $\langle D_{x}\rangle^{m_1}\Delta_{12}A \langle D_{x}\rangle^{m_2}, \langle D_{x}\rangle^{m_1}\Delta_{12}B \langle D_{x}\rangle^{m_2}$ with $m_1+m_2=\rho-1$ 
are bounded operators on $H^{s}$
(see \eqref{Mdrittaconlai} and Def. \ref{TameConstants}).
In order to prove \eqref{composizioneTamedei} one can bound 
the two summand in \eqref{Extron4} by following the same procedure 
used to prove \eqref{composizioneTame}.
\end{proof}

The next Lemma shows that, if $\rho\geq 3$, $OPS^{-\rho}\subset \gotL_{\rho, p}$ (see Section \ref{sezione functional setting} for the definition of $OPS^m$).
\begin{lem}\label{INCLUSIONEpseudoInclasseL}
Fix $\rho\geq 3$ and consider a symbol $a=a(\oo, \mathfrak{I}(\omega))$ 
in $S^{-\rho}$
depending on   $\oo\in \calO\subset \mathbb{R}^{\nu}$ and on $\mathfrak{I}$  
in a Lipschitz way.
One has that $A:={\rm op}(a(\f,x,\x))\in \gotL_{\rho, p}$ (see \ref{ellerho})
and 
\begin{equation}\label{costsimbolo}
\begin{aligned}
\mathbb{M}^{\gamma}_{A}(s,\tb)&\leq_{s,\rho} |a|^{\gamma,\calO}_{-\rho,s+\rho,0},\qquad \;\;\; 
\mathbb{M}_{\Delta_{12}A }(p,\tb)\leq_{p,\rho} 
|\Delta_{12}a |_{-\rho,p+\rho,0}.
\end{aligned}
\end{equation}
\end{lem}

\begin{proof}
Let $m_{1},m_{2}\in \mathbb{R}$, $m_{1},m_{2}\geq0$ and $m_{1}+m_{2}=\rho$.
We need to show that 
$\langle D_{x}\rangle^{m_1}A\langle D_{x}\rangle^{m_2}$
satisfies item $(i)$ of Definition \ref{ellerho}. By definition it is the composition of three pseudo differential operators hence,
by Lemma \ref{actiononsobolev} and by formula \eqref{tazza} of Lemma \ref{PROP} one has that
\begin{align}
\gotM^{\gamma}_{\langle D_{x}\rangle^{m_1}A
\langle D_{x}\rangle^{m_2}}(0, s) &\le_s  |\langle D_{x}\rangle^{m_1}A
\langle D_{x}\rangle^{m_2}|^{\gamma,\calO}_{0,s,0} \le_s
 |\langle D_{x}\rangle^{m_1}|_{m_1,s,0} |a|^{\gamma,\calO}_{-\rho,s+|m_1|, 0} 
 |\langle D_{x}\rangle^{m_2}|_{m_2,s+|m_1|+\rho,0}
 \nonumber
 \\ &\le_s |a|^{\gamma,\calO}_{-\rho,s+|m_1|, 0 }\label{Extron5}
\end{align}
This means that
\[
\gotM_{A}^{\gamma}(-\rho,s)
\leq_{s}|a|^{\gamma,\calO}_{-\rho,s+\rho,0}. 
\]
Secondly we consider the operator $(\del_{\f}^{\vec{\mathtt{b}}}{\rm op}(a(\f,x,\x)))=
{\rm op}(\del_{\f}^{\vec{\mathtt{b}}}a(\f,x,\x))$
for $\vec{\tb}\in\mathbb{N}^{\nu}$ and $|\vec{\tb}|\leq \rho-2$.
It is pseudo differential and its symbol $\del_{\f}^{\vec{\mathtt{b}}}a(\f,x,\x)$
is such that
\[
|\del_{\f}^{\vec{\mathtt{b}}}a|^{\gamma,\calO}_{-\rho,s,\al}\leq |a|^{\gamma,\calO}_{-\rho,s+|\vec{\tb}|,\al}.
\]
Following the same reasoning used in \eqref{Extron5} (recall that $m_1+m_2=\rho-|\vec{\tb}|$) one obtains
\[
\gotM_{\del_{\f}^{\vec{\mathtt{b}}}A}^{\gamma}(-\rho+|\vec{\tb}|,s)\leq_{s} |a|^{\gamma,\calO}_{-\rho,s+|\vec{\tb}|+
	(\rho-|\vec{\tb}|),0}=C(s)\,|a|^{\gamma,\calO}_{-\rho,s+\rho,0}.
\]
The operator $[A,\del_x]= A\del_x-\del_x A$  can be treated in the same way, discussing each of the two summands separately, 
 (we are not taking advantage of the pseudo dfferential structure in order to control the order of the commutator), with $m_1+m_2=\rho-1$,
\begin{equation*}
\gotM^{\gamma}_{
	\langle D_{x}\rangle^{m_1}\del_x A
	\langle D_{x}\rangle^{m_2}}(0, s) \le_s  |\langle D_{x}\rangle^{m_1}\del_x A
\langle D_{x}\rangle^{m_2}|^{\gamma,\calO}_{0,s,0} \le_s  |a|^{\gamma,\calO}_{-\rho,s+\rho, 0}.
\end{equation*}
 The same strategy holds for $[\del_{\f}^{\vec{\mathtt{b}}}A,\del_{x}]$
Hence one gets the first of \eqref{costsimbolo}.
The second bound in \eqref{costsimbolo}
can be obtained by noting that $\Delta_{12}A ={\rm op}(\Delta_{12}a )[\cdot]$
and then following almost word by word the discussion above.
\end{proof}
The next Lemma shows that $\gotL_{\rho, p}$ is closed under left and right multiplication by  operators in $S^0$.
\begin{lem}\label{idealeds}
	Let $a\in S^{0}$ and $B\in \gotL_{\rho, p}$, then $\op(a)\circ B, B\circ \op(a)\in \gotL_{\rho, p}$ and satisfy the bounds
	\begin{align}\label{idealerho}
\mathbb{M}_{\op(a)\circ B}^\g(s,\tb) & 
\le_{s, \rho} |a|^{\g,\calO}_{0,s+\rho,0} \mathbb{M}_{ B}^\g(s_0,\tb) 
+ |a|^{\g,\calO}_{0,s_0+\rho,0} \mathbb{M}_{ B}^\g(s,\tb)\\ \notag
\mathbb{M}_{\Delta_{12} (\op(a)\circ B)  }(p,\tb)&
\le_{p, \rho} 
|\Delta_{12}a|_{1,p+\rho,0}\mathbb{M}_B(p,\mathtt{b})
+|a|_{0,p+\rho,0}\mathbb{M}_{\Delta_{12} B  }(p,\mathtt{b})\,,\
\end{align}
for all $s_0\le s\le \mathcal{S}$.
 Moreover if $B\in \gotL_{\rho+1}$  then $ 	\del_{\f_m} B, [\del_x,B]$, $m=1,\dots,\nu$, are in $\gotL_{\rho, p}$
and satisfy the bounds
 \begin{equation}\label{miserialadra}
  \begin{aligned} &	\mathbb M_{\del_{\f_m} B}^\g (s,\tb),\mathbb M_{[\del_{x}, B]}^\g (s,\tb) \le 
  \mathbb M_{ B}^\g (s,\tb+1) \,,\quad \tb \le \rho-2\\  
	&	\mathbb M_{\del_{\f_m}\Delta_{12} B} (p,\tb),\mathbb M_{[\del_{x}, \Delta_{12}B]} (p,\tb) 
\le \mathbb M_{\Delta_{12} B} (p,\tb+1) \,,\quad \tb \le \rho-3
 \end{aligned}
  \end{equation}
 for all $s_0\le s\le \mathcal{S}$. Note that in \eqref{miserialadra} the constants 
 in the right hand side control the tameness constants of $B$ as an element of $\gotL_{\rho+1}$.
\end{lem}
\begin{proof}
We start by studying the Lip-$0$-tame norm of
\[
\langle D_{x}\rangle^{m_1}\del_\f^{\vec{\tb}_1} \op(a) \circ \del_\f^{\vec{\tb}_2} B 
\langle D_{x}\rangle^{m_2} =  
\langle D_{x}\rangle^{m_1}\del_\f^{\vec{\tb}_1} \op(a)\langle D_{x}\rangle^{-m_1} \circ \langle D_{x}\rangle^{m_1}
\del_\f^{\vec{\tb}_2} B \langle D_{x}\rangle^{m_2},
\]
with $\lvert \vec{\tb}_1 \rvert+\lvert \vec{\tb}_2 \rvert=\lvert \vec{\tb} \rvert$ and $m_1+m_2=\rho-\lvert \vec{\tb} \rvert$.
By Lemma \ref{actiononsobolev}  and formula \eqref{tazza}
\[
\gotM^\g_{ \langle D_{x}\rangle^{m_1}\del_\f^{\vec{\tb}_1} \op(a)\langle D_{x}\rangle^{-m_1}}(0,s)\le_s 
|a|^{\g, \calO}_{0,s+|\vec{\tb}_1|+m_1,0}\le_s |a|^{\g, \calO}_{0,s+\rho,0}\,
\]
 hence by Lemma \ref{lem: 2.3.6} we have 
\[
\gotM^\g_{\langle D_{x}\rangle^{m_1}\del_\f^{\vec{\tb}} (\op(a)  B)\langle D_{x}\rangle^{m_2}}(-\rho+|\vec{\tb}|,s) \le_{s, \rho} 
|a|^{\g, \calO}_{0,s+\rho,0} \mathbb M^\g_B(s_0,\tb) + |a|^{\g, \calO}_{0,s_0+\rho,0} \mathbb M^\g_B(s,\tb)\,.
\]	
Regarding 
	\[
\langle D_{x}\rangle^{m_1}\del_\f^{\vec{\tb}} [\del_x,\op(a)  B]\langle D_{x}\rangle^{m_2}=\langle D_{x}\rangle^{m_1}\del_\f^{\vec{\tb}} ([\del_x,\op(a)]  B)\langle D_{x}\rangle^{m_2} + \langle D_{x}\rangle^{m_1}\del_\f^{\vec{\tb}} (\op(a) [\del_x, B])\langle D_{x}\rangle^{m_2}
\]
we only need to consider the first summand as the second can be discussed exactly as above.
Recalling that by definition $m_1+m_2= \rho-|\vec{\tb}|-1$ we write for $\lvert \vec{\tb}_1 \rvert+\lvert \vec{\tb}_2 \rvert=\lvert \vec{\tb} \rvert$ and $m_1+m_2=\rho-\lvert \vec{\tb} \rvert$
\[
\langle D_{x}\rangle^{m_1}\del_\f^{\vec{\tb}_1} [\del_x,\op(a)] \del_\f^{\vec{\tb}_2} B \langle D_{x}\rangle^{m_2} =
\langle D_{x}\rangle^{m_1}\del_\f^{\vec{\tb}_1} [\del_x,\op(a)]\langle D_{x}\rangle^{-m_1-1}\langle D_{x}\rangle^{m_1+1} \del_\f^{\vec{\tb}_2} B \langle D_{x}\rangle^{m_2}
\]
and the result follows by recalling that
\[
\gotM^\g_{ \langle D_{x}\rangle^{m_1}\del_\f^{\vec{\tb}_1} [\del_x,\op(a)]\langle D_{x}\rangle^{-m_1-1}}(0,s)\le_s 
|a|^{\g, \calO}_{0,s+|\vec{\tb}_1|+m_1,0}\le_s |a|^{\g, \calO}_{0,s+\rho,0}\,.
\]
The bounds \eqref{miserialadra} follows by the fact that
$\partial_{\varphi}^{\vec{\tb}}\partial_{\varphi_m}=\partial_{\varphi}^{\vec{\tb}_0}$ with $\lvert \vec{\tb}_0 \rvert=\lvert \vec{\tb} \rvert+1$
and  $\mathbb{M}_A^{\g}(s, \tb)\le \mathbb{M}_A^{\g}(s, \tb+1)$ if $A\in\gotL_{\rho+1}$.
\end{proof}

The next Lemma gives a canonical way to write the composition of two pseudo differential operators as a pseudo differential operator plus a remainder in $\gotL_{\rho, p}$. Of course Lemma \ref{PROP} says that such a composition is itself a pseudo differential operator, so in principle one could take the remainder to be zero. The purpose of this Lemma is to get better bounds with respect to \eqref{tazza}, the price to pay is that we do not control the symbol of the composition but only an approximation up to a smoothing remainder of order $-\rho$.

\begin{lem}[{\bf Composition}]\label{James}
Let $a=a(\omega)\in S^m$, $b=b(\omega)\in S^{m'}$ be defined on some subset $\calO\subset \mathbb{R}^{\nu}$ with $m, m'\in\mathbb{R}$ and consider any  $\rho \geq \max\{-(m+m'+1), 3\}$.  Assume also that $a$ and $b$ depend in a Lipschitz way 
on the parameter $\mathfrak{I}$.
There exist an operator $R_{\rho}\in \gotL_{\rho, p}$ such that (recall Definition \eqref{cancelletti}) setting $N=m+m'+\rho\ge 1$
\begin{equation*}
\op(a\# b)=\op(c)+R_{\rho}, \qquad c:=a\#_{<N} b\in S^{m+m'}
\end{equation*}
where
\begin{equation}\label{crawford}
\begin{aligned}
\lvert c \rvert^{\g, \calO}_{m+m', s, \alpha}\le_{s, \rho, \alpha, m, m'} & \lvert a \rvert^{\g, \calO}_{m, s,N-1+\alpha}\lvert b \rvert^{\g, \calO}_{m', s_0+N-1, \alpha}
+\lvert a \rvert^{\g, \calO}_{m, s_0, N-1+\alpha}\lvert b \rvert^{\g, \calO}_{m', s+N-1, \alpha},
\end{aligned}
\end{equation}
\begin{equation}\label{jamal}
\begin{aligned}
\mathbb{M}^{\g}_{R_{\rho}}(s, \tb)
&\le_{s, \rho,  m , m'} \lvert a \rvert^{\g, \calO}_{m,s+\rho, N}\lvert b \rvert^{\g, \calO}_{m', s_0+2 N+\lvert m \rvert, 0}+\lvert a \rvert^{\g, \calO}_{m, s_0, N}\lvert b \rvert^{\g, \calO}_{m', s+\rho+2N+\lvert m \rvert, 0}.
\end{aligned}
\end{equation}
for all $0\le \tb\le \rho-2$ and $s_0\le s\le \mathcal{S}$.
Moreover one has
\begin{equation}\label{crawford2}
\begin{aligned}
\lvert \Delta_{12} c   \rvert_{m+m', p, \alpha} &\le_{p, \alpha, \rho, m, m'} 
\lvert \Delta_{12} a   \rvert_{m, p, N-1+\alpha}\lvert b \rvert_{m', p+N-1, \alpha}\\
&+\lvert a \rvert_{m, p, N-1+\alpha}\lvert \Delta_{12} b   \rvert_{m', p+N-1, \alpha}
\end{aligned}
\end{equation}
\begin{equation}\label{jamal2}
\begin{aligned}
\mathbb{M}_{\Delta_{12} R_{\rho} } (p, \tb) &\le_{p, \rho, m, m'} 
\lvert \Delta_{12} a   \rvert_{m+1, p+\rho, N}\lvert b \rvert_{m', p+2 N+ \lvert m \rvert, 0}\\
&+\lvert a \rvert_{m, p+\rho, N}\lvert \Delta_{12} b   \rvert_{m'+1, p+2 N+ \lvert m \rvert, 0}.
\end{aligned}
\end{equation}
for all $0\le \tb \le \rho-3$ and where $p$ is the constant given in Definition \ref{ellerho}.
\end{lem}

\begin{proof}
To shorten the notation we write $\lVert \cdot \rVert_s:=\lVert \cdot \rVert_s^{\g, \calO}$. For $\beta\in\mathbb{R}$,
using formula \eqref{comp2} and by the tameness of the product, we have
\begin{equation*}
\lVert \partial_{\xi}^{\beta} c \rVert_s\le_s \sum_{k=0}^{N-1} \frac{1}{k!} \sum_{\beta_1+\beta_2=\beta} C_{\beta_1 \beta_2} \big( \lVert \partial_{\xi}^{\beta_1+k} a \rVert_{s}\,\lVert \partial_{\xi}^{\beta_2} \partial_x^k b \rVert_{s_0}
+\lVert \partial_{\xi}^{\beta_1+k} a \rVert_{s_0}\,\lVert \partial_{\xi}^{\beta_2} \partial_x^k b \rVert_{s}\big).
\end{equation*}
Thus, recalling \eqref{norma},
one gets
\begin{align*}
\lvert c \rvert_{m+m', s, \alpha} \le_s 
\le_{s, \alpha} & \sum_{k=0}^{N-1} \frac{1}{k!}   \,\sup_{\xi\in\mathbb{R}} \big( \max_{0\le \beta_1 \le \alpha} \lVert \partial_{\xi}^{\beta_1+k} a \rVert_{s} \langle \xi \rangle^{-m+\beta_1}\,\max_{0\le \beta_2 \le \alpha} \lVert \partial_{\xi}^{\beta_2} \partial_x^k b \rVert_{s_0}\langle \xi \rangle^{-m'+\beta_2}\\
&+\max_{0\le \beta_1 \le \alpha} \lVert \partial_{\xi}^{\beta_1+k} a \rVert_{s_0} \langle \xi \rangle^{-m+\beta_1}\,\max_{0\le \beta_2 \le \alpha} \lVert \partial_{\xi}^{\beta_2} \partial_x^k b \rVert_{s} \langle \xi \rangle^{-m'+\beta_2}\big),
\end{align*}
which implies \eqref{crawford}.
In the same way we obtain the bound \eqref{crawford2} by using the following fact
\[
\Delta_{12} (\partial_{\xi}^k a\,\partial_x^k b) =\partial_{\xi}^k (\Delta_{12} a  )\,\partial_x^k b+\partial_{\xi}^k a\,\partial_x^k (\Delta_{12} b  ).
\]
We remark that $R_{\rho}$ is the pseudo differential operator $R_N$ considered in Lemma \ref{PROP} (recall  $N=m+m'+\rho$). 
By Lemma \ref{INCLUSIONEpseudoInclasseL}
\[
\mathbb{M}^\g_{R_\rho} \stackrel{\eqref{costsimbolo}}{\le}_{s,\rho,m,m'} |R_\rho|_{-\rho,s+\rho,0}
\]
then  by formula \eqref{tazza1} of Lemma \ref{PROP} we get the bounds \eqref{jamal}.
The bounds  \eqref{jamal2}, follow in the same way.
\end{proof}
\begin{remark}
Note that if $m+m'\le -\rho\le - 3$ then by Lemma \ref{INCLUSIONEpseudoInclasseL} $\op(a)\circ\op(b)\in \gotL_{\rho, p}$.
\end{remark}

\begin{lem}\label{casabianca}
Fix $\rho\geq 3$ and $n\in\mathbb{N}$, $n<\rho$. Let $a\in S^{-1}$ depending in a  Lipschitz way on a parameter $i$. Then there exist a symbol $c^{(n)}\in S^{-n}$ and a operator $R_{\rho}^{(n)}\in\gotL_{\rho, p}$ such that
\begin{equation}
\op(a)^n=\op(c^{(n)})+R_{\rho}^{(n)}. 
\end{equation}
Moreover the following bounds hold
\begin{align}
\lvert c^{(n)} \rvert^{\g, \calO}_{-n, s, \alpha} &\le_{n, s, \alpha, \rho} \lvert a \rvert^{\g, \calO}_{-1, s+(n-1)(\rho-3), \alpha+\rho-3}\big(\lvert a \rvert^{\g, \calO}_{-1, s_0+(n-1)(\rho-3), \alpha+\rho-3}\big)^{n-1},\label{kessie}\\
\lvert \Delta_{12} c^{(n)}   \rvert_{-n, p, \alpha, \rho} &\le  
\lvert \Delta_{12} a   \rvert_{-1, p+(n-1)(\rho-3), \alpha+\rho-3}
\lvert a \rvert^{n-1}_{-1, p+(n-1)(\rho-3), \alpha+\rho-3} \label{jacob}\\
\mathbb{M}^{\g}_{ R_{\rho}^{(n)}}(s, \tb)&\le_{s, \rho, \tb, n} \lvert a \rvert^{\g, \calO}_{-1, s+n (\rho-3)+\rho, \rho-2} \big(\lvert a \rvert^{\g,\calO}_{-1, s_0+n (\rho-3)+\rho, \rho-2}\big)^{n-1}\label{kessie3}\\
\mathbb{M}_{\Delta_{12}R^{(n)}_{\rho} }(p,\tb) & \le_{p, n, \tb}  
\lvert \Delta_{12} a   \rvert_{0, p+n(\rho-3)+\rho, \rho-2}
\big(\lvert a \rvert_{-1, p+n(\rho-3)+\rho, \rho-2}\big)^{n-1}\label{kessie4}
\end{align}
for all  $s_0\le s\le \mathcal{S}$ and where $p$ is the constant given in Definition \ref{ellerho}.

\end{lem}

\begin{proof}
We define $c^{(1)}:=a\in S^{-1},$ and, for $n\geq 2$,
\[
c^{(n)}:=a\#_{<\rho-2} c^{(n-1)}, \qquad
\qquad
\;\;\; R_{\rho}^{(n)}:=\sum_{k=0}^{n-2} [\op(a)]^{k} \op(a\#_{\geq \rho-2} c^{(n-k-1)})
\]
By using Lemma \ref{James} 
we have that \eqref{kessie} is satisfied for $n=2$.
Now given \eqref{kessie} for $n$ we prove it for $n+1$. For simplicity we write $\le_{n, s, \alpha}=\le$. We have
\begin{align*}
\lvert a\#_{<\rho-2} c^{(n)} \rvert^{\g, \calO}_{-n-1, s, \alpha} &\le \lvert a \rvert^{\g, \calO}_{-1, s, \alpha+\rho-3} \lvert c^{(n)} \rvert^{\g, \calO}_{- n, s_0+\rho-3, \alpha}+\lvert a \rvert^{\g, \calO}_{-1, s_0, \alpha+\rho-3}\lvert a^{(n)} \rvert^{\g, \calO}_{-n, s+\rho-3, \alpha}
\\
&\le \lvert a \rvert^{\g, \calO}_{-1, s+n(\rho-3), \alpha+\rho-3}\big(\lvert a \rvert^{\g, \calO}_{-n, s_0+n (\rho-3), \alpha+\rho-3}\big)^{n},  
\end{align*}
hence \eqref{kessie}  is proved. Arguing as above one can prove \eqref{jacob}.

\noindent
Now fix $2\le k\in\mathbb{N}$ and define
$r_k:= a\#_{\geq \rho-2} c^{(k-1)}\in S^{-\rho}$.
We apply repeatedly  \eqref{idealerho} in oder to get
\begin{align*}
\mathbb{M}^{\g}_{R_{k}}(s,\tb)  \le_{s, \rho, \tb} (|a|^{\g, \calO}_{-1,s_0+\rho,0})^{k-1}  \left(  |a|^{\g, \calO}_{-1,s+\rho,0} \mathbb{M}^{\g}_{\op(r_{n-k})}(s_0,\tb) +  |a|^{\g, \calO}_{-1,s_0+\rho,0} \mathbb{M}^{\g}_{\op(r_{n-k})}(s,\tb)\right),
\end{align*}
with $R_{k}:=(\op(a)^k \op(r_{n-k}))$.
 Now by Lemma \ref{INCLUSIONEpseudoInclasseL} we have that for all $k\ge 2$
\[
\mathbb{M}^{\g}_{\op(r_{k})}(s, \tb)\le_{s, \rho, \tb}  \lvert r_{k} \rvert^{\g, \calO}_{-\rho, s+\rho, 0}
\]
Now by \eqref{tazza1} with $m=-1,m'=-k+1,N=\rho-2$ we have
\begin{align*}
\lvert r_{k} \rvert^{\g, \calO}_{-\rho, s, 0} &\le \lvert r_{k} \rvert^{\g, \calO}_{-\rho-k+2, s, 0}
\stackrel{\eqref{kessie}}{ \le} 
 |a|^{\g, \calO}_{-1,s+ k(\rho-3),\rho-2}(|a|^{\g, \calO}_{-1,s_0+ k(\rho-3),\rho-2})^{k-1}
\end{align*}
Then
\[
\mathbb{M}^{\g}_{R_{\rho}^{(n)}}(s, \tb)\le_{s, \rho, \tb}  \lvert a \rvert^{\g, \calO}_{-1, s+n(\rho-3)+\rho, \rho-2}\big( \lvert a \rvert^{\g, \calO}_{-1, s_0+n(\rho-3)+\rho, \rho-2}\big)^{n-1}.
\]
We follow the same strategy in order to study the operator
\[
\Delta_{12} \big(\op(a)^k R_{\rho^{(n-k)}}\big)  =k \op(a)^{k-1} \op(\Delta_{12} a  ) R_{\rho^{(n-k)}}+\op(a)^k \Delta_{12} R_{\rho^{(n-k)}}  
\]
and we get \eqref{kessie4}.
\end{proof}
\begin{remark}
	Note that if $n\ge \rho \geq 3$ and $a\in S^{-1}$ then $\op(a)^n\in \gotL_{\rho, p}$, by Lemma \ref{INCLUSIONEpseudoInclasseL}. 
\end{remark}
\begin{coro}\label{InvertibilityUtile}
Let $a\in S^{-1}$ and consider $\mathrm{I}-(\op(a)+T)$, where $T\in\gotL_{\rho, p}=\gotL_{\rho, p}$ (recall Def. \ref{ellerho}) with $\rho\geq 3$. There exist a constant $C(\mathcal{S}, \alpha, \rho)$ such that if
\begin{equation}\label{piccoloperNeumann}
C(\mathcal{S}, \alpha, \rho) \Big(\lvert a \rvert^{\g, \calO}_{-1, p+(\rho-1)(\rho-2)+3, \rho-2}+\mathbb{M}^{\g}_T(s_0, \tb)\Big)<1,
\end{equation}
where $\mathcal{S}$ is a fixed constant appearing in the Def. \ref{ellerho},
then $\mathrm{I}-(\op(a)+T)$ is invertible and
\begin{equation}\label{finaleUff}
(\mathrm{I}-(\op(a)+T))^{-1}=\mathrm{I}+\op(c)+R_{\rho}
\end{equation}
where
\begin{equation}\label{kepalle}
\lvert c \rvert^{\g, \calO}_{-1, s, \alpha} \le_{s, \alpha, \rho} \lvert a \rvert^{\g, \calO}_{-1, s+(\rho-2)(\rho-3), \alpha+\rho-3},
\quad \lvert \Delta_{12} c   \rvert_{-1, s, \alpha} \le  \lvert \Delta_{12} a   \rvert_{-1, p+(\rho-2)(\rho-3), \alpha+\rho-3}
\end{equation}
and $R_{\rho}\in \gotL_{\rho, p}$ with
\begin{equation}\label{casabianca0}
\mathbb{M}^{\g}_{R_{\rho}}(s, \tb)\le \lvert a \rvert^{\g, \calO}_{-1, s+(\rho-1)(\rho-2)+3, \rho-2} 
+ \mathbb{M}^{\g}_{T}(s, \tb),\;\;\;\;
0\le\tb\le \rho-2,
\end{equation}
\begin{equation}\label{kepalle100}
\mathbb{M}_{\Delta_{12} R_{\rho}  }(p, \tb)\le \lvert \Delta_{12} a   \rvert_{-1, p+(\rho-1)(\rho-2)+3, \rho-2}
+\mathbb{M}_{\Delta_{12} T }(p,\tb)\,,\;\;\;0\le\tb\le \rho-3,
\end{equation}
for all  $s_0\le s\le \mathcal{S}$.
\end{coro}
\begin{proof}
To shorten the notation we write $\lvert \cdot \rvert_{m, s, \alpha}^{\g, \calO}=\lvert \cdot \rvert_{m, s, \alpha}$.
We have by \eqref{piccoloperNeumann} and Neumann series
\[
\begin{aligned}
(\mathrm{I}-(\op(a)+T))^{-1} &=\mathrm{I}+\sum_{n\geq 1}  (\op(a)+T)^n=\mathrm{I}+\sum_{n= 1}^{\rho-1} \Big(\op(a)^n+ \sum_{n= 1}^{\infty}\tilde{R}^{(n)}_{\rho}\Big)+\sum_{n\ge\rho} \op(a)^n\\
&=\mathrm{I}+\sum_{n= 1}^{\rho-1} \Big(\op(c^{(n)})
+R^{(n)}_{\rho}+\tilde{R}^{(n)}_{\rho}\Big)+\sum_{n\ge \rho} \Big(\op(a)^n 
+ \tilde{R}^{(n)}_{\rho}\Big)
\end{aligned}
\]
where $\tilde{R}^{(n)}_{\rho}:=(\op(a)+T)^n-\op(a)^n$ 
and $c^{(n)}$ and $R^{(n)}_{\rho}$ are given by Lemma \ref{casabianca}
(and we are setting $R^{(1)}_\rho=0$).
We define the symbol $c$ and the operator $R_{\rho}$ in \eqref{finaleUff} as
\begin{equation}\label{defFinale}
\begin{aligned}
c:=\sum_{n=1}^{\rho-1} c^{(n)},
\qquad
R_{\rho}:= \sum_{n= 1}^{\rho-1}(R^{(n)}_{\rho}+\tilde{R}^{(n)}_{\rho})+ \sum_{n\ge \rho} \tilde R_\rho^{(n)}+ \sum_{n\ge\rho}\op(a)^n. 
\end{aligned}
\end{equation}


\noindent
By  using \eqref{kessie} we get
\[
\begin{aligned}
\lvert c \rvert_{-1, s, \alpha} &\le_{s, \alpha, \rho} \sum_{n=1}^{\rho-1}\lvert a \rvert_{-1, s+(n-1)(\rho-3), \alpha+\rho-3}\big(\lvert a \rvert_{-1, s_0+(n-1)(\rho-3), \alpha+\rho-3}\big)^{n-1}
\end{aligned} 
\]
which implies the first of \eqref{kepalle}.
The second one in \eqref{kepalle} is obtained as above by using \eqref{jacob}.\
The bounds \eqref{casabianca0}, \eqref{kepalle100} on $R_{\rho}$ in \eqref{defFinale}
can be proved similarly by using Lemmata \ref{INCLUSIONEpseudoInclasseL}, \ref{idealeds}, \ref{James} and  \ref{casabianca}.

\noindent
In order to bound the $\mathfrak{I}$ variation we note
\[
\Delta_{12} (1-(\op(a)+T))^{-1}  =-(1-(\op(a)+T))^{-1}  (\op(\Delta_{12} a  )+\Delta_{12} T ) (1-(\op(a)+T))^{-1}\,,
\]
and proceed as above.
\end{proof}

\subsection{The torus diffeomorphisms}\label{someprop}

In this Section we
 wish to study conjugation of elements of  $\mathfrak{L}_\rho$ 
 under the action of the map $\mathcal{A}^{\tau}$ introduced in \eqref{flussoKDV}.
%
We first give some properties of $\mathcal A^\tau$ defined in \eqref{ignobel}. 
\begin{lem}\label{bastalapasta}
Assume that 
$\beta:=\beta(\omega, \mathfrak{I}(\omega))\in H^s(\T^{\nu+1})$ 
for some $s\geq s_0$, is  Lipschitz in $\omega\in \calO \subseteq \Omega_{\varepsilon}$ 
and Lipschitz  in the variable $i$.
If $\| \beta\|^{\gamma, \calO}_{s_0+\mu}\leq 1$,
for some $\mu\gg1$, 
then, for any $s\geq s_0$ and $u\in H^{s}$ with $u=u(\oo)$ depending in a Lipschitz way on $\oo\in \calO $,
one has
\begin{equation}\label{pasta7}
\begin{aligned}
&\sup_{\tau\in[0,1]} \lVert \mathcal{A}^{\tau} u \rVert_s^{\gamma,  \calO }, 
\sup_{\tau\in[0,1]} \lVert (\mathcal{A}^{\tau})^{*} u \rVert_s^{\gamma,  \calO }
\le_s 
\left(
\|u\|_{s}^{\gamma,  \calO }+\| \be\|_{s+\s}^{\gamma,  \calO }\|u\|_{s_0}^{\gamma,  \calO }
\right)
\end{aligned}
\end{equation}
\begin{equation}\label{Amenouno}
\begin{aligned}
&\sup_{\tau\in[0,1]} 
\lVert (\mathcal{A}^{\tau}-\mathrm{I})u \rVert_s^{\gamma,  \calO }, 
\sup_{\tau\in[0,1]} \lVert ((\mathcal{A}^{\tau})^{*}-\mathrm{I})u \rVert_s^{\gamma,  \calO }
\le_s \left(
\lVert \be \rVert_{s_0+\s}^{\gamma,  \calO } 
\|u\|_{s+1}^{\gamma,  \calO }+\| \be\|_{s+\s}^{\gamma,  \calO }\|u\|^{\gamma,  \calO }_{s_0}
\right)
\end{aligned}
\end{equation}
	for some $\sigma=\s(s_0)$.
	The inverse map $(\mathcal{A}^{\tau})^{-1}$ satisfies the same estimates but with possibly larger $\s$.
\end{lem}
\begin{proof}
	The bounds \eqref{pasta7}-\eqref{Amenouno}
	in norm $\|\cdot\|_{s}$
	follows by an explicit computation using 
	the formula \eqref{ignobel} and applying 
	Lemma $A.3$
	in Appendix $A$ in \cite{FGMP}.
	If $\be=\be(\oo)$ is a function of the parameters $\oo\in  \calO $, hence
	we need 
	to study the 
	term
	\begin{equation}\label{apple1}
	\sup_{\oo_1\neq\oo_{2}}\frac{\|(\mathcal{A}^{\tau}(\oo_{1})-\mathcal{A}^{\tau}(\oo_{2}))u\|_{s-1}}{|\oo_1-\oo_2|}
	\end{equation}
	in order to estimate the
	Lip-norm introduced in \eqref{tazza10}. We reason as follows.
	By \eqref{ignobel}
	we have for $\oo_{1},\oo_{2}\in \calO$
	\begin{equation}\label{apple}
	\begin{aligned}
	(\mathcal{A}^{\tau}(\oo_{1})-\mathcal{A}^{\tau}(\oo_{2}))u&=
	(1+\tau \be_{x}(\oo_1))\big[
	u(\oo_{1},x+\be(\oo_1))-u(\oo_{1},x+\be(\oo_2))
	\big]\\
	&+(1+\tau \be_{x}(\oo_1))\big[
	u(\oo_{1},x+\be(\oo_2))-u(\oo_{2},x+\be(\oo_2))\big]\\
	&+\tau u(\oo_{1},x+\be(\oo_2))(\be_{x}(\oo_1)-\be_{x}(\oo_2)).
	\end{aligned}
	\end{equation}
	Using the estimates in Lemma $A.3$ in \cite{FGMP} 
	and interpolation arguments we get
	\[
	\begin{aligned}
	\|u(\oo_{1},x+\be(\oo_1))-u(\oo_{1},x+\be(\oo_2))\|_{s-1}&\leq_s
	\|\be(\oo_1)-\be(\oo_2)\|_{s_0}\|u\|_{s}\\
	&+
	\|\be(\oo_1)-\be(\oo_2)\|_{s+1}\|u\|_{s_0}\\
	&\leq_s
	\Big(\|\be\|_{s+s_0+1}^{\gamma, \calO }\|u\|^{\gamma, \calO }_{s_0}+
	\|\be\|_{s_0}^{\gamma, \calO }\|u\|^{\gamma, \calO }_{s}
	\Big)|\oo_1-\oo_2|.
	\end{aligned}
	\]
	The term we have estimated above is the most critical one among the summand in \eqref{apple}.
	The other estimates follow by the fact that $u(\oo,\f,x)$ and $\be(\oo,\f,x)$ are Lipschitz functions of $\oo\in  \calO $. 
	One can reason in the same way to get the estimates on the inverse map $(\mathcal{A}^{\tau})^{-1}$ by recalling that it has the same form of $\mathcal{A}^{\tau}$
	(see \eqref{ignobel}) and $\beta=-\mathcal{A}^{\tau} \tilde{\beta}$. 
\end{proof}
\begin{lem}\label{buttalapasta2}
	Fix $\mathtt{b}\in \mathbb{N}$.
	For any $\alpha\in\mathbb{N}^{\nu}$, $|\al |\leq \mathtt{b}$, $m_1,m_{2}\in \mathbb{R}$ such that $m_1+m_{2}=|\al|$, for any 
	$s\geq s_0$ there exists a constant $\mu=\mu(|\al|,m_1,m_2)$ and $\delta=\delta(m_1,s)$
	such that if 
	\begin{equation}\label{buttalapasta3}
	\|\be\|_{2s_0+|m_1|+2}\leq \delta,\quad
	\|\be\|^{\gamma, \calO }_{s_0+\mu}\leq1,
	\end{equation}
	then one has
	\begin{equation}\label{pasta1} 
	\sup_{\tau\in[0,1]}\|\langle D_{x}\rangle^{-m_1}\del_{\f}^{\al}\mathcal{A}^{\tau}(\f)
	\langle D_{x}\rangle^{-m_2}u\|^{\gamma, \calO }_{s}
	\leq_{s,\mathtt{b},m_1,m_2}
	\| u \|_{s}+\|\be\|^{\gamma, \calO }_{s+\mu}\| u\|_{s_0}.
	\end{equation}
	The inverse map $(\mathcal{A}^{\tau})^{-1}$ satisfies the same estimate.
\end{lem}
\begin{proof}
	We  prove the bound \eqref{pasta1} for the  $\|\cdot\|_{s}$ norm
	since one can obtain the bound in the Lipschitz norm $\|\cdot\|^{\gamma, \calO }_{s}$
	using the same arguments (recall also the reasoning used in \eqref{apple}).
	We take $h\in C^{\infty}$, so that $\del_{\f}^{\al}\mathcal{A}^{\tau}(\f)h\in C^{\infty}$ for any $|\al|\leq \mathtt{b}$ and we prove the bound \eqref{pasta1} in this case.
	The thesis will follows by density.
	\\
	We argue  by induction on $\al$. 
	Given $\al\in \mathbb{N}^\nu$ we write $\al' \preceq \al$ if  $\al'_{n}\leq \al_{n}$ for any $n=1,\ldots,\nu$  and $\al'\neq\al$.
	\\
	Let us check \eqref{pasta1} for $\al=0$. 
	Let us define $\Psi^{\tau}:=\langle D_{x}\rangle^{m}\mathcal{A}^{\tau}(\f)
	\langle D_{x}\rangle^{-m}$ with $m=-m_1=m_2$. 
	One has that $\Psi^{0}:=\mathrm{I}$ (where $\mathrm{I}$ is the identity operator).
	One can check that $\Psi^{\tau}$ solves the problem (recall \eqref{flussoKDV})
	\begin{equation}\label{pasta5}
	\del_{\tau}\Psi^{\tau}=\mathtt{X}\Psi^{\tau}+G^{\tau}\Psi^{\tau},
	\end{equation} 
	where $G^{\tau}:=[\langle D_{x}\rangle^{m},\mathtt{X}]\,\langle D_{x}\rangle^{-m}$.
	Therefore by Duhamel principle one has
	\[
	\Psi^{\tau}=\mathcal{A}^{\tau}+\mathcal{A}^{\tau}\int_{0}^{\tau}(\mathcal{A}^{t})^{-1}G^{t}\Psi^{t}\,dt.
	\]
By Lemma \ref{PROP} and \eqref{tazza8} one has that
$|G^{\tau}|_{0,s,0}\leq_{s}\|\be\|_{s+m+3}$, for $s\geq s_0,$
hence by estimate \eqref{pasta7}, Lemma \ref{actiononsobolev} we have
\begin{equation}\label{apple4}
\begin{aligned}
\sup_{\tau\in[0,1]}\|\Psi^{\tau}h\|_{s}&\leq_{s}\|h\|_s+\|\be\|_{s+\s} 
\|h\|_{s_0}+
\|\be\|_{s_0+m+3}\sup_{\tau\in[0,1]}\|\Psi^{\tau}h\|_{s}\\
&+(\|\be\|_{s+m+3}+\|\be\|_{s+\s})\sup_{\tau\in[0,1]}\|\Psi^{\tau}h\|_{s_0}
\end{aligned}
\end{equation}
	for some $\s>0$ given in Lemma \ref{bastalapasta}.
	For $\delta$ in \eqref{buttalapasta3} small enough, then the \eqref{apple4} for $s=s_0$ implies
	that $\sup_{\tau\in[0,1]}\|\Psi^{\tau}h\|_{s_0}\leq_{s_0} \|h\|_{s_0}$. Using this bound in \eqref{apple4}
	one gets the \eqref{pasta1}.
	\\
	Now assume that the bound \eqref{pasta1} holds for any $\al'\preceq \al$ with $|\al|\leq \mathtt{b}$
	and $m_1,m_2\in \mathbb{R}$ with $m_1+m_2=|\al'|$.
	We now prove the estimate \eqref{pasta1}
	for the operator $\langle D_{x}\rangle^{-m_1}\del_{\f}^{\al}\mathcal{A}^{\tau}(\f)
	\langle D_{x}\rangle^{-m_2}$ for $m_1+m_2=|\al|$. 
Differentiating the \eqref{flussoKDV} and using the Duhamel formula
we get that
\begin{equation}\label{apple5}
\begin{aligned}
&\del_{\f}^{\al}\mathcal{A}^{\tau}(\f)=\int_{0}^{\tau}\,\mathcal{A}^{\tau}(\f)
(\mathcal{A}^{t}(\f))^{-1}F_{\al}^{t}dt,\qquad F_{\al}^{t}
:=\sum_{\substack{\al_1+\al_2=\al,\\ \lvert \alpha_2 \rvert+1\le\alpha}}C(\al_1,\al_2)\,\del_{x}\, [(\del_{\f}^{\al_1}b)
\del_{\f}^{\al_2}\mathcal{A}^{t}(\f)].
\end{aligned}
\end{equation}
For any $m_1+m_2=|\al|$ and any $\tau,s\in [0,1]$ 
we write
\begin{equation}\label{apple6}
\begin{aligned}
&\langle D_{x}\rangle^{-m_1}
\del_{x}(\del_{\f}^{\al_1}b)
\del_{\f}^{\al_2}\mathcal{A}^{t}(\f)
\langle D_{x}\rangle^{-m_2}\\
&=\langle D_{x}\rangle^{-m_1}\del_{x} (\del_{\f}^{\al_1}b)
\langle D_{x}\rangle^{-m_2+|\al_2|}\langle D_{x}\rangle^{m_2-|\al_2|}
\del_{\f}^{\al_2}\mathcal{A}^{t}(\f)\langle D_{x}\rangle^{-m_2}.
\end{aligned}
\end{equation}
	Hence  in order to estimate the operator 
	$\langle D_{x}\rangle^{-m_1}\mathcal{A}^{\tau}\,(\mathcal{A}^{t}(\f))^{-1}F_{\al}^{t}\langle D_{x}\rangle^{-m_2}$
	we need to estimate, uniformly in $\tau,s\in[0,1]$ the term
	\begin{equation}\label{apple7}
	\!\!\!\!\!\Big(
	\langle D_{x}\rangle^{-m_1}\mathcal{A}^{\tau}(\mathcal{A}^{t})^{-1}\langle D_{x}\rangle^{m_1}
	\Big)
	\Big(
	\langle D_{x}\rangle^{-m_1}\del_{x} (\del_{\f}^{\al_1}b)
	\langle D_{x}\rangle^{-m_2+|\al_2|}
	\Big)
	\Big(
	\langle D_{x}\rangle^{m_2-|\al_2|}
	\del_{\f}^{\al_2}\mathcal{A}^{t}(\f)\langle D_{x}\rangle^{-m_2}
	\Big).
	\end{equation}
	For $s\geq s_0$, by the inductive hypothesis one has
	\begin{equation}\label{apple8}
	\|\langle D_{x}\rangle^{-m_1}\mathcal{A}^{\tau}
	(\mathcal{A}^{t})^{-1}\langle D_{x}\rangle^{m_1}h\|_{s}\leq _{s,m_1}\|h\|_{s}+
	\|\be\|_{s+\mu}^{\gamma, \calO }\|h\|_{s_0},
	\end{equation}
	\begin{equation}\label{apple9}
	\|\langle D_{x}\rangle^{m_2-|\al_2|}
	\del_{\f}^{\al_2}\mathcal{A}^{t}(\f)\langle D_{x}\rangle^{-m_2}h\|_{s}\leq_{s,\mathtt{b},m_2}
	\|h\|_{s}+
	\|\be\|_{s+\mu}^{\gamma, \calO }\|h\|_{s_0}.
	\end{equation}
	provided that $\al_1\neq 0$.
	We estimate the second factor in \eqref{apple7}. We first note that
	\begin{equation*}
	-m_1-m_2+1+|\al_{2}|=1+|\al_{2}|-|\al|\leq 0.
	\end{equation*}
	This implies that 
	$\langle D_{x}\rangle^{-m_1}\del_{x} (\del_{\f}^{\al_1}b)
	\langle D_{x}\rangle^{-m_2+|\al_2|}$ belongs to $OPS^{0}$,
	and in particular, using Lemma \ref{PROP} and \eqref{norma3},
	we obtain
	\begin{equation}\label{apple10}
	|\langle D_{x}\rangle^{-m_1}\del_{x} (\del_{\f}^{\al_1}b)
	\langle D_{x}\rangle^{-m_2+|\al_2|}|_{0,s,0}\leq_{\mathtt{b},m_1,m_2}
	\|a\|^{\gamma, \calO }_{s+|m_1|+|\al_2|}.
	\end{equation}
	To obtain the bound \eqref{pasta1} it is enough to
	use bounds \eqref{apple8}, \eqref{apple9},\eqref{apple10}, Lemma \ref{actiononsobolev}
	and recall the smallness assumption \eqref{buttalapasta3}.\\
	About the estimate for the inverse of $\mathcal{A}^{\tau}$, we note that 
$\partial_{\tau} (\mathcal{A}^{\tau})^{-1}=\big(\partial_y\circ \tilde{b}\big)\, (\mathcal{A}^{\tau})^{-1}$
	with $\tilde{b}:=\frac{\partial_{\tau} \tilde{\beta}}{1+\tilde{\beta}_y}$ and $\lVert \tilde{b} \rVert_s\le \lVert \beta \rVert_{s+\tilde{\s}}$ for some $\tilde{\s}>0$. Then one can follow the same arguments above with $\partial_y\circ \tilde{b}$ instead of $\mathtt{X}$ and $\tilde{b}$ instead of $b$.
	\end{proof}
\begin{lem}\label{buttalapasta60}
	Let $\mathtt{b}\in \mathbb{N}$ and let $p>0$ be the constant given in Def. \ref{ellerho}.
	For any $|\al |\leq \mathtt{b}$, $m_1,m_{2}\in \mathbb{R}$ such that $m_1+m_{2}=|\al|+1$, 
	for any 
	$s\geq s_0$ there exists a constant $\mu=\mu(|\al|,m_1,m_2)$, $\s=\s(|\al|,m_1,m_2)$ 
	and $\delta=\delta(s, m_1)>0$
	such that if $\|\be\|_{s_0+\mu}\leq \delta$ and   $\|\be\|_{p+\s}\leq 1$
then one has
\begin{equation}\label{pasta100} 
\begin{aligned}
\sup_{\tau\in[0,1]} & \| \langle D_{x}\rangle^{-m_1}\del_{\f}^{\al}\Delta_{12}\mathcal{A}^{\tau}(\f) 
\langle D_{x}\rangle^{-m_2}u\|_{p}  \leq_{s,\mathtt{b},m_1,m_2}
\lVert u \rVert_{p}\lVert \Delta_{12} \be   \rVert_{p+\mu}
\end{aligned}
\end{equation}
The operators $\Delta_{12}(\mathcal{A}^{\tau})^{*} $, $\Delta_{12}(\mathcal{A}^{\tau})^{-1} $ satisfy the same estimate.
\end{lem}
\begin{proof}
The Lemma can be proved arguing as in the proof of Lemma \ref{buttalapasta2} using 
$(\mathcal{A}^{\tau})^*=(1+\tau \beta)^{-1}\,\mathcal{A}^{\tau}$.
\end{proof}

We have the following Lemma.

\begin{lem}\label{preparailsugo}
	Fix $\rho\geq3$, consider $ \calO \subset\R^\nu$ 
	and let $R\in \gotL_{\rho, p}( \calO )$
	(see Def. 
	\ref{ellerho}).
	Consider a function $\be$ such that
	$\beta:=\beta(\omega, i(\omega))\in H^s(\T^{\nu+1})$  
	for some $s\geq s_0$, assume that it is  Lipschitz in $\omega\in  \calO $ and $i$.  Let $\mathcal{A}^{\tau}$ be the operator defined in \eqref{ignobel}.
	There exists $\mu=\mu(\mathtt{\rho})\gg1$, $\s=\s(\rho)$ and $\delta>0$ small such that
	if $\| \beta\|^{\gamma,  \calO }_{s_0+\mu}\leq \delta$ and $\| \beta\|^{\gamma,  \calO }_{p+\s}\leq 1$, then
	the operator $M^{\tau}:=\mathcal{A}^{\tau}R(\mathcal{A}^{\tau})^{-1}$
	belongs to the class $\gotL_\rho$.
	In particular one has, for $s_0\le s\le \mathcal{S}$,
	\begin{equation}\label{casalotti}
	\begin{aligned}
	\mathbb{M}^{\gamma}_{M^{\tau}}(s,\mathtt{b}) &\leq \mathbb{M}^{\gamma}_{R}(s,\mathtt{b})+
	\|\be\|_{s+\mu}^{\gamma, \calO }\mathbb{M}^{\gamma}_{R}(s_0,\mathtt{b}),\qquad \mathtt{b}\leq \rho-2
	\end{aligned}
	\end{equation}
\begin{equation}
\begin{aligned}
\mathbb{M}_{\Delta_{12} M^{\tau}  }(p,\mathtt{b}) &\leq  
\mathbb{M}_{\Delta_{12} R^{\tau}  }(p,\mathtt{b})
+\lVert \Delta_{12} \beta   \rVert_{p+\mu}
\mathbb{M}_{R^{\tau}}^{\gamma}(p,\mathtt{b}), \qquad \mathtt{b}\leq \rho-3.
\end{aligned}
\end{equation}
\end{lem}

\begin{proof}
	We start by showing that $M^{\tau}$  satisfies item $(i)$ of Definition \ref{ellerho}.
	Let $m_{1},m_{2}\in \mathbb{R}$, $m_{1},m_{2}\geq0$
	and $m_{1}+m_{2}=\rho$.
	We write
	\begin{equation*}
	\begin{aligned}
	\langle D_{x}\rangle^{m_1}M^{\tau} \langle D_{x}\rangle^{m_2}&=
	\langle D_{x}\rangle^{m_1}\mathcal{A}^{\tau}\langle D_{x}\rangle^{-m_1}
	\langle D_{x}\rangle^{m_1}R\langle D_{x}\rangle^{m_2}
	\langle D_{x}\rangle^{-m_2}(\mathcal{A}^{\tau})^{-1}\langle D_{x}\rangle^{m_2}.
	\end{aligned}
	\end{equation*}
	Recall that by hypothesis the operator 
	$\langle D_{x}\rangle^{m_1}R\langle D_{x}\rangle^{m_2}$
	is Lip-$0$-tame with constants $\gotM_{R}^{\gamma}(-\rho,s)$ see \eqref{megaTame2}. 
	Lemma \eqref{buttalapasta2} implies the estimates
	\begin{equation*}
	\|\langle D_{x}\rangle^{m_1}\mathcal{A}^{\tau}(\f)
	\langle D_{x}\rangle^{-m_1}u\|^{\gamma, \calO }_{s},\|\langle D_{x}\rangle^{-m_2}(\mathcal{A}^{\tau}(\f))^{-1}
	\langle D_{x}\rangle^{m_2}u\|^{\gamma, \calO }_{s}
	\leq_{s,\rho}
	\| u \|_{s}+\|\be\|^{\gamma, \calO }_{s+\mu}\| u\|_{s_0},
	\end{equation*}
	for $\tau\in [0,1]$,
	which implies that $\langle D_{x}\rangle^{m_1}M^{\tau} \langle D_{x}\rangle^{m_2}$
	is Lip-$0-$tame with constant 
	\begin{equation}\label{apple14}
	\gotM^{\gamma}_{\langle D_{x}\rangle^{m_1}M^{\tau} \langle D_{x}\rangle^{m_2}}(0,s)
	\leq_{s,\rho}
	\gotM_{R}^{\gamma}(-\rho,s)
	+\|\be\|_{s+\mu}^{\gamma, \calO }\gotM^{\gamma}_{R}(-\rho,s_0).
	\end{equation}
	Hence  $M^{\tau}$ is Lip-$(-\rho)$-tame
	with constant $\gotM_{M^{\tau}}^{\gamma}(-\rho,s)=\sup_{\substack{m_{1}+m_{2}=\rho\\
			m_{1},m_{2}\geq0}}\gotM^{\gamma}_{\langle D_{x}\rangle^{m_1}M^{\tau} \langle D_{x}\rangle^{m_2}}(0,s)$.
	Fix $\mathtt{b}\leq \rho-2$
	and 
	let $m_{1},m_{2}\in \mathbb{R}$, $m_{1},m_{2}\geq0$
	and $m_{1}+m_{2}=\rho-\mathtt{b}$.
	We note that for any $\vec{\tb}\in\mathbb{N}^{\nu}$ with $|\vec{\tb}|=\tb$ 
	\begin{equation}\label{apple11}
	\del_{\f}^{\vec{\mathtt{b}}}M=\sum_{\vec{\mathtt{b}_1}+\vec{\mathtt{b}_2}+\vec{\mathtt{b}_{3}}=\vec{\mathtt{b}}}
	C(|\vec{\mathtt{b}_{1}}|,|\vec{\mathtt{b}_{2}}|,|\vec{\mathtt{b}_{3}}|)
	(\del_{\f}^{\vec{\mathtt{b}_1}}\mathcal{A}^{\tau})\del_{\f}^{\vec{\mathtt{b}_2}}R(\del_{\f}^{\vec{\mathtt{b}_3}}(\mathcal{A}^{\tau})^{-1}),
	\end{equation}
	for some constants $C(|\vec{\mathtt{b}_{1}}|,|\vec{\mathtt{b}_{2}}|,|\vec{\mathtt{b}_{3}}|)>0$, 
	hence
	we need to show that each summand in \eqref{apple11}
	satisfies item $(i)$ of Definition \ref{ellerho}.
	We write
	\begin{equation}\label{apple13}
	\begin{aligned}
	&\langle D_{x}\rangle^{m_1}
	(\del_{\f}^{\vec{\mathtt{b}_1}}\mathcal{A}^{\tau})\del_{\f}^{\vec{\mathtt{b}_2}}R(\del_{\f}^{\vec{\mathtt{b}_3}}(\mathcal{A}^{\tau})^{-1})
	\langle D_{x}\rangle^{m_2}=\\
	&=
	\langle D_{x}\rangle^{m_1}
	(\del_{\f}^{\vec{\mathtt{b}_1}}\mathcal{A}^{\tau})
	\langle D_{x}\rangle^{y}
	\langle D_{x}\rangle^{-y}
	(\del_{\f}^{\vec{\mathtt{b}_2}}R)
	\langle D_{x}\rangle^{z}
	\langle D_{x}\rangle^{-z}
	(\del_{\f}^{\vec{\mathtt{b}_3}}(\mathcal{A}^{\tau})^{-1})
	\langle D_{x}\rangle^{m_2},
	\end{aligned}
	\end{equation}
	where $y=-|\vec{\mathtt{b}_{1}}|-m_1$, $z=\rho-|\vec{\mathtt{b}_2}|-|\vec{\mathtt{b}_1}|-m_1$.
	Since
	$y+m_1=-|\vec{\mathtt{b}_1}|$ and $-z+m_2=-|\vec{\mathtt{b}_3}|$, hence
	by Lemma \ref{buttalapasta2} the operators
	\[
	\langle D_{x}\rangle^{m_1}
	(\del_{\f}^{\vec{\mathtt{b}_1}}\mathcal{A}^{\tau})
	\langle D_{x}\rangle^{y}, \quad 
	\langle D_{x}\rangle^{-z}
	(\del_{\f}^{\vec{\mathtt{b}_3}}(\mathcal{A}^{\tau})^{-1})
	\langle D_{x}\rangle^{m_2}, 
	\]
	satisfy bounds like \eqref{pasta1}.
	Moreover $-y+z=\rho-|\vec{\mathtt{b}_2}|$ and $-y,z\geq0$, hence, 
	by the definition of the class
	$\gotL_{\rho, p}$, we have that the operator $\langle D_{x}\rangle^{-y}
	(\del_{\f}^{\vec{\mathtt{b}_2}}R)
	\langle D_{x}\rangle^{z}$ is Lip-$0$-tame. 
	Following the reasoning used to prove \eqref{apple14}
	one obtains
	\begin{equation}\label{apple15}
	\gotM^{\gamma}_{\langle D_{x}\rangle^{m_1}
		\del_{\f}^{\vec{\mathtt{b}}}M^{\tau} \langle D_{x}\rangle^{m_2}}(0,s)
	\leq_{s,\rho}
	\mathbb{M}^{\gamma}_{R}(s, \tb)
	+\|\be\|_{s+\mu}^{\gamma, \calO }\mathbb{M}^{\gamma}_{R}(s, \tb).
	\end{equation}
	Let us consider the operator $[M^{\tau},\del_{x}]$. We write
	\begin{equation}\label{apple16}
	[M^{\tau},\del_{x}]=\mathcal{A}^{\tau}[R,\del_{x}](\mathcal{A}^{\tau})^{-1}+
	\mathcal{A}^{\tau}R[(\mathcal{A}^{\tau})^{-1},\del_{x}]+
	[\mathcal{A}^{\tau},\del_{x}]R(\mathcal{A}^{\tau})^{-1},
	\end{equation}
	for $\tau\in[0,1]$. We need to show that
	each summand in \eqref{apple16} satisfies item $(ii)$ in Definition \eqref{ellerho}.
	Let $m_{1},m_{2}\in \mathbb{R}$, $m_{1},m_{2}\geq0$
	and $m_{1}+m_{2}=\rho-1$.
	We first note that
	\begin{equation}\label{apple17}
	\begin{aligned}
	&\langle D_{x}\rangle^{m_1}\mathcal{A}^{\tau}[R,\del_{x}](\mathcal{A}^{\tau})^{-1}
	\langle D_{x}\rangle^{m_2}= \\
	&=
	\langle D_{x}\rangle^{m_1}
	\mathcal{A}^{\tau}
	\langle D_{x}\rangle^{-m_1}\langle D_{x}\rangle^{m_1}
	[R,\del_{x}]
	\langle D_{x}\rangle^{m_2}\langle D_{x}\rangle^{-m_2}
	(\mathcal{A}^{\tau})^{-1}
	\langle D_{x}\rangle^{m_2},
	\end{aligned}
	\end{equation}
	hence, by applying Lemma \ref{buttalapasta2} to estimate the terms
	\[
	\langle D_{x}\rangle^{-m_2}
	(\mathcal{A}^{\tau})^{-1}
	\langle D_{x}\rangle^{m_2},\quad
	\langle D_{x}\rangle^{m_1}
	(\mathcal{A}^{\tau})^{-1}
	\langle D_{x}\rangle^{-m_1}
	\] 
	and using the tameness of the operator 
	$\langle D_{x}\rangle^{m_1}
	[R,\del_{x}]
	\langle D_{x}\rangle^{m_2}$ (recall that $R\in\gotL_{\rho, p}$) one gets
	\begin{equation}\label{apple18}
	\gotM^{\gamma}_{\langle D_{x}\rangle^{m_1} 
		\mathcal{A}^{\tau}[R,\del_{x}](\mathcal{A}^{\tau})^{-1}
		\langle D_{x}\rangle^{m_2}}(0,s)
	\leq_{s,\rho}
	\mathbb{M}^{\gamma}_{R}(s,\mathtt{b})
	+\|\be\|_{s+\mu}^{\gamma, \calO }\mathbb{M}^{\gamma}_{R}(s_0,\mathtt{b}).
	\end{equation}
	The term $[\mathcal{A}^{\tau},\del_{x}]R(\mathcal{A}^{\tau})^{-1}$ in \eqref{apple16}
	is more delicate.
	Let $m_{1},m_{2}\in \mathbb{R}$, $m_{1},m_{2}\geq0$
	and $m_{1}+m_{2}=\rho-1$.
	We write
	\begin{equation}\label{apple19}
	\langle D_{x}\rangle^{m_1}[\mathcal{A}^{\tau},\del_{x}]\langle D_{x}\rangle^{-m_1-1}
	\langle D_{x}\rangle^{m_1+1}R\langle D_{x}\rangle^{m_2}
	\langle D_{x}\rangle^{-m_2}(\mathcal{A}^{\tau})^{-1}
	\langle D_{x}\rangle^{m_2}.
	\end{equation}
	By Lemma \ref{buttalapasta2} we have that $\langle D_{x}\rangle^{-m_2}(\mathcal{A}^{\tau})^{-1}
	\langle D_{x}\rangle^{m_2}$ satisfies a bound like \eqref{pasta1} with $\al=0$.
	The operator $\langle D_{x}\rangle^{m_1+1}R\langle D_{x}\rangle^{m_2}\langle D_{x}\rangle^{m_1+1}R\langle D_{x}\rangle^{m_2}$ is Lip-$0$-tame
	since $R\in \gotL_{\rho, p}$ and $m_1+m_2+1=\rho$.
	Moreover by an explicit computation (using formula \eqref{ignobel}) we get
	\begin{equation}\label{apple20}
	[\mathcal{A}^{\tau},\del_{x}]=\tau \frac{\be_{xx}}{1+\tau\be_{x}}\mathcal{A}^{\tau}+\tau\be_{x}\mathcal{A}^{\tau}\del_{x}.
	\end{equation}
	We claim that, for $s\geq s_0$ and $u\in H^{s}$, one has
	\begin{equation}\label{apple21}
	\|
	\langle D_{x}\rangle^{m_1}
	[\mathcal{A}^{\tau},\del_{x}]
	\langle D_{x}\rangle^{-m_1-1}
	u\|_{s}^{\gamma, \calO }\leq_{s,\rho}
	\|\be\|_{s_0+\mu}^{\gamma, \calO } \| u \|_{s}+\|\be\|^{\gamma, \calO }_{s+\mu}
	\| u\|_{s_0},
	\end{equation}
	for some $\mu>0$ depending only on $s,\rho$.
	The first summand in \eqref{apple20}
	satisfies the bound \eqref{apple21} thanks to Lemma \ref{PROP} for the estimate of 
	$\langle D_{x}\rangle^{m_1}\be_{xx}(1+\tau\be_{x})^{-1}\langle D_{x}\rangle^{-m_1}$ and thanks Lemma 
	\ref{buttalapasta2} to estimate $\langle D_{x}\rangle^{m_1}\mathcal{A}^{\tau}
	\langle D_{x}\rangle^{-m_1}$. For the second summand we reason as follow:
	we write
	\[
	\begin{aligned}
	\langle D_{x}\rangle^{m_1}
	\tau\be_{x}\mathcal{A}^{\tau}\del_{x}
	\langle D_{x}\rangle^{-m_1-1}
	&=
	\Big(\langle D_{x}\rangle^{m_1}\be_{x}\langle D_{x}\rangle^{-m_1}\Big)
	\Big(
	\langle D_{x}
	\rangle^{m_1}
	\mathcal{A}^{\tau}
	\langle D_{x}\rangle^{-m_1}
	\Big)
	\del_{x}\langle D_{x}\rangle^{-1}
	\end{aligned}
	\]
	and we note that the operator $\del_{x}\langle D_{x}\rangle^{-1}$ is bounded on $H^{s}$.
	Hence the bound \eqref{apple21} follows by applying Lemmata \ref{PROP} and \ref{buttalapasta2}. 
	By the discussion above one gets
	\begin{equation}\label{apple22}
	\gotM^{\gamma}_{\langle D_{x}\rangle^{m_1} 
		[\mathcal{A}^{\tau},\del_{x}]R(\mathcal{A}^{\tau})^{-1}
		\langle D_{x}\rangle^{m_2}}(0,s)
	\leq_{s,\rho}
	\mathbb{M}^{\gamma}_{R}(s, \tb)
	+\|\be\|_{s+\mu}^{\gamma, \calO }\mathbb{M}^{\gamma}_{R}(s, \tb).
	\end{equation}
	One can study the tameness constant of the operator $\mathcal{A}^{\tau}R[(\mathcal{A}^{\tau})^{-1},\del_{x}]$ in \eqref{apple16} by using the same arguments above.
	
	\noindent
	We check now that $M^{\tau}$ satisfies item $(iii)$ of Def. \ref{ellerho}. 
	Let $m_{1},m_{2}\in \mathbb{R}$, $m_{1},m_{2}\geq0$
	and $m_{1}+m_{2}=\rho-\lvert \vec{\mathtt{b}} \rvert-1$.
	We write for $\vec{\tb}\in \mathbb{N}^{\nu}$, $|\vec{\tb}|=\tb$
	\begin{equation}\label{apple23}
	\begin{aligned}
	\; [\del_{\f}^{\vec{\mathtt{b}}}\mathcal{A}^{\tau}R (\mathcal{A}^{\tau})^{-1},\del_{x} ]&=\sum_{\vec{\mathtt{b}_1}+\vec{\mathtt{b}_2}+\vec{\mathtt{b}_3}=\vec{\mathtt{b}}}
	C(|\vec{\mathtt{b}_{1}}|,|\vec{\mathtt{b}_{2}}|,|\vec{\mathtt{b}_{3}}|)
	\Big[
	(\del_{\f}^{\vec{\mathtt{b}_1}}\mathcal{A}^{\tau})
	(\del_{\f}^{\vec{\mathtt{b}_2}}R)(\del_{\f}^{\vec{\mathtt{b}_3}}(\mathcal{A}^{\tau})^{-1})
	,\del_{x}\Big]
	\end{aligned}
	\end{equation}
	and we note that
	\begin{equation}\label{apple24}
	\begin{aligned}
	\;[(\del_{\f}^{\vec{\mathtt{b}_1}}\mathcal{A}^{\tau})
	(\del_{\f}^{\vec{\tb_2}}R)(\del_{\f}^{\vec{\mathtt{b}_1}}(\mathcal{A}^{\tau})^{-1})
	,\del_{x}]&=
	(\del_{\f}^{\vec{\mathtt{b}_1}}\mathcal{A}^{\tau})\Big[(\del_{\f}^{\vec{\mathtt{b}_2}}R),\del_{x}\Big]
	(\del_{\f}^{\vec{\mathtt{b}_3}}\mathcal{A}^{\tau})^{-1})\\
	&+
	(\del_{\f}^{\vec{\mathtt{b}_1}}\mathcal{A}^{\tau})
	(\del_{\f}^{\vec{\mathtt{b}_2}}R)
	\Big[(\del_{\f}^{\vec{\mathtt{b}_3}}(\mathcal{A}^{\tau})^{-1}),\del_{x}\Big]\\
	&+
	\Big[(\del_{\f}^{\vec{\mathtt{b}_1}}\mathcal{A}^{\tau}),\del_{x}\Big]
	(\del_{\f}^{\vec{\mathtt{b}_2}}R)
	(\del_{\f}^{\vec{\mathtt{b}_3}}(\mathcal{A}^{\tau})^{-1}).
	\end{aligned}
	\end{equation}
	The most difficult term to study is the last summand in \eqref{apple24}.
	We have that
	\begin{equation}\label{apple25}
	\begin{aligned}
	&\langle D_{x}\rangle^{m_{1}}
	\Big[(\del_{\f}^{\vec{\mathtt{b}_1}}\mathcal{A}^{\tau}),\del_{x}\Big]
	(\del_{\f}^{\vec{\mathtt{b}_2}}R)
	(\del_{\f}^{\vec{\mathtt{b}_3}}(\mathcal{A}^{\tau})^{-1})
	\langle D_{x}\rangle^{m_{2}}=\\
	&=
	\langle D_{x}\rangle^{m_{1}}
	\Big[(\del_{\f}^{\vec{\mathtt{b}_1}}\mathcal{A}^{\tau}),\del_{x}\Big]
	\langle D_{x}\rangle^{-y}
	\langle D_{x}\rangle^{y}
	(\del_{\f}^{\vec{\mathtt{b}_2}}R)
	\langle D_{x}\rangle^{z}
	\langle D_{x}\rangle^{-z}
	(\del_{\f}^{\vec{\mathtt{b}_3}}(\mathcal{A}^{\tau})^{-1})
	\langle D_{x}\rangle^{m_{2}},
	\end{aligned}
	\end{equation}
	with $z=m_{2}+|\vec{\mathtt{b}_{3}}|$ and $y=\rho-|\vec{\mathtt{b}_{2}}|-|\vec{\mathtt{b}_{3}}|-m_{2}$.
	Note the operator $\langle D_{x}\rangle^{-z}
	(\del_{\f}^{\vec{\mathtt{b}_3}}(\mathcal{A}^{\tau})^{-1})
	\langle D_{x}\rangle^{m_{2}}$ satisfies bound like \eqref{pasta1} with $\al=\vec{\tb}_3$;
	moreover the operator
	$\langle D_{x}\rangle^{y}
	(\del_{\f}^{\vec{\mathtt{b}_2}}R)
	\langle D_{x}\rangle^{z}$ is Lip-$0$-tame
	since $y+z=\rho-|\vec{\mathtt{b}_{2}}|$ and $R\in\gotL_{\rho, p}$. Note also that, since $m_1+m_{2}=\rho-|\vec{\mathtt{b}}|-1$, 
	one has $y=m_{1}+|\vec{\mathtt{b}_{1}}|+1$. We now study the tameness constant of
	\[
	\langle D_{x}\rangle^{m_1}
	\Big[(\del_{\f}^{\vec{\mathtt{b}_1}}\mathcal{A}^{\tau}),\del_{x}\Big]
	\langle D_{x}\rangle^{-m_{1}-|\vec{\mathtt{b}_1}|-1}. 
	\]
By differentiating the \eqref{apple20} we get
\begin{equation}\label{apple26}
\del_{\f}^{\vec{\mathtt{b}_1}}[\mathcal{A}^{\tau},\del_{x}]=
\sum_{\vec{\mathtt{b}}_{1}'+\vec{\mathtt{b}}_{1}''=\vec{\mathtt{b}}_{1}}
\tau(\del_{\f}^{\vec{\mathtt{b}}_{1}'}g) (\del_{\f}^{\vec{\mathtt{b}}_{1}''}\mathcal{A}^{\tau})+
\tau(\del_{\f}^{\vec{\mathtt{b}}_{1}'}\be_{x})(\del_{\f}^{\vec{\mathtt{b}}_{1}''}\mathcal{A}^{\tau})\del_{x},
\end{equation}
where $g=\be_{xx}/(1+\tau\be_{x})$. We claim that
\begin{equation}\label{apple27}
\|
\langle D_{x}\rangle^{m_1}
[\del_{\f}^{\vec{\mathtt{b}}_1}\mathcal{A}^{\tau},\del_{x}]
\langle D_{x}\rangle^{-m_1-|\vec{\mathtt{b}_1}|-1}
u\|_{s}^{\gamma, \calO }\leq_{s,\rho}
\| u \|_{s}\|\be\|_{s_0+\mu}^{\gamma, \calO }+\|\be\|^{\gamma, \calO }_{s+\mu}\| u\|_{s_0},
\end{equation}
for some $\mu>0$ depending on $s,\rho$. 
We study the most difficult summand in \eqref{apple26}.
We have
\begin{equation}\label{apple28}
\begin{aligned}
\langle D_{x}\rangle^{m_1}(\del_{\f}^{\vec{\mathtt{b}}_{1}'}\be_{x})
(\del_{\f}^{\vec{\mathtt{b}}_{1}''}\mathcal{A}^{\tau})\del_{x}\langle D_{x}\rangle^{-m_1-|\vec{\mathtt{b}}_1|-1}&=
\langle D_{x}\rangle^{m_1}(\del_{\f}^{\vec{\mathtt{b}}_{1}'}\be_{x})
\langle D_{x}\rangle^{-m_1-|\vec{\mathtt{b}_1}|+|\vec{\mathtt{b}_1''}|}\\
&\times
\langle D_{x}\rangle^{m_1+|\vec{\mathtt{b}}_1|-|\vec{\mathtt{b}}_1''|}
(\del_{\f}^{\vec{\mathtt{b}}_{1}''}\mathcal{A}^{\tau})
\langle D_{x}\rangle^{-m_1-\lvert\vec{\mathtt{b}}_1\rvert}\del_{x}\langle D_{x}\rangle^{-1}.
\end{aligned}
\end{equation}
The \eqref{apple27} follows for the term in \eqref{apple28} 
by using Lemmata \ref{PROP},
\ref{buttalapasta2} and the fact that 
$\del_{x}\langle D_{x}\rangle^{-1}$ is bounded on $H^{s}$. 
On the other summand in \eqref{apple26} one uses similar arguments.
By the discussion above one can check that
\begin{equation}\label{apple30}
\gotM^{\gamma}_{\langle D_{x}\rangle^{m_1} 
[\del_{\f}^{\vec{\mathtt{b}}}\mathcal{A}^{\tau},\del_{x}]R(\mathcal{A}^{\tau})^{-1}
\langle D_{x}\rangle^{m_2}}(0,s)
\leq_{s,\rho}
\mathbb{M}^{\gamma}_{R}(s,\mathtt{b})
+\|\be\|_{s+\mu}^{\gamma, \calO }\mathbb{M}^{\gamma}_{R}(s_0,\mathtt{b}).
\end{equation}
The fact  that the operator $M$ satisfies items $(iii)$-$(iv)$ of Definition \eqref{ellerho}
can be proved arguing as done above for items $(i)$-$(ii)$. 
\end{proof}


\begin{thebibliography}{10}

\bibitem{Avila}
Artur Avila, Bassam Fayad, and Rapha\"el Krikorian.
\newblock A {KAM} scheme for {${\rm SL}(2, \mathbb{R})$} cocycles with {L}iouvillean
  frequencies.
\newblock {\em Geom. Funct. Anal.}, 21(5):1001--1019, 2011.

\bibitem{BBHM}
P.~Baldi, M.~Berti, E.~Haus, and R.~Montalto.
\newblock Time quasiperiodic gravity water waves in finite depth.
\newblock preprint 2017, arXiv:1708.01517.

\bibitem{Airy}
P.~Baldi, M.~Berti, and R.~Montalto.
\newblock K{AM} for quasi-linear and fully nonlinear forced perturbations of
  {A}iry equation.
\newblock {\em Math. Ann.}, 359(1-2):471--536, 2014.

\bibitem{BBM16}
P.~Baldi, M.~Berti, and R.~Montalto.
\newblock K{AM} for autonomous quasi-linear perturbations of {K}d{V}.
\newblock {\em Ann. Inst. H. Poincar\'e Anal. Non Lin\'eaire},
  33(6):1589--1638, 2016.

\bibitem{Bam}
D.~Bambusi.
\newblock Reducibility of 1-d {S}chr\"odinger equation with time quasiperiodic
  unbounded perturbations, {II}.
\newblock {\em Comm. Math. Phys.}, 353(1):353--378, 2017.

\bibitem{BGMR2}
D.~Bambusi, B.~Gr\'ebert, A.~Maspero, and D.~Robert.
\newblock {G}rowth of {S}obolev norms for abstract linear {S}chr\"odinger
  equations.
\newblock arXiv:1706.09708.

\bibitem{BGMR}
D.~Bambusi, B.~Gr\'ebert, A.~Maspero, and D.~Robert.
\newblock Reducibility of the quantum harmonic oscillator in d-dimensions with
  polynomial time-dependent perturbation.
\newblock {\em Anal. PDE}, 11(3):775--799, 2018.

\bibitem{BBiP1}
M.~Berti, L.~Biasco, and M.~Procesi.
\newblock {KAM} theory for the {H}amiltonian derivative wave equation.
\newblock {\em Annales Scientifiques de l'ENS}, 46(2):299--371, 2013.

\bibitem{BBiP2}
M.~Berti, P.~Biasco, and M.~Procesi.
\newblock {KAM} theory for reversible derivative wave equations.
\newblock {\em Archive for Rational Mechanics and Analysis}, 212(3):905--955,
  2014.

\bibitem{BM1}
M.~Berti and R.~Montalto.
\newblock Quasi-periodic standing wave solutions of gravity-capillary water
  waves.
\newblock Preprint 2016 arXiv:1602.02411, To appear on Memoirs Am. Math. Soc.
  MEMO-891, 2016.

\bibitem{CY}
L.~Chierchia and J.~You.
\newblock Kam tori for 1d nonlinear wave equations with periodic boundary
  conditions.
\newblock {\em Comm. Math. Phys.}, 211:497--525, 2000.

\bibitem{Comb87}
M.~Combescure.
\newblock The quantum stability problem for time-periodic perturbations of the
  harmonic oscillator.
\newblock {\em Ann. Inst. H. Poincar\'e Phys. Th\'eor.}, 47(1):63--83, 1987.

\bibitem{Cons}
Adrian Constantin and Rossen Ivanov.
\newblock Dressing method for the {D}egasperis-{P}rocesi equation.
\newblock {\em Stud. Appl. Math.}, 138(2):205--226, 2017.

\bibitem{ConsIvLe}
Adrian Constantin, Rossen~I. Ivanov, and Jonatan Lenells.
\newblock Inverse scattering transform for the {D}egasperis-{P}rocesi equation.
\newblock {\em Nonlinearity}, 23(10):2559--2575, 2010.

\bibitem{ConstLannes}
Adrian Constantin and David Lannes.
\newblock The hydrodynamical relevance of the {C}amassa-{H}olm and
  {D}egasperis-{P}rocesi equations.
\newblock {\em Arch. Ration. Mech. Anal.}, 192(1):165--186, 2009.

\bibitem{CFP}
L.~Corsi, R.~Feola, and M.~Procesi.
\newblock {F}inite dimensional invariant {K}{A}{M} tori for tame vector fields.
\newblock preprint 2016.

\bibitem{Bambusi-Graffi}
Bambusi D. and Graffi S.
\newblock Time quasi-periodic unbounded perturbations of schr\"odinger
  operators and kam methods.
\newblock {\em Commun. Math. Phys.}, 219:465--480, 2001.

\bibitem{Deg}
A.~Degasperis, D.~D. Holm, and A.~N.~I. Hone.
\newblock A new integrable equation with peakon solutions.
\newblock {\em Teoret. Mat. Phys.}, 133(2):170--183, 2002.

\bibitem{DegPro}
A.~Degasperis and M.~Procesi.
\newblock Asymptotic integrability.
\newblock In {\em Symmetry and perturbation theory ({R}ome, 1998)}, pages
  23--37. World Sci. Publ., River Edge, NJ, 1999.

\bibitem{EK2}
H.~L. Eliasson and S.~B. Kuksin.
\newblock On reducibility of {S}chr\"odinger equations with quasiperiodic in
  time potentials.
\newblock {\em Comm. Math. Phys.}, 286(1):125--135, 2009.

\bibitem{Eli}
L.~H. Eliasson.
\newblock Almost reducibility of linear quasi-periodic systems.
\newblock In {\em Smooth ergodic theory and its applications ({S}eattle, {WA},
  1999)}, volume~69 of {\em Proc. Sympos. Pure Math.}, pages 679--705. Amer.
  Math. Soc., Providence, RI, 2001.

\bibitem{EGK}
L.~H. Eliasson, B.~Gr{\'e}bert, and S.B. Kuksin.
\newblock K{AM} for the nonlinear beam equation.
\newblock {\em Geom. Funct. Anal.}, 26(6):1588--1715, 2016.

\bibitem{EK}
L.~H. Eliasson and S.~B. Kuksin.
\newblock K{AM} for the nonlinear {S}chr\"odinger equation.
\newblock {\em Ann. of Math. (2)}, 172(1):371--435, 2010.

\bibitem{FGMP}
R.~Feola, F.~Giuliani, R.~Montalto, and M.~Procesi.
\newblock Reducibility of first order linear operators on tori via moser's
  theorem.
\newblock preprint 2018, arXiv:1801.04224.

\bibitem{FP}
R.~Feola and M.~Procesi.
\newblock Quasi-periodic solutions for fully nonlinear forced reversible
  {S}chr\"odinger equations.
\newblock {\em J. Differential Equations}, 259(7):3389--3447, 2015.

\bibitem{GY}
J.~Geng and J.~You.
\newblock A {KAM} theorem for {H}amiltonian partial differential equations in
  higher dimensional spaces.
\newblock {\em Comm. Math. Phys.}, 262(2):343--372, 2006.

\bibitem{GYX}
J.~Geng, J.~You, and X.~Xu.
\newblock {KAM} tori for cubic {NLS} with constant potentials.
\newblock Preprint.

\bibitem{Gi}
F.~Giuliani.
\newblock {Q}uasi-periodic solutions for quasi-linear generalized {K}d{V}
  equations.
\newblock {\em Journal of Differential Equations}, 262(10):5052 -- 5132, 2017.

\bibitem{GP}
B.~Gr{\'e}bert and E.~Paturel.
\newblock {KAM} for the {K}lein-{G}ordon equation on {${\mathbb S}^d$}.
\newblock {\em Boll. Unione Mat Ital.}, 9(2):237--288, 2016.

\bibitem{You}
Hai-Long Her and Jiangong You.
\newblock Full measure reducibility for generic one-parameter family of
  quasi-periodic linear systems.
\newblock {\em J. Dynam. Differential Equations}, 20(4):831--866, 2008.

\bibitem{Hou}
Yu~Hou, Peng Zhao, Engui Fan, and Zhijun Qiao.
\newblock Algebro-geometric solutions for the {D}egasperis-{P}rocesi hierarchy.
\newblock {\em SIAM J. Math. Anal.}, 45(3):1216--1266, 2013.

\bibitem{Jorba}
\`Angel Jorba and Carles Sim\'o.
\newblock On the reducibility of linear differential equations with
  quasiperiodic coefficients.
\newblock {\em J. Differential Equations}, 98(1):111--124, 1992.

\bibitem{Ku}
S.~B. Kuksin.
\newblock {\em Nearly integrable infinite-dimensional {H}amiltonian systems},
  volume 1556 of {\em Lecture Notes in Mathematics}.
\newblock Springer-Verlag, Berlin, 1993.

\bibitem{Ku2}
S.~B. Kuksin.
\newblock A {KAM} theorem for equations of the {K}orteweg-de {V}ries type.
\newblock {\em Rev. Math. Phys.}, 10(3):1--64, 1998.

\bibitem{KP}
S.~B. Kuksin and J.~P{\"o}schel.
\newblock Invariant {C}antor manifolds of quasi-periodic oscillations for a
  nonlinear {S}chr\"odinger equation.
\newblock {\em Ann. of Math.}, 143(1):149--179, 1996.

\bibitem{LY}
J.~Liu and X.~Yuan.
\newblock A {KAM} {T}heorem for {H}amiltonian partial differential equations
  with unbounded perturbations.
\newblock {\em Comm. Math. Phys,}, 307:629--673, 2011.

\bibitem{M18}
A.~Maspero.
\newblock Lower bounds on the growth of sobolev norms in some linear time
  dependent schrödinger equations.
\newblock preprint 2018, arXiv:1801.06813.

\bibitem{Meti}
Guy M\'etivier.
\newblock {\em Para-differential calculus and applications to the {C}auchy
  problem for nonlinear systems}, volume~5 of {\em Centro di Ricerca Matematica
  Ennio De Giorgi (CRM) Series}.
\newblock Edizioni della Normale, Pisa, 2008.

\bibitem{Mon2}
R.~Montalto.
\newblock A {R}educibility result for a class of {L}inear {W}ave {E}quations on
  {${\mathbb T}^d$}.
\newblock {\em International Mathematics Research Notices}, page rnx167, 2017.

\bibitem{M2}
Ju. Moser.
\newblock A rapidly convergent iteration method and nonlinear differential
  equations.
\newblock {\em Uspehi Mat. Nauk}, 23(4 (142)):179--238, 1968.

\bibitem{P1}
J.~P{\"o}schel.
\newblock Quasi-periodic solutions for a nonlinear wave equation.
\newblock {\em Comment. Math. Helv.}, 71(2):269--296, 1996.

\bibitem{PP}
C.~Procesi and M.~Procesi.
\newblock A normal form for the {S}chr\"odinger equation with analytic
  non-linearities.
\newblock {\em Communications in Mathematical Physics}, 312(2):501--557, 2012.

\bibitem{W}
C.~Eugene Wayne.
\newblock Periodic and quasi-periodic solutions of nonlinear wave equations via
  {KAM} theory.
\newblock {\em Comm. Math. Phys.}, 127(3):479--528, 1990.

\end{thebibliography}

\def\cprime{$'$}

\end{document}